\documentclass[11pt]{amsart}

\usepackage{amsmath,amsfonts,amscd,amssymb,verbatim,float}

\makeatletter

\theoremstyle{plain}
\@addtoreset{equation}{section}
\def\theequation{\thesection.\arabic{equation}}
\newtheorem{theorem}[equation]{Theorem}
\newtheorem{thm}[equation]{Theorem}

\newtheorem{proposition}[equation]{Proposition}

\newtheorem{lemma}[equation]{Lemma}

\newtheorem{cor}[equation]{Corollary}
\newtheorem{corollary}[equation]{Corollary}
\theoremstyle{definition}

\newtheorem{remark}[equation]{Remark}

\newtheorem{definition}[equation]{Definition}

\newtheorem{notation}[equation]{Notation}
\newtheorem{example}[equation]{Example}

\def\namedthm#1{\let\oldtheeq\theequation\def\theequation{#1}\begin{thm}}
\def\endnamedthm{\end{thm}\let\theequation\oldtheeq}

\def\smallmatrix#1#2#3#4{\genfrac{(}{.}{0pt}{1}{#1}{#3}\genfrac{.}{)}{0pt}{1}{#2}{#4}}

\def\choice#1#2#3#4{{\def\arraystretch{0.7}
\Bigl\{\!\!\begin{array}{ll}
   \scriptstyle #1\!\!\!&\scriptstyle #2\cr
   \scriptstyle #3\!\!\!&\scriptstyle #4\end{array}}\!\!\Bigr\}\,}

\def\leftchoice#1#2#3#4{{\def\arraystretch{0.7}
\Bigl\{\!\!\begin{array}{ll}
   \scriptstyle #1\!\!\!&\scriptstyle #2\cr
   \scriptstyle #3\!\!\!&\scriptstyle #4\end{array}}}

\def\bigleftchoice#1#2#3#4{{\def\arraystretch{1}
  \biggl\{\!\!\begin{array}{ll}#1&#2\cr#3&#4\end{array}}}

\let\iff\Leftrightarrow

\def\injects{\lhook\joinrel\rightarrow}
\def\longinjects{\lhook\joinrel\lar}

\let\lar\longrightarrow
\let\iso\cong
\let\tensor\otimes

\def\beq{$$\begin{array}{llllllllllllllll}}
\def\eeq{\end{array}$$}
\def\beqn{\begin{equation}\begin{array}{llllllllllllllll}}
\def\eeqn{\end{array}\end{equation}}
\def\beql#1{\begin{equation}\label{#1}}
\def\eeql{\end{equation}}

\def\tbuildrel#1\over#2{\buildrel\text{\rm\normalsize\smaller[3]#1}\over{#2}}

\def\thincdots{\raise1.3pt\hbox{$\scriptscriptstyle
  \>\cdot\>\cdot\>\cdot\>\cdot\hskip0.3pt$}}

\def\II{\text{\rm II}}

\def\IV{\text{\rm IV}}

\def\comment{}
\def\endcomment{}

\def\<{\raise0.5pt\hbox{$\,\scriptstyle<\,$}}

\def\arrowlim#1#2{\mathop{\underset{\scriptstyle #1}{\underset
    {\raisebox{0ex}[0.25ex][-0.5ex]{$#2$}}{\operatorname{lim}}}}}
\newcommand{\invlim}[1][]{\arrowlim{#1}{\longleftarrow}}        
\def\isoarrow{\xrightarrow{\raise-6pt\hbox{\smash{$\scriptscriptstyle\iso\,$}}}}

\def\bb@symb#1|#2{\expandafter\def\csname #2#1\endcsname{\mathbb{#1}}}
\def\bbsymbols#1#2{\@for\@tmpz:=#2\do{\expandafter\bb@symb\@tmpz|{#1}}}

\def\cal@symb#1|#2{\expandafter\def\csname #2#1\endcsname{\mathcal{#1}}}
\def\calsymbols#1#2{\@for\@tmpz:=#2\do{\expandafter\cal@symb\@tmpz|{#1}}}

\def\dmth@p#1|{\expandafter\let\csname#1\endcsname\relax
  \expandafter\DeclareMathOperator\csname#1\endcsname{#1}}
\def\operators#1{\@for\@tmpz:=#1\do{\expandafter\dmth@p\@tmpz|}}

\operators{Tr,rk,ord,Gal,BSD,Hom,tors,coker,Ind,Res,Frob,GL,PGL,SL,Aut,Mack,Funct,per}
\operators{Sel,End,Maps,Reg,Sy,Syl,sgn,id,Jac,Sp,lift}

\DeclareMathOperator\vchar{char}

\bbsymbols{}{R,C,Z,N,Q,F}
\calsymbols{}{O}
\calsymbols{c}{C}

\def\triv{\boldsymbol{1}}

\def\mapiso{\>{\buildrel\raise-2pt\hbox{$\scriptscriptstyle\iso$}\over\lar}\,}
\def\mm#1#2#3{m^{\scriptscriptstyle #1}_{#2,#3}}

\def\centralC2{C_2^z}

\long\def\introthm#1#2#3\endthm{%
  \expandafter\gdef\csname #2\endcsname
  {\begin{theorem}[=Theorem \ref{i#1}]\label{#1}#3\end{theorem}}%
  \begin{theorem}[=Theorem \ref{#1}]\label{i#1}#3\end{theorem}}

\def\Tcropt#1{\rule{0pt}{#1ex}\cr}          
\def\Bcropt#1{\rule[-#1ex]{0pt}{0pt}\cr}    

\def\Bcr{\Bcropt{1}}

\def\optional#1{{\color{cyan}#1}}

\usepackage{cleveref,tikz-cd,pbox,wrapfig}
\usepackage{enumitem}

                             \long\def\comment#1\endcomment{}

\operators{Stab,Sel,Gal,Id,red,const,genus,parent,div,Spec,Div,Pic,order,Type,tame,wild,Irr,denom,nr,sep,rk,im,val}

\calsymbols{c}{C,K,O,A,R,P,Z,D,S,N}
\let\O\cO
\bbsymbols{}{A,P,G}

\def\kbar{\bar k}
\def\m{\mathfrak{m}}

\def\Ks{K^{\sep}}                     
\def\Kbar{\bar K}
\def\Knr{K^{\nr}}
\def\Fnr{F^{\nr}}
\def\OKnr{\cO_{K^{\nr}}}
\def\OFnr{\cO_{F^{\nr}}}

\def\II{{\mathcal I}}            
\def\JJ{{\mathcal J}}
\def\KK{{\mathcal K}}

\def\Cmp{\Gamma}                 
\def\Cmpn{\tilde\Gamma}          

\def\GKK{\Gal(\bar K/K)}

\usepackage{graphicx}

\def\floor*#1{\lfloor #1\rfloor}
\def\ceil*#1{\lceil #1\rceil}
\usepackage[OT2,OT1]{fontenc}
\@addtoreset{equation}{section}
\def\theequation{\thesection.\arabic{equation}}

\def\ZZ{\mathbb{Z}}

\def\cO{\mathcal{O}}

\def\b{\beta}

\def\c{{\mathfrak s}}
\def\v{v}
\def\Pa{P}
\def\ub{\"ubereven}
\def\neck#1{#1^*}

\def\r{{\mathfrak r}}

\def\overarrow#1{\>\>{\buildrel{#1}\over\lar}\>\>}
\def\overlarrow#1{\>\>{\buildrel{#1}\over{\hbox to 3em{\rightarrowfill}}}\>\>}

\def\overhline#1{\>\>{\buildrel{#1}\over{\raise0.33em\hbox to 3em{\hrulefill}}}\>\>}

\let\s\c
\def\so{\tilde\s}
\operators{lca,rhs,pr}
\def\I{{\rm \hskip 1pt I}}

\def\ss{{\rm ss}}

\def\Cbar{{\bar C}}
\def\Cbarn{\tilde{\bar C}}

\def\H{H^1_{\text{\'et}}}

\def\t{{\mathfrak t}}

\DeclareFontFamily{U}{mathb}{\hyphenchar\font45}
\DeclareFontShape{U}{mathb}{m}{n}{
      <5> <6> <7> <8> <9> <10> gen * mathb
      <10.95> mathb10 <12> <14.4> <17.28> <20.74> <24.88> mathb12
      }{}
\DeclareSymbolFont{mathb}{U}{mathb}{m}{n}
\DeclareMathSymbol{\lefttorightarrow}{3}{mathb}{"FC}
\def\acts{\mathop{\raise 0.13ex\hbox{\protect\scalebox{0.8}[0.9]{$\lefttorightarrow$}}}}

\crefname{equation}{}{}

\usepackage{pbox}
\usepackage{tkz-graph}
\usetikzlibrary{snakes,fit,positioning,calc}

\definecolor{amethyst}{rgb}{0.6, 0.4, 0.8}
\definecolor{atomictangerine}{rgb}{1.0, 0.6, 0.4}
\definecolor{deeppeach}{rgb}{1.0, 0.8, 0.64}
\definecolor{eggshell}{rgb}{0.94, 0.92, 0.84}
\definecolor{lightapricot}{rgb}{0.99, 0.84, 0.69}
\definecolor{lemonchiffon}{rgb}{1.0, 0.98, 0.8}
\definecolor{roundabout}{rgb}{1.0, 0.91, 0.75}
\definecolor{atomictangerine}{rgb}{1.0, 0.6, 0.4}

\def\rootsep{0.03}               
\def\clustersep{0.06}            
\def\cnamescale{0.4}             
\def\cdepthscale{0.4}            
\def\cltopskip{1pt}              
\def\clbottomskip{1pt}           

\def\rootscale{0.5}   \def\rootcolor{amethyst}
\def\rootscaleA{0.7}  \def\rootcolorA{yellow}
\def\rootscaleB{0.5}  \def\rootcolorB{green}
\def\rootscaleC{0.5}  \def\rootcolorC{red}
\def\rootscaleD{0.5}  \def\rootcolorD{blue}
\tikzset{
  clA/.style = {very thick,black},
  clB/.style = {thick,purple}
}

\def\graphdslabelscale{0.6}

\def\GraphScale{0.6}
\def\SnakeWiggle{4pt}
\def\SnakeAmplitude{0.8pt}

\tikzset{
  root/.style = {circle,scale=\rootscale,ball color=\rootcolor},
    rc/.style 2 args = {right=#1*1.5*\clustersep of {#2.east|-first},root}, rr/.style = {right=\rootsep of {#1.east|-first},root},
  roott/.style = {circle,inner sep=-2pt,minimum size=5pt,black,font=\ttfamily\footnotesize},
    rct/.style 2 args = {right=#1*1.5*\clustersep of {#2.east|-first},roott}, rrt/.style = {right=\rootsep of {#1.east|-first},roott},
  rootA/.style = {circle,scale=\rootscaleA,ball color=\rootcolorA},
    rcA/.style 2 args = {right=#1*1.5*\clustersep of {#2.east|-first},rootA}, rrA/.style = {right=\rootsep of {#1.east|-first},rootA},
  rootB/.style = {circle,scale=\rootscaleB,ball color=\rootcolorB},
    rcB/.style 2 args = {right=#1*1.5*\clustersep of {#2.east|-first},rootB}, rrB/.style = {right=\rootsep of {#1.east|-first},rootB},
  rootC/.style = {circle,scale=\rootscaleC,ball color=\rootcolorC},
    rcC/.style 2 args = {right=#1*1.5*\clustersep of {#2.east|-first},rootC}, rrC/.style = {right=\rootsep of {#1.east|-first},rootC},
  rootD/.style = {circle,scale=\rootscaleD,ball color=\rootcolorD},
    rcD/.style 2 args = {right=#1*1.5*\clustersep of {#2.east|-first},rootD}, rrD/.style = {right=\rootsep of {#1.east|-first},rootD},
  cluster/.style = {draw=blue!70,thick,rounded corners,inner sep=22*\clustersep,outer xsep=22*\clustersep,fit=#1},
  clabel/.style  = {anchor=west,scale=\cdepthscale,black,inner sep=0,outer xsep=1,outer ysep=0},
  clabelL/.style = {above right=-\clustersep of #1t.north east,clabel},
  clabelD/.style = {below right=-\clustersep of #1t.south east,clabel},
  clouter/.style = {inner sep=0,outer sep=0,fit=#1}
}

\def\Cluster #1 = #2;{\node[cluster=#2] (#1) {};}
\def\ClusterL #1[#2] = #3;{
  \node[cluster=#3] (#1t) {}; \node[clabelL=#1] (#1l) {$#2$}; \node[clouter=(#1t)(#1l)] (#1) {};}
\def\ClusterD #1[#2] = #3;{
  \node[cluster=#3] (#1t) {}; \node[clabelD=#1] (#1d) {$#2$}; \node[clouter=(#1t)(#1d)] (#1) {};}
\def\ClusterLD #1[#2][#3] = #4;{
  \node[cluster=#4] (#1t) {}; \node[clabelL=#1] (#1l) {$#2$}; 
  \node[clabelD=#1] (#1d) {$#3$}; \node[clouter=(#1t)(#1l)(#1d)] (#1) {};}
\def\ClusterLDName #1[#2][#3][#4] = #5;{
  \node[cluster=#5] (#1t) {}; \node[clabelL=#1] (#1l) {$#2$}; 
  \node[clabelD=#1] (#1d) {$#3$}; 
  \node[scale=\cnamescale,above=\clustersep/3 of #1t,inner sep=0, outer sep=0] (#1n) {$#4$}; 
  \node[clouter=(#1l)(#1d)(#1t)] (#1) {};}

\newcommand{\Root}[4][]{
  \ifx\relax#2\relax\node[rr#1=#3] (#4) {};\else\node[rc#1={#2}{#3}] (#4) {};\fi}
\newcommand{\RootT}[5][]{
  \ifx\relax#2\relax\node[rrt#1=#3] (#4) {#5};\else\node[rct#1={#2}{#3}] (#4) {#5};\fi}

\def\frob(#1)(#2){\path[draw,thick,shorten <=-22*\clustersep,shorten >=-22*\clustersep](#1.east)--(#2.west|-#1){};}

\long\def\clusterpicture#1\endclusterpicture{\pb{\vbox to \cltopskip{\vfill}\\%
  \begin{tikzpicture}\node[coordinate] (first) {};#1\end{tikzpicture}\\[-11pt]\vbox to \clbottomskip{\vfill}}}   
\long\def\clusterpictureopt#1#2\endclusterpicture{\pb{\vbox to \cltopskip{\vfill}\\%
  \begin{tikzpicture}[#1]\node[coordinate] (first) {};#2\end{tikzpicture}\\[-11pt]\vbox to \clbottomskip{\vfill}}}   

\def\VSwap#1#2#3#4{\path[draw](#1) edge[<->,#3,shorten >=#4pt,shorten <=#4pt] (#2){};}
\def\VArr#1#2#3#4{\path[draw](#1) edge[->,#3,shorten >=#4pt,shorten <=#4pt] (#2){};}

\def\ESwapOfs#1#2#3#4#5#6#7#8{\VSwap{$(#1)!0.5!(#2) + (#6)$}{$(#3)!0.5!(#4) + (#7)$}{#5}{#8}}

\def\EArrOfs#1#2#3#4#5#6#7#8{\VArr{$(#1)!0.5!(#2) + (#6)$}{$(#3)!0.5!(#4) + (#7)$}{#5}{#8}}

\def\EdgeSign(#1)(#2)#3(#4)#5{
  \node at ($(#1)!#3!(#2) + (#4)$) [color=black, scale=\graphdslabelscale] {$\scriptstyle #5$};
}

\def\TreeEdgeSignDist{0.33}
\def\GraphEdgeSignDist{0.55}

\def\TreeEdgeSignN(#1)(#2)#3#4{\EdgeSign(#1)(#2)#3(0,\TreeEdgeSignDist){#4}}
\def\TreeEdgeSignS(#1)(#2)#3#4{\EdgeSign(#1)(#2)#3(0,-\TreeEdgeSignDist){#4}}
\def\TreeEdgeSignW(#1)(#2)#3#4{\EdgeSign(#1)(#2)#3(-\TreeEdgeSignDist,0){#4}}
\def\TreeEdgeSignE(#1)(#2)#3#4{\EdgeSign(#1)(#2)#3(\TreeEdgeSignDist,0){#4}}
\def\GraphEdgeSignN(#1)(#2)#3#4{\EdgeSign(#1)(#2)#3(0,\GraphEdgeSignDist){#4}}
\def\GraphEdgeSignS(#1)(#2)#3#4{\EdgeSign(#1)(#2)#3(0,-\GraphEdgeSignDist){#4}}
\def\GraphEdgeSignW(#1)(#2)#3#4{\EdgeSign(#1)(#2)#3(-\GraphEdgeSignDist,0){#4}}
\def\GraphEdgeSignE(#1)(#2)#3#4{\EdgeSign(#1)(#2)#3(\GraphEdgeSignDist,0){#4}}
\def\TreeSignAt(#1)(#2)#3{\EdgeSign(#1)(#1){0.5}(#2){#3}}

\def\ecluster(#1)(#2)#3{
  \path(#1) coordinate (#2a){}; 
  \path($(#1)+(0.3,0)$) coordinate (#2c){}; 
  \path($(#1)+(#3/2,0)-(0,#3)$) coordinate (#2t); 
  \path[draw,-,thick] (#2a)--($(#2a)-(0,0.15)$);
  \path[draw,-,thick] (#2a) edge[out=90,in=90] (#2c);
  \path[draw,-,thick] (#2c) edge[out=270,in=180] (#2t);
  \expandafter\edef\csname lastnode#2\endcsname{#2t}
  \expandafter\let\csname cltype#2\endcsname\epsilon
}

\def\ocluster(#1)(#2)#3{
  \path(#1) coordinate (#2a); 
  \path($(#1)-(0,0.1)$) coordinate (#2b); 
  \path($(#1)+(#3/2,0)-(0,#3)$) coordinate (#2t); 
  \path[draw,-,thick] ($(#2a)+(0,0.15)$)--(#2b);
  \path[draw,-,thick] (#2b) edge[out=270,in=180] (#2t);
  \expandafter\edef\csname lastnode#2\endcsname{#2t}
  \expandafter\let\csname cltype#2\endcsname\omega
}

\def\tocluster(#1)(#2)#3{
  \path(#1) coordinate (#2a); 
  \path(#1) coordinate (#2b); 
  \path(#1) coordinate (#2t); 
  \expandafter\edef\csname lastnode#2\endcsname{#2t}
  \expandafter\let\csname cltype#2\endcsname\omega
}

\def\tecluster(#1)(#2)#3{
  \path(#1) coordinate (#2a); 
  \path(#1) coordinate (#2b); 
  \path(#1) coordinate (#2t); 
  \expandafter\edef\csname lastnode#2\endcsname{#2t}
  \expandafter\let\csname cltype#2\endcsname\epsilon
}

\def\tucluster(#1)(#2)#3{
  \path(#1) coordinate (#2a); 
  \path(#1) coordinate (#2c); 
  \path(#1) coordinate (#2t); 
  \path($(#2c)-(0,0.1)$) coordinate (#2b); 
  \path(#2b) coordinate (#2d); 
  \expandafter\edef\csname lastnode#2\endcsname{#2t}
  \expandafter\let\csname cltype#2\endcsname\upsilon
}

\def\uclustergen(#1)(#2)#3#4{
  \path($(#1)+(0.1,0)$) coordinate (#2c){}; 
  \path($(#2c)-(0,0.1)$) coordinate (#2d){}; 
  \path($(#2c)+(#3/2,0)-(0,#3)$) coordinate (#2t); 
  \path[draw,-,thick] ($(#2c)+(0,0.15)$)--(#2d);
  \path[draw,-,thick] (#2d) edge[out=270,in=180] (#2t);
  \path($(#1)$) coordinate (#2a){}; 
  \path($(#2a)-(0,0.1)$) coordinate (#2b){}; 
  \path($(#2d)+(#3/2,0)-(0,#3)$) coordinate (#2u); 
  \path[draw,-,thick#4] ($(#2a)+(0,0.15)$)--(#2b);          
  \path[draw,-,thick#4] (#2u) edge[out=180,in=270] (#2b);   
  \expandafter\edef\csname lastnode#2\endcsname{#2t}
  \expandafter\let\csname cltype#2\endcsname\upsilon
}

\def\ucluster(#1)(#2)#3{\uclustergen(#1)(#2){#3}{}}
\def\uclusterbr(#1)(#2)#3{\uclustergen(#1)(#2){#3}{,shorten >=1.1pt}}

\def\clusterNWlabel(#1)#2{
  \path(#1a) node[color=blue,anchor=east] {$#2$}; 
}
\def\clusterSWlabel(#1)#2{
  \path(#1t) node[color=blue,anchor=east] {$#2$}; 
}

\def\subroot #1-#2 #3{
  \edef\lastnode{\csname lastnode#1\endcsname}
  \path($(\lastnode)+(#3,0)$) node[root](#2){};
  \path[draw,-,thick](\lastnode)--(#2){};
  \expandafter\def\csname lastnode#1\endcsname{#2}
}

\def\subrootname #1-#2 #3#4{
  \subroot #1-#2 {#3}
  \node (#2n) [color=purple!50!black,below=of #2.south,scale=0.7,yshift=0pt,font=\ttfamily, anchor=north] {$#4$};  
}

\def\endcluster #1 #2{
  \edef\lastnode{\csname lastnode#1\endcsname}
  \path($(\lastnode)+(#2,0)$) node[coordinate](#1end){};
  \path[draw,-,thick](\lastnode)--(#1end){};
  \expandafter\ifx\csname cltype#1\endcsname\upsilon
    \path[draw,-,thick]($(\lastnode)-(0,0.1)$)--($(#1end)-(0,0.1)$){};  
  \else\fi
  \expandafter\def\csname lastnode#1\endcsname{#1end}
}

\def\subcluster #1 #2-#3 (#4,#5){
  \edef\lastnode{\csname lastnode#2\endcsname}
  \expandafter\ifx\csname cltype#2\endcsname\upsilon
    \expandafter\ifx #1u
      \uclusterbr($(\lastnode)+(#4,0)$)(#3){#5}               
    \else
      \csname#1cluster\endcsname($(\lastnode)+(#4,0)$)(#3){#5}
    \fi
  \else
    \csname#1cluster\endcsname($(\lastnode)+(#4,0)$)(#3){#5}
  \fi
  \path[draw,-,thick](\lastnode)--(#3a){};
  \expandafter\ifx\csname cltype#2\endcsname\upsilon
    \path[draw,-,thick]($(\lastnode)-(0,0.1)$)--($(#3a)-(0,0.1)$){};  
    \expandafter\ifx\csname cltype#3\endcsname\upsilon      
      \path($(#3c)+(1.1pt,0)$) node[coordinate] (#3e) {};
      \path($(#3d)+(1.1pt,0)$) node[coordinate] (#3f) {};
      \path[draw,-,thick] (#3a)--($(#3e)-(2.2pt,0)$) {};  
      \path[draw,-,thick] (#3b)--(#3f) {};  
      \expandafter\def\csname lastnode#2\endcsname{#3e}
    \else
      \expandafter\def\csname lastnode#2\endcsname{#3a}
    \fi 
  \else
    \expandafter\def\csname lastnode#2\endcsname{#3a}
  \fi
}

\def\pb#1{\pbox[c]{\textwidth}{\hfil #1\hfil}}

\def\GraphVertices{\SetVertexNormal[Shape=circle, FillColor=blue!50, LineColor=blue!50, LineWidth=0.8pt]
  \tikzset{VertexStyle/.append style = {inner sep=0.5pt,minimum size=0.3em,font = \tiny\bfseries}}}
\let\BlueVertices\GraphVertices

\def\YellowEdges{\SetUpEdge[lw=1pt,color=orange]
   \tikzset{EdgeStyle/.append style = {snake=snake,segment length=\SnakeWiggle,
    segment amplitude=\SnakeAmplitude,shorten <=1pt,shorten >=1pt}}}
\def\BlueEdges{  \SetUpEdge[lw=0.8pt,color=blue!70]
   \tikzset{EdgeStyle/.append style = {shorten <=0.5pt,shorten >=0.5pt}}}
\def\BendedEdges{  \SetUpEdge[lw=0.8pt,color=blue!70]
   \tikzset{EdgeStyle/.append style = {shorten <=0.5pt,shorten >=0.5pt,bend left=30}}}
\def\LoopW(#1){
  \path[draw,-,thick,color=blue!70] (#1) edge[out=155,in=90] ($(#1)-(1.3,0)$);
  \path[draw,-,thick,color=blue!70] (#1) edge[out=210,in=270] ($(#1)-(1.3,0)$);
}
\def\LoopE(#1){
  \path[draw,-,thick,color=blue!70] (#1) edge[out=25,in=90] ($(#1)+(1.3,0)$);
  \path[draw,-,thick,color=blue!70] (#1) edge[out=-25,in=270] ($(#1)+(1.3,0)$);
}
\def\LoopS(#1){
  \path[draw,-,thick,color=blue!70] (#1) edge[out=115,in=180] ($(#1)+(0,1.2)$);
  \path[draw,-,thick,color=blue!70] (#1) edge[out=65,in=0] ($(#1)+(0,1.2)$);
}
\def\LoopN(#1){
  \path[draw,-,thick,color=blue!70] (#1) edge[out=-115,in=180] ($(#1)-(0,1.2)$);
  \path[draw,-,thick,color=blue!70] (#1) edge[out=-65,in=0] ($(#1)-(0,1.2)$);
}
\def\EdgeW(#1){
  \path[draw,-,thick,color=blue!70] (#1+) edge[out=180,in=90] ($(#1+)-(1.3,0.3)$);
  \path[draw,-,thick,color=blue!70] (#1-) edge[out=180,in=270] ($(#1+)-(1.3,0.3)$);
}
\def\EdgeE(#1){
  \path[draw,-,thick,color=blue!70] (#1+) edge[out=0,in=90] ($(#1-)+(1.3,0.3)$);
  \path[draw,-,thick,color=blue!70] (#1-) edge[out=0,in=270] ($(#1-)+(1.3,0.3)$);
}
\def\EdgeS(#1){
  \path[draw,-,thick,color=blue!70] (#1+) edge[out=90,in=0] ($(#1-)+(0.3,1.3)$);
  \path[draw,-,thick,color=blue!70] (#1-) edge[out=90,in=180] ($(#1-)+(0.3,1.3)$);
}
\def\EdgeN(#1){
  \path[draw,-,thick,color=blue!70] (#1+) edge[out=270,in=0] ($(#1+)-(0.3,1.3)$);
  \path[draw,-,thick,color=blue!70] (#1-) edge[out=270,in=180] ($(#1+)-(0.3,1.3)$);
}
\def\GCircle(#1,#2)(#3,#4){
  \path(#1,#2) node[coordinate] (1) {};
  \path(#3,#4) node[coordinate] (2) {};
  \path[draw,-,thick,color=blue!70] (1) edge[out=90,in=90] (2);
  \path[draw,-,thick,color=blue!70] (2) edge[out=270,in=270] (1);
}

\def\vertlabelscale{0.6}
\def\vertlabeldist{0.3}
\def\VertexLabelAt#1(#2)#3{
  \node at ($(#1)+(#2)$) [scale=\vertlabelscale,color=black] {#3};
}
\def\VertexLN[#1]#2#3{\Vertex[#1]{#2}\VertexLabelAt#2(0,\vertlabeldist){#3}}
\def\VertexLS[#1]#2#3{\Vertex[#1]{#2}\VertexLabelAt#2(0,-\vertlabeldist){#3}}
\def\VertexLW[#1]#2#3{\Vertex[#1]{#2}\VertexLabelAt#2(-\vertlabeldist,0){#3}}
\def\VertexLE[#1]#2#3{\Vertex[#1]{#2}\VertexLabelAt#2(\vertlabeldist,0){#3}}
\def\VertexLSE[#1]#2#3{\Vertex[#1]{#2}\VertexLabelAt#2(\vertlabeldist,-\vertlabeldist){#3}}
\def\VertexLNW[#1]#2#3{\Vertex[#1]{#2}\VertexLabelAt#2(-\vertlabeldist/1.5,\vertlabeldist/1.8){#3}}
\def\VertexLNE[#1]#2#3{\Vertex[#1]{#2}\VertexLabelAt#2(\vertlabeldist/1.5,\vertlabeldist/1.8){#3}}

\newcommand{\mysetminusD}{\hbox{\tikz{\draw[line width=0.6pt,line cap=round] (3pt,0) -- (0,6pt);}}}
\newcommand{\mysetminusT}{\mysetminusD}
\newcommand{\mysetminusS}{\hbox{\tikz{\draw[line width=0.45pt,line cap=round] (2pt,0) -- (0,4pt);}}}
\newcommand{\mysetminusSS}{\hbox{\tikz{\draw[line width=0.4pt,line cap=round] (1.5pt,0) -- (0,3pt);}}}

\let\setminus\mysetminus

\def\FrobL{\raise 0.3pt\hbox{\scalebox{0.3}{$\,\sim\,$}}}
\def\FrobDot{\hbox{$\hspace{0.6pt}\raise4pt\hbox{\rlap{\scalebox{0.26}{$\kern 0.14em\sim$}}}\cdot\hspace{0.6pt}$}}
\def\FrobCol{\hbox{$\hspace{0.6pt}\raise5pt\hbox{\rlap{\scalebox{0.26}{$\kern 0.14em\sim$}}}\hbox{:}\hspace{0.6pt}$}}
\def\FrobO{\hbox{$\hspace{0.6pt}\raise5pt\hbox{\rlap{\scalebox{0.4}{$\>\sim$}}}\circ\hspace{0.6pt}$}}
\def\FrobX{\hbox{$\hspace{0.6pt}\raise5pt\hbox{\rlap{\scalebox{0.4}{$\kern 0.6em\sim$}}}\times\hspace{0.6pt}$}}

\def\FrobDot{\raise1pt\hbox{\scalebox{0.7}{$\sim$}}}
\let\FrobCol\FrobDot

\makeatletter
\def\hetype#1{\hbox{$\tn@main#1\tn@end\tn@end$}}
\def\tn@und#1#2{{\tn@dound #2\tn@end\tn@end}\tn@main}
\def\tn@tri#1#2#3{{\smash{\raise2pt\hbox{$\scriptscriptstyle #3$}}}\tn@main}
\def\tn@main#1{%
   \ifx#1\tn@end\relax\else   
   \ifx#1__\expandafter\tn@und\else
   \ifx#1^ ^\expandafter\tn@tri\else     
   \ifx#1.\hbox{$\cdot$}\else
   \ifx#1:\hbox{:}\else
   \ifx#1o\hbox{$\hspace{0.6pt}\circ\hspace{0.6pt}$}\else
   \ifx#1x\hbox{$\hspace{0.6pt}\times\hspace{0.6pt}$}\else   
   \ifx#1I\hbox{I}\else   
   \ifx#1U\hbox{U}\else   
   #1\fi\fi\fi\fi\fi\fi
   \expandafter\expandafter\expandafter\expandafter
   \expandafter\expandafter\expandafter\tn@main\fi\fi\fi}
\def\tn@dound#1{%
  \ifx#1\tn@end\relax\else
  \ifx#1~\hbox{$\FrobL$}\else #1\fi\expandafter\tn@dound\fi}
\let\tn@end\relax

\def\clgBal{\clusterpicture            
  \Root {1} {first} {r1};
  \Root {} {r1} {r2};
  \Root {} {r2} {r3};
  \Root {} {r3} {r4};
  \Root {} {r4} {r5};
  \Root {} {r5} {r6};
  \ClusterD c1[0] = (r1)(r2)(r3)(r4)(r5)(r6);
\endclusterpicture}

\def\clgAc{\clusterpicture            
  \Root {1} {first} {r1};
  \Root {} {r1} {r2};
  \Root {} {r2} {r3};
  \Root {} {r3} {r4};
  \Root {} {r4} {r5};
  \Cluster c1 = (r1)(r2)(r3)(r4)(r5),clA;
\endclusterpicture}

\def\clgBalc{\clusterpicture            
  \Root {1} {first} {r1};
  \Root {} {r1} {r2};
  \Root {} {r2} {r3};
  \Root {} {r3} {r4};
  \Root {} {r4} {r5};
  \Root {} {r5} {r6};
  \Cluster c1 = (r1)(r2)(r3)(r4)(r5)(r6),clA;
\endclusterpicture}

\def\clgCc{\clusterpicture            
  \Root {1} {first} {r1};
  \Root {1} {r1} {r2};
  \Root {} {r2} {r3};
  \Root {} {r3} {r4};
  \Root {} {r4} {r5};
  \Root {} {r5} {r6};
  \Cluster c1 = (r2)(r3)(r4)(r5)(r6),clA;
  \Cluster c2 = (r1)(c1);
\endclusterpicture}

\def\clnBals{\clusterpicture            
  \Root {1} {first} {r1};
  \Root {} {r1} {r2};
  \Root {} {r2} {r3};
  \Root {} {r3} {r4};
  \Root {1} {r4} {r5};
  \Root {} {r5} {r6};
  \ClusterLD c1[{+}][{\nn}] = (r5)(r6);
  \ClusterD c2[0] = (r1)(r2)(r3)(r4)(c1);
\endclusterpicture}

\def\clnBalns{\clusterpicture            
  \Root {1} {first} {r1};
  \Root {} {r1} {r2};
  \Root {} {r2} {r3};
  \Root {} {r3} {r4};
  \Root {1} {r4} {r5};
  \Root {} {r5} {r6};
  \ClusterLD c1[{-}][{\nn}] = (r5)(r6);
  \ClusterD c2[0] = (r1)(r2)(r3)(r4)(c1);
\endclusterpicture}

\def\clnAce{\clusterpicture            
  \Root {1} {first} {r1};
  \Root {} {r1} {r2};
  \Root {} {r2} {r3};
  \Root {1} {r3} {r4};
  \Root {} {r4} {r5};
  \ClusterLD c1[{\ee}][{\nn}] = (r4)(r5);
  \Cluster c2 = (r1)(r2)(r3)(c1),clA;
\endclusterpicture}

\def\clnBalce{\clusterpicture            
  \Root {1} {first} {r1};
  \Root {} {r1} {r2};
  \Root {} {r2} {r3};
  \Root {} {r3} {r4};
  \Root {1} {r4} {r5};
  \Root {} {r5} {r6};
  \ClusterLD c1[{\ee}][{\nn}] = (r5)(r6);
  \Cluster c2 = (r1)(r2)(r3)(r4)(c1),clA;
\endclusterpicture}

\def\clnCce{\clusterpicture            
  \Root {1} {first} {r1};
  \Root {} {r1} {r2};
  \Root {1} {r2} {r3};
  \Root {} {r3} {r4};
  \Root {} {r4} {r5};
  \Root {} {r5} {r6};
  \ClusterLD c1[{\ee}][{\nn}] = (r3)(r4)(r5)(r6),clA;
  \Cluster c2 = (r1)(r2)(c1);
\endclusterpicture}

\def\clnDce{\clusterpicture            
  \Root {2} {first} {r1};
  \Root {} {r1} {r2};
  \ClusterD c1[{\nn\!-\!t}] = (r1)(r2);
  \Root {1} {c1} {r3};
  \Root {} {r3} {r4};
  \Root {} {r4} {r5};
  \Root {} {r5} {r6};
  \ClusterD c2[t] = (r3)(r4)(r5)(r6),clA;
  \ClusterL c3[{\ee}] = (c1)(c2);
\endclusterpicture}

\def\clnEce{\clusterpicture            
  \Root {1} {first} {r1};
  \Root {1} {r1} {r2};
  \Root {} {r2} {r3};
  \Root {} {r3} {r4};
  \Root {1} {r4} {r5};
  \Root {} {r5} {r6};
  \ClusterLD c1[{\ee}][{\nn}] = (r5)(r6);
  \Cluster c2 = (r2)(r3)(r4)(c1),clA;
  \Cluster c3 = (r1)(c2);
\endclusterpicture}

\def\clnFce{\clusterpicture            
  \Root {1} {first} {r1};
  \Root {1} {r1} {r2};
  \Root {} {r2} {r3};
  \Root {} {r3} {r4};
  \Root {} {r4} {r5};
  \ClusterLD c1[{\ee}][{\nn}] = (r2)(r3)(r4)(r5),clA;
  \Cluster c2 = (r1)(c1);
\endclusterpicture}

\def\clnGce{\clusterpicture            
  \Root {1} {first} {r1};
  \Root {1} {r1} {r2};
  \Root {1} {r2} {r3};
  \Root {} {r3} {r4};
  \Root {} {r4} {r5};
  \Root {} {r5} {r6};
  \ClusterLD c1[{\ee}][{\nn}] = (r3)(r4)(r5)(r6),clA;
  \Cluster c2 = (r2)(c1);
  \Cluster c3 = (r1)(c2);
\endclusterpicture}

\def\clnmBalss{\clusterpicture            
  \Root {1} {first} {r1};
  \Root {} {r1} {r2};
  \Root {1} {r2} {r3};
  \Root {} {r3} {r4};
  \ClusterLD c1[{+}][{\nn}] = (r3)(r4);
  \Root {1} {c1} {r5};
  \Root {} {r5} {r6};
  \ClusterLD c2[{+}][{\mm}] = (r5)(r6);
  \ClusterD c3[0] = (r1)(r2)(c1)(c2);
\endclusterpicture}

\def\clnmBalnsns{\clusterpicture            
  \Root {1} {first} {r1};
  \Root {} {r1} {r2};
  \Root {1} {r2} {r3};
  \Root {} {r3} {r4};
  \ClusterLD c1[{-}][{\nn}] = (r3)(r4);
  \Root {1} {c1} {r5};
  \Root {} {r5} {r6};
  \ClusterLD c2[{-}][{\mm}] = (r5)(r6);
  \ClusterD c3[0] = (r1)(r2)(c1)(c2);
\endclusterpicture}

\def\clnmBalsns{\clusterpicture            
  \Root {1} {first} {r1};
  \Root {} {r1} {r2};
  \Root {1} {r2} {r3};
  \Root {} {r3} {r4};
  \ClusterLD c1[{+}][{\nn}] = (r3)(r4);
  \Root {1} {c1} {r5};
  \Root {} {r5} {r6};
  \ClusterLD c2[{-}][{\mm}] = (r5)(r6);
  \ClusterD c3[0] = (r1)(r2)(c1)(c2);
\endclusterpicture}

\def\clnmAce{\clusterpicture            
  \Root {1} {first} {r1};
  \Root {1} {r1} {r2};
  \Root {} {r2} {r3};
  \ClusterLD c1[{\ee}][{\nn}] = (r2)(r3);
  \Root {1} {c1} {r4};
  \Root {} {r4} {r5};
  \ClusterLD c2[{\dd}][{\mm}] = (r4)(r5);
  \Cluster c3 = (r1)(c1)(c2),clA;
\endclusterpicture}

\def\clnmBalce{\clusterpicture            
  \Root {1} {first} {r1};
  \Root {} {r1} {r2};
  \Root {1} {r2} {r3};
  \Root {} {r3} {r4};
  \ClusterLD c1[{\ee}][{\nn}] = (r3)(r4);
  \Root {1} {c1} {r5};
  \Root {} {r5} {r6};
  \ClusterLD c2[{\dd}][{\mm}] = (r5)(r6);
  \Cluster c3 = (r1)(r2)(c1)(c2),clA;
\endclusterpicture}

\def\clnmCce{\clusterpicture            
  \Root {1} {first} {r1};
  \Root {} {r1} {r2};
  \Root {1} {r2} {r3};
  \Root {} {r3} {r4};
  \Root {1} {r4} {r5};
  \Root {} {r5} {r6};
  \ClusterLD c1[{\ee}][{\nn}] = (r5)(r6);
  \ClusterLD c2[{\dd}][{\mm}] = (r3)(r4)(c1),clA;
  \Cluster c3 = (r1)(r2)(c2);
\endclusterpicture}

\def\clnmDce{\clusterpicture            
  \Root {2} {first} {r1};
  \Root {} {r1} {r2};
  \ClusterD c1[{\mm\!-\!t}] = (r1)(r2);
  \Root {1} {c1} {r3};
  \Root {} {r3} {r4};
  \Root {1} {r4} {r5};
  \Root {} {r5} {r6};
  \ClusterLD c2[{\ee}][{\nn}] = (r5)(r6);
  \ClusterD c3[t] = (r3)(r4)(c2),clA;
  \ClusterL c4[{\dd}] = (c1)(c3);
\endclusterpicture}

\def\clnmEce{\clusterpicture            
  \Root {1} {first} {r1};
  \Root {1} {r1} {r2};
  \Root {1} {r2} {r3};
  \Root {} {r3} {r4};
  \ClusterLD c1[{\ee}][{\nn}] = (r3)(r4);
  \Root {1} {c1} {r5};
  \Root {} {r5} {r6};
  \ClusterLD c2[{\dd}][{\mm}] = (r5)(r6);
  \Cluster c3 = (r2)(c1)(c2),clA;
  \Cluster c4 = (r1)(c3);
\endclusterpicture}

\def\clnmFce{\clusterpicture            
  \Root {1} {first} {r1};
  \Root {1} {r1} {r2};
  \Root {} {r2} {r3};
  \Root {1} {r3} {r4};
  \Root {} {r4} {r5};
  \ClusterLD c1[{\ee}][{\nn}] = (r4)(r5);
  \ClusterLD c2[{\dd}][{\mm}] = (r2)(r3)(c1),clA;
  \Cluster c3 = (r1)(c2);
\endclusterpicture}

\def\clnmGce{\clusterpicture            
  \Root {1} {first} {r1};
  \Root {1} {r1} {r2};
  \Root {1} {r2} {r3};
  \Root {} {r3} {r4};
  \Root {1} {r4} {r5};
  \Root {} {r5} {r6};
  \ClusterLD c1[{\ee}][{\nn}] = (r5)(r6);
  \ClusterLD c2[{\dd}][{\mm}] = (r3)(r4)(c1),clA;
  \Cluster c3 = (r2)(c2);
  \Cluster c4 = (r1)(c3);
\endclusterpicture}

\def\clnnBals{\clusterpicture            
  \Root {1} {first} {r1};
  \Root {} {r1} {r2};
  \Root {1} {r2} {r3};
  \Root {} {r3} {r4};
  \ClusterL c1[{+}] = (r3)(r4);
  \Root {1} {c1} {r5};
  \Root {} {r5} {r6};
  \ClusterLD c2[{+}][{\nn}] = (r5)(r6);
  \ClusterD c3[0] = (r1)(r2)(c1)(c2);
  \frob(c1t)(c2);
\endclusterpicture}

\def\clnnBalns{\clusterpicture            
  \Root {1} {first} {r1};
  \Root {} {r1} {r2};
  \Root {1} {r2} {r3};
  \Root {} {r3} {r4};
  \ClusterL c1[{-}] = (r3)(r4);
  \Root {1} {c1} {r5};
  \Root {} {r5} {r6};
  \ClusterLD c2[{+}][{\nn}] = (r5)(r6);
  \ClusterD c3[0] = (r1)(r2)(c1)(c2);
  \frob(c1t)(c2);
\endclusterpicture}

\def\clnnAce{\clusterpicture            
  \Root {1} {first} {r1};
  \Root {1} {r1} {r2};
  \Root {} {r2} {r3};
  \ClusterL c1[{\eta}] = (r2)(r3);
  \Root {1} {c1} {r4};
  \Root {} {r4} {r5};
  \ClusterLD c2[{\ee\eta}][{\nn}] = (r4)(r5);
  \Cluster c3 = (r1)(c1)(c2),clA;
  \frob(c1t)(c2);
\endclusterpicture}

\def\clnnBalce{\clusterpicture            
  \Root {1} {first} {r1};
  \Root {} {r1} {r2};
  \Root {1} {r2} {r3};
  \Root {} {r3} {r4};
  \ClusterL c1[{\eta}] = (r3)(r4);
  \Root {1} {c1} {r5};
  \Root {} {r5} {r6};
  \ClusterLD c2[{\ee\eta}][{\nn}] = (r5)(r6);
  \Cluster c3 = (r1)(r2)(c1)(c2),clA;
  \frob(c1t)(c2);
\endclusterpicture}

\def\clnnCce{\clusterpicture            
  \Root {1} {first} {r1};
  \Root {1} {r1} {r2};
  \Root {1} {r2} {r3};
  \Root {} {r3} {r4};
  \ClusterL c1[{\eta}] = (r3)(r4);
  \Root {1} {c1} {r5};
  \Root {} {r5} {r6};
  \ClusterLD c2[{\ee\eta}][{\nn}] = (r5)(r6);
  \Cluster c3 = (r2)(c1)(c2),clA;
  \Cluster c4 = (r1)(c3);
  \frob(c1t)(c2);
\endclusterpicture}

\def\clUnmkBals{\clusterpicture            
  \Root {2} {first} {r1};
  \Root {} {r1} {r2};
  \ClusterD c1[{\nn}] = (r1)(r2);
  \Root {1} {c1} {r3};
  \Root {} {r3} {r4};
  \ClusterD c2[{\mm}] = (r3)(r4);
  \Root {1} {c2} {r5};
  \Root {} {r5} {r6};
  \ClusterD c3[{\kk}] = (r5)(r6);
  \ClusterLD c4[{+}][0] = (c1)(c2)(c3);
\endclusterpicture}

\def\clUnmkBalns{\clusterpicture            
  \Root {2} {first} {r1};
  \Root {} {r1} {r2};
  \ClusterD c1[{\nn}] = (r1)(r2);
  \Root {1} {c1} {r3};
  \Root {} {r3} {r4};
  \ClusterD c2[{\mm}] = (r3)(r4);
  \Root {1} {c2} {r5};
  \Root {} {r5} {r6};
  \ClusterD c3[{\kk}] = (r5)(r6);
  \ClusterLD c4[{-}][0] = (c1)(c2)(c3);
\endclusterpicture}

\def\clUnmkAce{\clusterpicture            
  \Root {1} {first} {r1};
  \Root {2} {r1} {r2};
  \Root {} {r2} {r3};
  \ClusterD c1[{\nn}] = (r2)(r3);
  \Root {1} {c1} {r4};
  \Root {} {r4} {r5};
  \ClusterD c2[{\mm}] = (r4)(r5);
  \ClusterLD c3[{\ee}][{\kk}] = (c1)(c2),clA;
  \Cluster c4 = (r1)(c3);
\endclusterpicture}

\def\clUnmkBalce{\clusterpicture            
  \Root {2} {first} {r1};
  \Root {} {r1} {r2};
  \ClusterD c1[{\nn}] = (r1)(r2);
  \Root {1} {c1} {r3};
  \Root {} {r3} {r4};
  \ClusterD c2[{\mm}] = (r3)(r4);
  \Root {1} {c2} {r5};
  \Root {} {r5} {r6};
  \ClusterD c3[{\kk}] = (r5)(r6);
  \ClusterL c4[{\ee}] = (c1)(c2)(c3),clA;
\endclusterpicture}

\def\clUnmkCce{\clusterpicture            
  \Root {1} {first} {r1};
  \Root {} {r1} {r2};
  \Root {2} {r2} {r3};
  \Root {} {r3} {r4};
  \ClusterD c1[{\nn}] = (r3)(r4);
  \Root {1} {c1} {r5};
  \Root {} {r5} {r6};
  \ClusterD c2[{\mm}] = (r5)(r6);
  \ClusterLD c3[{\ee}][{\kk}] = (c1)(c2),clA;
  \Cluster c4 = (r1)(r2)(c3);
\endclusterpicture}

\def\clUnmkDce{\clusterpicture            
  \Root {2} {first} {r1};
  \Root {} {r1} {r2};
  \ClusterD c1[{\kk\!-\!t}] = (r1)(r2);
  \Root {2} {c1} {r3};
  \Root {} {r3} {r4};
  \ClusterD c2[{\nn}] = (r3)(r4);
  \Root {1} {c2} {r5};
  \Root {} {r5} {r6};
  \ClusterD c3[{\mm}] = (r5)(r6);
  \ClusterD c4[t] = (c2)(c3),clA;
  \ClusterL c5[{\ee}] = (c1)(c4);
\endclusterpicture}

\def\clUnmkEce{\clusterpicture            
  \Root {1} {first} {r1};
  \Root {1} {r1} {r2};
  \Root {2} {r2} {r3};
  \Root {} {r3} {r4};
  \ClusterD c1[{\nn}] = (r3)(r4);
  \Root {1} {c1} {r5};
  \Root {} {r5} {r6};
  \ClusterD c2[{\mm}] = (r5)(r6);
  \ClusterLD c3[{\ee}][{\kk}] = (c1)(c2),clA;
  \Cluster c4 = (r2)(c3);
  \Cluster c5 = (r1)(c4);
\endclusterpicture}

\def\clUnnkBals{\clusterpicture            
  \Root {2} {first} {r1};
  \Root {} {r1} {r2};
  \Cluster c1 = (r1)(r2);
  \Root {1} {c1} {r3};
  \Root {} {r3} {r4};
  \ClusterD c2[{\nn}] = (r3)(r4);
  \Root {1} {c2} {r5};
  \Root {} {r5} {r6};
  \ClusterD c3[{\kk}] = (r5)(r6);
  \ClusterLD c4[{+}][0] = (c1)(c2)(c3);
  \frob(c1)(c2);
\endclusterpicture}

\def\clUnnkBalns{\clusterpicture            
  \Root {2} {first} {r1};
  \Root {} {r1} {r2};
  \Cluster c1 = (r1)(r2);
  \Root {1} {c1} {r3};
  \Root {} {r3} {r4};
  \ClusterD c2[{\nn}] = (r3)(r4);
  \Root {1} {c2} {r5};
  \Root {} {r5} {r6};
  \ClusterD c3[{\kk}] = (r5)(r6);
  \ClusterLD c4[{-}][0] = (c1)(c2)(c3);
  \frob(c1)(c2);
\endclusterpicture}

\def\clUnnkAce{\clusterpicture            
  \Root {1} {first} {r1};
  \Root {2} {r1} {r2};
  \Root {} {r2} {r3};
  \Cluster c1 = (r2)(r3);
  \Root {1} {c1} {r4};
  \Root {} {r4} {r5};
  \ClusterD c2[{\nn}] = (r4)(r5);
  \ClusterLD c3[{\ee}][{\kk}] = (c1)(c2),clA;
  \Cluster c4 = (r1)(c3);
  \frob(c1)(c2);
\endclusterpicture}

\def\clUnnkBalce{\clusterpicture            
  \Root {2} {first} {r1};
  \Root {} {r1} {r2};
  \Cluster c1 = (r1)(r2);
  \Root {1} {c1} {r3};
  \Root {} {r3} {r4};
  \ClusterD c2[{\nn}] = (r3)(r4);
  \Root {1} {c2} {r5};
  \Root {} {r5} {r6};
  \ClusterD c3[{\kk}] = (r5)(r6);
  \ClusterL c4[{\ee}] = (c1)(c2)(c3),clA;
  \frob(c1)(c2);
\endclusterpicture}

\def\clUnnkCce{\clusterpicture            
  \Root {1} {first} {r1};
  \Root {} {r1} {r2};
  \Root {2} {r2} {r3};
  \Root {} {r3} {r4};
  \Cluster c1 = (r3)(r4);
  \Root {1} {c1} {r5};
  \Root {} {r5} {r6};
  \ClusterD c2[{\nn}] = (r5)(r6);
  \ClusterLD c3[{\ee}][{\kk}] = (c1)(c2),clA;
  \Cluster c4 = (r1)(r2)(c3);
  \frob(c1)(c2);
\endclusterpicture}

\def\clUnnkDce{\clusterpicture            
  \Root {2} {first} {r1};
  \Root {} {r1} {r2};
  \ClusterD c1[{\kk\!-\!t}] = (r1)(r2);
  \Root {2} {c1} {r3};
  \Root {} {r3} {r4};
  \Cluster c2 = (r3)(r4);
  \Root {1} {c2} {r5};
  \Root {} {r5} {r6};
  \ClusterD c3[{\nn}] = (r5)(r6);
  \ClusterD c4[t] = (c2)(c3),clA;
  \ClusterL c5[{\ee}] = (c1)(c4);
  \frob(c2)(c3);
\endclusterpicture}

\def\clUnnkEce{\clusterpicture            
  \Root {1} {first} {r1};
  \Root {1} {r1} {r2};
  \Root {2} {r2} {r3};
  \Root {} {r3} {r4};
  \Cluster c1 = (r3)(r4);
  \Root {1} {c1} {r5};
  \Root {} {r5} {r6};
  \ClusterD c2[{\nn}] = (r5)(r6);
  \ClusterLD c3[{\ee}][{\kk}] = (c1)(c2),clA;
  \Cluster c4 = (r2)(c3);
  \Cluster c5 = (r1)(c4);
  \frob(c1)(c2);
\endclusterpicture}

\def\clUnnnBals{\clusterpicture            
  \Root {2} {first} {r1};
  \Root {} {r1} {r2};
  \Cluster c1 = (r1)(r2);
  \Root {1} {c1} {r3};
  \Root {} {r3} {r4};
  \Cluster c2 = (r3)(r4);
  \Root {1} {c2} {r5};
  \Root {} {r5} {r6};
  \ClusterD c3[{\nn}] = (r5)(r6);
  \ClusterLD c4[{+}][0] = (c1)(c2)(c3);
  \frob(c1)(c2);
  \frob(c2)(c3);
\endclusterpicture}

\def\clUnnnBalns{\clusterpicture            
  \Root {2} {first} {r1};
  \Root {} {r1} {r2};
  \Cluster c1 = (r1)(r2);
  \Root {1} {c1} {r3};
  \Root {} {r3} {r4};
  \Cluster c2 = (r3)(r4);
  \Root {1} {c2} {r5};
  \Root {} {r5} {r6};
  \ClusterD c3[{\nn}] = (r5)(r6);
  \ClusterLD c4[{-}][0] = (c1)(c2)(c3);
  \frob(c1)(c2);
  \frob(c2)(c3);
\endclusterpicture}

\def\clUnnnBalce{\clusterpicture            
  \Root {2} {first} {r1};
  \Root {} {r1} {r2};
  \Cluster c1 = (r1)(r2);
  \Root {1} {c1} {r3};
  \Root {} {r3} {r4};
  \Cluster c2 = (r3)(r4);
  \Root {1} {c2} {r5};
  \Root {} {r5} {r6};
  \ClusterD c3[{\nn}] = (r5)(r6);
  \ClusterL c4[{\ee}] = (c1)(c2)(c3),clA;
  \frob(c1)(c2);
  \frob(c2)(c3);
\endclusterpicture}

\def\clInImBalss{\clusterpicture            
  \Root {2} {first} {r1};
  \Root {1} {r1} {r2};
  \Root {} {r2} {r3};
  \ClusterLD c1[{+}][{\nn}] = (r2)(r3);
  \ClusterD c2[r] = (r1)(c1);
  \Root {1} {c2} {r4};
  \Root {1} {r4} {r5};
  \Root {} {r5} {r6};
  \ClusterLD c3[{+}][{\mm}] = (r5)(r6);
  \ClusterD c4[{r}] = (r4)(c3);
  \ClusterD c5[0] = (c2)(c4);
\endclusterpicture}

\def\clInImBalnsns{\clusterpicture            
  \Root {2} {first} {r1};
  \Root {1} {r1} {r2};
  \Root {} {r2} {r3};
  \ClusterLD c1[{-}][{\nn}] = (r2)(r3);
  \ClusterD c2[r] = (r1)(c1);
  \Root {1} {c2} {r4};
  \Root {1} {r4} {r5};
  \Root {} {r5} {r6};
  \ClusterLD c3[{-}][{\mm}] = (r5)(r6);
  \ClusterD c4[{r}] = (r4)(c3);
  \ClusterD c5[0] = (c2)(c4);
\endclusterpicture}

\def\clInImBalsns{\clusterpicture            
  \Root {2} {first} {r1};
  \Root {1} {r1} {r2};
  \Root {} {r2} {r3};
  \ClusterLD c1[{+}][{\nn}] = (r2)(r3);
  \ClusterD c2[r] = (r1)(c1);
  \Root {1} {c2} {r4};
  \Root {1} {r4} {r5};
  \Root {} {r5} {r6};
  \ClusterLD c3[{-}][{\mm}] = (r5)(r6);
  \ClusterD c4[{r}] = (r4)(c3);
  \ClusterD c5[0] = (c2)(c4);
\endclusterpicture}

\def\clInImAce{\clusterpicture            
  \Root {2} {first} {r1};
  \Root {} {r1} {r2};
  \ClusterLD c1[{\ee}][{\nn}] = (r1)(r2);
  \Root {1} {c1} {r3};
  \Root {1} {r3} {r4};
  \Root {} {r4} {r5};
  \ClusterLD c2[{\dd}][{\mm}] = (r4)(r5);
  \ClusterD c3[{2r}] = (r3)(c2);
  \Cluster c4 = (c1)(c3),clA;
\endclusterpicture}

\def\clInImBalce{\clusterpicture            
  \Root {2} {first} {r1};
  \Root {1} {r1} {r2};
  \Root {} {r2} {r3};
  \ClusterLD c1[{\ee}][{\nn}] = (r2)(r3);
  \ClusterD c2[t] = (r1)(c1),clA;
  \Root {1} {c2} {r4};
  \Root {1} {r4} {r5};
  \Root {} {r5} {r6};
  \ClusterLD c3[{\dd}][{\mm}] = (r5)(r6);
  \ClusterD c4[{2r\!-\!t}] = (r4)(c3);
  \Cluster c5 = (c2)(c4);
\endclusterpicture}

\def\clInImCce{\clusterpicture            
  \Root {1} {first} {r1};
  \Root {} {r1} {r2};
  \Root {1} {r2} {r3};
  \Root {1} {r3} {r4};
  \Root {1} {r4} {r5};
  \Root {} {r5} {r6};
  \ClusterLD c1[{\ee}][{\nn}] = (r5)(r6);
  \ClusterD c2[{2r}] = (r4)(c1);
  \ClusterLD c3[{\dd}][{\mm}] = (r3)(c2),clA;
  \Cluster c4 = (r1)(r2)(c3);
\endclusterpicture}

\def\clInImDce{\clusterpicture            
  \Root {2} {first} {r1};
  \Root {} {r1} {r2};
  \ClusterD c1[{\nn\!-\!t}] = (r1)(r2);
  \Root {1} {c1} {r3};
  \Root {1} {r3} {r4};
  \Root {1} {r4} {r5};
  \Root {} {r5} {r6};
  \ClusterLD c2[{\dd}][{\mm}] = (r5)(r6);
  \ClusterD c3[{2r}] = (r4)(c2);
  \ClusterD c4[t] = (r3)(c3),clA;
  \ClusterL c5[{\ee}] = (c1)(c4);
\endclusterpicture}

\def\clInImEce{\clusterpicture            
  \Root {1} {first} {r1};
  \Root {2} {r1} {r2};
  \Root {} {r2} {r3};
  \ClusterLD c1[{\ee}][{\nn}] = (r2)(r3);
  \Root {1} {c1} {r4};
  \Root {1} {r4} {r5};
  \Root {} {r5} {r6};
  \ClusterLD c2[{\dd}][{\mm}] = (r5)(r6);
  \ClusterD c3[{2r}] = (r4)(c2);
  \Cluster c4 = (c1)(c3),clA;
  \Cluster c5 = (r1)(c4);
\endclusterpicture}

\def\clInImFce{\clusterpicture            
  \Root {1} {first} {r1};
  \Root {1} {r1} {r2};
  \Root {1} {r2} {r3};
  \Root {1} {r3} {r4};
  \Root {} {r4} {r5};
  \ClusterLD c1[{\ee}][{\nn}] = (r4)(r5);
  \ClusterD c2[{2r}] = (r3)(c1);
  \ClusterLD c3[{\dd}][{\mm}] = (r2)(c2),clA;
  \Cluster c4 = (r1)(c3);
\endclusterpicture}

\def\clInImGce{\clusterpicture            
  \Root {1} {first} {r1};
  \Root {1} {r1} {r2};
  \Root {} {r2} {r3};
  \ClusterLD c1[{\ee}][{\nn}] = (r2)(r3);
  \Root {1} {c1} {r4};
  \Root {1} {r4} {r5};
  \Root {} {r5} {r6};
  \ClusterLD c2[{\dd}][{\mm}] = (r5)(r6);
  \ClusterD c3[{2r}] = (r4)(c2);
  \Cluster c4 = (r1)(c1)(c3),clA;
\endclusterpicture}

\def\clInImHce{\clusterpicture            
  \Root {1} {first} {r1};
  \Root {1} {r1} {r2};
  \Root {1} {r2} {r3};
  \Root {1} {r3} {r4};
  \Root {1} {r4} {r5};
  \Root {} {r5} {r6};
  \ClusterLD c1[{\ee}][{\nn}] = (r5)(r6);
  \ClusterD c2[{2r}] = (r4)(c1);
  \ClusterLD c3[{\dd}][{\mm}] = (r3)(c2),clA;
  \Cluster c4 = (r2)(c3);
  \Cluster c5 = (r1)(c4);
\endclusterpicture}

\def\clInInBals{\clusterpicture            
  \Root {2} {first} {r1};
  \Root {1} {r1} {r2};
  \Root {} {r2} {r3};
  \ClusterL c1[{+}] = (r2)(r3);
  \Cluster c2 = (r1)(c1);
  \Root {1} {c2} {r4};
  \Root {1} {r4} {r5};
  \Root {} {r5} {r6};
  \ClusterLD c3[{+}][{\nn}] = (r5)(r6);
  \ClusterD c4[r] = (r4)(c3);
  \ClusterD c5[0] = (c2)(c4);
  \frob(c2)(c4);
\endclusterpicture}

\def\clInInBalns{\clusterpicture            
  \Root {2} {first} {r1};
  \Root {1} {r1} {r2};
  \Root {} {r2} {r3};
  \ClusterL c1[{+}] = (r2)(r3);
  \Cluster c2 = (r1)(c1);
  \Root {1} {c2} {r4};
  \Root {1} {r4} {r5};
  \Root {} {r5} {r6};
  \ClusterLD c3[{-}][{\nn}] = (r5)(r6);
  \ClusterD c4[r] = (r4)(c3);
  \ClusterD c5[0] = (c2)(c4);
  \frob(c2)(c4);
\endclusterpicture}

\def\clInInBalce{\clusterpicture            
  \Root {2} {first} {r1};
  \Root {1} {r1} {r2};
  \Root {} {r2} {r3};
  \ClusterL c1[{\eta}] = (r2)(r3);
  \Cluster c2 = (r1)(c1),clA;
  \Root {1} {c2} {r4};
  \Root {1} {r4} {r5};
  \Root {} {r5} {r6};
  \ClusterLD c3[{\ee\eta}][{\nn}] = (r5)(r6);
  \ClusterD c4[r] = (r4)(c3);
  \Cluster c5 = (c2)(c4);
  \frob(c2)(c4);
\endclusterpicture}

\def\cloInBals{\clusterpicture            
  \Root {2} {first} {r1};
  \Root {} {r1} {r2};
  \Root {} {r2} {r3};
  \ClusterD c1[r] = (r1)(r2)(r3);
  \Root {1} {c1} {r4};
  \Root {1} {r4} {r5};
  \Root {} {r5} {r6};
  \ClusterLD c2[{+}][{\nn}] = (r5)(r6);
  \ClusterD c3[{r}] = (r4)(c2);
  \ClusterD c4[0] = (c1)(c3);
\endclusterpicture}

\def\cloInBalns{\clusterpicture            
  \Root {2} {first} {r1};
  \Root {} {r1} {r2};
  \Root {} {r2} {r3};
  \ClusterD c1[r] = (r1)(r2)(r3);
  \Root {1} {c1} {r4};
  \Root {1} {r4} {r5};
  \Root {} {r5} {r6};
  \ClusterLD c2[{-}][{\nn}] = (r5)(r6);
  \ClusterD c3[{r}] = (r4)(c2);
  \ClusterD c4[0] = (c1)(c3);
\endclusterpicture}

\def\cloInAce{\clusterpicture            
  \Root {1} {first} {r1};
  \Root {} {r1} {r2};
  \Root {1} {r2} {r3};
  \Root {1} {r3} {r4};
  \Root {} {r4} {r5};
  \ClusterLD c1[{\ee}][{\nn}] = (r4)(r5);
  \ClusterD c2[{2r}] = (r3)(c1);
  \Cluster c3 = (r1)(r2)(c2),clA;
\endclusterpicture}

\def\cloInBalce{\clusterpicture            
  \Root {2} {first} {r1};
  \Root {} {r1} {r2};
  \Root {} {r2} {r3};
  \ClusterD c1[t] = (r1)(r2)(r3),clA;
  \Root {1} {c1} {r4};
  \Root {1} {r4} {r5};
  \Root {} {r5} {r6};
  \ClusterLD c2[{\ee}][{\nn}] = (r5)(r6);
  \ClusterD c3[{2r\!-\!t}] = (r4)(c2);
  \Cluster c4 = (c1)(c3);
\endclusterpicture}

\def\cloInCce{\clusterpicture            
  \Root {1} {first} {r1};
  \Root {} {r1} {r2};
  \Root {1} {r2} {r3};
  \Root {1} {r3} {r4};
  \Root {} {r4} {r5};
  \Root {} {r5} {r6};
  \ClusterD c1[{2r}] = (r4)(r5)(r6);
  \ClusterLD c2[{\ee}][{\nn}] = (r3)(c1),clA;
  \Cluster c3 = (r1)(r2)(c2);
\endclusterpicture}

\def\cloInDce{\clusterpicture            
  \Root {2} {first} {r1};
  \Root {} {r1} {r2};
  \ClusterD c1[{\nn\!-\!t}] = (r1)(r2);
  \Root {1} {c1} {r3};
  \Root {1} {r3} {r4};
  \Root {} {r4} {r5};
  \Root {} {r5} {r6};
  \ClusterD c2[{2r}] = (r4)(r5)(r6);
  \ClusterD c3[t] = (r3)(c2),clA;
  \ClusterL c4[{\ee}] = (c1)(c3);
\endclusterpicture}

\def\cloInEce{\clusterpicture            
  \Root {1} {first} {r1};
  \Root {1} {r1} {r2};
  \Root {} {r2} {r3};
  \Root {1} {r3} {r4};
  \Root {1} {r4} {r5};
  \Root {} {r5} {r6};
  \ClusterLD c1[{\ee}][{\nn}] = (r5)(r6);
  \ClusterD c2[{2r}] = (r4)(c1);
  \Cluster c3 = (r2)(r3)(c2),clA;
  \Cluster c4 = (r1)(c3);
\endclusterpicture}

\def\cloInFce{\clusterpicture            
  \Root {2} {first} {r1};
  \Root {} {r1} {r2};
  \ClusterLD c1[{\ee}][{\nn}] = (r1)(r2);
  \Root {1} {c1} {r3};
  \Root {} {r3} {r4};
  \Root {} {r4} {r5};
  \ClusterD c2[{2r}] = (r3)(r4)(r5);
  \Cluster c3 = (c1)(c2),clA;
\endclusterpicture}

\def\cloInGce{\clusterpicture            
  \Root {1} {first} {r1};
  \Root {1} {r1} {r2};
  \Root {} {r2} {r3};
  \ClusterLD c1[{\ee}][{\nn}] = (r2)(r3);
  \Root {1} {c1} {r4};
  \Root {} {r4} {r5};
  \Root {} {r5} {r6};
  \ClusterD c2[{2r}] = (r4)(r5)(r6);
  \Cluster c3 = (r1)(c1)(c2),clA;
\endclusterpicture}

\def\cloInHce{\clusterpicture            
  \Root {1} {first} {r1};
  \Root {2} {r1} {r2};
  \Root {} {r2} {r3};
  \ClusterLD c1[{\ee}][{\nn}] = (r2)(r3);
  \Root {1} {c1} {r4};
  \Root {} {r4} {r5};
  \Root {} {r5} {r6};
  \ClusterD c2[{2r}] = (r4)(r5)(r6);
  \Cluster c3 = (c1)(c2),clA;
  \Cluster c4 = (r1)(c3);
\endclusterpicture}

\def\cloInIce{\clusterpicture            
  \Root {1} {first} {r1};
  \Root {1} {r1} {r2};
  \Root {1} {r2} {r3};
  \Root {} {r3} {r4};
  \Root {} {r4} {r5};
  \ClusterD c1[{2r}] = (r3)(r4)(r5);
  \ClusterLD c2[{\ee}][{\nn}] = (r2)(c1),clA;
  \Cluster c3 = (r1)(c2);
\endclusterpicture}

\def\cloInJce{\clusterpicture            
  \Root {1} {first} {r1};
  \Root {} {r1} {r2};
  \Root {} {r2} {r3};
  \Root {1} {r3} {r4};
  \Root {1} {r4} {r5};
  \Root {} {r5} {r6};
  \ClusterLD c1[{\ee}][{\nn}] = (r5)(r6);
  \ClusterD c2[{2r}] = (r4)(c1);
  \Cluster c3 = (r1)(r2)(r3)(c2),clA;
\endclusterpicture}

\def\cloInKce{\clusterpicture            
  \Root {1} {first} {r1};
  \Root {1} {r1} {r2};
  \Root {1} {r2} {r3};
  \Root {1} {r3} {r4};
  \Root {} {r4} {r5};
  \Root {} {r5} {r6};
  \ClusterD c1[{2r}] = (r4)(r5)(r6);
  \ClusterLD c2[{\ee}][{\nn}] = (r3)(c1),clA;
  \Cluster c3 = (r2)(c2);
  \Cluster c4 = (r1)(c3);
\endclusterpicture}

\def\clooBal{\clusterpicture            
  \Root {2} {first} {r1};
  \Root {} {r1} {r2};
  \Root {} {r2} {r3};
  \ClusterD c1[r] = (r1)(r2)(r3);
  \Root {1} {c1} {r4};
  \Root {} {r4} {r5};
  \Root {} {r5} {r6};
  \ClusterD c2[{r}] = (r4)(r5)(r6);
  \ClusterD c3[0] = (c1)(c2);
\endclusterpicture}

\def\clooAc{\clusterpicture            
  \Root {1} {first} {r1};
  \Root {} {r1} {r2};
  \Root {1} {r2} {r3};
  \Root {} {r3} {r4};
  \Root {} {r4} {r5};
  \ClusterD c1[{2r}] = (r3)(r4)(r5);
  \Cluster c2 = (r1)(r2)(c1),clA;
\endclusterpicture}

\def\clooBalc{\clusterpicture            
  \Root {2} {first} {r1};
  \Root {} {r1} {r2};
  \Root {} {r2} {r3};
  \ClusterD c1[t] = (r1)(r2)(r3),clA;
  \Root {1} {c1} {r4};
  \Root {} {r4} {r5};
  \Root {} {r5} {r6};
  \ClusterD c2[{2r\!-\!t}] = (r4)(r5)(r6);
  \Cluster c3 = (c1)(c2);
\endclusterpicture}

\def\clooCc{\clusterpicture            
  \Root {1} {first} {r1};
  \Root {1} {r1} {r2};
  \Root {} {r2} {r3};
  \Root {1} {r3} {r4};
  \Root {} {r4} {r5};
  \Root {} {r5} {r6};
  \ClusterD c1[{2r}] = (r4)(r5)(r6);
  \Cluster c2 = (r2)(r3)(c1),clA;
  \Cluster c3 = (r1)(c2);
\endclusterpicture}

\def\clooDc{\clusterpicture            
  \Root {1} {first} {r1};
  \Root {} {r1} {r2};
  \Root {} {r2} {r3};
  \Root {1} {r3} {r4};
  \Root {} {r4} {r5};
  \Root {} {r5} {r6};
  \ClusterD c1[{2r}] = (r4)(r5)(r6);
  \Cluster c2 = (r1)(r2)(r3)(c1),clA;
\endclusterpicture}

\def\cloof{\clusterpicture            
  \Root {2} {first} {r1};
  \Root {} {r1} {r2};
  \Root {} {r2} {r3};
  \Cluster c1 = (r1)(r2)(r3);
  \Root {1} {c1} {r4};
  \Root {} {r4} {r5};
  \Root {} {r5} {r6};
  \ClusterD c2[{r}] = (r4)(r5)(r6);
  \ClusterD c3[0] = (c1)(c2);
  \frob(c1)(c2);
\endclusterpicture}

\def\cloofc{\clusterpicture            
  \Root {2} {first} {r1};
  \Root {} {r1} {r2};
  \Root {} {r2} {r3};
  \Cluster c1 = (r1)(r2)(r3),clA;
  \Root {1} {c1} {r4};
  \Root {} {r4} {r5};
  \Root {} {r5} {r6};
  \ClusterD c2[{r}] = (r4)(r5)(r6);
  \Cluster c3 = (c1)(c2);
  \frob(c1)(c2);
\endclusterpicture}

\def\triple{\clusterpicture            
  \Root {1} {first} {r1};
  \Root {} {r1} {r2};
  \Root {} {r2} {r3};
  \ClusterD c1[{\s}] = (r1)(r2)(r3);
\endclusterpicture}

\def\childparent{\clusterpicture            
  \Root {1} {first} {r4};
  \Root {} {r4} {r5};
  \Root {} {r5} {r6};
  \Root {2} {r6} {r1};
  \Root {} {r1} {r2};
  \Root {} {r2} {r3};
  \ClusterD c1[{\s^{\prime}}] = (r1)(r2)(r3);
  \ClusterD c2[{\s}] = (r4)(r5)(r6)(c1);
\endclusterpicture}

\def\childparentn{\clusterpicture            
  \Root {1} {first} {r4};
  \Root {} {r4} {r5};
  \Root {} {r5} {r6};
  \Root {2} {r6} {r1};
  \Root {} {r1} {r2};
  \Root {} {r2} {r3};
  \ClusterD c1[{\s}] = (r1)(r2)(r3);
  \ClusterD c2[{P(\s)}] = (r4)(r5)(r6)(c1);
\endclusterpicture}

\def\event{\clusterpicture            
  \Root {1} {first} {r1};
  \Root {} {r1} {r2};
  \Root {} {r2} {r3};
  \Root {} {r3} {r4};
  \ClusterD c1[{\s}] = (r1)(r2)(r3)(r4);
\endclusterpicture}

\def\oddt{\clusterpicture            
  \Root {1} {first} {r1};
  \Root {} {r1} {r2};
  \Root {} {r2} {r3};
  \Root {} {r3} {r4};
  \Root {} {r4} {r5};
  \ClusterD c1[{\s}] = (r1)(r2)(r3)(r4)(r5);
\endclusterpicture}

\def\ubereven{\clusterpicture            
  \Root {2} {first} {r1};
  \Root {} {r1} {r2};
  \Cluster c1 = (r1)(r2);
  \Root {1} {c1} {r3};
  \Root {} {r3} {r4};
  \Cluster c2 = (r3)(r4);
  \Root {1} {c2} {r5};
  \Root {} {r5} {r6};
  \Cluster c3 = (r5)(r6);
  \ClusterD c4[{\s}] = (c1)(c2)(c3);
\endclusterpicture}

\def\twin{\clusterpicture            
  \Root {1} {first} {r1};
  \Root {} {r1} {r2};
  \ClusterD c1[{\s}] = (r1)(r2);
\endclusterpicture}

\def\cotwino{\clusterpicture            
  \Root {1} {first} {r5};
  \Root {2} {r5} {r1};
  \Root {} {r1} {r2};
  \Root {} {r2} {r3};
  \Root {} {r3} {r4};
  \Cluster c1 = (r1)(r2)(r3)(r4);
  \ClusterD c2[{\s}] = (r5)(c1);
\endclusterpicture}

\def\cotwint{\clusterpicture            
  \Root {1} {first} {r5};
  \Root {} {r5} {r6};
  \Root {2} {r6} {r1};
  \Root {} {r1} {r2};
  \Root {} {r2} {r3};
  \Root {} {r3} {r4};
  \Cluster c1 = (r1)(r2)(r3)(r4);
  \ClusterD c2[{\s}] = (r5)(r6)(c1);
\endclusterpicture}

\def\clusterlcas{\clusterpicture            
  \Root {2} {first} {r1};
  \Root {} {r1} {r2};
  \ClusterD c1[{\s^{\prime}}] = (r1)(r2);
  \Root {1} {c1} {r5};
  \Root {2} {r5} {r3};
  \Root {} {r3} {r4};
  \ClusterD c2[{\s}] = (r3)(r4);
  \Cluster c3 = (r5)(c2);
  \ClusterD c4[{\s^{\prime}\wedge\s}] = (c1)(c3);
\endclusterpicture}

\def\clusternecks{\clusterpicture            
  \Root {1} {first} {r5};
  \Root {3} {r5} {r1};
  \Root {} {r1} {r2};
  \ClusterD c1[{\s}] = (r1)(r2);
  \Root {1} {c1} {r3};
  \Root {} {r3} {r4};
  \Cluster c2 = (r3)(r4);
  \ClusterD c3[{\s^*}] = (c1)(c2);
  \Cluster c4 = (r5)(c3);
\endclusterpicture}

\def\RebclnCce{\clusterpicture            
  \Root {1} {first} {r1};
  \Root {} {r1} {r2};
  \Root {1} {r2} {r3};
  \Root {} {r3} {r4};
  \Root {} {r4} {r5};
  \Root {} {r5} {r6};
  \ClusterD c1[{n}] = (r3)(r4)(r5)(r6);
  \Cluster c2 = (r1)(r2)(c1);
\endclusterpicture}

\def\RebclnDce{\clusterpicture            
  \Root {2} {first} {r1};
  \Root {} {r1} {r2};
  \ClusterD c1[{n\!-\!t}] = (r1)(r2);
  \Root {1} {c1} {r3};
  \Root {} {r3} {r4};
  \Root {} {r4} {r5};
  \Root {} {r5} {r6};
  \ClusterD c2[t] = (r3)(r4)(r5)(r6);
  \Cluster c3 = (c1)(c2);
\endclusterpicture}

\def\RebclnFce{\clusterpicture            
  \Root {1} {first} {r1};
  \Root {1} {r1} {r2};
  \Root {} {r2} {r3};
  \Root {} {r3} {r4};
  \Root {} {r4} {r5};
  \ClusterD c1[{n}] = (r2)(r3)(r4)(r5);
  \Cluster c2 = (r1)(c1);
\endclusterpicture}

\def\RebclnGce{\clusterpicture            
  \Root {1} {first} {r1};
  \Root {1} {r1} {r2};
  \Root {1} {r2} {r3};
  \Root {} {r3} {r4};
  \Root {} {r4} {r5};
  \Root {} {r5} {r6};
  \ClusterD c1[{n}] = (r3)(r4)(r5)(r6);
  \Cluster c2 = (r2)(c1);
  \Cluster c3 = (r1)(c2);
\endclusterpicture}

\def\clgBal{\clusterpicture            
  \Root {1} {first} {r1};
  \Root {} {r1} {r2};
  \Root {} {r2} {r3};
  \Root {} {r3} {r4};
  \Root {} {r4} {r5};
  \Root {} {r5} {r6};
  \ClusterD c1[0] = (r1)(r2)(r3)(r4)(r5)(r6);
\endclusterpicture}

\def\clgAc{\clusterpicture            
  \Root {1} {first} {r1};
  \Root {} {r1} {r2};
  \Root {} {r2} {r3};
  \Root {} {r3} {r4};
  \Root {} {r4} {r5};
  \Cluster c1 = (r1)(r2)(r3)(r4)(r5),clA;
\endclusterpicture}

\def\clgBalc{\clusterpicture            
  \Root {1} {first} {r1};
  \Root {} {r1} {r2};
  \Root {} {r2} {r3};
  \Root {} {r3} {r4};
  \Root {} {r4} {r5};
  \Root {} {r5} {r6};
  \Cluster c1 = (r1)(r2)(r3)(r4)(r5)(r6),clA;
\endclusterpicture}

\def\clgCc{\clusterpicture            
  \Root {1} {first} {r1};
  \Root {1} {r1} {r2};
  \Root {} {r2} {r3};
  \Root {} {r3} {r4};
  \Root {} {r4} {r5};
  \Root {} {r5} {r6};
  \Cluster c1 = (r2)(r3)(r4)(r5)(r6),clA;
  \Cluster c2 = (r1)(c1);
\endclusterpicture}

\def\clnBals{\clusterpicture            
  \Root {1} {first} {r1};
  \Root {} {r1} {r2};
  \Root {} {r2} {r3};
  \Root {} {r3} {r4};
  \Root {1} {r4} {r5};
  \Root {} {r5} {r6};
  \ClusterLD c1[{+}][{\nn}] = (r5)(r6);
  \ClusterD c2[0] = (r1)(r2)(r3)(r4)(c1);
\endclusterpicture}

\def\clnBalns{\clusterpicture            
  \Root {1} {first} {r1};
  \Root {} {r1} {r2};
  \Root {} {r2} {r3};
  \Root {} {r3} {r4};
  \Root {1} {r4} {r5};
  \Root {} {r5} {r6};
  \ClusterLD c1[{-}][{\nn}] = (r5)(r6);
  \ClusterD c2[0] = (r1)(r2)(r3)(r4)(c1);
\endclusterpicture}

\def\clnAce{\clusterpicture            
  \Root {1} {first} {r1};
  \Root {} {r1} {r2};
  \Root {} {r2} {r3};
  \Root {1} {r3} {r4};
  \Root {} {r4} {r5};
  \ClusterLD c1[{\ee}][{\nn}] = (r4)(r5);
  \Cluster c2 = (r1)(r2)(r3)(c1),clA;
\endclusterpicture}

\def\clnBalce{\clusterpicture            
  \Root {1} {first} {r1};
  \Root {} {r1} {r2};
  \Root {} {r2} {r3};
  \Root {} {r3} {r4};
  \Root {1} {r4} {r5};
  \Root {} {r5} {r6};
  \ClusterLD c1[{\ee}][{\nn}] = (r5)(r6);
  \Cluster c2 = (r1)(r2)(r3)(r4)(c1),clA;
\endclusterpicture}

\def\clnCce{\clusterpicture            
  \Root {1} {first} {r1};
  \Root {} {r1} {r2};
  \Root {1} {r2} {r3};
  \Root {} {r3} {r4};
  \Root {} {r4} {r5};
  \Root {} {r5} {r6};
  \ClusterLD c1[{\ee}][{\nn}] = (r3)(r4)(r5)(r6),clA;
  \Cluster c2 = (r1)(r2)(c1);
\endclusterpicture}

\def\clnDce{\clusterpicture            
  \Root {2} {first} {r1};
  \Root {} {r1} {r2};
  \ClusterD c1[{\nn\!-\!t}] = (r1)(r2);
  \Root {1} {c1} {r3};
  \Root {} {r3} {r4};
  \Root {} {r4} {r5};
  \Root {} {r5} {r6};
  \ClusterD c2[t] = (r3)(r4)(r5)(r6),clA;
  \ClusterL c3[{\ee}] = (c1)(c2);
\endclusterpicture}

\def\clnEce{\clusterpicture            
  \Root {1} {first} {r1};
  \Root {1} {r1} {r2};
  \Root {} {r2} {r3};
  \Root {} {r3} {r4};
  \Root {1} {r4} {r5};
  \Root {} {r5} {r6};
  \ClusterLD c1[{\ee}][{\nn}] = (r5)(r6);
  \Cluster c2 = (r2)(r3)(r4)(c1),clA;
  \Cluster c3 = (r1)(c2);
\endclusterpicture}

\def\clnFce{\clusterpicture            
  \Root {1} {first} {r1};
  \Root {1} {r1} {r2};
  \Root {} {r2} {r3};
  \Root {} {r3} {r4};
  \Root {} {r4} {r5};
  \ClusterLD c1[{\ee}][{\nn}] = (r2)(r3)(r4)(r5),clA;
  \Cluster c2 = (r1)(c1);
\endclusterpicture}

\def\clnGce{\clusterpicture            
  \Root {1} {first} {r1};
  \Root {1} {r1} {r2};
  \Root {1} {r2} {r3};
  \Root {} {r3} {r4};
  \Root {} {r4} {r5};
  \Root {} {r5} {r6};
  \ClusterLD c1[{\ee}][{\nn}] = (r3)(r4)(r5)(r6),clA;
  \Cluster c2 = (r2)(c1);
  \Cluster c3 = (r1)(c2);
\endclusterpicture}

\def\clnmBalss{\clusterpicture            
  \Root {1} {first} {r1};
  \Root {} {r1} {r2};
  \Root {1} {r2} {r3};
  \Root {} {r3} {r4};
  \ClusterLD c1[{+}][{\nn}] = (r3)(r4);
  \Root {1} {c1} {r5};
  \Root {} {r5} {r6};
  \ClusterLD c2[{+}][{\mm}] = (r5)(r6);
  \ClusterD c3[0] = (r1)(r2)(c1)(c2);
\endclusterpicture}

\def\clnmBalnsns{\clusterpicture            
  \Root {1} {first} {r1};
  \Root {} {r1} {r2};
  \Root {1} {r2} {r3};
  \Root {} {r3} {r4};
  \ClusterLD c1[{-}][{\nn}] = (r3)(r4);
  \Root {1} {c1} {r5};
  \Root {} {r5} {r6};
  \ClusterLD c2[{-}][{\mm}] = (r5)(r6);
  \ClusterD c3[0] = (r1)(r2)(c1)(c2);
\endclusterpicture}

\def\clnmBalsns{\clusterpicture            
  \Root {1} {first} {r1};
  \Root {} {r1} {r2};
  \Root {1} {r2} {r3};
  \Root {} {r3} {r4};
  \ClusterLD c1[{+}][{\nn}] = (r3)(r4);
  \Root {1} {c1} {r5};
  \Root {} {r5} {r6};
  \ClusterLD c2[{-}][{\mm}] = (r5)(r6);
  \ClusterD c3[0] = (r1)(r2)(c1)(c2);
\endclusterpicture}

\def\clnmAce{\clusterpicture            
  \Root {1} {first} {r1};
  \Root {1} {r1} {r2};
  \Root {} {r2} {r3};
  \ClusterLD c1[{\ee}][{\nn}] = (r2)(r3);
  \Root {1} {c1} {r4};
  \Root {} {r4} {r5};
  \ClusterLD c2[{\dd}][{\mm}] = (r4)(r5);
  \Cluster c3 = (r1)(c1)(c2),clA;
\endclusterpicture}

\def\clnmBalce{\clusterpicture            
  \Root {1} {first} {r1};
  \Root {} {r1} {r2};
  \Root {1} {r2} {r3};
  \Root {} {r3} {r4};
  \ClusterLD c1[{\ee}][{\nn}] = (r3)(r4);
  \Root {1} {c1} {r5};
  \Root {} {r5} {r6};
  \ClusterLD c2[{\dd}][{\mm}] = (r5)(r6);
  \Cluster c3 = (r1)(r2)(c1)(c2),clA;
\endclusterpicture}

\def\clnmCce{\clusterpicture            
  \Root {1} {first} {r1};
  \Root {} {r1} {r2};
  \Root {1} {r2} {r3};
  \Root {} {r3} {r4};
  \Root {1} {r4} {r5};
  \Root {} {r5} {r6};
  \ClusterLD c1[{\ee}][{\nn}] = (r5)(r6);
  \ClusterLD c2[{\dd}][{\mm}] = (r3)(r4)(c1),clA;
  \Cluster c3 = (r1)(r2)(c2);
\endclusterpicture}

\def\clnmDce{\clusterpicture            
  \Root {2} {first} {r1};
  \Root {} {r1} {r2};
  \ClusterD c1[{\mm\!-\!t}] = (r1)(r2);
  \Root {1} {c1} {r3};
  \Root {} {r3} {r4};
  \Root {1} {r4} {r5};
  \Root {} {r5} {r6};
  \ClusterLD c2[{\ee}][{\nn}] = (r5)(r6);
  \ClusterD c3[t] = (r3)(r4)(c2),clA;
  \ClusterL c4[{\dd}] = (c1)(c3);
\endclusterpicture}

\def\clnmEce{\clusterpicture            
  \Root {1} {first} {r1};
  \Root {1} {r1} {r2};
  \Root {1} {r2} {r3};
  \Root {} {r3} {r4};
  \ClusterLD c1[{\ee}][{\nn}] = (r3)(r4);
  \Root {1} {c1} {r5};
  \Root {} {r5} {r6};
  \ClusterLD c2[{\dd}][{\mm}] = (r5)(r6);
  \Cluster c3 = (r2)(c1)(c2),clA;
  \Cluster c4 = (r1)(c3);
\endclusterpicture}

\def\clnmFce{\clusterpicture            
  \Root {1} {first} {r1};
  \Root {1} {r1} {r2};
  \Root {} {r2} {r3};
  \Root {1} {r3} {r4};
  \Root {} {r4} {r5};
  \ClusterLD c1[{\ee}][{\nn}] = (r4)(r5);
  \ClusterLD c2[{\dd}][{\mm}] = (r2)(r3)(c1),clA;
  \Cluster c3 = (r1)(c2);
\endclusterpicture}

\def\clnmGce{\clusterpicture            
  \Root {1} {first} {r1};
  \Root {1} {r1} {r2};
  \Root {1} {r2} {r3};
  \Root {} {r3} {r4};
  \Root {1} {r4} {r5};
  \Root {} {r5} {r6};
  \ClusterLD c1[{\ee}][{\nn}] = (r5)(r6);
  \ClusterLD c2[{\dd}][{\mm}] = (r3)(r4)(c1),clA;
  \Cluster c3 = (r2)(c2);
  \Cluster c4 = (r1)(c3);
\endclusterpicture}

\def\clnnBals{\clusterpicture            
  \Root {1} {first} {r1};
  \Root {} {r1} {r2};
  \Root {1} {r2} {r3};
  \Root {} {r3} {r4};
  \ClusterL c1[{+}] = (r3)(r4);
  \Root {1} {c1} {r5};
  \Root {} {r5} {r6};
  \ClusterLD c2[{+}][{\nn}] = (r5)(r6);
  \ClusterD c3[0] = (r1)(r2)(c1)(c2);
  \frob(c1t)(c2);
\endclusterpicture}

\def\clnnBalns{\clusterpicture            
  \Root {1} {first} {r1};
  \Root {} {r1} {r2};
  \Root {1} {r2} {r3};
  \Root {} {r3} {r4};
  \ClusterL c1[{-}] = (r3)(r4);
  \Root {1} {c1} {r5};
  \Root {} {r5} {r6};
  \ClusterLD c2[{+}][{\nn}] = (r5)(r6);
  \ClusterD c3[0] = (r1)(r2)(c1)(c2);
  \frob(c1t)(c2);
\endclusterpicture}

\def\clnnAce{\clusterpicture            
  \Root {1} {first} {r1};
  \Root {1} {r1} {r2};
  \Root {} {r2} {r3};
  \ClusterL c1[{\eta}] = (r2)(r3);
  \Root {1} {c1} {r4};
  \Root {} {r4} {r5};
  \ClusterLD c2[{\ee\eta}][{\nn}] = (r4)(r5);
  \Cluster c3 = (r1)(c1)(c2),clA;
  \frob(c1t)(c2);
\endclusterpicture}

\def\clnnBalce{\clusterpicture            
  \Root {1} {first} {r1};
  \Root {} {r1} {r2};
  \Root {1} {r2} {r3};
  \Root {} {r3} {r4};
  \ClusterL c1[{\eta}] = (r3)(r4);
  \Root {1} {c1} {r5};
  \Root {} {r5} {r6};
  \ClusterLD c2[{\ee\eta}][{\nn}] = (r5)(r6);
  \Cluster c3 = (r1)(r2)(c1)(c2),clA;
  \frob(c1t)(c2);
\endclusterpicture}

\def\clnnCce{\clusterpicture            
  \Root {1} {first} {r1};
  \Root {1} {r1} {r2};
  \Root {1} {r2} {r3};
  \Root {} {r3} {r4};
  \ClusterL c1[{\eta}] = (r3)(r4);
  \Root {1} {c1} {r5};
  \Root {} {r5} {r6};
  \ClusterLD c2[{\ee\eta}][{\nn}] = (r5)(r6);
  \Cluster c3 = (r2)(c1)(c2),clA;
  \Cluster c4 = (r1)(c3);
  \frob(c1t)(c2);
\endclusterpicture}

\def\clUnmkBals{\clusterpicture            
  \Root {2} {first} {r1};
  \Root {} {r1} {r2};
  \ClusterD c1[{\nn}] = (r1)(r2);
  \Root {1} {c1} {r3};
  \Root {} {r3} {r4};
  \ClusterD c2[{\mm}] = (r3)(r4);
  \Root {1} {c2} {r5};
  \Root {} {r5} {r6};
  \ClusterD c3[{\kk}] = (r5)(r6);
  \ClusterLD c4[{+}][0] = (c1)(c2)(c3);
\endclusterpicture}

\def\clUnmkBalns{\clusterpicture            
  \Root {2} {first} {r1};
  \Root {} {r1} {r2};
  \ClusterD c1[{\nn}] = (r1)(r2);
  \Root {1} {c1} {r3};
  \Root {} {r3} {r4};
  \ClusterD c2[{\mm}] = (r3)(r4);
  \Root {1} {c2} {r5};
  \Root {} {r5} {r6};
  \ClusterD c3[{\kk}] = (r5)(r6);
  \ClusterLD c4[{-}][0] = (c1)(c2)(c3);
\endclusterpicture}

\def\clUnmkAce{\clusterpicture            
  \Root {1} {first} {r1};
  \Root {2} {r1} {r2};
  \Root {} {r2} {r3};
  \ClusterD c1[{\nn}] = (r2)(r3);
  \Root {1} {c1} {r4};
  \Root {} {r4} {r5};
  \ClusterD c2[{\mm}] = (r4)(r5);
  \ClusterLD c3[{\ee}][{\kk}] = (c1)(c2),clA;
  \Cluster c4 = (r1)(c3);
\endclusterpicture}

\def\clUnmkBalce{\clusterpicture            
  \Root {2} {first} {r1};
  \Root {} {r1} {r2};
  \ClusterD c1[{\nn}] = (r1)(r2);
  \Root {1} {c1} {r3};
  \Root {} {r3} {r4};
  \ClusterD c2[{\mm}] = (r3)(r4);
  \Root {1} {c2} {r5};
  \Root {} {r5} {r6};
  \ClusterD c3[{\kk}] = (r5)(r6);
  \ClusterL c4[{\ee}] = (c1)(c2)(c3),clA;
\endclusterpicture}

\def\clUnmkCce{\clusterpicture            
  \Root {1} {first} {r1};
  \Root {} {r1} {r2};
  \Root {2} {r2} {r3};
  \Root {} {r3} {r4};
  \ClusterD c1[{\nn}] = (r3)(r4);
  \Root {1} {c1} {r5};
  \Root {} {r5} {r6};
  \ClusterD c2[{\mm}] = (r5)(r6);
  \ClusterLD c3[{\ee}][{\kk}] = (c1)(c2),clA;
  \Cluster c4 = (r1)(r2)(c3);
\endclusterpicture}

\def\clUnmkDce{\clusterpicture            
  \Root {2} {first} {r1};
  \Root {} {r1} {r2};
  \ClusterD c1[{\kk\!-\!t}] = (r1)(r2);
  \Root {2} {c1} {r3};
  \Root {} {r3} {r4};
  \ClusterD c2[{\nn}] = (r3)(r4);
  \Root {1} {c2} {r5};
  \Root {} {r5} {r6};
  \ClusterD c3[{\mm}] = (r5)(r6);
  \ClusterD c4[t] = (c2)(c3),clA;
  \ClusterL c5[{\ee}] = (c1)(c4);
\endclusterpicture}

\def\clUnmkEce{\clusterpicture            
  \Root {1} {first} {r1};
  \Root {1} {r1} {r2};
  \Root {2} {r2} {r3};
  \Root {} {r3} {r4};
  \ClusterD c1[{\nn}] = (r3)(r4);
  \Root {1} {c1} {r5};
  \Root {} {r5} {r6};
  \ClusterD c2[{\mm}] = (r5)(r6);
  \ClusterLD c3[{\ee}][{\kk}] = (c1)(c2),clA;
  \Cluster c4 = (r2)(c3);
  \Cluster c5 = (r1)(c4);
\endclusterpicture}

\def\clUnnkBals{\clusterpicture            
  \Root {2} {first} {r1};
  \Root {} {r1} {r2};
  \Cluster c1 = (r1)(r2);
  \Root {1} {c1} {r3};
  \Root {} {r3} {r4};
  \ClusterD c2[{\nn}] = (r3)(r4);
  \Root {1} {c2} {r5};
  \Root {} {r5} {r6};
  \ClusterD c3[{\kk}] = (r5)(r6);
  \ClusterLD c4[{+}][0] = (c1)(c2)(c3);
  \frob(c1)(c2);
\endclusterpicture}

\def\clUnnkBalns{\clusterpicture            
  \Root {2} {first} {r1};
  \Root {} {r1} {r2};
  \Cluster c1 = (r1)(r2);
  \Root {1} {c1} {r3};
  \Root {} {r3} {r4};
  \ClusterD c2[{\nn}] = (r3)(r4);
  \Root {1} {c2} {r5};
  \Root {} {r5} {r6};
  \ClusterD c3[{\kk}] = (r5)(r6);
  \ClusterLD c4[{-}][0] = (c1)(c2)(c3);
  \frob(c1)(c2);
\endclusterpicture}

\def\clUnnkAce{\clusterpicture            
  \Root {1} {first} {r1};
  \Root {2} {r1} {r2};
  \Root {} {r2} {r3};
  \Cluster c1 = (r2)(r3);
  \Root {1} {c1} {r4};
  \Root {} {r4} {r5};
  \ClusterD c2[{\nn}] = (r4)(r5);
  \ClusterLD c3[{\ee}][{\kk}] = (c1)(c2),clA;
  \Cluster c4 = (r1)(c3);
  \frob(c1)(c2);
\endclusterpicture}

\def\clUnnkBalce{\clusterpicture            
  \Root {2} {first} {r1};
  \Root {} {r1} {r2};
  \Cluster c1 = (r1)(r2);
  \Root {1} {c1} {r3};
  \Root {} {r3} {r4};
  \ClusterD c2[{\nn}] = (r3)(r4);
  \Root {1} {c2} {r5};
  \Root {} {r5} {r6};
  \ClusterD c3[{\kk}] = (r5)(r6);
  \ClusterL c4[{\ee}] = (c1)(c2)(c3),clA;
  \frob(c1)(c2);
\endclusterpicture}

\def\clUnnkCce{\clusterpicture            
  \Root {1} {first} {r1};
  \Root {} {r1} {r2};
  \Root {2} {r2} {r3};
  \Root {} {r3} {r4};
  \Cluster c1 = (r3)(r4);
  \Root {1} {c1} {r5};
  \Root {} {r5} {r6};
  \ClusterD c2[{\nn}] = (r5)(r6);
  \ClusterLD c3[{\ee}][{\kk}] = (c1)(c2),clA;
  \Cluster c4 = (r1)(r2)(c3);
  \frob(c1)(c2);
\endclusterpicture}

\def\clUnnkDce{\clusterpicture            
  \Root {2} {first} {r1};
  \Root {} {r1} {r2};
  \ClusterD c1[{\kk\!-\!t}] = (r1)(r2);
  \Root {2} {c1} {r3};
  \Root {} {r3} {r4};
  \Cluster c2 = (r3)(r4);
  \Root {1} {c2} {r5};
  \Root {} {r5} {r6};
  \ClusterD c3[{\nn}] = (r5)(r6);
  \ClusterD c4[t] = (c2)(c3),clA;
  \ClusterL c5[{\ee}] = (c1)(c4);
  \frob(c2)(c3);
\endclusterpicture}

\def\clUnnkEce{\clusterpicture            
  \Root {1} {first} {r1};
  \Root {1} {r1} {r2};
  \Root {2} {r2} {r3};
  \Root {} {r3} {r4};
  \Cluster c1 = (r3)(r4);
  \Root {1} {c1} {r5};
  \Root {} {r5} {r6};
  \ClusterD c2[{\nn}] = (r5)(r6);
  \ClusterLD c3[{\ee}][{\kk}] = (c1)(c2),clA;
  \Cluster c4 = (r2)(c3);
  \Cluster c5 = (r1)(c4);
  \frob(c1)(c2);
\endclusterpicture}

\def\clUnnnBals{\clusterpicture            
  \Root {2} {first} {r1};
  \Root {} {r1} {r2};
  \Cluster c1 = (r1)(r2);
  \Root {1} {c1} {r3};
  \Root {} {r3} {r4};
  \Cluster c2 = (r3)(r4);
  \Root {1} {c2} {r5};
  \Root {} {r5} {r6};
  \ClusterD c3[{\nn}] = (r5)(r6);
  \ClusterLD c4[{+}][0] = (c1)(c2)(c3);
  \frob(c1)(c2);
  \frob(c2)(c3);
\endclusterpicture}

\def\clUnnnBalns{\clusterpicture            
  \Root {2} {first} {r1};
  \Root {} {r1} {r2};
  \Cluster c1 = (r1)(r2);
  \Root {1} {c1} {r3};
  \Root {} {r3} {r4};
  \Cluster c2 = (r3)(r4);
  \Root {1} {c2} {r5};
  \Root {} {r5} {r6};
  \ClusterD c3[{\nn}] = (r5)(r6);
  \ClusterLD c4[{-}][0] = (c1)(c2)(c3);
  \frob(c1)(c2);
  \frob(c2)(c3);
\endclusterpicture}

\def\clUnnnBalce{\clusterpicture            
  \Root {2} {first} {r1};
  \Root {} {r1} {r2};
  \Cluster c1 = (r1)(r2);
  \Root {1} {c1} {r3};
  \Root {} {r3} {r4};
  \Cluster c2 = (r3)(r4);
  \Root {1} {c2} {r5};
  \Root {} {r5} {r6};
  \ClusterD c3[{\nn}] = (r5)(r6);
  \ClusterL c4[{\ee}] = (c1)(c2)(c3),clA;
  \frob(c1)(c2);
  \frob(c2)(c3);
\endclusterpicture}

\def\clInImBalss{\clusterpicture            
  \Root {2} {first} {r1};
  \Root {1} {r1} {r2};
  \Root {} {r2} {r3};
  \ClusterLD c1[{+}][{\nn}] = (r2)(r3);
  \ClusterD c2[r] = (r1)(c1);
  \Root {1} {c2} {r4};
  \Root {1} {r4} {r5};
  \Root {} {r5} {r6};
  \ClusterLD c3[{+}][{\mm}] = (r5)(r6);
  \ClusterD c4[{r}] = (r4)(c3);
  \ClusterD c5[0] = (c2)(c4);
\endclusterpicture}

\def\clInImBalnsns{\clusterpicture            
  \Root {2} {first} {r1};
  \Root {1} {r1} {r2};
  \Root {} {r2} {r3};
  \ClusterLD c1[{-}][{\nn}] = (r2)(r3);
  \ClusterD c2[r] = (r1)(c1);
  \Root {1} {c2} {r4};
  \Root {1} {r4} {r5};
  \Root {} {r5} {r6};
  \ClusterLD c3[{-}][{\mm}] = (r5)(r6);
  \ClusterD c4[{r}] = (r4)(c3);
  \ClusterD c5[0] = (c2)(c4);
\endclusterpicture}

\def\clInImBalsns{\clusterpicture            
  \Root {2} {first} {r1};
  \Root {1} {r1} {r2};
  \Root {} {r2} {r3};
  \ClusterLD c1[{+}][{\nn}] = (r2)(r3);
  \ClusterD c2[r] = (r1)(c1);
  \Root {1} {c2} {r4};
  \Root {1} {r4} {r5};
  \Root {} {r5} {r6};
  \ClusterLD c3[{-}][{\mm}] = (r5)(r6);
  \ClusterD c4[{r}] = (r4)(c3);
  \ClusterD c5[0] = (c2)(c4);
\endclusterpicture}

\def\clInImAce{\clusterpicture            
  \Root {2} {first} {r1};
  \Root {} {r1} {r2};
  \ClusterLD c1[{\ee}][{\nn}] = (r1)(r2);
  \Root {1} {c1} {r3};
  \Root {1} {r3} {r4};
  \Root {} {r4} {r5};
  \ClusterLD c2[{\dd}][{\mm}] = (r4)(r5);
  \ClusterD c3[{2r}] = (r3)(c2);
  \Cluster c4 = (c1)(c3),clA;
\endclusterpicture}

\def\clInImBalce{\clusterpicture            
  \Root {2} {first} {r1};
  \Root {1} {r1} {r2};
  \Root {} {r2} {r3};
  \ClusterLD c1[{\ee}][{\nn}] = (r2)(r3);
  \ClusterD c2[t] = (r1)(c1),clA;
  \Root {1} {c2} {r4};
  \Root {1} {r4} {r5};
  \Root {} {r5} {r6};
  \ClusterLD c3[{\dd}][{\mm}] = (r5)(r6);
  \ClusterD c4[{2r\!-\!t}] = (r4)(c3);
  \Cluster c5 = (c2)(c4);
\endclusterpicture}

\def\clInImCce{\clusterpicture            
  \Root {1} {first} {r1};
  \Root {} {r1} {r2};
  \Root {1} {r2} {r3};
  \Root {1} {r3} {r4};
  \Root {1} {r4} {r5};
  \Root {} {r5} {r6};
  \ClusterLD c1[{\ee}][{\nn}] = (r5)(r6);
  \ClusterD c2[{2r}] = (r4)(c1);
  \ClusterLD c3[{\dd}][{\mm}] = (r3)(c2),clA;
  \Cluster c4 = (r1)(r2)(c3);
\endclusterpicture}

\def\clInImDce{\clusterpicture            
  \Root {2} {first} {r1};
  \Root {} {r1} {r2};
  \ClusterD c1[{\nn\!-\!t}] = (r1)(r2);
  \Root {1} {c1} {r3};
  \Root {1} {r3} {r4};
  \Root {1} {r4} {r5};
  \Root {} {r5} {r6};
  \ClusterLD c2[{\dd}][{\mm}] = (r5)(r6);
  \ClusterD c3[{2r}] = (r4)(c2);
  \ClusterD c4[t] = (r3)(c3),clA;
  \ClusterL c5[{\ee}] = (c1)(c4);
\endclusterpicture}

\def\clInImEce{\clusterpicture            
  \Root {1} {first} {r1};
  \Root {2} {r1} {r2};
  \Root {} {r2} {r3};
  \ClusterLD c1[{\ee}][{\nn}] = (r2)(r3);
  \Root {1} {c1} {r4};
  \Root {1} {r4} {r5};
  \Root {} {r5} {r6};
  \ClusterLD c2[{\dd}][{\mm}] = (r5)(r6);
  \ClusterD c3[{2r}] = (r4)(c2);
  \Cluster c4 = (c1)(c3),clA;
  \Cluster c5 = (r1)(c4);
\endclusterpicture}

\def\clInImFce{\clusterpicture            
  \Root {1} {first} {r1};
  \Root {1} {r1} {r2};
  \Root {1} {r2} {r3};
  \Root {1} {r3} {r4};
  \Root {} {r4} {r5};
  \ClusterLD c1[{\ee}][{\nn}] = (r4)(r5);
  \ClusterD c2[{2r}] = (r3)(c1);
  \ClusterLD c3[{\dd}][{\mm}] = (r2)(c2),clA;
  \Cluster c4 = (r1)(c3);
\endclusterpicture}

\def\clInImGce{\clusterpicture            
  \Root {1} {first} {r1};
  \Root {1} {r1} {r2};
  \Root {} {r2} {r3};
  \ClusterLD c1[{\ee}][{\nn}] = (r2)(r3);
  \Root {1} {c1} {r4};
  \Root {1} {r4} {r5};
  \Root {} {r5} {r6};
  \ClusterLD c2[{\dd}][{\mm}] = (r5)(r6);
  \ClusterD c3[{2r}] = (r4)(c2);
  \Cluster c4 = (r1)(c1)(c3),clA;
\endclusterpicture}

\def\clInImHce{\clusterpicture            
  \Root {1} {first} {r1};
  \Root {1} {r1} {r2};
  \Root {1} {r2} {r3};
  \Root {1} {r3} {r4};
  \Root {1} {r4} {r5};
  \Root {} {r5} {r6};
  \ClusterLD c1[{\ee}][{\nn}] = (r5)(r6);
  \ClusterD c2[{2r}] = (r4)(c1);
  \ClusterLD c3[{\dd}][{\mm}] = (r3)(c2),clA;
  \Cluster c4 = (r2)(c3);
  \Cluster c5 = (r1)(c4);
\endclusterpicture}

\def\clInInBals{\clusterpicture            
  \Root {2} {first} {r1};
  \Root {1} {r1} {r2};
  \Root {} {r2} {r3};
  \ClusterL c1[{+}] = (r2)(r3);
  \Cluster c2 = (r1)(c1);
  \Root {1} {c2} {r4};
  \Root {1} {r4} {r5};
  \Root {} {r5} {r6};
  \ClusterLD c3[{+}][{\nn}] = (r5)(r6);
  \ClusterD c4[r] = (r4)(c3);
  \ClusterD c5[0] = (c2)(c4);
  \frob(c2)(c4);
\endclusterpicture}

\def\clInInBalns{\clusterpicture            
  \Root {2} {first} {r1};
  \Root {1} {r1} {r2};
  \Root {} {r2} {r3};
  \ClusterL c1[{+}] = (r2)(r3);
  \Cluster c2 = (r1)(c1);
  \Root {1} {c2} {r4};
  \Root {1} {r4} {r5};
  \Root {} {r5} {r6};
  \ClusterLD c3[{-}][{\nn}] = (r5)(r6);
  \ClusterD c4[r] = (r4)(c3);
  \ClusterD c5[0] = (c2)(c4);
  \frob(c2)(c4);
\endclusterpicture}

\def\clInInBalce{\clusterpicture            
  \Root {2} {first} {r1};
  \Root {1} {r1} {r2};
  \Root {} {r2} {r3};
  \ClusterL c1[{\eta}] = (r2)(r3);
  \Cluster c2 = (r1)(c1),clA;
  \Root {1} {c2} {r4};
  \Root {1} {r4} {r5};
  \Root {} {r5} {r6};
  \ClusterLD c3[{\ee\eta}][{\nn}] = (r5)(r6);
  \ClusterD c4[r] = (r4)(c3);
  \Cluster c5 = (c2)(c4);
  \frob(c2)(c4);
\endclusterpicture}

\def\cloInBals{\clusterpicture            
  \Root {2} {first} {r1};
  \Root {} {r1} {r2};
  \Root {} {r2} {r3};
  \ClusterD c1[r] = (r1)(r2)(r3);
  \Root {1} {c1} {r4};
  \Root {1} {r4} {r5};
  \Root {} {r5} {r6};
  \ClusterLD c2[{+}][{\nn}] = (r5)(r6);
  \ClusterD c3[{r}] = (r4)(c2);
  \ClusterD c4[0] = (c1)(c3);
\endclusterpicture}

\def\cloInBalns{\clusterpicture            
  \Root {2} {first} {r1};
  \Root {} {r1} {r2};
  \Root {} {r2} {r3};
  \ClusterD c1[r] = (r1)(r2)(r3);
  \Root {1} {c1} {r4};
  \Root {1} {r4} {r5};
  \Root {} {r5} {r6};
  \ClusterLD c2[{-}][{\nn}] = (r5)(r6);
  \ClusterD c3[{r}] = (r4)(c2);
  \ClusterD c4[0] = (c1)(c3);
\endclusterpicture}

\def\cloInAce{\clusterpicture            
  \Root {1} {first} {r1};
  \Root {} {r1} {r2};
  \Root {1} {r2} {r3};
  \Root {1} {r3} {r4};
  \Root {} {r4} {r5};
  \ClusterLD c1[{\ee}][{\nn}] = (r4)(r5);
  \ClusterD c2[{2r}] = (r3)(c1);
  \Cluster c3 = (r1)(r2)(c2),clA;
\endclusterpicture}

\def\cloInBalce{\clusterpicture            
  \Root {2} {first} {r1};
  \Root {} {r1} {r2};
  \Root {} {r2} {r3};
  \ClusterD c1[t] = (r1)(r2)(r3),clA;
  \Root {1} {c1} {r4};
  \Root {1} {r4} {r5};
  \Root {} {r5} {r6};
  \ClusterLD c2[{\ee}][{\nn}] = (r5)(r6);
  \ClusterD c3[{2r\!-\!t}] = (r4)(c2);
  \Cluster c4 = (c1)(c3);
\endclusterpicture}

\def\cloInCce{\clusterpicture            
  \Root {1} {first} {r1};
  \Root {} {r1} {r2};
  \Root {1} {r2} {r3};
  \Root {1} {r3} {r4};
  \Root {} {r4} {r5};
  \Root {} {r5} {r6};
  \ClusterD c1[{2r}] = (r4)(r5)(r6);
  \ClusterLD c2[{\ee}][{\nn}] = (r3)(c1),clA;
  \Cluster c3 = (r1)(r2)(c2);
\endclusterpicture}

\def\cloInDce{\clusterpicture            
  \Root {2} {first} {r1};
  \Root {} {r1} {r2};
  \ClusterD c1[{\nn\!-\!t}] = (r1)(r2);
  \Root {1} {c1} {r3};
  \Root {1} {r3} {r4};
  \Root {} {r4} {r5};
  \Root {} {r5} {r6};
  \ClusterD c2[{2r}] = (r4)(r5)(r6);
  \ClusterD c3[t] = (r3)(c2),clA;
  \ClusterL c4[{\ee}] = (c1)(c3);
\endclusterpicture}

\def\cloInEce{\clusterpicture            
  \Root {1} {first} {r1};
  \Root {1} {r1} {r2};
  \Root {} {r2} {r3};
  \Root {1} {r3} {r4};
  \Root {1} {r4} {r5};
  \Root {} {r5} {r6};
  \ClusterLD c1[{\ee}][{\nn}] = (r5)(r6);
  \ClusterD c2[{2r}] = (r4)(c1);
  \Cluster c3 = (r2)(r3)(c2),clA;
  \Cluster c4 = (r1)(c3);
\endclusterpicture}

\def\cloInFce{\clusterpicture            
  \Root {2} {first} {r1};
  \Root {} {r1} {r2};
  \ClusterLD c1[{\ee}][{\nn}] = (r1)(r2);
  \Root {1} {c1} {r3};
  \Root {} {r3} {r4};
  \Root {} {r4} {r5};
  \ClusterD c2[{2r}] = (r3)(r4)(r5);
  \Cluster c3 = (c1)(c2),clA;
\endclusterpicture}

\def\cloInGce{\clusterpicture            
  \Root {1} {first} {r1};
  \Root {1} {r1} {r2};
  \Root {} {r2} {r3};
  \ClusterLD c1[{\ee}][{\nn}] = (r2)(r3);
  \Root {1} {c1} {r4};
  \Root {} {r4} {r5};
  \Root {} {r5} {r6};
  \ClusterD c2[{2r}] = (r4)(r5)(r6);
  \Cluster c3 = (r1)(c1)(c2),clA;
\endclusterpicture}

\def\cloInHce{\clusterpicture            
  \Root {1} {first} {r1};
  \Root {2} {r1} {r2};
  \Root {} {r2} {r3};
  \ClusterLD c1[{\ee}][{\nn}] = (r2)(r3);
  \Root {1} {c1} {r4};
  \Root {} {r4} {r5};
  \Root {} {r5} {r6};
  \ClusterD c2[{2r}] = (r4)(r5)(r6);
  \Cluster c3 = (c1)(c2),clA;
  \Cluster c4 = (r1)(c3);
\endclusterpicture}

\def\cloInIce{\clusterpicture            
  \Root {1} {first} {r1};
  \Root {1} {r1} {r2};
  \Root {1} {r2} {r3};
  \Root {} {r3} {r4};
  \Root {} {r4} {r5};
  \ClusterD c1[{2r}] = (r3)(r4)(r5);
  \ClusterLD c2[{\ee}][{\nn}] = (r2)(c1),clA;
  \Cluster c3 = (r1)(c2);
\endclusterpicture}

\def\cloInJce{\clusterpicture            
  \Root {1} {first} {r1};
  \Root {} {r1} {r2};
  \Root {} {r2} {r3};
  \Root {1} {r3} {r4};
  \Root {1} {r4} {r5};
  \Root {} {r5} {r6};
  \ClusterLD c1[{\ee}][{\nn}] = (r5)(r6);
  \ClusterD c2[{2r}] = (r4)(c1);
  \Cluster c3 = (r1)(r2)(r3)(c2),clA;
\endclusterpicture}

\def\cloInKce{\clusterpicture            
  \Root {1} {first} {r1};
  \Root {1} {r1} {r2};
  \Root {1} {r2} {r3};
  \Root {1} {r3} {r4};
  \Root {} {r4} {r5};
  \Root {} {r5} {r6};
  \ClusterD c1[{2r}] = (r4)(r5)(r6);
  \ClusterLD c2[{\ee}][{\nn}] = (r3)(c1),clA;
  \Cluster c3 = (r2)(c2);
  \Cluster c4 = (r1)(c3);
\endclusterpicture}

\def\clooBal{\clusterpicture            
  \Root {2} {first} {r1};
  \Root {} {r1} {r2};
  \Root {} {r2} {r3};
  \ClusterD c1[r] = (r1)(r2)(r3);
  \Root {1} {c1} {r4};
  \Root {} {r4} {r5};
  \Root {} {r5} {r6};
  \ClusterD c2[{r}] = (r4)(r5)(r6);
  \ClusterD c3[0] = (c1)(c2);
\endclusterpicture}

\def\clooAc{\clusterpicture            
  \Root {1} {first} {r1};
  \Root {} {r1} {r2};
  \Root {1} {r2} {r3};
  \Root {} {r3} {r4};
  \Root {} {r4} {r5};
  \ClusterD c1[{2r}] = (r3)(r4)(r5);
  \Cluster c2 = (r1)(r2)(c1),clA;
\endclusterpicture}

\def\clooBalc{\clusterpicture            
  \Root {2} {first} {r1};
  \Root {} {r1} {r2};
  \Root {} {r2} {r3};
  \ClusterD c1[t] = (r1)(r2)(r3),clA;
  \Root {1} {c1} {r4};
  \Root {} {r4} {r5};
  \Root {} {r5} {r6};
  \ClusterD c2[{2r\!-\!t}] = (r4)(r5)(r6);
  \Cluster c3 = (c1)(c2);
\endclusterpicture}

\def\clooCc{\clusterpicture            
  \Root {1} {first} {r1};
  \Root {1} {r1} {r2};
  \Root {} {r2} {r3};
  \Root {1} {r3} {r4};
  \Root {} {r4} {r5};
  \Root {} {r5} {r6};
  \ClusterD c1[{2r}] = (r4)(r5)(r6);
  \Cluster c2 = (r2)(r3)(c1),clA;
  \Cluster c3 = (r1)(c2);
\endclusterpicture}

\def\clooDc{\clusterpicture            
  \Root {1} {first} {r1};
  \Root {} {r1} {r2};
  \Root {} {r2} {r3};
  \Root {1} {r3} {r4};
  \Root {} {r4} {r5};
  \Root {} {r5} {r6};
  \ClusterD c1[{2r}] = (r4)(r5)(r6);
  \Cluster c2 = (r1)(r2)(r3)(c1),clA;
\endclusterpicture}

\def\cloof{\clusterpicture            
  \Root {2} {first} {r1};
  \Root {} {r1} {r2};
  \Root {} {r2} {r3};
  \Cluster c1 = (r1)(r2)(r3);
  \Root {1} {c1} {r4};
  \Root {} {r4} {r5};
  \Root {} {r5} {r6};
  \ClusterD c2[{r}] = (r4)(r5)(r6);
  \ClusterD c3[0] = (c1)(c2);
  \frob(c1)(c2);
\endclusterpicture}

\def\cloofc{\clusterpicture            
  \Root {2} {first} {r1};
  \Root {} {r1} {r2};
  \Root {} {r2} {r3};
  \Cluster c1 = (r1)(r2)(r3),clA;
  \Root {1} {c1} {r4};
  \Root {} {r4} {r5};
  \Root {} {r5} {r6};
  \ClusterD c2[{r}] = (r4)(r5)(r6);
  \Cluster c3 = (c1)(c2);
  \frob(c1)(c2);
\endclusterpicture}

\def\triple{\clusterpicture            
  \Root {1} {first} {r1};
  \Root {} {r1} {r2};
  \Root {} {r2} {r3};
  \ClusterD c1[{\s}] = (r1)(r2)(r3);
\endclusterpicture}

\def\childparent{\clusterpicture            
  \Root {1} {first} {r4};
  \Root {} {r4} {r5};
  \Root {} {r5} {r6};
  \Root {2} {r6} {r1};
  \Root {} {r1} {r2};
  \Root {} {r2} {r3};
  \ClusterD c1[{\s^{\prime}}] = (r1)(r2)(r3);
  \ClusterD c2[{\s}] = (r4)(r5)(r6)(c1);
\endclusterpicture}

\def\childparentn{\clusterpicture            
  \Root {1} {first} {r4};
  \Root {} {r4} {r5};
  \Root {} {r5} {r6};
  \Root {2} {r6} {r1};
  \Root {} {r1} {r2};
  \Root {} {r2} {r3};
  \ClusterD c1[{\s}] = (r1)(r2)(r3);
  \ClusterD c2[{P(\s)}] = (r4)(r5)(r6)(c1);
\endclusterpicture}

\def\event{\clusterpicture            
  \Root {1} {first} {r1};
  \Root {} {r1} {r2};
  \Root {} {r2} {r3};
  \Root {} {r3} {r4};
  \ClusterD c1[{\s}] = (r1)(r2)(r3)(r4);
\endclusterpicture}

\def\oddt{\clusterpicture            
  \Root {1} {first} {r1};
  \Root {} {r1} {r2};
  \Root {} {r2} {r3};
  \Root {} {r3} {r4};
  \Root {} {r4} {r5};
  \ClusterD c1[{\s}] = (r1)(r2)(r3)(r4)(r5);
\endclusterpicture}

\def\ubereven{\clusterpicture            
  \Root {2} {first} {r1};
  \Root {} {r1} {r2};
  \Cluster c1 = (r1)(r2);
  \Root {1} {c1} {r3};
  \Root {} {r3} {r4};
  \Cluster c2 = (r3)(r4);
  \Root {1} {c2} {r5};
  \Root {} {r5} {r6};
  \Cluster c3 = (r5)(r6);
  \ClusterD c4[{\s}] = (c1)(c2)(c3);
\endclusterpicture}

\def\twin{\clusterpicture            
  \Root {1} {first} {r1};
  \Root {} {r1} {r2};
  \ClusterD c1[{\s}] = (r1)(r2);
\endclusterpicture}

\def\cotwino{\clusterpicture            
  \Root {1} {first} {r5};
  \Root {2} {r5} {r1};
  \Root {} {r1} {r2};
  \Root {} {r2} {r3};
  \Root {} {r3} {r4};
  \Cluster c1 = (r1)(r2)(r3)(r4);
  \ClusterD c2[{\s}] = (r5)(c1);
\endclusterpicture}

\def\cotwint{\clusterpicture            
  \Root {1} {first} {r5};
  \Root {} {r5} {r6};
  \Root {2} {r6} {r1};
  \Root {} {r1} {r2};
  \Root {} {r2} {r3};
  \Root {} {r3} {r4};
  \Cluster c1 = (r1)(r2)(r3)(r4);
  \ClusterD c2[{\s}] = (r5)(r6)(c1);
\endclusterpicture}

\def\clusterlcas{\clusterpicture            
  \Root {2} {first} {r1};
  \Root {} {r1} {r2};
  \ClusterD c1[{\s^{\prime}}] = (r1)(r2);
  \Root {1} {c1} {r5};
  \Root {2} {r5} {r3};
  \Root {} {r3} {r4};
  \ClusterD c2[{\s}] = (r3)(r4);
  \Cluster c3 = (r5)(c2);
  \ClusterD c4[{\s^{\prime}\wedge\s}] = (c1)(c3);
\endclusterpicture}

\def\clusternecks{\clusterpicture            
  \Root {1} {first} {r5};
  \Root {3} {r5} {r1};
  \Root {} {r1} {r2};
  \ClusterD c1[{\s}] = (r1)(r2);
  \Root {1} {c1} {r3};
  \Root {} {r3} {r4};
  \Cluster c2 = (r3)(r4);
  \ClusterD c3[{\s^*}] = (c1)(c2);
  \Cluster c4 = (r5)(c3);
\endclusterpicture}

\def\RebclnCce{\clusterpicture            
  \Root {1} {first} {r1};
  \Root {} {r1} {r2};
  \Root {1} {r2} {r3};
  \Root {} {r3} {r4};
  \Root {} {r4} {r5};
  \Root {} {r5} {r6};
  \ClusterD c1[{n}] = (r3)(r4)(r5)(r6);
  \Cluster c2 = (r1)(r2)(c1);
\endclusterpicture}

\def\RebclnDce{\clusterpicture            
  \Root {2} {first} {r1};
  \Root {} {r1} {r2};
  \ClusterD c1[{n\!-\!t}] = (r1)(r2);
  \Root {1} {c1} {r3};
  \Root {} {r3} {r4};
  \Root {} {r4} {r5};
  \Root {} {r5} {r6};
  \ClusterD c2[t] = (r3)(r4)(r5)(r6);
  \Cluster c3 = (c1)(c2);
\endclusterpicture}

\def\RebclnFce{\clusterpicture            
  \Root {1} {first} {r1};
  \Root {1} {r1} {r2};
  \Root {} {r2} {r3};
  \Root {} {r3} {r4};
  \Root {} {r4} {r5};
  \ClusterD c1[{n}] = (r2)(r3)(r4)(r5);
  \Cluster c2 = (r1)(c1);
\endclusterpicture}

\def\RebclnGce{\clusterpicture            
  \Root {1} {first} {r1};
  \Root {1} {r1} {r2};
  \Root {1} {r2} {r3};
  \Root {} {r3} {r4};
  \Root {} {r4} {r5};
  \Root {} {r5} {r6};
  \ClusterD c1[{n}] = (r3)(r4)(r5)(r6);
  \Cluster c2 = (r2)(c1);
  \Cluster c3 = (r1)(c2);
\endclusterpicture}

\makeatother

\advance\oddsidemargin -1cm
\advance\evensidemargin -1cm
\advance\textwidth 2cm
\advance\textheight 0.6cm 

\begin{document}

\title{Arithmetic of hyperelliptic curves over local fields}
\author{Tim and Vladimir Dokchitser, C\'eline Maistret, Adam Morgan}

\address{Department of Mathematics, University of Bristol, Bristol BS8 1TW, UK}
\email{tim.dokchitser@bristol.ac.uk}

\address{King's College London, Strand, London WC2R 2LS, UK}
\email{vladimir.dokchitser@kcl.ac.uk}

\address{Department of Mathematics, University of Bristol, Bristol BS8 1TW, UK}
\email{celine.maistret@bristol.ac.uk}

\address{King's College London, Strand, London WC2R 2LS, UK}
\email{adam.morgan@kcl.ac.uk}

\begin{abstract}
We study hyperelliptic curves $y^2 = f(x)$ over local fields of odd residue characteristic. 
We introduce the notion of a ``cluster picture'' associated to the curve, that describes the $p$-adic distances between the roots of $f(x)$, and show that this elementary combinatorial object encodes the curve's Galois representation, conductor, whether the curve is semistable, and if so, the special fibre of its minimal regular model, the discriminant of its minimal Weierstrass equation and other invariants. 
\end{abstract}


\maketitle

\vfill

\setcounter{tocdepth}{1}

\tableofcontents

\vfill

\section{Introduction}

In this paper we study hyperelliptic curves $y^2 = f(x)$ over local fields of odd residue characteristic. 
To a curve we associate a ``cluster picture'', defined by the combinatorics of the root configuration of $f$, 
and show that it encodes many arithmetic invariants of the curve and its Jacobian. 
We use cluster pictures to get hold of a curve's Galois representation and conductor, determine whether 
it is semistable and if so obtain the special fibre of its minimal regular model, the discriminant 
of its minimal Weierstrass model and other invariants. 

%
%

Throughout the paper $K$ will be a local field, with normalised valuation $v$, ring of integers $\cO_K$, uniformiser $\pi$,
and finite residue field $k$ of characteristic $p\ne 2$. We use the shorthand mod $\mathfrak{m}$ to denote reduction to the residue field. 
We write $G_K=\Gal(\Ks/K)$ for the absolute Galois group, and $I_K<G_K$ for the inertia subgroup.

We work with hyperelliptic curves\footnote{in this paper, hyperelliptic curves will always be assumed to have genus at least 2}
$C/K$ given by Weierstrass equations
$$
  C\colon y^2=f(x).
$$
We write $\cR$ for the set of roots of $f(x)$ in $\Ks$ and $c_f$ for its leading coefficient, so that
$$
f(x) = c_f\prod_{r\in\cR} (x-r). 
$$
We denote by $g$ the genus of the curve so that $|\cR| = 2g+1$ or $2g+2$. 

The main invariant that we are interested in is the configuration of distances between 
the roots of $f$. This is captured in the following: 
\begin{definition}
\label{defclusters}
A \emph{cluster} is a non-empty subset $\c\subset\cR$ of the form $\c = D \cap \cR$
for some disc $D=\{x\!\in\! \Kbar \mid v(x-z)\!\geq\! d\}$
for some $z\in \Kbar$ and $d\in \Q$. 
If $|\s|>1$, we say that $\s$ is a \emph{proper} cluster and define its \emph{depth} $d_\s$ to be 
$$
  d_\s = \min_{r,r' \in \mathfrak{s}} v(r-r').
$$
It is the minimal $d$ for which $\s$ is cut out by such a disc. 
\end{definition}


It turns out that the cluster data carries a huge amount of information about the arithmetic of $C/K$. To fix ideas, let us begin with an example.

\begin{example}\label{introexample}
Let $C/\Q_p$ be the hyperelliptic curve of genus 3 given by
$$
  C/\Q_p\colon y^2 = (x\!-\!1)\,(x\!-\!(1\!+\!p^2))\,(x\!-\!(1\!-\!p^2))\,(x\!-\!p)\,x\,(x\!-\!p^3)\,(x\!+\!p^3).
$$
The set of roots is $\cR=\{1,1\!+\!p^2,1\!-\!p^2,p,0,p^3,-p^3\}$. There are four proper clusters:
$$
 \{1,1\!+\!p^2,1\!-\!p^2\}, \quad
  \{0,p^3,-p^3\}, \quad 
 \{p,0,p^3,-p^3\}, \quad 
  \cR,
$$
of depths 2,3,1 and 0, respectively. 
We draw cluster pictures by drawing roots $r \in\cR$ as
\smash{\raise4pt\hbox{\clusterpicture\Root{1}{first}{r1};\endclusterpicture}},
and draw ovals around roots to represent a cluster:
$$
\scalebox{1.4}{\clusterpicture            
  \Root {1} {first} {r1};
  \Root {} {r1} {r2};
  \Root {} {r2} {r3};
  \Root {3} {r3} {r4};
  \Root {1} {r4} {r5};
  \Root {} {r5} {r6};
  \Root {} {r6} {r7};
  \ClusterD c1[{2}] = (r1)(r2)(r3);
  \ClusterD c2[{2}] = (r5)(r6)(r7);
  \ClusterD c3[{1}] = (r4)(c2);
  \ClusterD c4[{0}] = (c1)(c2)(c3);
\endclusterpicture}
$$
Here we have ordered the roots as they appear in the equation for $C$. The subscript of the top cluster $\cR$ is its depth. For all other clusters it is their ``relative depth'', that is the difference between their depth and that of their parent cluster. 
%

From this picture our results let us extract the following information.

\noindent (1) $C/\Q_p$ is semistable (Thm. \ref{redcond} (1)), with
conductor exponent 1 (Thm. \ref{th:condexpogen}),
the model for $C$ is a minimal Weierstrass equation (Thm. \ref{introminimal}) and has discriminant of valuation
$v(\Delta_C) = 36$ (Thm. \ref{thmmindisc}).


\noindent\begin{minipage}[t]{.7\textwidth}
(2)
The special fibre $\cC_{\bar{\mathbb{F}}_p}$ of the minimal regular model $\cC/\Z_p^{nr}$ has 4 components,
as shown on the right (Thms. \ref{introhomology}(2), \ref{introspecial}). Both genus 1 curves 
are given by $y^2 = x^3-x$. 
\end{minipage}%
\smash{$\quad$\raise -10pt\hbox{
\begin{minipage}{0.2\textwidth}
  \includegraphics[height=1.7cm]{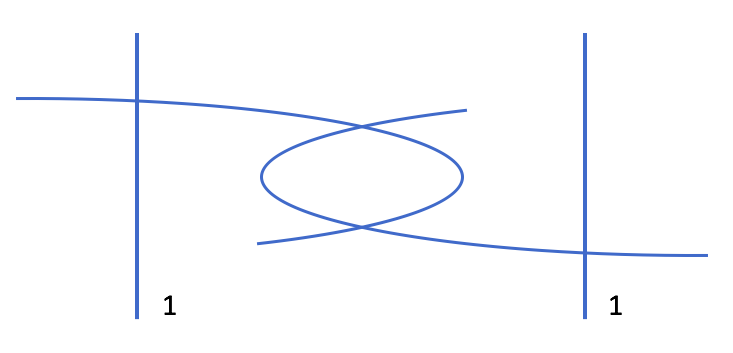}
\end{minipage}}
}
\vskip 2pt
    
\noindent (3) The homology of the dual graph of $\cC_{\bar{\mathbb{F}}_p}$ is isomorphic to $\Z = \langle \ell\rangle $ 
with length pairing $\langle \ell, \ell \rangle =2$. The absolute Galois group of $\F_p$ acts naturally
on it, and the Frobenius element $\Phi\colon a\mapsto a^p$ acts by 
$\Phi(\ell) = \epsilon\ell$ with $\epsilon=\bigl(\frac{-1}{p}\bigr)$  (Thm. \ref{introhomology}).

\noindent (4) Fix $l\ne p$ and a Frobenius element $\Frob\in G_{\Q_p}$ (a lift of $\Phi$). Let 
$\tau\colon I_{\Q_p}\to\Z_l$ be an $l$-adic tame character. Then
there is a basis for $\H(C,\Z_l)\tensor_{\Z_l}{\bar \Q_l}$ such that $\sigma \in I_{\Q_p}$ and $\Frob$ 
act as
$$
\scalebox{0.7}{$
\begin{pmatrix}
1 &0&0&0&&\\
0 &1&0&0&&\\
0 &0&1&0&&\\
0 &0&0&1&&\\
&&&&1&\tau(\sigma)\\
&&&&0&1
\end{pmatrix}$} \qquad \text{and} \qquad 
\scalebox{0.7}{$
\begin{pmatrix}
\alpha &0&0&0&&\\
0 &\beta&0&0&&\\
0 &0&\alpha&0&&\\
0 &0&0&\beta&&\\
&&&&\epsilon&0\\
&&&&0&\epsilon p
\end{pmatrix}$} 
$$
respectively,
where 
$\alpha, \beta$ are the Frobenius eigenvalues for $y^2 = x^3-x$ over $\F_p$ (Thm.~\ref{hetmain} and Remark~\ref{semsible action on cohomology remark}).

\noindent (5) If $g(x)$ is any polynomial of degree 7 whose coefficients are congruent to those of $f(x)$ modulo $p^4$, then the data listed above is the same for the curve $y^2 =g(x)$
(Thm.~\ref{introcontinuity}). 

To deduce (1) to (5) we need very little information beyond the cluster picture displayed above. To be precise, we obviously need the valuation of the leading term $c_f$ (in this case $0$, as $f(x)$ is monic). 
Moreover the explicit description of $\epsilon$ for the Galois representation and the Galois action on homology, and the explicit equations of the curves in the special fibre (but not their genera) also require more delicate information about the roots of $f(x)$ (see Definition \ref{de:epsilon} and Theorem \ref{introspecial}).
\end{example}

\medskip

We now present our main results. We first need some terminology to work with clusters. 

\begin{definition}
If $\s'\subsetneq \s$ is a maximal subcluster, we write $\s'<\s$ and 
refer to $\s'$ as a \emph{child} of $\s$, and to $\s$ as the \emph{parent} of $\s'$. We write $\s =P(\s')$.

For two clusters (or roots)  $\c_1$, $\c_2$ we write $\c_1\wedge \c_2$ for the smallest cluster that contains them.  
\end{definition}

\begin{definition} \label[definition]{principaldefi}
A cluster $\s$ is
a \emph{twin} if $|\s|=2$,
and it is \emph{odd}/\emph{even} if its size is odd/even. 
A proper cluster is \emph{\"ubereven} if it has only even children.
A cluster $\s$ is {\em principal} if $|\s|\ge 3$, except if either $\s=\cR$ is even and has exactly two children, or if $\s$ has a child of size $2g$.
(See also Table \ref{clusternotation} for a summary.)
\end{definition}

\begin{definition}\label{de:relativedepth}
For a proper cluster $\s \ne \cR$ we define its \emph{relative depth} to be 
$$
\delta_\s = d_\s -d_{P(\s)}.
$$ 
\end{definition}

\subsection{Reduction of $C$ and its Jacobian}
One of the main outcomes of the cluster approach is that it easily detects the reduction behaviour of the curve $C$ and its Jacobian. It requires one extra invariant (that feels the leading coefficient of $f(x)$): 
\begin{definition}\label[definition]{no:nu}
For a cluster $\c$ set
$$
 \nu_\c = v(c_f)+\sum_{r\in\cR} d_{r\wedge \s}.
$$  
\end{definition}

\begin{definition} \label[definition]{semistability criterion}
We say that $C/K$ satisfies the \emph{semistability criterion} if
the following conditions hold:
\begin{enumerate}
\item
The extension $K(\cR)/K$ has ramification degree at most 2.
\item
Every proper cluster is $I_K$-invariant. 
\item
Every principal cluster $\s$ has $d_\c \in\Z$ and 
$\nu_\s \in 2\Z.$
\end{enumerate}
\end{definition}


\begin{theorem}[=Theorem \ref{reduction}]
\label[theorem]{redcond}
Let $C/K$ be a hyperelliptic curve of genus~$g$. Write $\Jac C$ for its Jacobian. Then
\begin{enumerate}
\item
$C$ is semistable ($\iff$ $\Jac C$ semistable) 
\\$\iff$ $C/K$ satisfies the semistability criterion.
\item
$C$ has good reduction 
\\$\iff$ $K(\cR)/K$ is unramified, there are no proper clusters of size $<2g+1$ and $\nu_s\in 2\Z$ for the unique principal cluster.
\item
$C$ has potentially good reduction 
\\$\iff$ there are no proper clusters of size $<2g+1$.
\item
$C$ is tame\footnote{By `tame' we mean semistable over some tamely ramified extension of $K$; this is automatically the case if $p>2g+1$}
 ($\iff$ $\Jac C$ is tame)
\\ $\iff$ $K(\cR)/K$ is tame.
\item
$\Jac C$ has good reduction 
\\$\iff$  $K(\cR)/K$ is unramified, 
all clusters $\s\ne\cR$ are odd, and principal clusters 
have $\nu_\s \in 2\Z$.
\item
$\Jac C$ has potentially good reduction  
\\$\iff$ all clusters $\s\ne\cR$ are odd.
\item
The potential toric rank of $\Jac C$ equals the number of even non-\"ubereven clusters 
excluding $\cR$, less 1 if $\cR$ is \"ubereven.
\item
$\Jac C$ has potentially totally toric reduction 
\\ $\iff$ every cluster has at most two odd children. 
\end{enumerate}
\end{theorem}

\subsection{Special fibre of the minimal regular model}
For semistable curves we write down explicit charts for a regular model (\Cref{normalisation of charts}). 
This then gives us an explicit construction for the special fibre of the minimal regular model in terms of clusters (see Theorem \ref{th:DualGraph}). We give a simplified statement here.

\begin{definition}\label{de:centre}
A \emph{centre} $z_\s$ of a proper cluster $\s$ is any element $z_\s \in \Ks$ such that $v(z_\s-r) \geq d_\s$ for all $r \in \s$; equivalently the cluster $\s$ can be written as $D \cap \cR$ for the disc $D = z_s + \pi^{d_\s} \cO_{\Ks}$. 
If $\s = \{r\}$ is a singleton, its centre is $z_\s = r$. 
\end{definition}

\begin{theorem}
\label{introspecial}
Suppose $C/K$ is semistable.
The special fibre of the minimal regular model of $C$ over $\OKnr$ consists of 
components $\Cmp_\s$ for every principal cluster $\s$, linked by chains of $\P^1$s. The normalisation of $\Cmp_s$ is given as follows. Fix a choice of centre $z_\s$ for each cluster. Then
$$
 \tilde{\Gamma}_\s \colon\> y^2 = 
   c_\s
    \prod_{{\text{odd }}{{\mathfrak o} < \s}}(x-\red_\s(\mathfrak o)),
$$
where $c_\s=\frac{c_f}{\pi^{v(c_f)}}\prod_{r \notin \c} \frac{z_\c-r}{\pi^{d_{\c\wedge r}}} \mod \m$ and $\red_\s(\mathfrak{o}) = \frac{z_\s-z_{\mathfrak{o}}}{\pi^{d_\s}} \mod \m$. 

If $\cR$ is principal then the chains of $\P^1$s are given explicitly as follows. 
If $\s'<\s$ are principal, then there is one chain with $\frac{\delta_{\s'}}{2}\!-\! 1$ $\P^1$s (resp. two chains with $\delta_{\s'}\!-\! 1$ $\P^1$s) between $\Gamma_\s$ and $\Gamma_{\s'}$ if $\s'$ is odd (resp. even); if a chain has length 0 the components $\Gamma_\s$ and $\Gamma_{\s'}$ intersect. Every twin $\t<\s$ with $\delta_\t>\frac 12$ gives a chain of $2\delta_\t\!-\! 1$ $\P^1$s from $\Gamma_\s$ to itself. 
\end{theorem}

In particular we obtain a description of the dual graph of the special fibre and its homology in terms of clusters. The latter has a particularly simple description: roughly speaking, the set of even clusters corresponds to a basis of the homology group. Recall that this homology group is related to the Galois representation, the Tamagawa number of the Jacobian (Lemma \ref{lemtam}) and the character group of the toric part of the Raynaud parametrisation of the Jacobian (Lemma \ref{le:isolattice}).
In order to keep track of the Galois action on the dual graph (an analogue of split/non-split multiplicative reduction for an elliptic curves) we need an extra invariant $\epsilon_\s$,  which generalises $\epsilon$ in Example \ref{introexample}.

\begin{definition}
A {\em cotwin} is a non-\ub\ cluster that has a child of size~$2g$.

For a cluster $\s$ that is not a cotwin we write $\s^*$ for the smallest cluster $\s^*\supseteq\s$ whose parent is not \ub\ (and $\s^*=\cR$ if no such cluster exists). If $\s$ is a cotwin, we write $\s^*$ for its child of size $2g$.
\end{definition}

\begin{definition}\label[definition]{de:epsilon}
For even clusters $\s$ fix a choice of $\theta_\c = \sqrt{c_f\prod\nolimits_{r \notin \c} (z_\c-r)}$, where $z_s$ is some (any) centre for $\s$. 
If $\c$ is either even or a cotwin, define 
$\epsilon_\s:G_K\to \{\pm 1\}$ by
$$
  \epsilon_{\c}(\sigma) \equiv \frac{\sigma(\theta_{\c^*})}{\theta_{\neck{(\sigma\s)}}} \mod \mathfrak{m}.
$$
For all other clusters $\s$, set $\epsilon_\s(\sigma)=0$. 
\end{definition}

\begin{remark}
Note that $\epsilon_\s$ does not depend on the choice of centre $z_\s$: if $z'_\s$ is another centre and $r \notin \s$ then $v(z_\s-r) < v(z_\s-z'_\s)$ so the leading term in the $p$-adic expansion of $z'_s-r$ is the same as that of $z_\s-r$. 
Moreover, $\epsilon_\s$ restricts to a character on the stabiliser of $\s$ that is also independent of the choice of the sign of $\theta_\s$; this character is unramified if and only if $|I_K/I_\s| (v(c_f)+\sum_{r\notin\s}d_{r\wedge\s})$ is even, where $I_\s<I_K$ is the stabiliser of $\s$.
\end{remark}

\begin{theorem}[see Theorem \ref{th:Homology}, Corollary \ref{numcomponents}]
\label{introhomology}
Suppose $C/K$ is semistable. Let $\Upsilon_C$ denote the dual graph of the special fibre of the minimal regular model of $C$ over $\OKnr$,
with its natural action of $\Gal(\kbar/k)$.
Let $A$ be the set of even non-\ub\ clusters excluding $\cR$.
Then
\begin{enumerate}
\item 
$
rk_\Z( H_1(\Upsilon_C,\Z)) = \bigleftchoice {\#A-1}{\mbox{ if } \cR \mbox{ is \ub,}}{\#A}{\mbox{ otherwise.}}
$
\item The number of components in the special fibre is
 $$
 m_C = \sum_{\substack{\s \neq \cR, \\ \text{odd, proper}}} \frac{\delta_{\s}}{2} +  \sum_{\substack{\s \neq \cR, \\ \text{even}}} 2\delta_{\s} +1-\rk H_1(\Upsilon_C,\Z).
 $$
\item $H_1(\Upsilon_{C},\ZZ) = \Bigl\{ \sum_{\c\in A} a_{\c}\ell_{\c} \Bigm|\> a_\s\in\Z, \>\>\sum_{\c\in B} a_{\c}=0\Bigr\},$
where $B$ is the subset of clusters $\c \in A$ such that $\c^*=\cR$. 
\item the length pairing is given by
$$
\langle \ell_{\c_1},\ell_{\c_2} \rangle=\left\{
    \begin{array}{ll}
        0&  \mbox{ if } \c^*_1 \neq \c^*_2, \\
        2(d_{(\c_1\wedge\c_2)}-d_{P(\c^*_1)})& \mbox{ if } \c^*_1 =\c^*_2 \ne \cR, \\
        2(d_{(\c_1\wedge\c_2)}-d_{\cR)}& \mbox{ if } \c^*_1 =\c^*_2 = \cR.\\
    \end{array}
\right.
$$
\item $\sigma \in \Gal(\bar{k}/k)$ acts
on $H_1(\Upsilon_C,\Z)$ by $\sigma(\ell_{\c}) = \epsilon_{\c}(\sigma) \ell_{\sigma(\c)}.$
\end{enumerate}
\end{theorem}

Our description of the special fibre of the minimal regular model of $C$ also allows us to easily determine whether $C$ is deficient, i.e. has no $K$-rational divisor of degree $g-1$. Deficiency is used to determine whether the Tate-Shafarevich group of the Jacobian of a curve over a number field has square order (see Section 8 of \cite{PS}).

\begin{theorem}[=Theorem \ref{th:deficiency}]\label{introdeficiency}
Suppose $C/K$ is semistable. Then
$C$ is deficient if and only if it has even genus and either
\begin{enumerate}
\item $\cR = \s_1 \coprod \s_2$ with $\s_1, \s_2$ odd, $G_K$-conjugate and $\delta_{\s_i}$ odd, or
\item $\cR$ is \ub\ with $\epsilon_{\cR}(\Frob) = -1$ and for all non-\ub\ $\s$ such that $\neck{\s} =\cR$, either $d_\s \notin \Z$ or the $\Frob$-orbit of $\c$ has even size, or
\item $\cR$ is a cotwin, its principal child $\r$ is \ub\ with $\epsilon_{\r}(\Frob)\! =\! -1$, and for all $\s$ such that $\neck{\s}\!=\!\r$, either $d_\s \notin \Z$ or the $\Frob$-orbit of $\s$ has even size.
\end{enumerate}
\end{theorem}

%

In fact, even for curves that are not semistable, we construct a regular model over a field $F/K$ where the curve becomes semistable. The special fibre of the minimal regular model of $C$ over $\OFnr$ comes with a natural action of $G_K$, not just $G_{\Fnr}(=I_F)$, 
see \S\ref{ss:gcurves}. We describe it explicitly in Theorems \ref{th:DualGraph} and \ref{de:gammatilde} as well.

\subsection{Galois representation}

Knowing the explicit Galois action on the special fibre lets us determine the $l$-adic Galois representation 
of arbitrary hyperelliptic curves. 
We use the following shorthand notation:

\begin{notation}
For a curve $X/k$ and a prime $l\ne p$ write
$$
\H(X)=\H(X_{\bar k},\Q_l),
$$
and similarly for curves over $K$. 
\end{notation}

\begin{notation}\label{Gs}
For a cluster $\s$ we write  $G_\s=\Stab(\s)$ for its stabiliser in $G_K$ and $I_\s<G_\s$ for the corresponding inertia subgroup.
\end{notation}

As mentioned above, if $C$ acquires semistable reduction over $F$, the full Galois group $G_K$ acts on the special fibre of the minimal regular model of $C$ over $\OFnr$. In particular, the \'etale cohomology groups $\H(\Gamma_\s)$ of the components have an induced action of the stabiliser $G_\s$, which is closely linked to the \'etale cohomology of $C$ and which we are able to control explicitly.
We obtain the following description.

%
%
%
%

\begin{notation}\label{de:lambdatilde}
For a cluster $\s$ we define $\so$ to be the set of odd children of~$\s$ and write
$$
\tilde\lambda_\s 
= \frac 12 (v(c_f) +  |\so|d_{\s} +  \sum_{r\notin\s}d_{\s\wedge r}).
$$
\end{notation}

\begin{theorem}[= Theorem \ref{het}]
\label{hetmain}
Let $C/K$ be a hyperelliptic curve. Let $\H(C)=\H(C)_{ab}\oplus \H(C)_{t}\tensor\Sp_2$ be the decomposition into `toric' and `abelian' 
parts. 
Then 
$$
  \H(C)_{t} =\bigoplus_\s \Ind_{G_\s}^{G_K} \!\epsilon_\s \quad\ominus\epsilon_\cR,
$$
the sum taken over representatives of $G_K$-orbits of even non-\"ubereven clusters. 
Writing $\Gamma_\s$ for the components of the special fibre in Theorem \ref{introspecial} over a Galois extension where $C$ acquires semistable reduction, we have
$$
  \H(C)_{ab}=\bigoplus_\s \Ind_{G_\s}^{G_K} \H(\Cmpn_\s),
$$
the sum taken over representatives of $G_K$-orbits of principal non-\"ubereven 
clusters. For every such cluster $\s$, there is an isomorphism of $I_\s$-modules
$$
  \H(\Cmpn_\s) \>\>\iso\>\> \tilde\gamma_s \otimes (\Q_l[\so] \ominus \triv) \quad\ominus \epsilon_\s,  
$$
where $\tilde\gamma_\s\colon I_\s\to \bar\Q_l^\times$ is any character%
\footnote{Although $\tilde\gamma_\s$ is $\bar\Q_l^\times$-valued, 
the tensored representation is realisable over $\Q_l$, and we view it as 
a $\Q_l$-representation in this formula}
whose order is the prime-to-$p$ part of the denominator of 
$|I_K/I_\s|\,\tilde\lambda_\s$.
\end{theorem}

When $C/K$ is semistable one can in fact 
also recover the Frobenius action on $\H(\Cmpn_\s)$,
as we explain in \Cref{semsible action on cohomology remark}.
Briefly, 
one can pick the centres $z_\s$ in Theorem \ref{introspecial} to lie in $K_\s$, the field cut out by $G_\s$,
which makes $\widetilde{\Gamma}_\mathfrak{s}$ naturally a curve over the residue field $k_\s$ of $K_\s$. The action of $G_\s$ on $\H(\Cmpn_\s)$ in Theorem \ref{hetmain} is then simply via the quotient  $G_\s/I_\s=\textup{Gal}(\bar{k}/k_\s)$ and the usual action of $\Gal(\bar{k}/k_\s)$ on \'{e}tale cohomology. One may recover the Frobenius eigenvalues for this action from point counts on $\tilde\Gamma_\s$ over extensions of $k_\s$.
In a forthcoming paper \cite{weil2} we will explain how to reconstruct 
$\H(C)$ as a full $G_K$-representation from point counts even when $C/K$ is not semistable.

From the explicit description of the Galois representation we deduce the following formula for the conductor.

\begin{theorem}[Theorem \ref{condmain}, Corollary \ref{th:condexpo}]\label{th:condexpogen}
Let $C/K$ be a hyperelliptic curve. Decompose the conductor exponent of $\Jac C$ into its tame and wild parts,
$n_C = n_{C,\text{tame}}+n_{C,\text{wild}}$.

$(i)$ The wild part is given by
$$
n_{C,\text{wild}} = \sum_{r\in S} v(\Delta_{K(r)/K})-[K(r):K]+f_{K(r)/K}, 
$$
where $S$ is any set of representatives of $G_K$-orbits of $\cR$, $\Delta_{K(r)/K}$ is the 
discriminant of $K(r)$ over $K$, and $f_{K(r)/K}$ is the residue degree.

$(ii)$ The tame part is given by $n_{C,\text{tame}} = 2g-\dim \H(C)^{I_K}$ with
$$
\begin{array}{cll}
  \dim \H(C)^{I_K} &=&
  \#\{I_K\text{-orbits on }U\}\> - \cr
  &&\qquad\#\{I_K\text{-orbits on }V\} - \leftchoice{1}{\text{if $|\cR|$ and $v(c_f)$ are even,}}{0}{\text{otherwise,}}\cr
U &=& \{\s\ne\cR \text{ odd cluster} \,\bigm|\,  \xi_{P(\s)}(\tilde\lambda_{P(\s)})\le \xi_{P(\s)}(d_{P(\s)}) \bigr\},\cr
V &=& \{\s \text{ proper non-\"ubereven cluster} \,\bigm|\, \xi_{P(\s)}(\tilde\lambda_{P(\s)})=0 \bigr\};\cr
\end{array}
$$
here $\xi_\s(a)=\ord_2(b)$ where $b$ is the denominator of $|I_K/I_{\s}|a$, with $\xi_\s(0)=0$.

$(iii)$ If $C/K$ is semistable then 
$$
n_C =  \#A - \leftchoice{1}{\text{if $\cR$ is \ub,}}{0}{\text{otherwise,}}
$$
where $A$ is the set of even non-\ub\ clusters 
$\s\ne\cR$. 
\end{theorem}

\subsection{Weierstrass equations}
By their very nature, clusters are designed to work with Weierstrass equations. We establish criteria in terms of clusters for an equation to be integral (Theorem \ref{th:integralpoly}) and to be a minimal Weierstrass equation, and give a formula for its discriminant. 

\begin{definition}
We say that the Weierstrass equation $y^2=f(x)$ for $C$ is \textit{integral} if $f(x)\in \cO_{K}[x]$. Its discriminant is 
$\Delta_C=  16^g c_f^{4g+2}\text{disc}\bigl(\frac{1}{c_f}f(x)\bigr)$. We say that such an equation is \textit{minimal} if $v(\Delta)$ is minimal amongst all integral Weierstrass equations for~$C$. 
\end{definition}

\begin{remark}
One can consider more general Weierstrass equations for $C$ of the form $y^2+Q(x)y=P(x)$ for $Q,P$ polynomials of degree at most $g+1$ and $2g+2$ respectively, and define integral and minimal equations accordingly. Since we work in odd residue characteristic, we can always find a minimal equation with $Q(x)=0$. 

Our notion of integral Weierstrass equation differs slightly from that of Liu used in \cite[Definition 2]{Liu}. However the resulting notion of minimal equation and minimal discriminant (Definition 3 of op. cit.) is easily seen to coincide. Several additional notions of minimal discriminant appear in the literature for both hyperelliptic curves and more general curves. See for example \cite{Kau}, \cite{Sri} and \cite{Sai} for a discussion of these, and \cite{LiuB}  for the relationship between them for hyperelliptic curves of genus 2. 
\end{remark}

\begin{theorem}[see Theorem \ref{th:mainmini}]\label{introminimal}
Suppose $C:y^2=f(x)$ is a semistable hyperelliptic curve over $K$ with $f(x)\in\cO_K[x]$, and that $|k|>2g+1$. 
Then $C$ defines a minimal Weierstrass equation if and only if one of the following conditions hold:
\begin{enumerate}
\item there are two clusters of size $g+1$ that are swapped by Frobenius, $d_\cR=0$ and $v(c_f)\in\{0,1\}$,
\item there is no cluster of size $>\!g\!+\!1$ with depth $>0$, but there is some $G_K$-stable cluster $\s$ of size $|\s|\ge g+1$ with  $d_\s\ge 0$ and $v(c_f)=-\sum_{r\notin\s}d_{r\wedge\s}$.
\end{enumerate}
\end{theorem} 
\noindent 
Note that even if $C$ is not semistable or if $|k|\le2g\!+\!1$ but $y^2\!=\!f(x)$ satisfies (2) in the above theorem, then it is automatically a minimal Weierstrass equation, as it becomes minimal after a finite field extension.
In particular (taking $\s\!=\!\cR$ and $d_\cR\!=\!0$ in (2)), we immediately obtain the following general criterion for the equation to be minimal:

\begin{corollary}
Let $C:y^2\!=\!f(x)$ be a hyperelliptic curve over $K$ with $f(x)\in\cO_K[x]$.
If $f(x)\!\mod \m$ has at least two distinct roots in $\bar{k}$, but no root of multiplicity $>\!g\!+\!1$, and the leading coefficient of $f(x)$ is a unit, then this is a minimal Weierstrass equation.

\end{corollary}

We also obtain the following result on the discriminant.

\begin{theorem}[=Theorem \ref{thmmindisc}]
\label{thmmindisc}
Let $C/K$ be a hyperelliptic curve. The valuation of its discriminant $\Delta_C$ is given by \vskip-10pt
$$
 v(\Delta_C)=
   v(c_f)(4g+2) + \sum_{\s \text{ proper}} d_\s \Bigl(|\s|^2-\mathop{\Sigma}\limits_{\s'<\s}|\s'|^2\Bigr).
$$
If $C/K$ is semistable and $|k|>2g\!+\!1$, then the valuation of the discriminant $\Delta_C^{min}$ of a minimal Weierstrass model of $C$ is determined by the formula
$$
  \frac{v(\Delta_C)-v(\Delta_C^{min})}{4g+2}=v(c_f)-E+ d_{\cR}(|\cR|-g-1))+\sum_{g+1<|\s|<|\cR|}\delta_\s   (|\s|-g-1)),
$$
where $E=0$ unless there are two clusters of size $g+1$ that are permuted by Frobenius and $v(c_f)$ is odd, in which case $E=1$.
\end{theorem}

\subsection{Cluster pictures}

As we discussed at the beginning, we extract most of our arithmetic 
data purely from the `cluster picture' of the roots of $f(x)$. 
Effectively this is how we propose to think about hyperelliptic curves over local 
fields. To formalise this slightly, we consider an abstract cluster picture 
purely combinatorially, without reference to roots of a polynomial.

\begin{definition}
\label{defclpic}
Let $X$ be a finite set and $\Sigma$ a collection of non-empty subsets of $X$;
elements of $\Sigma$ are called \emph{clusters}. 
Attach \emph{depth} $d_\s\in\Q$ to every cluster $\c$ of size $> 1$.
Then $\Sigma$ (or $(\Sigma,X,d)$) is a \emph{cluster picture} if
\begin{enumerate}
\item
Every singleton (`root') is a cluster, and $X$ is a cluster. 
\item
Two clusters are either disjoint or contained in one another.
\item
$d_\t >d_\s$ if $\t \subsetneq \s$.
\end{enumerate}
Two cluster pictures $(\Sigma,X,d)$ and $(\Sigma',X',d')$ are isomorphic if there is a bijection $\phi\colon X\to X'$ which induces a bijection from $\Sigma$ to $\Sigma'$ and $d_\s = d'_{\phi(\s)}$. 
\end{definition}


For a polynomial $f(x)\in K[x]$ or a hyperelliptic curve $C\colon y^2=f(x)$,
the cluster picture $\Sigma_f$ or $\Sigma_C$ is the collection of all clusters of the roots of $f$,
as in Definition \ref{defclusters}.
%
%
%
%
%
There is then a purely combinatorial notion of ``equivalence'' 
of cluster pictures (Definition \ref{realiseequiv}) that keeps track of isomorphisms of curves:

\begin{theorem}[Theorem \ref{th:isoequiv}, Corollary \ref{realiseequiv}]
If $C$ and $C'$ are isomorphic hyperelliptic curves over $K$, then their cluster pictures are equivalent.
Conversely, if an abstract cluster picture $\Sigma$ is equivalent to $\Sigma_C$ for some hyperelliptic curve $C$, 
then there is a hyperelliptic curve $C'/\bar{K}$ that is $\bar{K}$-isomorphic to $C$ and 
whose cluster picture is $\Sigma$.
\end{theorem}
In fact, in every equivalence class of cluster pictures there is a 
canonical representative, a `balanced cluster picture' (see Lemma \ref{balanced}). 
For semistable hyperelliptic curves we explain how to find an isomorphic curve 
with such a cluster picture; see 
Theorem \ref{Ombalanseringsatz} and Corollary \ref{co:evengenus}.

As we have already seen, also keeping track of the Frobenius action on clusters and $\epsilon$ lets one control 
many arithmetic invariants: 

\begin{theorem}[see Theorem \ref{th:classification}, Lemma \ref{le:reductiontypeGeneral}, Theorem \ref{thmmindisc}]
The cluster picture of a semistable hyperelliptic curve, together with the action of Frobenius on clusters and the values of $\epsilon_\s(\Frob)$ for even clusters $\s$,
determines the dual graph of the special fibre of its minimal regular model (with genera of components and Frobenius action), the conductor exponent, whether the curve is deficient, and the Tamagawa number and root number of the Jacobian. If $|k|>2g+1$, it also determines the valuation of the minimal discriminant of the curve.
\end{theorem}

We introduce a notation for cluster pictures that carries this extra data (Notation \ref{no:adddata}).
In view of the semistability criterion (Theorem \ref{redcond}(1)), it is easy to list and classify
all possible cluster pictures that correspond to semistable hyperelliptic curves. 
We give an example of such a classification by considering semistable genus 2 curves (Theorem \ref{th:unbal}, Table \ref{tb:unbal}) and give their associated 
invariants (Theorem \ref{th:g2bible}, Table \ref{tb:genus2bible}).

Finally observe that the cluster picture is a fairly coarse invariant 
of the polynomial $f(x)$. In particular small perturbations of the coefficients 
of $f(x)$ will not change its associated cluster picture and hence many 
of the invariants of the corresponding curve. Here is a precise statement 
(see Theorem \ref{continuity} for a more general result).

\begin{theorem}[=Corollary \ref{co:continuity}]\label{introcontinuity}
Suppose $C_1\colon y^2 = c_1g_1(x)$ and $C_2\colon y^2 = c_2g_2(x)$ are two hyperelliptic curves with $c_1, c_2 \in K^{\times}$ and $g_1(x), g_2(x) \in \cO_K[x]$ monic polynomials.
If $\frac{c_1}{c_2} \in K^{\times 2}$ and $g_1(x) \equiv g_2(x) \mod \pi^{d+1}$ where $d$ is the largest depth among the depths of all proper clusters of $C_1$, then 
\begin{itemize}[leftmargin=*]
\item $\H(C_1) \iso \H(C_2)$ as $G_K$-modules for every $l\ne p$, and $C_1$ and $C_2$ have the same conductor exponent and the same root number. 
\item If $C_1$ is semistable then so is $C_2$. In this case, the special fibres of their minimal regular 
models over $\cO_{K^{nr}}$ are isomorphic as curves with an action of Frobenius,
their Jacobians have the same Tamagawa number, $C_2$ is deficient if and only if $C_1$ is and, 
if $|k|>\deg g_1(x)$, the valuations of their minimal discriminants are equal.
\end{itemize}
\end{theorem}

\subsection{Related work} \label[section]{related work}

A similar construction to our clusters was already used by Bosch \cite[\S5]{Bosch} to determine the stable type of hyperelliptic curves. 

Based on the present article, for semistable hyperelliptic curves Betts \cite{Betts} has given a description of the Tamagawa number of the Jacobian in terms of clusters, and Kunzweiler (work in progress) has found a simple formula for the other local ``fudge factor'' $|\omega/\omega^o|$ in the Birch--Swinnerton-Dyer formula for the $L$-value, as well as a description of an $\cO_{K^{\textup{nr}}}$-basis of integral differentials on the minimal regular model of $C$. Faraggi and Nowell have found a way to describe the special fibre of the minimal SNC model from the cluster picture for all hyperelliptic curves with tame reduction (work in progress). 

For the purposes of applications, a key feature of cluster pictures is that they let one study whole classes of hyperelliptic curves. For instance \cite{CelineThesis} uses a case-by-case analysis of cluster pictures as an ingredient in establishing the parity conjecture for a large class of abelian surfaces, and \cite{AD} uses them to construct explicit hyperelliptic curves of arbitrarily high genus over number fields whose Galois image on $\Jac C [\ell]$ is the maximal possible.

We would like to alert the reader to a forthcoming paper \cite{hyxit}, where
the authors plan to summarise various results on hyperelliptic curves, and illustrate
them with examples. This explains why examples are virtually non-existent in this,
already too long, paper.

\subsection{Layout} \label[section]{layout}

In \S\ref{sbackground} we review some facts about models of general curves over local
fields and \'etale cohomology. 

In \S\ref{models of P1} we show that cluster configurations 
(or, rather, certain collections of discs) give rise to regular model of $\P^1_{K^\textup{nr}}$
and describe its properties (Proposition \ref{main model properties}).
In \S\ref{ydisc_section} and \S\ref{scdisc}, we study double covers of those models, 
and deduce explicit regular models of hyperelliptic curves that satisfy 
the semistability criterion \ref{semistability criterion}. 
This approach is similar to that of Kausz \cite{Kau}, and has also been exploited by  Bouw-Wewers \cite{BW} and Srinivasan \cite{Sri}, though each of these works in a slightly different generality to us. In particular, we construct the models under the assumptions of the semistability criterion and so in particular do not assume that all Weierstrass points are rational. 
We find the minimal regular model
(\S\ref{ssminregmod}), describe the dual graph of its special fibre 
(Theorem \ref{main dual graph thingy}), give explicit equations 
for its components (Theorem \ref{components of min}) and describe the reduction map
from the generic to the special fibre (Prop. \ref{reducing points});
we deduce the stable model~in~\S\ref{stable model subsection}.

In \S\ref{sgalois} we turn to non-semistable curves and study the natural 
Galois action on the model that we have over an extension where the curve 
becomes semistable. We then deduce the semistability criterion 
(Theorem \ref{the semistability theorem}), and in \S\ref{sfibre} describe the Galois action
on the special fibre in terms of clusters. 
In \S\ref{shomology} we extract the homology of the dual graph of the special fibre
(Theorem \ref{th:Homology}), and, consequently, the toric part of the 
\'etale cohomology (Corollary \ref{co:EtaleCohoToric}). The abelian part 
is addressed in \S\ref{sgalrep}, and as a consequence, we get 
Theorems \ref{hetmain} (=\ref{galois rep theorem}) and 
\ref{redcond} (=\ref{reduction}). 
In \S\ref{sconductor} we then find the formula for the conductor,
and classify deficient curves in \S\ref{sdeficiency}.


In \S\ref{s:discriminant} and \S\ref{s:MWEquation} we study the discriminant
and the minimal Weierstrass equation of a hyperelliptic curve,
proving Theorems \ref{introminimal} (=\ref{th:mainmini})
and \ref{thmmindisc} (=\ref{thmmindisc}). This is primarily a combinatorial 
cluster yoga, relying on the semistability criterion to convert semistability into
cluster language. 
In \S\ref{se:bible} we propose a notion of a `reduction type' of a semistable curve, 
and give classfication in genus 2. 
In \S\ref{scontiunity} we study the variation of the coefficients of a curve that
does not affect its primary arithmetic invariants, and prove 
Theorem \ref{continuity} and 
Corollary \ref{introcontinuity} (=\ref{co:continuity}).

In Appendix \ref{ap:apphyp} we review affine automorphisms of (possibly singular)
hyperelliptic equations. Appendices \ref{s:app:centres} and \ref{s:app:equi}
prove some technical results concerning centres of clusters and equivalent forms 
of the semistability criterion. 
Finally, Appendix \ref{hyble appendix} links the results of this paper to its
combinatorial predecessor \cite{hyble}.

%

\subsection{Notation}


For the reader's convenience, the following tables gather the general notation and terminology that are used throughout the paper. We reserve gothic letters $\s$, $\t$, $\s_1$ etc. for clusters (except for ``mod $\mathfrak{m}$''). Tables \ref{fieldtable} and \ref{curvetable} list the general notation associated to fields and hyperelliptic curves. Tables \ref{clustertable} and \ref{NotationClusters} summarise the notation and terminology associated to a cluster $\s$, and the main functions and invariants associated to clusters. Table \ref{Adamtable} presents the main notation associated to a disc $D$, as used in \S \ref{models of P1}--\ref{scriterion}.


\begin{table}[H]
\caption{General notation associated to fields}
\begin{tabular}{ll}
$K$      & local field of odd residue characteristic \cr
$\cO_K$ & ring of integers of $K$\cr
$k$      &  residue field of $K$\cr
$\bar{k}$ & algebraic closure of $k$\cr
$\pi$ & uniformiser of $K$ \cr
$v$     & normalised valuation with respect to $K$ \cr
$\Knr$      & maximal unramified extension of $K$ \cr
$\Ks$ &  separable closure of $K$\cr
$\Kbar$ & algebraic closure of $K$ \cr
$G_K$ & $\Gal(\Ks/K)$ \cr
$I_K$ & inertia subgroup of $G_K$ \cr
$\Frob$ &a choice of Frobenius element in $G_K$\cr
$\Sp_2$ & representation $G_K\to\GL_2(\Q_l)$, given by $\sigma\mapsto \smallmatrix 1{\tau(\sigma)}01$ for $\sigma\in I_K$, where \cr
& $\tau: I_K\to\Z_l$ is the $l$-adic tame character, 
and $\Frob\mapsto \smallmatrix 100{|k|^{-1}}$; see \cite[4.1.4]{TatN}\cr
$\hat{c}$ & $=\frac{c}{\pi^{v(c)}}$ \cr
$F$ & a finite Galois extension of $K$ where $C$ is semistable\cr
$\pi_F$ & uniformiser of $F$, \cr
$\chi$ &$\chi(\sigma)= \frac{\sigma(\pi_F)}{\pi_F} \mod \m,$ for $\sigma \in G_K$, see Definition \ref{de:characters}\cr 
$e$ & residue degree of a finite Galois extension $F/K$\cr
mod $\mathfrak{m}$ & reduction to the residue field \cr
\cr
\end{tabular}
\label{fieldtable}
\end{table}


\begin{table}[H]
\caption{General notation associated to a hyperelliptic curve}
\begin{tabular}{ll}
$C$ & hyperelliptic curve given by $y^2 = f(x)$ \cr
$c_f$ & leading term of $f(x)$ \cr
$\cR$ & set of roots of $f(x)$ in $\Ks$ \cr
$g$ & genus of $C$ \cr
$\Delta_C$ & discriminant of $C$\cr
$\Delta_C^{min}$ & discriminant of a minimal Weierstrass equation of $C$ \cr
$\cC_{min}$ & minimal regular model of $C$ \cr
$\cC_{\text{disc}}$ & regular model of $C$ of \Cref{main regular model theorem} \cr
$\cC_{\text{st}}$ & stable model of $C$ \cr
$\cC_{min,\bar{k}}$ & special fibre of $\cC_{min}$ \cr
$\Upsilon_C$ &dual graph of $\cC_k$\cr
$\H(C)$ & $\H(C_{\Kbar},\Q_l)$\cr
Jac $C$ & Jacobian of $C$ \cr
$V_l A$ & $T_l A \otimes \Q_l$, where $T_l A$ denotes the $l$-adic Tate module of A \cr
$\iota$ & hyperelliptic involution \cr
\cr
\end{tabular}
\label{curvetable}
\end{table}

\vfill

\renewcommand{\arraystretch}{1.2}%
\begin{table}[H]  \label[table]{clusternotation}
\caption{Terminology and notation for clusters}
\begin{tabular}{lll}
\multicolumn{1}{l}{Terminology and notation} & Definition & Example\\
\multicolumn{1}{l}{for a cluster $\s\subset\cR$}\\
\hline
child of $\s$, $\s'<\s$ & maximal subcluster of $\s$&  \smash{\childparent}\\
parent of $\s$, $P(\s)$ & cluster $P(\s)$ in which $\s$ is maximal &   \smash{\childparentn}\\
proper cluster & cluster of size $>$ 1&  \triple \\
even cluster &cluster of even size&
\event\\
odd cluster  &cluster of odd size (e.g. singleton)&
\oddt\\
\"ubereven cluster  &even cluster with only even children&
\ubereven\\
twin  &cluster of size 2&
\twin\\
cotwin & cluster with a child of size $2g$&
\cotwino{} or\\
&whose complement is not a twin&\cotwint\\
principal cluster& proper, not a twin or a cotwin and if\\
& $|\c| = 2g+2$ then $\s$ has $\ge 3$ children\\
 $\neck{\c}$& smallest cluster $\neck{\c} \supseteq \c$ that  &\clusternecks\\
& does not have an \ub\ parent &\\
& or proper child of $\s$ if $\s$ cotwin\\
 $\c \wedge \c'$ & smallest cluster containing $\c$ and $\c'$& \clusterlcas\\ 
\end{tabular}
\label{clustertable}
\end{table}
\renewcommand{\arraystretch}{1}%

\vfill




\begin{table}[H]
\caption{Functions and invariants associated to clusters}
\begin{tabular}{ll}
$\delta_\c$ & relative depth of $\s$, $\delta_\s\!=\!d_\c\! -\! d_{\Pa(\c)}$ \hfill (
\ref{de:relativedepth})\cr 


$\so$ & set of odd children of  $\s$ \hfill (
\ref{de:lambdatilde})\cr
  
$g(\c)$ & genus of $\s$: $|\so|\!=\!2g(\c)\!+\!1$ or $2g(\c)\!+\!2$, or $g(\c)\!=\!0$ if $\c$ is \"ubereven \hfill (
\ref{se:specialfibre})\cr

$G_\s$ &stabiliser of $\s$ in $G_K$ \hfill (
\ref{Gs})\cr

$I_\s$ & inertia subgroup of $G_\s$ \hfill (
\ref{Gs}) \cr

$z_\s$ & (choice of) centre of $\s$; $z_\s\in \Ks$ with $\min_{r\in\s}v(z_\s-r)=d_\s$\hfill (
\ref{de:centre})\cr

$\red_\s$ & reduction map relative to $\s$: $\red_\s(x)=\frac{x-z_\s}{\pi^{d_\s}}\mod\m$ \hfill (
\ref{de:gammas})\cr

$c_\s$ & $\hat{c}_f \prod_{r \notin \c}\widehat{(z_\c-r)}  \mbox{ mod } \m $ \hfill (
\ref{de:gammas})\cr

$\Cmp_s$ &component of $\cC_k$ associated to $\s$ \hfill (
\ref{de:gammas})\cr

$\tilde{\Cmp}_\s$ & normalisation of $\Gamma_\c$; $y^2=c_\s\prod_{\mathfrak{o}\in\so }(x-\red_\s(z_\mathfrak{o}))$  \hfill (
\ref{de:gammatilde})\cr

$\nu_\s$ & $=v(c_f)+\sum_{r\in\cR} d_{r\wedge \s}$ \hfill(
\ref{no:nu})\cr


$\tilde{\lambda}_\s$ & 
$
= \frac{\nu_{\c}}{2} -  d_{\s}\!\sum_{\s'<\s} \lfloor \frac{|\c'|}{2}\rfloor $  \hfill (
\ref{de:lambdatilde}) \cr

$ \alpha_\c$ &$\alpha_\c(\sigma)=\chi(\sigma)^{d_\c}$ for $\sigma \in G_K$ \hfill (
\ref{de:characters}) \cr

$\beta_{\c}$ & $\beta_{\c}(\sigma)=\frac{\sigma(z_{\c})-z_{\sigma\c}}{\pi^{d_{\c}}} \mod \m$ for $\sigma \in G_K$ \hfill (
\ref{de:characters}) \cr

$ \gamma_\c$ & $\gamma_\c(\sigma)=\chi(\sigma)^{\tilde{\lambda}_\c}$ for $\sigma \in G_K$ \hfill (
\ref{de:characters}) \cr

$\epsilon_\s$  &$ \epsilon_\s(\sigma)=\frac{\sigma(\theta_{\c^*})}{\theta_{\neck{(\sigma\s)}}} \mod \m \in\{\pm1\}$ 
if $\s$ even or cotwin, \cr 
& $\epsilon_\s(\sigma) =0$ otherwise;  here $\theta_\c = \sqrt{c_f\prod\nolimits_{r \notin \c} (z_\c-r)}$ and $\sigma\in G_K$  \hfill (
\ref{de:characters})\cr
\cr
\end{tabular}
\label{NotationClusters}
\end{table}


\begin{table}[H]
\caption{General notation associated to discs}
\begin{tabular}{lllll}
$z_D$ &  a choice of centre for $D$ &\phantom{.................}& $\mathcal{Y}_{\text{disc}}$ & \hfill (\ref{the divisor B}) \cr
$P(D)$ & parent disc of $D$ \hfill (\ref{disc notation})&&$E_D$ & \hfill (\ref{the model of P1}) \cr
$\lambda_D$ & \hfill (\ref{the characters defi})&&$\widehat{\Upsilon_C}$ & \hfill (\ref{main dual graph thingy}) \cr
$\nu_D$ & \hfill (\ref{nu and kappa sub})&&$\mathcal{P}_D, \mathcal{Q}_D$ & \hfill (\ref{model components})\cr
$\kappa_D$ & \hfill (\ref{kappa def})&&
$\mathcal{U}_D, \mathcal{W}_D$ &\hfill (\ref{the charts of Cdisc})\cr
$\red_D$ & \hfill (\ref{model components})&&
$c_D$ & \hfill (\ref{the leading term defi}) \cr
$D_{max}$ & \hfill (\ref{valid discs defi}) &&Type I to VI & \hfill (\ref{types of valid disc}) \cr
$f_D, g_D$ & \hfill (\ref{defi of the polynomials})&&$\omega_D(f)$ & \hfill (\ref{the divisor B}) \cr
$U_D, W_D$ & \hfill (\ref{schemes defi}) &&$\Gamma_D$ & \hfill (\ref{components defi})\cr
$D(\s)$ & \hfill (\ref{no:disc}) &&$h_D$ & \hfill (\ref{the poly h def})\cr
Valid disc & \hfill (\ref{valid discs defi}) &&$\beta_D, \lambda_D$ & \hfill (\ref{the characters defi})\cr
\cr
\end{tabular}
\label{Adamtable}
\end{table}

\subsection*{Acknowledgements}
We would like to thank the Warwick Mathematics Institute
, where parts of this research were carried out, and Matthew Bisatt for helpful comments.
This research is supported by EPSRC grants EP/M016838/1 and EP/M016846/1 `Arithmetic of hyperelliptic curves'.
The second author is supported by a Royal Society University Research Fellowship.

\section{Curves and Jacobians over local fields}
\label{sbackground}

\def\H{H^1_{\text{\'et}}}
\def\Hn{H^n_{\text{\'et}}}
\def\Hgen#1{H^{#1}_{\text{\'et}}}
\def\cA{{\mathcal A}}
\def\B{{\mathcal B}}
\def\O{{\mathcal O}}
\def\cI{{\mathcal I}}
\def\T{{\mathcal T}}
\def\Gr#1{\text{\rm gr}_#1}

In this section we review some facts about models of curves over local fields and \'etale cohomology.
We refer the reader to \cite[\S2]{CFKS}, \cite{Gro}, \cite{BW}, \cite{ST}, and especially 
\cite{skew} for details. 
All of this is standard, except we want the residue field to be non-algebraically closed, and so have to
keep track of the Galois action throughout the section. 

\medskip

Let $K$ be a local\footnote{In fact, here and in \S\ref{ss:sjacs}, $K$ could be 
any complete discretely valued field with perfect residue field.} field, with uniformiser $\pi$ and  
residue field $k$.
Suppose $C/K$ is a non-singular projective curve, of genus $g\ge 2$.

A \emph{model} of $C/K$ is a flat proper scheme
$\cC/\cO_K$ together with a $K$-isomorphism of its generic fibre with $C$. 
It is a \emph{regular model} if $\cC$ is regular, and such a model can always be obtained from 
a given model by repeated blowups. Among regular models, there is a unique one 
dominated by all the others, the
%
\emph{minimal regular model}
$$
  \cC_{\min} \lar \Spec \cO_K.
$$

A model is \emph{semistable} if its special fibre $\cC_k$ is geometrically reduced and has only ordinary double points as 
singularities, and when such a model exists we say that $C/K$ is semistable or has \emph{semistable reduction}.
Such a model always exists over some finite extension $F/K$ \cite{DM}. When one exists over $K$, the minimal regular model is semistable as well, and
blowing down certain components of the special fibre yields a \emph{stable model}
$$
  \cC_{\textup{st}} \lar \Spec \cO_K, 
$$
characterised among semistable models by the fact that its special fibre has a finite 
automorphism group (i.e. it is a stable curve).
It is again unique, though it is not necessarily regular, and it commutes with base change,
as opposed to the regular model.

\medskip

\def\ecpic#1#2{
  \path($(7,6)+(#1,#2)$)  coordinate (p1) {}; 
  \path($(5,9)+(#1,#2)$)  coordinate (p2) {}; 
  \path($(4,10)+(#1,#2)$) coordinate (p3) {}; 
  \path($(5,11)+(#1,#2)$) coordinate (p4) {}; 
  \path($(7,14)+(#1,#2)$) coordinate (p5) {}; 
  \path[draw,-,thick] (p1) edge[out=100,in=40] (p2);
  \path[draw,-,thick] (p2) edge[out=220,in=270] (p3);
  \path[draw,-,thick] (p3) edge[out=90,in=140] (p4);
  \path[draw,-,thick] (p4) edge[out=320,in=260] (p5);
}
\def\cusppic#1#2{
  \path($(7,6)    +(#1,#2)$) coordinate (p1) {}; 
  \path($(6.8,8)  +(#1,#2)$) coordinate (p2) {}; 
  \path($(5.5,10) +(#1,#2)$) coordinate (p3) {}; 
  \path($(6.8,12) +(#1,#2)$) coordinate (p4) {}; 
  \path($(7,14)   +(#1,#2)$) coordinate (p5) {}; 
  \path[draw,-,thick] (p1) edge[out=100,in=270] (p2);
  \path[draw,-,thick] (p2) edge[out=90,in=0] (p3);
  \path[draw,-,thick] (p3) edge[out=0,in=270] (p4);
  \path[draw,-,thick] (p4) edge[out=90,in=260] (p5);
}
\def\IVpic#1#2{
  \path($(5,6)    +(#1,#2)$) coordinate (p1) {}; 
  \path($(5,14)   +(#1,#2)$) coordinate (p2) {}; 
  \path($(3.5,6.4)    +(#1,#2)$) coordinate (p3) {}; 
  \path($(6.5,13.6)   +(#1,#2)$) coordinate (p4) {}; 
  \path($(6.5,6.4)    +(#1,#2)$) coordinate (p5) {}; 
  \path($(3.5,13.6)   +(#1,#2)$) coordinate (p6) {}; 
  \path[draw,-,thick] (p1)--(p2);
  \path[draw,-,thick] (p3)--(p4);
  \path[draw,-,thick] (p5)--(p6);
}
\def\Qpmark#1#2#3#4{
  \path[draw,-,thin] ($(4,4)+(#1,#2)$)--($(13.5,4)+(#1,#2)$);
  \path($(6,4)+(#1,#2)$)    node[circle,scale=0.5,ball color=black!50](0,0){};
  \path($(11.5,4)+(#1,#2)$) node[circle,scale=0.3,ball color=black!50](0,0){};
  \path($(6,3)+(#1,#2)$) node(0,0) {$\scriptstyle #3$};
  \path($(11.5,3)+(#1,#2)$) node(0,0) {$\scriptstyle #4$};
}
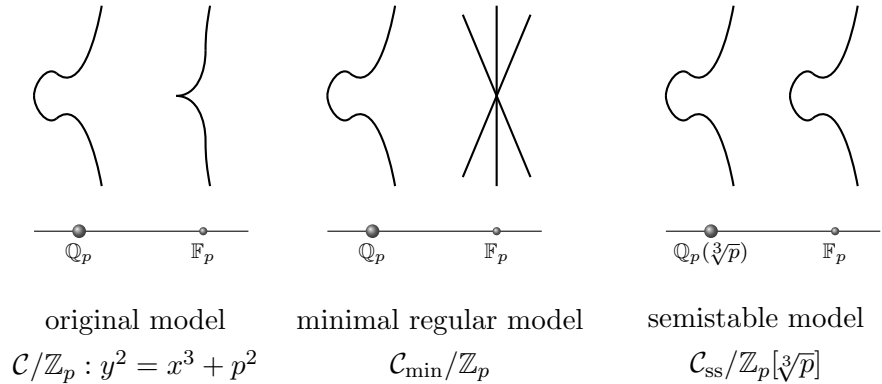
\begin{figure}[!htbp]
\begin{tikzpicture}[scale=0.3]
  \ecpic{2}{0}
  \cusppic{6.8}{0}
  \ecpic{15}{0}
  \IVpic{21.5}{0}
  \ecpic{30}{0}
  \ecpic{35.5}{0}
  \Qpmark{2}{0}{\Q_p}{\F_p}
  \Qpmark{15}{0}{\Q_p}{\F_p}
  \Qpmark{30}{0}{\Q_p(\!\sqrt[3]p)}{\F_p}
  \path(10.5,-2) node(0,0) {$\cC/\Z_p: y^2=x^3+p^2$};
  \path(24,-2) node(0,0) {$\cC_{\min}/\Z_p$};
  \path(38,-2) node(0,0) {$\cC_{\ss}/\Z_p[\!\sqrt[3]p]$};
  \path(10.5,0) node(0,0) {original model};
  \path(24,0) node(0,0) {minimal regular model};
  \path(38,0.2) node(0,0) {semistable model};
\end{tikzpicture}
\caption{Models of $y^2=x^3+p^2$ over $\Q_p$ and $\Q_p(\sqrt[3]{p})$}
\end{figure}

\begin{example}
\label{exbgr1}
Take $K=\Q_p$ ($p>3$), and 
$$
  C/K: y^2=x^3+p^2,
$$ 
an elliptic curve with additive reduction. The defining equation gives a model $\cC/\Z_p$ that
is neither regular (the ideal $(0,0,p)$ gives a singular point) nor semistable (the special fibre
has a cusp). Its minimal regular model $\cC_{\min}/\Z_p$ has three $\P^1$s meeting at a point as a special
fibre (Kodaira type~$\IV$). The curve $C$ becomes semistable over $\Q_p(\sqrt[3]{p})$, since
\begin{equation}\label{C3var}
  y^2 = x^3 + \pi^6 \quad \iso \quad y^2=x^3+1 \qquad (\pi=\sqrt[3]{p}),
\end{equation}
and the latter has good reduction: the special fibre is an elliptic curve over $y^2=x^3+1$ 
over $\F_p$.
\end{example}

We begin by reviewing special fibres of semistable models.

\subsection{Semistable curves over the residue field}
\label{ss:scurves}

\begingroup
\def\Cbar{Y}
\def\Cbarn{{\tilde Y}}
\def\Ish{{\mathbb I}}
\def\nmap{n}
\def\Yb{Y_{\kbar}}
\def\Ybn{\tilde Y_{\kbar}}

We follow \cite[pp. 469--474]{Gro} closely, except our description of $X(T)$ 
in \eqref{xtdesc} is slightly tweaked.

Let $\Cbar/k$ be a \emph{semistable} curve, that is
$\Cbar$ is complete, connected, not necessarily 
irreducible, and the only singularities of $\Yb=\Cbar\times_k\kbar$ are ordinary double points. Write

\medskip

\begin{tabular}{llllll}
$\nmap$      &=& normalisation map $\Cbarn\to\Cbar$,\cr
$\II$        &=& set of singular (ordinary double) points of $\Yb$,\cr
$\JJ$        &=& set of irreducible components of $\Yb$ (=connected comps. of $\Ybn$),\cr
$\KK$        &=& $\nmap^{-1}(\II)$; this comes with two canonical maps\cr
&& $\phi: \KK\to \II$, $P\mapsto\nmap(P),$\cr
&& $\psi: \KK\to \JJ$, $P\mapsto$ component of $\Ybn$ on which $P$ lies,\cr
$\Upsilon$ &=& dual graph of $\Yb$.\cr 
\end{tabular}
\smallskip

\subsubsection{Dual graph}
\label{sss:dualgraph}
By a \emph{metric graph} we mean a topological space $G$ homeomorphic
to a finite (combinatorial) graph, equipped with a set $V(G)$ of \emph{vertices} 
(containing (at least) all points $x\in G$ of degree $\ne 2$), a set $E(G)$ of \emph{edges}, and a length function $l:E(G)\rightarrow \mathbb{R}_{>0}$.
\emph{Graph isomorphisms} are homotopy classes of homeomorphisms that 
preserve vertices, edges and lengths. We allow loops and multiple edges and note that  automorphisms may permute multiple edges and reverse the direction of loops. Note that automorphisms act naturally on  the \emph{first singular homology group} $H_1(G,\Z)$.

The dual graph $\Upsilon$ is the metric graph with vertex set $\JJ$ and edge set $\II$. The set $\KK$ is the set of edge endpoints, the maps $\phi$ and $\psi$ specify adjacency, and each edge is given length 1.  It comes with a natural `genus' marking $g:V(\Upsilon)\rightarrow \mathbb{Z}_{\geq 0}$ which associates to each vertex the (geometric) genus of the component to which it corresponds.  Note that a graph automorphism of $\Upsilon$ is precisely
the data of bijections $\KK\to\KK$, $\II\to\II$ and $\JJ\to\JJ$ that commute with $\phi$ and $\psi$. 


\subsubsection{Character group}
\label{sss:chgroup}

The normalisation map $\nmap$ is an isomorphism outside $\II$,
and yields an exact sequence of sheaves on $\Cbar$,
$$
  1 \lar O_{\Cbar}^\times \lar \nmap_* O_{\Cbarn}^\times \lar \Ish \lar 0,
$$
with $\Ish$ concentrated in $\II$. Consider the long exact sequence on cohomology,
$$
  0 \to H^0(\Cbar,O_{\Cbar}^\times) \to H^0(\Cbarn,O_{\Cbarn}^\times) \to H^0(\Cbar,\Ish) \to 
    H^1(\Cbar,O_{\Cbar}^\times) \to H^1(\Cbarn,O_{\Cbarn}^\times) \to 0.
$$
The global sections of $\Ish$ are in bijection with invertible functions 
on $\KK$ modulo those pulled back from $\II$. In other words,
$$
  H^0(\Cbar,\Ish) = \coker((\kbar^\times)^\II\overarrow{\phi^*}(\kbar^\times)^\KK),
$$
where $\phi^*$ takes a function $\II\to \kbar^\times$ to $\KK\to \kbar^\times$ by 
composing it with $\phi$. With $\psi^*$ defined in the same way, 
the exact sequence above becomes
\begin{equation}\label{PicSeq}
  0 \lar \kbar^\times \lar (\kbar^\times)^\JJ \overarrow{\psi^*} 
    \frac{(\kbar^\times)^\KK}{\phi^*((\kbar^\times)^\II)} \lar \Pic \Cbar(\kbar) 
    \lar \Pic \Cbarn (\kbar) \lar 0. 
\end{equation}
In fact, there is an exact sequence of algebraic groups
\begin{equation}\label{alggpseq}
  0 \lar T \lar \Pic^0(\Cbar) \lar \Pic^0(\Cbarn) \lar 0,
\end{equation}
with $T$ the largest torus in $\Pic^0(\Cbar)$. 
Then \eqref{PicSeq} gives a canonical identification 
of its \emph{character group} $X(T)=\Hom(T_{\kbar},\G_{m,\kbar})$ as
\begin{equation}\label{xtdesc}
  X(T) = \ker(\Z^\KK\overarrow{(\phi,\psi)} \Z^\II\times\Z^\JJ).
\end{equation}
On the other hand, write the dual graph $\Upsilon$ as the union
$\Upsilon=U\cup V$, where $U$ 
is the union of open edges, and $V$ is the union of small open neighbourhoods of
the vertices. Then the Mayer-Vietoris sequence reads
$$
  0 \lar H_1(\Upsilon,\Z) \lar \Z^\KK \overarrow{(\phi,\psi)} \Z^\II\times\Z^\JJ \lar \Z \lar 0,
$$
since $H_0(U)=\Z^\II$, $H_0(V)=\Z^\JJ$, $H_0(U\cap V)=\Z^\KK$ 
and all their higher homology groups vanish. 
Therefore, the character group of $T$ and its $\Z$-linear dual are, canonically,
\begin{equation}\label{XvsUps}
  X(T)=H_1(\Upsilon,\Z), \qquad 
  X(T)^*=H^1(\Upsilon,\Z).
\end{equation}
On the level of Tate modules $T_l$, $l\ne\vchar k$, the sequence \eqref{alggpseq} becomes
\begin{equation}\label{curvetate}
  0 \lar X(T)^*\tensor_\Z\Z_l \lar T_l\Pic^0 \Cbar \lar T_l \Pic^0 (\Cbarn) \lar 0.
\end{equation}
There is a \emph{length pairing} on $H_1(\Upsilon,\Z)$: 
let $\langle e,e\rangle=1$ and $\langle e,e'\rangle=0$ for edges $e\ne e'$ of $\Upsilon$,  
and extend to singular chains by linearity. This descends to a pairing on~$H_1$,
\begin{equation}\label{curvepair}
  \langle , \rangle: H_1(\Upsilon,\Z)\times H_1(\Upsilon,\Z)\lar \Z. 
%
\end{equation}


Finally, it is clear that 
$\II$, $\JJ$, $\KK$, $X(T)$, $X(T)^*$, $H_1(\Upsilon,\Z)$, $T_l$
are all $G_k$-modules, and 
\eqref{XvsUps}, \eqref{curvetate}, \eqref{curvepair} are compatible with $G_k$-action,
as everything is canonical.

\endgroup

\subsection{Semistable Jacobians}
\label{ss:sjacs}

Now we go back to $C/K$, and suppose it has semistable reduction. 
Then the Jacobian $A=\Jac C$ is a semistable 
abelian variety over $K$. Let

\begin{tabular}{llllll}
$\cC/\cO_K$   & = & semistable model of $C$ over $\cO_K$ base changed to $\cO_{\Knr}$,\cr
&&with special fibre $\bar C/\kbar$.\cr
$\cN/\cO_K$   & = & N\'eron model of $A/\Knr$ base changed to $\cO_{\Knr}$,\cr
&&with special fibre $\bar N/\kbar$; the identity component $\bar N^0$ is $\Pic^0\bar C$.\cr
$\Phi(\kbar)$   & = & (finite) group of components $\bar N/\bar N^0$.\cr
$T/\kbar$       & = & toric part of $\Pic^0\bar C$, as in \eqref{alggpseq}.
\end{tabular}


By the work of Raynaud (\cite{Ray}, \cite{Gro} \S9),
there is a smooth commutative group scheme $\cA/\cO_{\Knr}$, 
unique up to a unique isomorphism, characterised by the following properties:
it is an extension
$$
  0 \lar \T \lar \cA \lar \B \lar 0,
$$
with $\T/\cO_{\Knr}$ a torus and $\B/\cO_{\Knr}$ an abelian scheme,
and $\cA\tensor(\O_{\Knr}/m_{\Knr}^i)$ is the identity component
of $\cN\tensor(\O_{\Knr}/m_{\Knr}^i)$.
Noting that $\B(\O_{\Knr})=\B(\Knr)$ as $\B$ is proper, from the commutative diagram
$$ 
\begin{tikzcd}[sep=1.2em]
  0 \rar& \T(\cO_{\Knr})\dar\rar& \cA(\cO_{\Knr}) \dar\rar& \B(\cO_{\Knr}) 
    \ar[d,dash,shift left=0.25ex]\ar[d,dash,shift right=0.25ex]\rar& 0 \\[0em]
  0 \rar& \T(\Knr) \rar& \cA(\Knr) \rar& \B(\Knr) \rar& 0 \cr
\end{tikzcd}
$$
we have
\begin{equation}\label{homXTZ}
  \frac{\cA(\Knr)}{\cA(\O_{\Knr})} \iso \frac{\T(\Knr)}{\T(\O_{\Knr})} = \frac{\Hom(X(\T),(\Knr)^\times)}{\Hom(X(\T),\cO_{\Knr}^\times)}
    = \Hom(X(\T),\Z) = X(\T)^*.
\end{equation}
By the rigidity of tori, we have $X(\T)=X(T)$ and so $X(\T)^*=X(T)^*$.

The dual abelian variety $A^t/\Knr$ has semistable reduction as well,
and there is a sequence as above with $\T^*, \cA^*$ and
$\B^*\iso \B^t$ (\cite{Gro}~Thm.~5.4).
Raynaud constructs a canonical map $X(\T^*)\hookrightarrow\cA(\Knr)$,
inducing a $G_{\Knr}$-isomorphism
$$
  A(\bar K) \iso \cA(\bar K)/X(\T^*).
$$
In the case of elliptic curves with split multiplicative reduction, this is 
Tate's parametrisation $E(\bar K)\iso \bar K^\times/q^\Z$.

Combining $X(\T^*)\injects\cA(\Knr)$ with \eqref{homXTZ}, we get an
inclusion 
\begin{equation}\label{nincl}
  n:\>\> X(\T^*) \injects \Hom(X(\T),\Z)
\end{equation}
with finite cokernel, which is canonically isomorphic to the group of components $\Phi(\kbar)$.
We may view $n$ as a non-degenerate bilinear pairing, the \emph{monodromy pairing},
\begin{equation}\label{ttspair}
  X(\T^*) \times X(\T) \lar \Z.
\end{equation}
If $K'/K$ is a finite extension, then $X(\T)$ and $X(\T^*)$
remain the same modules by uniqueness of Raynaud parametrisation, and
the map $n$ becomes $e_{K'/K}n$, see \cite[10.3.5]{Gro}.
Because $A$ is a Jacobian, it has a principal polarisation $\smash{A\overarrow{\raise-2pt\hbox{$\scriptscriptstyle\iso$}}A^t}$,
inducing $\cA\iso\cA^*$, $\T\iso\T^*$,  $\B\iso\B^*$.
The pairing \eqref{nincl} becomes a symmetric bilinear pairing (\cite{Gro} \S10.2) 
\begin{equation}\label{abpair}
  X(\T) \times X(\T) \lar \Z,
\end{equation}
and it coincides with \eqref{curvepair}, up to identifying 
$X(\T)=X(T)$.
Because $\Phi(\kbar)\iso \coker n$, we get a perfect symmetric pairing
\begin{equation}\label{phipair}
  \Phi(\kbar) \times \Phi(\kbar) \lar \Q/\Z.
\end{equation}
Finally, as in \S\ref{ss:scurves}, 
$G_k$ acts on everything, and \eqref{ttspair}, \eqref{abpair}, \eqref{phipair} are
$G_k$-equivariant.

%
%
%

\comment

Now suppose $A/K$ has {\em semistable\/} reduction.
The reduction becomes split semistable over some finite unramified extension
of $K$, and we take $L$ to be the smallest such field;
so now $\Gal(L/K)$ is cyclic, generated by Frobenius.
To describe the Tamagawa number $c_v(A/K)$ and
the action of inertia on $T_p(A)$ we
use the monodromy pairing
$$
  X(\T^*) \times X(\T) \lar \Z.
$$
This is a non-degenerate $\Gal(L/K)$-invariant pairing, and induces
a Galois-equivariant inclusion of lattices
$$
  N: X(\T^*) \longinjects \Hom(X(\T),\Z).
$$
These have the same $\Z$-rank, so $N$ has finite cokernel.
Moreover, $N$ is covariantly functorial with respect to isogenies
of semistable abelian varieties. 
Any polarisation on $A$ gives a map
$X(\T^*) \to X(\T)$, and the induced pairing
$$
  X(\T^*) \times X(\T^*) \lar \Z
$$
is symmetric (\cite{Gro} \S10.2).
In particular, if $A$ is principally polarised, we get a perfect
Galois-equivariant symmetric pairing 
$$
  \coker N \times \coker N \lar \Q/\Z.
$$
If $K'/K$ is a finite extension, then $X(\T)$ and $X(\T^*)$
remain the same modules (restricted to $\Gal(LK'/K')\subset \Gal(L/K)$)
by uniqueness of Raynaud parametrisation.
The map $N$ becomes $e_{K'/K}N$, see \cite{Gro} 10.3.5.

The $G_K$-module $\Gr2\oplus \Gr1\oplus \Gr0$ is unramified
and semisimple, so it is a semisimplification of $T_pA$.
With respect to this filtration, the inertia group acts on $T_pA$ by
$$
  I_K\ni\sigma\quad\longmapsto\quad
  \begin{pmatrix}
     1 & 0 & t_p(\sigma)N \cr
     0 & 1 & 0 \cr
     0 & 0 & 1 \cr
  \end{pmatrix}
  \in\Aut T_p(A)
$$
with $t_p: I_K\to \Z_p(1)$ defined by
$\sigma\mapsto \sigma(\pi_K^{1/p^n})/\pi_K^{1/p^n}\in\mu_{p^n}$
for any uniformiser $\pi_K$ of $K$
(\cite{Gro} \S\S 9.1--9.2).

Let $\Phi$ be the group scheme of connected components
of the special fibre of the N\'eron model of $A/\O_K$.
It is an \'etale group scheme over the residue field $k$ of $K$,
so
$\Phi(k)=\Phi(\kbar)^{\Gal(\kbar/k)}$
consists of components defined over $k$.
As $K$ is complete and $k$ is perfect, by \cite{BL} Lemma 2.1
the natural reduction map $A(K)\to\Phi(k)$ is onto, so
$c_v(A/K)=|\Phi(k)|$.
Finally, by \cite{Gro}~Thm.~11.5, $\Phi=\coker N$ as groups with
$\Gal(\kbar/k)=\Gal(K^{un}/K)$-action, so
$$
   c_v(A/K) = |(\coker N)^{\Gal(L/K)}|.
$$
\endcomment

\subsection{Galois and inertia} 
\label{ss:galinertia}

As $K$ is a local field, $G_K=\GKK$ fits into an exact sequence
$$
   1 \lar I_K \lar G_K \lar G_k \lar 1,
$$
with $I_K$ the \emph{inertia group}, and 
$G_k\iso\hat\Z$ topologically generated by the map $x\mapsto x^q$. 
Any of its lifts to $G_K$ is called an \emph{(arithmetic) 
Frobenius element} $\Frob$. 
Write $\chi_l: G_K\to\Z_l^\times=\GL_1(\Z_l)$,
$$
  \chi_l: \quad  I_K\mapsto 1, \>\> \Frob\mapsto q
$$
for the \emph{$l$-adic cyclotomic character}, and $\Z_l(n)=\chi_l^{\otimes n}$ for 
the Tate twist of the trivial module $\Z_l$.

The inertia group $I_K$ has a unique $p$-Sylow subgroup, the \emph{wild inertia} $P_K$,
and we have a short exact sequence
$$
  1 \lar P_K \lar I_K \lar \prod_{l\ne p} \Z_l \lar 1.
$$
The \emph{tame inertia} $I_K/P_K$ projects onto $\Z_l$ via the \emph{$l$-adic tame character}
$$
\begin{array}{cccccccccc}
  \tau:   & I_K     & \lar         & \invlim\mu_{l^n}=\Z_l(1) \cr
          & \sigma  & \longmapsto  & \frac{\sigma(\pi^{1/l^n})}{\pi^{1/l^n}}. \cr
\end{array}
$$

\comment
\subsection{Tate module and monodromy}

Let $X/K$ be any complete non-singular variety, $l\ne\vchar k$, and $0\le n\le 2\dim X$.
The group $G_K$ acts on the \'etale cohomology group
$$
  V = \Hn(X_{\Kbar},\Q_l).
$$
By the Grothendieck's monodromy theorem there is a finite Galois extension $F/K$ such
that $I_F$ acts unipotently on $V$ \cite[Thm. 1.2]{SGA7-I},
\begin{equation}\label{tameaction}
  \sigma\quad\longmapsto\quad \Id\,+\>t_l(\sigma) N,
\end{equation}
for a unique nilpotent operator $N\in\Aut V$.
The \emph{monodromy filtration} $\ker N^i\subset V$ is $G_K$-stable, and it follows 
$V$ decomposes as a $G_K$-module~as
\begin{equation}\label{weildec}
  V \iso \bigoplus V_i\tensor \Sp_{n_i}, 
\end{equation}
with $n_i$ distinct, $V_i$ uniquely determined $l$-adic representations of the group 
$\Gal(F^{nr}/K)$, and $\Sp_n$ the `special' $n$-dimensional representation (\cite[4.1.4]{TatN}). 
Explicitly,
$$
  \Sp(n) = \Q_l e_0 + \Q_l e_1 + \ldots + \Q_l e_{n-1},
$$
with $\Frob(e_i)=q^{-i}e_i$, and $I_K$ as in \eqref{tameaction} with $N(e_i)=e_{i+1}$ 
and $N(e_{n-1})=0$ (one Jordan block). Note that $\Sp(1)$ is the trivial representation, 
so $V=V_1\tensor\Sp(1)$ if and only if $V$ is unramified over $F$,
that is $I_F$ acts trivially.

When $X=C$ is a non-singular projective curve over $K$ of genus
$g\ge 2$, or $g=1$ with a marked point (an elliptic curve), we have

\medskip

\begin{tabular}{lllllllll}
$\Hgen0(C_{\bar K},\Q_l)$ &=& $\Q_l$ &&                                  &&& ($1$-dim)\cr
$\Hgen1(C_{\bar K},\Q_l)$ &=& $\Hgen1(J_{\bar K},\Q_l)$ &=& $(V_l J)^*$  &&& ($2g$-dim)\cr
$\Hgen2(C_{\bar K},\Q_l)$ &=& $\Q_l(1)$ &&                               &&& ($1$-dim).\cr
\end{tabular}

\smallskip

Let $V=\H(C_{\bar K},\Q_l)$, the only `interesting' cohomology group. 
Pick a Galois extension $F/K$ as above. It has a simple geometric interpretation:

\begin{center}
  $I_F$ acts unipotently on $V$ $\iff$
  $C/F$ is semistable $\iff$ $J/F$ is semistable.
\end{center}

\noindent
For $H^1$ of any variety, $N^2=0$, and so the decomposition \eqref{weildec} is of the form
\begin{equation}\label{vdec}
  V = V_{ab} \oplus (V_t\tensor\Sp_2),
\end{equation}
with $V_{ab}$ (the \emph{abelian part}) and $V_{t}$ (the \emph{toric part}). Both are
representations of $\Gal(L^{nr}/K)$, and we have
\begin{equation}\label{vtab}
  V_t \iso H^1(\Upsilon,\Q_l), \qquad V_{ab} \iso \bigoplus_i \H(\tilde{C_i},\Q_l),
\end{equation}
as $\Q_l$-vector spaces (for the moment). 
The first $H^1$ is the usual cohomology of a topological space, with $\Q_l$-coefficients.
(Its rank is the number of independent circuits in $\bar C$, including loops 
from self-intersections of components.)
Note that $J/F$ has \emph{good} reduction if and only if $V_t=0$, in other 
words the dual graph $\Upsilon$ of $\bar C$ is a tree.
\endcomment

\subsection{General curves}
\label[subsection]{ss:gcurves}

Fix

\begin{tabular}{l@{ }l}
$C/K$         & = arbitrary non-singular projective curve of genus $\ge 1$,\Tcropt{2.5}
$A/K$         & = Jacobian of $C$,\cr
$F/K$         & = finite Galois extension over which $C$ is semistable, with\cr
              & \ \ \ residue field $k_F$, ring of integers $\cO_F$ and uniformiser $\pi$,\cr
$\cC_{\min}/\cO_{\Fnr}$ & = minimal regular model of $C/\Fnr$, with special fibre $\Cbar$\cr
              & \ \ \ (semistable curve), and normalisation $\Cbarn$,\cr
$\cN/\cO_F$     & = N\'eron model of $A/\Fnr$, with special fibre $\bar N$,\cr
$\JJ$         & = set of connected comps. of $\Cbarn$,\cr 
$\Upsilon$    & = dual graph of $\bar C$, as in \S\ref{ss:scurves}.\Bcropt{1.3}
\end{tabular}

\noindent
For $\sigma\in G_K$, the model $\cC_{\min}^\sigma$ is again a stable model 
of $C/F$, and so $\cC_{\min}^\sigma\iso\cC_{\min}$, canonically. 
As explained in \cite{skew} (see also \cite[p. 13]{CFKS} and \cite[p. 497]{ST}), this implies 
that the Galois action $G_K\acts C(\Kbar)$,$A(\Kbar)$ extends to a \emph{semilinear} action 
on the geometric points of the special fibres $\Cbar$, $\bar N$ and $\Cbarn$,
\begin{equation}\label{saction}
  s: G_K \to \Aut\Cbar(\kbar)\>\> (\to \Aut\bar N(\kbar)).
\end{equation}

%
It is computed as follows. 
Let $\Cbar_{ns}(\kbar)\subset \Cbar(\kbar)$ be the non-singular locus.
Write $\red$ for the reduction map 
\begin{equation}\label{red}
  \red\colon C(F^{\textup{nr}}) \overarrow{=} \cC_{\min}(\cO_F^{\textup{nr}})  
  \overarrow{\rm reduce} \Cbar_{ns}(\kbar).
\end{equation}
It is surjective by Hensel's Lemma, so take a 
section $\red^{-1}\colon \Cbar_{ns}(\kbar)\to C(F^{\textup{nr}})$. 
Then on $\Cbar_{ns}(\kbar)$, the map $s$ is the composition (cf. \cite[Thm 1.5]{skew}, except there 
$\red^{-1}$ is chosen to land in $C(\bar F)$, which is a bit more general)
\begin{equation}\label{skew}
  s(\sigma)\colon\Cbar_{ns}(\kbar) \overarrow{\red^{-1}} C(F^{\textup{nr}}) 
    \overarrow{\sigma} C(F^{\textup{nr}}) 
  \overarrow{\rm red} \Cbar_{ns}(\kbar).
\end{equation}
The reduced curve $\Cbar$ has a natural structure of a $k$-scheme; denote 
by $\Phi: \Cbar\to \Cbar$ the absolute Frobenius map, acting on $K(C)$ by
raising everything to the power $|k|$. 
If $\sigma=\Frob^n\tau$ for some $n\in\Z_{\ge 0}$ and $\tau\in I_K$, then
$s(\sigma)=\Phi^n\tilde g$ for some $k$-linear automorphism $\tilde g$ 
of $\Cbar$. 
In particular, $n=0$ for $\sigma$ in the inertia group $I_K$, and
$s(I_K)=I_{F/K}$ is finite, acting through honest $k_F$-scheme automorphisms.
Note that for every $\sigma$ in the Weil group of $K$
(i.e. of the form $\sigma=\Frob^n\tau$ as above, but with $n\in\Z$), 
either $s(\sigma)$ or $s(\sigma^{-1})$ 
is a morphism of schemes, so that (a) \eqref{skew} determines it uniquely 
(on the whole of $G_K$, as the Weil group is dense in it), 
though it is only defined on the non-singular points, and (b) the action 
of $s$ extends naturally to the action on $\Cbarn(\kbar)$,
\begin{equation}\label{CbarnAction}
  \tilde s: \Cbarn(\kbar) \to \Cbarn(\kbar),
\end{equation}
on $\cN(\kbar)$, and on the dual graph $\Upsilon$.

Applying this to the $l$-power torsion points of $A=\Pic^0(C)$, from \eqref{XvsUps}, \eqref{curvetate}
and \cite[Prop. 2.6]{BW} (see also \cite[p.14]{CFKS}) we find an isomorphism of $G_K$-representations
$V_l A\iso (V_l A)_t\!\tensor\!\Sp_2\oplus (V_l A)_{ab}$, with
\begin{equation}\label{tatedec}
\begin{array}{llllllllllll}
  (V_l A)_t    &\iso&  H_1(\Upsilon_{C},\Z)\tensor_\Z\Q_l, \Bcr
  (V_l A)_{ab} &\iso&  \displaystyle \bigoplus_{\Gamma\in \JJ/G_K} \Ind_{\Stab(\Gamma)}^{G_K} V_l\Pic^0(\Gamma),\cr
\end{array}
\end{equation}
and $G_K$ acting through $s$ on the spaces on the right.
Twisting by $\Q_l(1)$ gives a similar decomposition for the \'etale cohomology
group $\H(C_{\Kbar},\Q_l)$.
See \cite[Cor 1.6]{skew} for details, noting that
$H_1(\Upsilon_{C},\Z)\tensor_\Z\Q_l\iso H^1(\Upsilon_{C},\Z)\tensor_\Z\Q_l$ as a $G_K$-module, since a rational
representation is self-dual.

\begin{example}
\label{exbgr2}
Consider the curve from Example \ref{exbgr1} over $K=\Q_p$ ($p>3$),
$$
  C/K\colon y^2=x^3+p^2.
$$ 
Fix a primitive 3rd root of unity $\zeta\in \bar K$ and $\pi=\sqrt[3]p$.
Let $K'=K(\pi)$, and $F=K(\zeta,\pi)$,
its Galois closure. Thus, $\Gal(F/K)\iso S_3$ if $p\equiv 2\mod 3$, and $F=K'$ with 
$\Gal(F/K)\iso C_3$ otherwise. In either case, the inertia group $I_{F/K}$ is $C_3$, generated
by $\tau$ that sends $\pi\to \zeta\pi$. Let $\Phi$ be a Frobenius element of $K'$; so
$\Phi$ fixes $\pi$ and sends $\zeta\to\zeta^p$. So
$$
  G\>=\>\Gal(F^{\textup{nr}}/K) \>=\> \langle \tau,\Phi \rangle \>\iso\> C_3\rtimes \hat\Z.
$$
Recall from \ref{exbgr1} that $C$ acquires good reduction over $K'$, and thus over $F$ as well,
and the special fibre of its minimal model is the curve $\Cbar: y^2=x^3+1$.  
Using the isomorphism \eqref{C3var}, the reduction map \eqref{red} becomes
$$
\begin{array}{llllllllllll}
\red\colon & C(F^{\textup{nr}}) & \lar & \Cbar(\bar\F_p) \cr
           & (x,y) & \longmapsto & (\tfrac{x}{\pi^2},\tfrac{y}{\pi^3})\mod \pi. \cr
\end{array}
$$
The semilinear action $s$ of $G_K$ in \eqref{skew} factors through $G$ and is given by
$$
\begin{array}{llllllllllll}
s(\tau)\colon & \Cbar(\bar\F_p) & \lar & \Cbar(\bar\F_p) \cr
              & (x,y) & \longmapsto & (\zeta x,y) \cr
\end{array}
\qquad 
\begin{array}{llllllllllll}
s(\Phi)\colon & \Cbar(\bar\F_p) & \lar & \Cbar(\bar\F_p) \cr
              & (x,y) & \longmapsto & (x^p,y^p) \cr
\end{array}
$$
We refer the reader to \cite[\S6]{hq} and \cite[\S3-4]{weil} and \cite{BW} for additional examples,
and explicit computations of Galois representations attached to curves.
\end{example}


We also recall the formula for the local root number of the Jacobian in the semistable case:

\begin{theorem}\label{th:rootnumber}
Let $C/K$ be a semistable curve. Then the local root number of the Jacobian 
$A=\Jac C$ is 
$$
w_A = (-1)^{a} ,
$$
where $a$ is the multiplicity of the trivial representation of $G_k$ 
in the homology of the dual graph $H_1(\Upsilon,\Q)$ of $C$. 
\end{theorem}
\begin{proof}
This is a standard root number computation, see e.g. \cite[Prop. 3.23]{RegConst} 
with $\tau = \triv$ and $X(\T) = X(T) = H_1(\Upsilon,\Z)$ by \eqref{XvsUps}.
\end{proof}

\begin{lemma}\label{le:isolattice}
Let $C/K$ be a semistable curve. 
Then 
$H_1(\Upsilon_C, \Z) \iso X(\T)$ as $\Z[\Frob]$-modules with pairing.
\end{lemma}
\begin{proof}
Again, $X(\T) = X(T) = H_1(\Upsilon,\Z)$ by \eqref{XvsUps}; for the compatibility
with the pairing, see \eqref{abpair}.
\end{proof}

\begin{lemma}
\label{lemtam}
Suppose $C/K$ is semistable. 
Let $\Phi$ be the component group of the N\'eron model of the Jacobian of $C$ over $K^{nr}$.
Then $\Phi$ is isomorphic, as a $G_k$-module, to the cokernel of 
$$
  H_1(\Upsilon_C,\Z) \lar \Hom(H_1(\Upsilon_C,\Z),\Z), \qquad \ell\mapsto \langle \ell,\cdot \rangle.
$$
\end{lemma}

\begin{proof}
Apply \cite[Thm. 9.6/1]{BLR} over $K^{nr}$.
\end{proof}

\begin{remark}
\label{remtam}
The size of the $G_k$-invariants of $\Phi$ is known as the \emph{Tamagawa number} of the Jacobian of $C$ over $K$. 
\end{remark}

\def\fr{{\mathfrak r}}
\def\R{{\mathcal R}}
\def\optional#1{{\color{cyan}#1}}
\def\d#1{d_#1}

%
\section{Regular semistable models of $\P^1$} \label[section]{models of P1}
%
In this section we show how certain finite collections of discs (see \Cref{admissible defi})  give rise to regular semistable models of the projective line over $K^{\textup{nr}}$. We then describe (\Cref{basic model lemma}) the divisor of a polynomial on this model. In the next section, given a hyperelliptic curve $C/K:y^2=f(x)$ satisfying the semistability criterion we use this to associate to $C$ a particular model of $\mathbb{P}^1_{K^{\textup{nr}}}$ on which the divisor of $f(x)$ has nice properties (see \Cref{branch locus lemma}). We then deduce that the normalisation of this model in the function field $K^{\textup{nr}}(C)$ is a regular model of $C/K^{\textup{nr}}$ (\Cref{main regular model theorem}) from which several of our main theorems follow.  

The relationship between discs and models of the projective line  is not new and roughly proceeds as follow.  To each disc there is a naturally associated valuation on $\bar{K}(x)$ (we recall this construction in \Cref{nu and kappa sub}). Now by \cite[Proposition 3.4]{Rut} (see also \cite[Section 5]{OW})  there is a one to one correspondence between normal models of $\mathbb{P}^1_{K^{\textup{nr}}}$ and  finite non-empty collections of `inductive' valuations on $K^{\textup{nr}}(x)$, the map taking a model to the set of valuations corresponding to the local rings at the generic points of the components of its special fibre. Our model is the one yielding the same collection of valuations as that associated to the collection of discs.  To facilitate in the analysis of the model however we construct it below using explicit charts. That these two descriptions agree follows from  \Cref{basic model lemma}. For a more general correspondence between normal  models of $\mathbb{P}^1_{K^{\textup{nr}}}$ and  collections  of `rigid diskoids'  see \cite[Proposition 4.4]{Rut}, \cite[Section 5.1]{OW}.

\subsection{Discs and associated valuations}

\subsubsection{Discs} \label{disc notation}

A \textit{disc} is a subset
\[D=D_{z,d}:=\{x\in \Kbar ~\mid~ v(x-z)\geq d\}\]
for some $z\in \Kbar$ and $d\in \Q$. Here $d$ is an invariant of the disc, its \textit{depth}, denoted $d_D$. If $D$ has depth $d_D$ and $z\in D$ then $D=D_{z,d_D}$; we call any $z\in D$ a \textit{centre} of $D$. 
We say a disc  is \textit{integral} if it has a centre in $\Knr$ and integer depth. For an integral disc $D$ we denote by $P(D)$ its `parent' integral disc $P(D)=D_{z_D,d_{D}-1}$ for any $z_D\in D$. We say  integral discs $D$ and $D'$ are \textit{adjacent} if one is the parent of the other.

\subsubsection{The valuation  associated to a disc} \label[section]{nu and kappa sub}

Each disc $D=D_{z_D,d_D}$ defines a valuation $\nu_D$ on the function field $\Kbar(x)$ extending $v$ (see e.g. \cite[Section 1.4.4]{Berk}). Explicitly, for a polynomial $f(x)\in \Kbar[x]$, letting $c_i$ denote the coefficient of $x^i$ in $f(x+z_D)$ 
 we have
\[\nu_D(f)=\textup{min}_i \{v(c_i)+d_Di\}.\]

Writing $\mathcal{R}\subseteq \bar{K}$ for the (multi)set of roots of $f(x)$ and $c_f$ for its leading coefficient, factoring $f(x)$ as a product of linear polynomials it follows from the fact that $\nu_D$ is a valuation extending $v$ that
\begin{equation} \label{formula for nu lemma}
\nu_D(f)=v(c_f)+\sum_{r\in \mathcal{R}} \mathrm{min}\{d_D,v(z_D-r)\}.
\end{equation}

%



\subsection{Admissible collections of discs}

The following collections of discs will correspond to regular semistable models of $\mathbb{P}^1_{K^{\textup{nr}}}$.

\begin{definition} \label[definition]{admissible defi}
Call a finite non-empty collection $\mathcal{D}$ of integral discs \emph{admissible} if
\begin{itemize}
\item[(i)]
$\mathcal{D}$ has a maximal element $D_{\textup{max}}$ with respect to inclusion,
\item[(ii)]
if $D_1,D_2\in \mathcal{D}$ with $D_1\subseteq D_2$ then every integral disc  $D_1\subseteq D\subseteq D_2$ is in $\mathcal{D}$ also. 
\end{itemize}
To such  $\mathcal{D}$  we associate the finite connected rooted tree $T_\mathcal{D}$ with vertices $\{v_D~~\mid~~D\in \mathcal{D}\}$ and root $v_{D_\textup{max}}$, where $v_D$ and $v_{D'}$ are joined by an edge when $D$ and $D'$ are adjacent.  We write $\mathcal{D}_i$ (resp. $\mathcal{D}_{\leq i}$) for the subset of $\mathcal{D}$ consisting of discs whose associated vertices are a distance $i$ (resp. at most $i$) from the root.  
\end{definition}

\begin{remark}
We will see in \Cref{basic model lemma1} that $T_\mathcal{D}$ is canonically the dual graph of the model of $\mathbb{P}^1_{K^{\textup{nr}}}$ associated to $\mathcal{D}$.
\end{remark}


\subsection{The model of $\P^1_{\Knr}$ associated to an admissible collection of discs}\label[section]{construction of the p1 sect}

\begin{notation} \label{centres notation}
For the rest of this section we fix an admissible collection of discs $\mathcal{D}$, along with a choice of centre $z_D\in \Knr$ for each $D\in \mathcal{D}$.
\end{notation}

In what follows denote by $\mathcal{O}$ the ring of integers of $\Knr$ and, as in the introduction, let $\pi$ denote a fixed choice of uniformiser for $K$. We now associate to $\mathcal{D}$ a model $\mathcal{Y}_\mathcal{D}/\mathcal{O}$ of $\P^1_{\Knr}$, first introducing some objects and notation which will be useful for the construction. 

\begin{remark}
The choice of centres above is minor - in particular the model $\mathcal{Y}_\mathcal{D}/\mathcal{O}$ which we associate to $\mathcal{D}$ in  \Cref{the model of P1} is, up to isomorphism over $\mathcal{O}$,  independent of this.
\end{remark}

\subsubsection{The schemes $U_D$, $W_D$ and $Y_D$}


\begin{definition} \label[definition]{schemes defi}
To each disc $D\in \mathcal{D}$ we associate  schemes \[U_{D}=\textup{Spec}~\mathcal{O}[x_D]\phantom{hi}\textup{and}\phantom{hi}W_D=\begin{cases}\textup{Spec}~\mathcal{O}[t_D]~~&~~D=D_{\textup{max},}\\\textup{Spec}~\mathcal{O}[s_D,t_D]/(\pi-s_Dt_D)~~&~~\textup{else}.\end{cases}\]
We denote by $Y_D$ the glueing of $U_D$ and $W_D$ over the subsets $\{x_D\neq0\}$ and $\{t_D\neq 0\}$  via the isomorphsim $x_D=1/t_D$. 

For $D=D_\textup{max}$ we have $Y_D=\P^1_\mathcal{O}$ with variable $x_D$. We denote its special fibre $E_{D_\textup{max}}$.  For $D\neq D_\textup{max}$, $Y_D$ is the result of blowing up $\mathbb{A}^1_\mathcal{O}$ with variable $s_D$ at the origin  on the special fibre (see e.g. \cite[Lemma 8.1.4]{LiuA}). Its special fibre consists of two irreducible components intersecting transversally at the single closed point $s_D=t_D=0$. One component (the exceptional fibre of the blow up) is  isomorphic to $\P^1_{\kbar}$ with variable $x_D$ and we denote it $E_D$. 
The other is isomorphic to $\mathbb{A}^1_{\kbar}$ with variable $s_D$. It is contained entirely in the complement of $U_D$ and we denote it $F_D$. 
%
\end{definition}

%



\begin{definition} \label[definition]{model components}
Let $D\in \mathcal{D}$. We define a `reduction' map $D\rightarrow \bar{k}$ (which depends on the choice of centre $z_D$ for $D$) by setting
\[\textup{red}_{D}(z)=\frac{z-z_{D}}{\pi^{d_{D}}}\quad (\textup{mod}~\mathfrak{m}).\]
Note that this gives a bijection between closed points on the special fibre of $U_D$ and maximal integral subdiscs $D'$ of $D$, sending $D'$ to the point $x_D=\textup{red}_D(z_{D'})$ for any $z_{D'}\in D$. Since this does not depend on the choice of centre $z_{D'}$ we henceforth write  $\textup{red}_D(D')$ in place of $\textup{red}_D(z_{D'})$. 

We denote by  $\mathcal{P}_D$ the finite set of points on the special fibre of $U_D$ corresponding to maximal integral subdsics of $D$ which are in $\mathcal{D}$. Similarly, let $\mathcal{Q}_D$ denote the finite set of closed points on the special fibre of $W_D$ of the form 
\[s_D=\frac{z_{D'}-z_D}{\pi^{d_D-1}}\quad(\textup{mod}~\mathfrak{m})\]
  for  $D'\in \mathcal{D}$ a `sibling' of  $D$ (i.e. such that $D'\neq D$ is a maximal integral subdisc of $P(D)$).
  
   In what follows we will at times wish to consider the scheme $W_D\setminus \mathcal{P}_\mathcal{D}$, viewing both $\mathcal{P}_D$ and $W_D$ as subsets of $Y_D$ to form the complement. 
\end{definition}


 \subsubsection{The model $\mathcal{Y}_\mathcal{D}$}
 
 We now glue the schemes $Y_D$ with certain points removed  to form the model $\mathcal{Y}_\mathcal{D}$. We will do this in such a way that for $D\neq D_{\textup{max}}$ the component $F_D$ of $Y_D$ glues onto the component $E_{P(D)}$ of $Y_{P(D)}$, identifying the set $\mathcal{Q}_{D}$ with $\mathcal{P}_{P(D)}$ less the point $x_D=\textup{red}_{P(D)}(D)$  in the process.
 
\begin{definition}[The model $\mathcal{Y}_\mathcal{D}$] \label[definition]{the model of P1}
For  $i\geq 0$ we construct inductively schemes $\mathcal{Y}_{\mathcal{D}_{\leq i}}$ which will be covered by
\[\{Y_D\setminus (\mathcal{P}_D\cup \mathcal{Q}_D)~~\mid D\in \mathcal{D}_{\leq i-1}\}\phantom{hi}\textup{and}\phantom{hi} \{Y_D\setminus \mathcal{Q}_D~~\mid~~D\in \mathcal{D}_i\}.\]
 In this way we  talk about components $E_D$ of $\mathcal{Y}_{\mathcal{D}_{\leq i}}$. 
 We will  denote by $\infty$ the point on $\mathcal{Y}_{\mathcal{D}_{\leq i}}$ corresponding to $t_{D_\textup{max}}=0$ on the generic fibre of $Y_{D_\textup{max}}$, and denote by $\overline{\{\infty\}}$ its closure in the model. We then define  $\mathcal{Y}_\mathcal{D}$ to be equal to $\mathcal{Y}_{\mathcal{D}_{\leq n}}$ for $n$ minimal such that $\mathcal{D}=\mathcal{D}_{\leq n}$.

First, set $\mathcal{Y}_{\mathcal{D}_{\leq 0}}=Y_{D_\textup{max}}$. We make this a model of $\P^1_{\Knr}$ (thought of with variable $x$) via the change of variable $x=\pi^{d_{D_\textup{max}}}x_{D_\textup{max}}+z_{D_\textup{max}}.$

Now given $\mathcal{Y}_{\mathcal{D}_{\leq i}}$ we obtain $\mathcal{Y}_{\mathcal{D}_{\leq i+1}}$ by blowing up $\mathcal{Y}_{\mathcal{D}_{\leq i}}$ at the finite set  $\bigcup_{D\in \mathcal{D}_i}\mathcal{P}_{D}$  of closed points on its special fibre.

 Explicitly, since blowing up is a local process, $\mathcal{Y}_{\mathcal{D}_{\leq i+1}}$ is given by glueing each of the schemes $Y_D\setminus \mathcal{Q}_D$ for $D\in \mathcal{D}_{i+1}$ onto $\mathcal{Y}_{\mathcal{D}_{\leq i}}':=\mathcal{Y}_{\mathcal{D}_{\leq i}}\setminus \bigcup_{D\in \mathcal{D}_i}\mathcal{P}_{D}$ over the open subsets given by removing $E_D$ from the special fibre of $Y_D\setminus \mathcal{Q}_D$ ($D\in \mathcal{D}_{i+1}$), and removing $\overline{\{\infty\}}\cup\bigcup_{D'\neq P(D)}E_{D'}$ from $\mathcal{Y}_{\mathcal{D}_{\leq i}}'$. The glueing maps are given, for $D\in  \mathcal{D}_{i+1}$, by
\[s_D=x_{P(D)}+\frac{z_{P(D)}-z_D}{\pi^{d_{P(D)}-1}}.\]
\end{definition}

\begin{remark} \label[remark]{function field relations hold} \label{disc polynomials remark}
In  the function field of $\mathcal{Y}_\mathcal{D}$ we have, for each $D\in \mathcal{D}$,
\[x_D=\frac{x-z_D}{\pi^{d_D}},\phantom{hello}t_D=1/x_D,\phantom{hi}\textup{and}\phantom{hi}s_Dt_D=\pi.\]
Moreover, we see from the construction that $\mathcal{Y}_{\mathcal{D}}$ is covered by the open subsets $U_D\setminus\mathcal{P}_D$ and $W_D\setminus (\mathcal{Q}_D\cup \mathcal{P}_D)$ as $D$ ranges over all elements of $\mathcal{D}$. In particular, every closed point of the special fibre of $\mathcal{Y}_D$ is either the point at infinity (i.e. $t_{D_\textup{max}}=0$) on $E_{D_\textup{max}}$, the single point of intersection between $E_D$ and $E_{P(D)}$  (visible as the point $s_D=t_D=0$ on  $W_D\setminus (\mathcal{Q}_D\cup \mathcal{P}_D)$) for some $D\neq D_\textup{max}$, or a point of the form $x_D=\textup{red}_D(D')$ on $E_{D}$ for $D'\notin \mathcal{D}$ a maximal proper integral subdisc of some $D
\in \mathcal{D}$, visible on $U_D\setminus\mathcal{P}_D$.
\end{remark}

\subsubsection{Properties of the model}

\begin{proposition} \label[proposition]{main model properties1}  \label[proposition]{basic model lemma1}
Let $\mathcal{D}$ be an admissible collection of discs and $\mathcal{Y}_{\mathcal{D}}/\mathcal{O}$  the associated model of $\P^1_{\Knr}$.
Then $\mathcal{Y}_\mathcal{D}$ is proper regular and semistable with dual graph $T_\mathcal{D}$ (\Cref{admissible defi}), the vertex $v_D$ corresponding to the component $E_D$. The valuation on $K^{\textup{nr}}(x)$ corresponding to $E_D$ is (the restriction to $K^{\textup{nr}}(x)$ of) $\nu_D$. 
\end{proposition}

\begin{proof}
Since  $\mathcal{Y}_\mathcal{D}$ is obtained by iteratively blowing up $\P^1_{\mathcal{O}}$ at a finite set of closed points on the special fibre it follows that $\mathcal{Y}_\mathcal{D}$ is proper and regular. Since none of these points were intersection points between components of the special fibre, $\mathcal{Y}_\mathcal{D}$ is  semistable (as can also be seen from the explicit charts covering $\mathcal{Y}_\mathcal{D}$). The dual graph is equal to $T_\mathcal{D}$ by construction.

Finally, write $\mu_D$ for the valuation associated to $E_D$. It is the valuation on $\Knr(x)$ associated to the prime ideal $(\pi)$ of $\mathcal{O}[x_D]$, where $x=\pi^{d_D}x_D+z_D$. Since both $\nu_D$ and $\mu_D$ are valuations, to show that they are equal it suffices to show they agree on all polynomials. Now for $f(x)\in K^{\textup{nr}}[x]$, it follows from the definition of $\nu_D$ that \[\pi^{-\nu_D(f)}f(x)=\pi^{-\nu_D(f)}f(\pi^{d_D}x_D+z_D)\] is in $\mathcal{O}[x_D]$ but not in the ideal $(\pi)$. Thus $\mu_D(\pi^{-\nu_D(f)}f(x))=0$ and the result follows.
\end{proof}

\subsection{The divisor of a polynomial on the model $\mathcal{Y}_\mathcal{D}$}

We now describe the divisor of a polynomial $f(x)\in K^{\textup{nr}}[x]$ on the model $\mathcal{Y}_\mathcal{D}$, first introducing some notation with which to describe the result. 

\begin{definition} \label[definition]{kappa def}	
Let $D$ be an integral disc and $f(x)\in K^{\textup{nr}}[x]$. Define
\[\kappa_D(f)=\nu_D(f)-\nu_{P(D)}(f).\]
\end{definition}

Note that if $f(x)$ is a polynomial with (multi)set of roots $\mathcal{R}$ then \Cref{formula for nu lemma} gives
\begin{equation} \label[equation]{formula for kappa lemma}
\kappa_D(f)= |D \cap \mathcal{R}|+\sum_{\substack{r\in \mathcal{R}\\ d_D-1<v(r-z_D)<d_D}}v\left(\frac{r-z_D}{\pi^{d_D-1}}\right)\quad\quad\geq 0.
\end{equation}


\begin{proposition} \label[proposition]{main model properties}  \label[proposition]{basic model lemma}
Let $\mathcal{D}$ be an admissible collection of discs, $\mathcal{Y}_{\mathcal{D}}/\mathcal{O}$  the associated model of $\P^1_{\Knr}$ and $f(x)\in K^{\textup{nr}}[x]$ a polynomial with (multi)set of roots $\mathcal{R}\subseteq \bar{K}$. Let $Z\in \textup{Div}(\mathcal{Y}_\mathcal{D})$ denote the divisor
\[Z=\textup{div}(f)+\textup{deg}(f)\overline{\{\infty\}},\]
and let $Z_\textup{ver}$ (resp. $Z_\textup{hor}$) denote its vertical (resp. horizontal) parts.
\begin{itemize}
\item[(i)] We have \[Z_\textup{ver}=\sum_{D\in \mathcal{D}}\nu_D(f)E_D.\]

\item[(ii)] $Z_\textup{hor}$
  does not meet any of the intersection points between components of the special fibre of $\mathcal{Y}_{\mathcal{D}}$ if and only if
\begin{equation} \label[equation]{intersection with special fibre condition 1}
\kappa_D(f)=|D\cap \mathcal{R}| ~~\quad \textup{ for all }D\in \mathcal{D} \textup{ with }D\neq D_\textup{max}.
\smallskip
\end{equation}
\end{itemize}
Suppose  that \Cref{intersection with special fibre condition 1} holds. Then
\begin{itemize}
\item[(iii)] for each $D\in \mathcal{D}$, $Z_\textup{hor}$ meets $E_D$ precisely at the points $x_D=\textup{red}_D(D')$  for $D'$ a maximal integral subdisc of $D$ with $D'\notin \mathcal{D}$ and $\kappa_{D'}(f)>0$, unless $D=D_\textup{max}$ and $\kappa_D(f)<\textup{deg}(f)$, in which case it additionaly meets $E_{D}$ at the point at infinity. 

\item[(iv)] $Z_\textup{hor}$ is regular if and only if 
\begin{equation*} \label[equation]{intersection with special fibre condition 2}
\kappa_{D_{\textup{max}}}(f)\in \{\textup{deg}(f),\textup{deg}(f)-1\}
\end{equation*}
and for all integral discs $D$ with $P(D)\in \mathcal{D}$ and $\kappa_D(f)\geq 2$, we have $D\in \mathcal{D}$ also.
\end{itemize}
\end{proposition}

We postpone the proof to the end of the section, beginning by defining certain auxhilliary polynomials associated to $f(x)$. 

\begin{definition} \label[definition]{defi of the polynomials}
Let $f(x)\in \Knr[x]$ and $D\in \mathcal{D}$. We define
\[f_D(x_D)=\pi^{-\nu_D(f)}f(\pi^{d_D}x_D+z_D)\in \mathcal{O}[x_D].\]
If $D=D_{\textup{max}}$ we define
\[g_D(t_D)=t_D^{\textup{deg}(f)}f_D(1/t_D)\in \mathcal{O}[t_D],\]
whilst if $D\neq D_\textup{max}$ we define
\[g_D(s_D,t_D)=\sum_{i\geq 0}\hat{c}_is_D^{v(c_i)+d_Di-\nu_D(f)}t_D^{v(c_i)+(d_D-1)i-\nu_{P(D)}(f)}\in \mathcal{O}[s_D,t_D],\]
where $c_i$ is the coefficient of $x^i$ in $f(x+z_D)$ and $\hat{c}_i=c_i\pi^{-v(c_i)}$.
\end{definition}

\begin{remark} \label[remark]{disc polynomials remark2}
Inside $K^{\textup{nr}}(\mathcal{Y}_\mathcal{D})$ we have (cf. \Cref{function field relations hold})
\[f_D(x_D)=\pi^{-\nu_D(f)}f(x),\phantom{hello}g_{D_{\textup{max}}}(t_{D_\textup{max}})=t_{D_\textup{max}}^{\textup{deg}(f)}\pi^{-\nu_D(f)}f(x)\phantom{hi}\]
and
\[\phantom{hi}g_D(s_D,t_D)=s_D^{-\nu_D(f)}t_D^{-v_{P(D)}(f)}f(x).\]
In particular, upon proving \Cref{basic model lemma}(i) it follows that, for each $D$, $f_D$ (resp. $g_D$) gives a local equation for $Z_\textup{hor}$ on $U_D\setminus  \mathcal{P}_D$ (resp. $W_D\setminus (\mathcal{P}_D \cup \mathcal{Q}_D)$).
\end{remark}

\begin{lemma} \label[lemma]{intersection lemma}
Let $f(x)\in \Knr[x]$ have (multi)set of roots $\mathcal{R}\subseteq \bar{K}$. Let $D\in \mathcal{D}$ with $D\neq D_\textup{max}$ and let $R_D$ denote the intersection point between the two components of the special fibre of $W_D$. Then $R_D$ lies on the closed subscheme $\{g_D(s_D,t_D)=0\}$ of $W_D$ if and only if 
\[\kappa_D(f)\neq| D \cap \mathcal{R}|.\]
\end{lemma}

\begin{proof}
Note that $R_D$ corresponds to the maximal ideal $(s_D,t_D)$ of $\mathcal{O}[s_D,t_D]/(\pi - s_Dt_D)$.
To simplify notation, for each $0\leq i \leq \textup{deg}(f)$ write $\lambda_i=v(c_i)+(d_D-1)i$, so that $\nu_D(f)=\textup{min}_i\{\lambda_i+i\}$, $\nu_{P(D)}(f)=\textup{min}_i\{\lambda_i\}$ and
\[g(s_D,t_D)=\sum_i \hat{c}_i s_D^{\lambda_i+i-\nu_D(f)}t_D^{\lambda_i-\nu_{P(D)}(f)}\]
for $\hat{c}_i$ as in \Cref{defi of the polynomials}.
Now $(s_D,t_D)$ does not lie on $\{g_D(s_D,t_D)=0\}$ if and only if 
 $g_D(s_D,t_D)$ has non-zero constant term, or equivalently if and only if
 \[\textup{min}_i\{\lambda_i+i\}=\textup{min}_i\{\lambda_i\}+j\]
where $j$ is the smallest integer such that $\lambda_j=\textup{min}_i\{\lambda_i\}$. That is, if and only if  $\kappa_D(f)=j$ for $0\leq j \leq \textup{deg}(f)$ as above.

Considering the Newton polygon of the polynomial $f(\pi^{d_D-1}x+z_D)$, the valuation of the $i$th coefficient of which is $\lambda_i$, we see that $j$ is equal to the sum of the lengths of the projections onto the horizontal axis of all segments in the Newton polygon having strictly negative slope. By standard properties of Newton polygons this is equal to the number of roots of $f(\pi^{d_D-1}x+z_D)$ having strictly positive valuation or, equivalently, equal to the number of roots $r$ of $f(x)$ satisfying
$v(r-z_D)>d_{D}-1.$
From \Cref{formula for kappa lemma}  we see that $\kappa_D(f)$ is equal to the number of such roots if and only if $\kappa_D(f)=|D \cap \mathcal{R}|$. 
\end{proof}

\begin{lemma} \label[lemma]{regularity lemma}
Let $f(x)\in \Knr[x]$. 
\begin{itemize}
\item[(i)] Let $D'$ be a maximal integral subdisc of some $D\in \mathcal{D}$. Then the closed subscheme $\{f_D(x_D)=0\}$ of $U_D$ meets the point $x_D=\textup{red}_{D}(D')$ on the special fibre of $U_D$ if and only if
$\kappa_{D'}(f)>0,$
and if this is the case, is regular at this point if and only if 
$\kappa_{D'}(f)=1.$

\item[(ii)] The closed subscheme $\{g_{D_\textup{max}}(t_{D_\textup{max}})=0\}$ of $W_{D_\textup{max}}$ meets the point  $t_{D_\textup{max}}=0$  on the special fibre if and only if $\kappa_{D_\textup{max}}(f)<\textup{deg}(f)$, in which case it is regular here if and only if $\kappa_{D_\textup{max}}(f)=\textup{deg}(f)-1$.
\end{itemize}
\end{lemma}

\begin{proof}
(i): Fix a centre $z_{D'}$ of $D'$ so that the point $x_D=\textup{red}_{D}(D')$ in question corresponds to the maximal ideal $\mathfrak{n}=(\pi,x_D-\alpha_{D'})$ of $\mathcal{O}[x_D]$, where $\alpha_{D'}=(z_{D'}-z_D)/\pi^{d_D}$. Denoting by $c_i$ the coefficient of $x^i$ in $f(x+z_{D'})$ we have  
\[v_{D'}(f)=\textup{min}_{i}\{v(c_i)+(d_{D}+1)i\}\]
and
\[f_D(x_D)=\sum_i c_i'(x_D-\alpha_{D'})^i\]
where $c_i'=c_i\pi^{id_D-\nu_D(f)}$. 
 Now $f_D(x_D)$ is in $\mathfrak{n}$ if and only if $v(c_0')>0$ and, when it is, $\mathcal{O}[x_D]/(f_D(x_D))$ is regular at $\mathfrak{n}$ if and only if $f_D(x_D)\notin\mathfrak{n}^2$ (see \cite[Lemma 4.2.2]{LiuA}). This happens if and only if either $v(c_0')=1$ or $v(c_1')=0$ (or both). 
 
 Now  $v(c_0')>0$ (resp. $v(c_0')=1$ or $v(c_1')=0$) if and only if
\[\textup{min}_i\{v(c_i')+i\}>0~~~(\textup{resp. }\textup{min}_i\{v(c_i')+i\}=1)\]
and since
\[\textup{min}_i\{v(c_i')+i\}=\textup{min}_i\{v(c_i)+(d_D+1)i-\nu_D(f)\}=\nu_{D'}(f)-\nu_D(f)=\kappa_{D'}(f)\]
we are done.

(ii): This may be proved analogously to (i).
\end{proof}

\begin{proof}[Proof of \Cref{main model properties}]
(i). Follows from \Cref{basic model lemma1} which shows that $\nu_D$ is the valuation on $K^{\textup{nr}}(x)$ corresponding to  $E_D$. 

(ii). Follows from \Cref{intersection lemma} noting that by \Cref{disc polynomials remark2}, for any $D\neq D_{\textup{max}}$, $g_D(s_D,t_D)$ is a local equation for $Z_\textup{hor}$ on $W_D\setminus (\mathcal{P}_D\cup \mathcal{Q}_D)$, where the  intersection point between $E_D$ and $E_{P(D)}$ corresponds, in the notation of the lemma, to the point $R_D$. 

(iii) and (iv). The argument is similar to that of (ii) and follows by combining  \Cref{disc polynomials remark2}, \Cref{regularity lemma}, (ii) and the description of the closed points on the special fibre of $\mathcal{Y}_{\mathcal{D}}$ afforded by \Cref{function field relations hold}. We note that the condition $D\in \mathcal{D}$ in (iv) ensures that the point $x_{P(D)}=\textup{red}_{P(D)}(D)$ is removed from $U_{P(D)}$ when forming the model.  
\end{proof}

\def\fr{{\mathfrak r}}
\def\R{{\mathcal R}}
\def\optional#1{{\color{cyan}#1}}
\def\d#1{d_#1}

%
\section{An explicit model of $\P^1$ assuming the semistability criterion} \label[section]{ydisc_section}

Suppose now that $C/K:y^2=f(x)$ is a hyperelliptic curve satisfying the semistability criterion  (\Cref{semistability criterion}). In this section we associate to $C$   an admissible  collection of discs (in the sense of \Cref{{admissible defi}}) which we call `valid discs', and which are closely related to the cluster picture of $C$. Using the results of the previous section we then study the divisor of $f(x)$ on the associated model  of $\P^1_{\Knr}$, the main result being \Cref{branch locus lemma}. This shows in particular that the conditions of \cite[Lemma 2.1]{Sri} are satisfied, so that  the  normalisation of this  model  of $\P^1_{\Knr}$ in the function field $K^{\textup{nr}}(C)$  is a regular proper model of $C/K^{\textup{nr}}$. This, and its consequences, are treated in \Cref{scdisc}. We remark that this approach to constructing regular models of hyperelliptic curves in not new and appears in  \cite{Kau}, \cite{Sri} and \cite{BW}, although each of these assumes that all roots of $f(x)$ lie in $K$ (or at least $K^{\textup{nr}}$). In particular, these constructions do not cover all curves satisfying the semistability criterion as is needed in our situation.

\subsection{Notation}
 In this section we freely apply the notation and terminology associated to clusters as set out in \Cref{clusternotation}.  As in the introduction we denote by $\mathcal{R}\subseteq\Kbar$ the set of roots of $f(x)$.   We also frequently use the following additional definition. 

\begin{definition} \label[definition]{no:disc}
If $\s\subseteq \mathcal{R}$ is a proper cluster we call the unique smallest disc cutting it out the \textit{defining disc of} $\s$, and denote it $D(\s)$. 
\end{definition}

Note that for any proper cluster $\s$, the disc $D(\s)$ has depth $d_\mathfrak{s}$ and, in the notation of \Cref{no:nu,nu and kappa sub}, it follows from \Cref{formula for nu lemma} that $\nu_{D(\s)}(f)=\nu_\s$. 

We also note here that by part (1) of the semistability criterion, all roots of $f(x)$ are defined over the unique quadratic extension of $K^{\textup{nr}}$. In particular, every proper cluster $\s\subseteq \mathcal{R}$ has depth $d_\s\in \frac{1}{2}\mathbb{Z}$, and if $r\in \mathcal{R}$ with $r\notin K^{\textup{nr}}$ then denoting by $r'$ its inertia-conjugate root we have  $v(r-r')\in 1/2+\mathbb{Z}$.


\subsection{The collection of valid discs}

 To first approximation the set of valid discs consists of all discs of the form $D(\s)$ for a proper cluster $\s$, along with all integral discs between them. However, the precise definition is slightly more complicated, mainly owing to the failure of some proper clusters to have integer depth. The following lemma characterises this phenomenon. 
 
\begin{lemma} \label[lemma]{integral clusters}
Let $\s\subseteq \mathcal{R}$ be a proper cluster. Then $\s$ has a centre in $\Knr$. Moreover, $d_\mathfrak{s}\in \mathbb{Z}$ unless we have either
\begin{itemize}
\item[(i)] $\s=\{r,r'\}$,
or
\item[(ii)] $\s=\mathcal{R}$ has a unique proper child $\s'$, which has size $2g$, and $\s\setminus \s'=\{r,r'\}$,
\end{itemize}
for two inertia-conjugate roots $r,r'\notin K^{\textup{nr}}$. In these cases we have $d_\s\in 1/2+\mathbb{Z}$.
\end{lemma}

\begin{proof}
Since the hyperelliptic curve $C:y^2=f(x)$ satisfies the semistability criterion, $\mathcal{R}$ is tamely ramified and every proper cluster inertia invariant. In particular, every proper cluster has centre in  $\Knr$ by \Cref{lem:invcentre}. Now for a proper cluster $\s$, if $\s=\{r\} \wedge \{ r'\}$ for two inertia-conjugate roots $r,r'\notin \Knr$, then $d_\s=v(r-r')\in 1/2+\mathbb{Z}$,  whence $\s$ is not principal (see \Cref{principaldefi}). An easy case by case analysis shows that this forces $\s$ to fall into one of cases (i) or (ii) above. Otherwise, for any two (not necessarily proper) children $\s_1$ and $\s_2$ of $\s$ we may find centres $z_1$ and $z_2$ in $\Knr$ respectively. Then $v(z_1-z_2)$ is an integer and it follows that $d_\s$, being the minimum of such valuations as $\s_1,\s_2$ vary over all children of $\s$, is an integer also.
\end{proof} 

\begin{remark}
The lemma shows that for any proper cluster $\s$, the disc $D(\s)$ is integral save when $\s$ falls into one of the cases (i) or (ii) above.
\end{remark}

\begin{definition}[Valid discs] \label[definition]{valid discs defi}
Call $\mathcal{R}$ \emph{exceptional} (of type $A$ resp. $B$) if:
\begin{itemize}
\item[(A)] ~~
~~ $\mathcal{R}$ has a unique proper child $\s$, which has size $2g$, and $d_\mathcal{R}\in 1/2+\mathbb{Z}$,
\item[(B)]  ~~$\mathcal{R}$ has size $2g+2$ and a child of size  $2g+1$.
\end{itemize}
We define the integral disc $D_\textup{max}$ as follows. If $\mathcal{R}$ is not exceptional set $D_\textup{max}=D(\mathcal{R})$. If $\mathcal{R}$ is exceptional of type A define $D_{\textup{max}}$ to be the maximal integral disc cutting out the child of size $2g$, and if $\mathcal{R}$ is exceptional of type B let $D_{\textup{max}}$ be the defining disc of the child of size $2g+1$. 

We now say that an integral disc $D$ is \emph{valid} if it is contained in $D_{\textup{max}}$ and $|D\cap \mathcal{R}|\geq 2$. We denote the collection of all valid discs by $\mathcal{D}$, noting that $\mathcal{D}$ is admissible in the sense of \Cref{admissible defi}.  
\end{definition}

\begin{remark}
When $\mathcal{R}$ is exceptional of type A, if $z_\s\in K^{\textup{nr}}$ is a centre for the child $\s$ of size $2g$ then $D_\textup{max}=D_{z_\s,d_\mathcal{R} - 1/2}$. Moreover, it follows from \Cref{integral clusters} that $d_\mathcal{R}\in 1/2+\mathbb{Z}$ alone forces $\mathcal{R}$ to be exceptional of type A.
\end{remark}

\begin{notation} \label[notation]{fix centres notation}
For the rest of this section fix a choice of centre $z_D\in K^{\textup{nr}}$ for each valid disc $D$ (this is necessary to use the notation of the previous section when talking about the model of $\mathbb{P}^1_{K^{\textup{nr}}}$ associated to the collection of valid discs, since such a choice was made in \Cref{construction of the p1 sect}). 
\end{notation}

\subsection{The quantity $\kappa_D(f)$}

By \Cref{main model properties}, to study the divisor of $f(x)$ on the model  associated to the collection of valid discs we wish to understand the quantities $\nu_D(f)$ and  $\kappa_D(f)$  for integral discs $D$. We begin by considering the second of these functions.

\begin{lemma} \label[lemma]{kappa proposition}
Let $D$ be an integral disc. Then  $\kappa_D(f)=|D\cap \mathcal{R}|$ unless either of the following hold:
\begin{itemize}
\item[(i)] $\mathcal{R}$ is exceptional of type $A$ and $D=D_{\textup{max}}$, in which case $\kappa_D(f)=\textup{deg}(f)-1$,
\item[(ii)] $D=D_{z_\t,\lceil d_\t\rceil}$ for $\t$  a twin with $d_\t\in \frac{1}{2}+\mathbb{Z}$ and $z_\t$ a centre\footnote{any two choices for $z_\t$ give the same disc, c.f. \Cref{max subdisc lemma} below.}  for $\t$ in $\Knr$, in which case $D\cap \mathcal{R}=\emptyset$ and  $\kappa_D(f)=1$. 
\end{itemize}
In particular, if $D\neq D_\textup{max}$ is a valid disc then $\kappa_D(f)=|D\cap \mathcal{R}|$.
\end{lemma}
 
 Will will need the following easy lemma, whose proof we omit. 

\begin{lemma} \label[lemma]{max subdisc lemma}
Let $D$ be a disc with centre in $\Knr$ and non-integral depth $d_D$. Then $D$ has a unique maximal proper integral subdisc given by $D_{z_D,\lceil d_D \rceil}$ for any $z_D\in D\cap \Knr$. 
\end{lemma}


\begin{proof}[Proof of \Cref{kappa proposition}]
Fix an integral disc $D=D_{z_D,d_D}$ with $z_D\in \Knr$ and suppose $\kappa_D(f)\neq |D\cap \mathcal{R}|$. By \Cref{formula for kappa lemma} we have
\begin{equation*} \label{formula for kappa recalled}
\kappa_D(f)- \mid  D \cap \mathcal{R}\mid =\sum_{\substack{r\in \mathcal{R}\\ d_D-1<v(r-z_D)<d_D}}v\left(\frac{r-z_D}{\pi^{d_D-1}}\right).
\end{equation*}
Fix $r\in \mathcal{R}$ contributing non trivially to the right hand side. Then necessarily $r\notin \Knr$ and, denoting by $r'$ its inertia-conjugate root, \[d_D-1/2=v\left(r-z_D\right)=v(r'-z_D).\] 
Let $\s=\{r\}\wedge\{r'\}$ and pick a centre $z_\s\in \Knr$. Note that $d_\s =v(r-r')\notin \mathbb{Z}$ and as also $v(z_D-r)\notin \mathbb{Z}$ we must have
\begin{equation*} \label{bracketed expression}
v(z_D-r)=v\left(z_D-z_\s+z_\s-r\right)=v(z_\s-r)=d_\s.
\end{equation*}
Thus $z_D$ is a centre for $\s$ and we deduce that $D$ is  the maximal integral subdisc of $D(\s)$ afforded by \Cref{max subdisc lemma}. By \Cref{integral clusters,max subdisc lemma}, $D$ is now one of the discs claimed and the rest follows.
%
\end{proof}

\subsection{The parity of $\nu_D(f)$}

%

\begin{lemma} \label[lemma]{nu parity app}
Let $\s$ be a proper cluster with $d_\s\in \mathbb{Z}$. Then $\nu_\mathfrak{s}$ is even unless
$\s=\mathcal{R}$ has size $2g+2$ and $\mathcal{R}=\s_{1}\sqcup \s_{2}$ is a union of two odd children with one of the $\delta_{\s_i}$  odd (equivalently both if they have size $\geq 3$, cf. \Cref{le:AdamsLemma}). 
\end{lemma}

\begin{proof}
By the semistability criterion every principal cluster $\s$ has $\nu_\s$ even. Now combine \Cref{integral clusters} with \Cref{change of nu cluster}. 
\end{proof}

We now use this to characterise the valid discs $D$ for which $\nu_D(f)$ is odd (resp. even).

\begin{lemma} \label[lemma]{parity of nu thing}
Let $D$ be a valid disc and $\s=D\cap \mathcal{R}$. Then $\nu_D(f)$ is odd if and only if one of the following hold:
\begin{itemize}
\item[(i)] Both $\s$ and $d_D-d_\s$ are odd 
(in particular $D$ is not the defining disc of a cluster) or,
\item[(ii)] $D=D_\textup{max}$ and $\mathcal{R}=\s_1\sqcup \s_2$ is a  union of two odd proper children with $\delta_{\s_1}$  odd (equivalently $\delta_{\s_2}$ odd, cf. \Cref{le:AdamsLemma}). 
\end{itemize}
\end{lemma}

\begin{proof}
First note that by \Cref{kappa proposition} if $D_1\subseteq D_2$ are valid discs  and $D'\cap \mathcal{R}=D_1\cap \mathcal{R}$ for each intermediate integral disc  $D_1\subseteq D' \subsetneq D_2$ then
\begin{equation}\label{change of nu for discs}
\nu_{D_1}(f)=\nu_{D_2}(f)+|D_1\cap \mathcal{R}|(d_{D_2}-d_{D_1}).
\end{equation}

To prove the lemma, note first that if $D(\s)$ is integral and $\nu_\s$ even then we conclude by applying \cref{change of nu for discs} with $D_1=D(\s)$ and $D_2=D$. If $D=D_\textup{max}$ and $\mathcal{R}=\s_1\sqcup \s_2$ is a  union of two odd proper children then we conclude by applying \cref{change of nu for discs}  with $D_1=D(\s_1)$ and $D_2=D_\textup{max}$. Now by 
\Cref{integral clusters,nu parity app} the only remaining case is  $\s$ a twin with $\delta_\s \in 1/2+\mathbb{Z}$. Applying \cref{change of nu for discs} with $D_1=D$ and $D_2=D(P(\s))$ shows that $\nu_D(f)$ is even as desired, since $\nu_{P(\s)}$ is even by \Cref{nu parity app}. 
\end{proof}

 \subsection{The model of $\mathbb{P}^1$ associated to the collection of valid discs}
 
 \begin{definition}\label[definition]{ydisc}
 Define $\mathcal{Y}_{\textup{disc}}/\mathcal{O}$ to be the regular, proper, semistable model of $\mathbb{P}^1_{K^{\textup{nr}}}$ associated to the collection of valid discs via \Cref{the model of P1}. 
 \end{definition}
 
 We retain the notation of \Cref{models of P1} to talk about $\mathcal{Y}_{\textup{disc}}$. Thus the components of its special fibre are indexed by discs $D\in \mathcal{D}$, the component corresponding to a disc $D$ being isomorphic to $\mathbb{P}^1_{\bar{k}}$ with variable $x_D$, and denoted $E_D$.  The point at infinity on the generic fibre of $\mathcal{Y}_{\textup{disc}}$ is denoted $\infty$, and  $\overline{\{\infty\}}$ denotes the closure of this point in $\mathcal{Y}_{\textup{disc}}$.  Its intersection with the special fibre  is the point at infinity  on the component $E_{D_\textup{max}}$. 
 
 \begin{definition} \label[definition]{the divisor B}
 We denote by $B\in \textup{Div}(\mathcal{Y}_\textup{disc})$ the divisor
 \[B=\textup{div}(f)_\textup{hor}+(2g+2)\overline{\{\infty\}}+\sum_{D\in \mathcal{D}}\omega_D(f)E_D\] \label{explain infinity bar}
where $\omega_D(f)\in \{0,1\}$ is such that $\omega_D(f)\equiv \nu_D(f)~~\textup{(mod 2)}$, and $\textup{div}(f)_\textup{hor}$ denotes the horizontal part of the divisor of $f(x)$.
 \end{definition}

\begin{remark}
In the next section we will consider the normalisation  of $\mathcal{Y}_{\textup{disc}}$ in the function field of $C:y^2=f(x)$, which we denote $\mathcal{C}_{\textup{disc}}$. We show in \Cref{main regular model theorem} that the divisor $B$ above is the branch locus of the associated morphism $\mathcal{C}_{\textup{disc}}\rightarrow \mathcal{Y}_{\textup{disc}}$.
\end{remark}

To describe the divisor of $f(x)$ on $\mathcal{Y}_{\textup{disc}}$ it will be convenient to introduce the following notation. 

\begin{notation} \label{reduction of clusters}
For a proper cluster $\s$ with $D(\s)$ strictly contained in a valid disc $D$, by $\textup{red}_{D}(\s)$ we mean $\textup{red}_D(z_\s)$ (\Cref{model components}) for any centre $z_\s$ of $\s$. Note that this is independent of the choice of centre for $\s$. For $\s=\{r\}$ a singleton, for any valid disc $D$ containing $r$ we define $\textup{red}_{D}(\s):=\textup{red}_{D}(r)$.
\end{notation}
 
 \begin{proposition} \label[proposition]{branch locus lemma}
 Let $C/K:y^2=f(x)$ be a hyperelliptic curve satisfying the semistability criterion and let $\mathcal{Y}_\textup{disc}/\mathcal{O}$  and $B\in \textup{Div}(\mathcal{Y}_{\textup{disc}})$ be as above. 
Then $B$ is effective, regular, and $\textup{div}(f)\equiv B~~\textup{(mod 2)}$
inside $\textup{Div}(\mathcal{Y}_{\textup{disc}})$.
 Moreover, the horizontal part  $B_\textup{hor}$ of $B$ meets  the special fibre of $\mathcal{Y}_{\textup{disc}}$ at precisely the following points:
 \begin{itemize}
 \item the point at infinity on $E_{D_\textup{max}}$ if $\textup{deg}(f)=2g+1$ or $\mathcal{R}$ is exceptional,
 \item the points $x_D=\textup{red}_D(\s')$ on $E_D$ if  $D=D(\s)$ for a proper cluster $\s$ and $\s'<\s$ is either a singleton, or a twin with $d_{\s'}=d_\s+1/2$ (note that these points are distinct as we vary $\s'$),
 \item the point $x_D=\textup{red}_D(\t)$  on $E_D$ if $D\cap \mathcal{R}=\t$ for $\t$ a twin with $d_\t=d_D+1/2$.
 \end{itemize}
 \end{proposition}
 
 \begin{proof}
It follows from \Cref{main model properties} (i) along with the definition of $\omega_D(f)$ that $B\equiv \textup{div}(f)~~\textup{(mod 2)}$, and its clear that $B$ is effective.

Combining \Cref{main model properties} (iii), (iv) and (v) with \Cref{kappa proposition}, we see that $B_\textup{hor}$ is 
regular and meets the special fibre $\mathcal{Y}_{\textup{disc},\bar{k}}$  of $\mathcal{Y}_{\textup{disc}}$ precisely at the points claimed (in particular $B_\textup{hor}$ does not meet any intersection points between the components of $\mathcal{Y}_{\textup{disc},\bar{k}}$).  

Next it follows from \Cref{parity of nu thing} that if $D_1$ and $D_2$ are adjacent valid discs then $\nu_{D_1}(f)$ and $\nu_{D_2}(f)$ are not both odd. As any two components $E_{D_1}$ and $E_{D_2}$ are individually regular, and intersect if and only if $D_1$ and $D_2$ are adjacent, we deduce that the vertical part $B_\textup{ver}$ of $B$ is regular.

To conclude we claim that $B_{\textup{hor}}$ and $B_{\textup{ver}}$ are disjoint, i.e. that if $B_\textup{hor}$ meets $E_D$ then $\nu_D(f)$ is even. This follows from \Cref{parity of nu thing}. 
 \end{proof}
 
\subsection{The dual graph of $\mathcal{Y}_{\textup{disc}}$} \label[section]{dual graph of y sub}
 
 By \Cref{basic model lemma} the dual graph of $\mathcal{Y}_{\textup{disc}}$ is the graph $T_\mathcal{D}$  associated to the collection of valid discs as in \Cref{admissible defi}. Here we describe this graph in terms of clusters. We begin with an explicit description, and then recast this in terms of the purely combinatorial notions introduced in \cite{hyble}. For convenience, we summarise the relevant parts of that paper in \Cref{hyble appendix}. 
 
  \begin{notation}
 Let  $\widehat{T_\mathcal{D}}$ denote the graph obtained from  $T_\mathcal{D}$ by removing from the vertex set (keeping the underlying metric space the same) all $v_D$ corresponding to discs $D\neq D_\textup{max}$ for which there is a smaller valid disc cutting out the cluster $D\cap \mathcal{R}$.
 \end{notation}
 
  \begin{lemma} \label[lemma]{graph as clusters}
The graph $\widehat{T_\mathcal{D}}$ is, via $v_D\leftrightarrow v_{D\cap \mathcal{R}}$, the metric graph with:
  \begin{itemize}
 \item one vertex $v_\s$ for each cluster $\s$ which is not
 \begin{itemize}
 \item a twin $\t$ with $\delta_\t=1/2$,
 \item  $\mathcal{R}$ if either $\mathcal{R}$ has size $2g+2$ and a child of size $2g+1$, or a child $\s$ with $|\s|=2g$ and $\delta_\s=1/2$,
 \end{itemize} 
\item an edge of length $\lfloor \delta_\s \rfloor$ between $v_\s$ and $v_{P(\s)}$ (should $P(\s)$ yield a vertex).
\end{itemize}
\end{lemma}

\begin{proof}
This follows from \Cref{integral clusters} and the definition of the collection of valid discs.
\end{proof}

In \Cref{clreal} it is shown how, to $f(x)$, one may associate a purely combinatorial object called a \emph{metric cluster picture }   (\Cref{defclpic}). To this one may then formally associate a certain metric graph called a BY tree (this procedure is explained in \Cref{BY tree const}) which is very closely related to the dual graph of (the special fibre of) $\mathcal{Y}_{\textup{disc}}$, as the following result shows. This will be important later since it enables us to apply combinatorial results of \cite{hyble} to hyperelliptic curves.
  
%
%

%
 
 \begin{lemma} \label[lemma]{graph comparison 2}
Let  $\Sigma$ denote the metric cluster picture associated to $f(x)$, and $T_\Sigma$ denote the associated BY tree.
 Define $T_\Sigma'$ to be the  graph obtained from $T_\Sigma$ by (in order)
 \begin{itemize}
 \item halving the length of all yellow edges,
 \item replacing each edge of  length $l\in 1/2+\mathbb{Z}$ joining a leaf $v$ to a vertex $v'$, with an edge of length $\lfloor l \rfloor$ (identifying $v$ and $v'$ if $\lfloor l \rfloor=0$), 
   \item adding a vertex $v_\mathfrak{\mathcal{R}}$ on the edge joining $v_{\mathfrak{s}_1}$ and $v_{\mathfrak{s}_2}$, at a distance $\delta_{\s_1}$ from $v_{\mathfrak{s}_1}$, if $\mathcal{R}$ has size $2g+2$ and is a disjoint union of two proper clusters $\mathfrak{s}_1$ and $\mathfrak{s}_2$.
  \end{itemize}
  Then the map sending a vertex $v_D\in \widehat{T_\mathcal{D}}$ to the vertex $v_{D\cap \mathcal{R}}\in T_\Sigma'$ is an isomorphism of metric graphs.
 \end{lemma}
 
 \begin{proof}
Combine \Cref{graph as clusters} with the description of $T_\Sigma$ given in \Cref{BY tree const}.
 \end{proof}

  \begin{remark} \label[remark]{over extensions}
 When all the roots of $f(x)$ lie in $\Knr,$ so that all relative depths are integers, it follows from \Cref{graph comparison 2} that $\widehat{T_\mathcal{D}}$ is  simply the result of halving the length of all yellow edges in $T_\Sigma$ (up to possibly adding in a vertex corresponding to the cluster $\mathcal{R}$). 
 \end{remark}

\def\fr{{\mathfrak r}}
\def\R{{\mathcal R}}
\def\optional#1{{\color{cyan}#1}}
\def\d#1{d_#1}



%

\section{A regular model of $C$ assuming the semistability criterion}
\label[section]{scdisc}

As in \Cref{ydisc_section} let $C/K:y^2=f(x)$ be a hyperelliptic curve satisfying the semistability criterion (\Cref{semistability criterion}). We begin by using the results of the previous section to give an explicit regular proper model of $C/K^{\textup{nr}}$.  We then use this to determine the minimal proper regular model of $C$ and show that it is semistable. Having done this, we preceed to describe the special fibre of the minimal regular model and relate this to the cluster picture of $f(x)$. 

We continue to denote the ring of intgers of $K^{\textup{nr}}$ by $\mathcal{O}$. As in Notation \ref{fix centres notation} we fix centres $z_D\in K^{\textup{nr}}$ for each valid disc $D$ (see \Cref{valid discs defi}).  Let $\mathcal{Y}_\textup{disc}/\mathcal{O}$  be the  model of  $\P^1_{\Knr}$ associated to the collection of valid discs as in \Cref{ydisc}, and let $B\in \textup{Div}(\mathcal{Y}_{\textup{disc}})$ be as in \Cref{the divisor B}, so that by \Cref{branch locus lemma} $B$ is effective, regular, and congruent to $\textup{div}(f)$ modulo 2.
 
\begin{definition} \label[definition]{definition of cdisc}
 Define the scheme $\mathcal{C}_{\textup{disc}}/\mathcal{O}$ as the normalisation of $\mathcal{Y}_{\textup{disc}}/\mathcal{O}$ in the function field $K^{\textup{nr}}(C)$ of $C/K^{\textup{nr}}$. Denote  by $\phi:\mathcal{C}_{\textup{disc}}\rightarrow \mathcal{Y}_{\textup{disc}}$ the associated normalisation morphism. We write $\mathcal{C}_{\textup{disc},\bar{k}}$ for the special fibre of $\mathcal{C}_{\textup{disc}}$.
\end{definition}

 \begin{theorem} \label[theorem]{main regular model theorem}
Let $C/K:y^2=f(x)$ be a hyperelliptic curve satisfying  the semistability criterion. Then $\mathcal{C}_{\textup{disc}}/\mathcal{O}$  is a proper regular model of $C$ and $\phi:\mathcal{C}_{\textup{disc}}\rightarrow \mathcal{Y}_{\textup{disc}}$
is finite flat of degree $2$ with branch locus  $B$.
 \end{theorem}
 
 \begin{proof}
 That the normalisation morphism is finite flat of degree $2$ is standard (the degree being equal to the degree of $\Knr(C)/\Knr(x)$). In particular, $\mathcal{C}_{\textup{disc}}/\mathcal{O}$ is proper since both $\phi$ and $\mathcal{Y}_{\textup{disc}}/\mathcal{O}$ are. That $\mathcal{C}_\textup{disc}$ is regular is a consequence of the regularity of both $\mathcal{Y}_{\textup{disc}}$ and $B$ (the latter requiring the semistability criterion, see \Cref{branch locus lemma}), and follows from \cite[Lemma 2.1]{Sri} (taking $Y=\mathcal{Y}_{\textup{disc}}$ and $f=f(x)$ in the statement). Indeed, from the definition of $B$, one sees that, writing $\textup{div}(f)=\sum_{i\in I}m_i\Gamma_i$ as in the notation of loc. cit., the divisor $B$ is equal to $\sum_{i\in I}\bar{m}_i\Gamma_i$ for $\bar{m}_i\in \{0,1\}$  congruent to $m_i$ modulo $2$.  It's then clear that regularity of $B$ forces the hypothesis (a) and (b) of loc. cit. to be satisfied, guaranteeing the regularity of the normalisation of $\mathcal{Y}_{\textup{disc}}$ in $K^\textup{nr}(C)$. Finally, the claim about the branch locus  follows from the  description of the normalisation of the local rings given in the proof of loc. cit., noting that  for a point $P$ of $\mathcal{Y}_{\textup{disc}}$, $P$ lies on $B$ if and only if, in the expression for $\textup{div}(f)$ as $\sum_{i\in I}\bar{m}_i\Gamma_i$ above, $m_i$ is odd for a $\Gamma_i$ which contains $P$. 
 \end{proof}
 
 
 \begin{remark}  \label[remark]{hyperelliptic involution extension}
 Since we have defined $\mathcal{C}_{\textup{disc}}$ as the normalisation in $K^{\textup{nr}}(C)$ of a model of $\mathbb{P}^1_{K^{\textup{nr}}}$ it follows that the hyperelliptic involution on $C$ (which on function fields is the generator of the Galois group of $K^{\textup{nr}}(C)/K^{\textup{nr}}(x)$) extends to an involution $\iota$ on $\mathcal{C}_\textup{disc}/\mathcal{O}$  and identifies $\mathcal{Y}_{\textup{disc}}$ with the quotient $\mathcal{C}_{\textup{disc}}/\iota$. In particular, $\iota$ swaps the two points in the fibre over any point of $\mathcal{Y}_{\textup{disc}}\setminus B$. 
 \end{remark}
 
 \subsection{Explicit equations for the model $\mathcal{C}_{\textup{disc}}$}
 
We  now give explicit charts covering the model $\mathcal{C}_{\textup{disc}}$. Recall from \Cref{disc polynomials remark} that $\mathcal{Y}_{\textup{disc}}$ is covered by the open subschemes $U_D\setminus\mathcal{P}_D$ and $W_D\setminus (\mathcal{Q}_D\cup \mathcal{P}_D)$ (\Cref{schemes defi,model components}) as $D$ ranges over all valid discs. 

\begin{definition} \label[definition]{the charts of Cdisc}
Define, for each valid disc $D$, the schemes 
 \[\mathcal{U}_D=\textup{Spec}~ \frac{\mathcal{O}[x_D,y_D]}{\left(y_D^2-\pi^{\omega_D(f)}f_D(x_D)\right)}\] 
 and
 \[\mathcal{W}_D=\begin{cases}\textup{Spec}~\frac{\mathcal{O}[t_D,z_D]}{\left(z_D^2-\pi^{\omega_D(f)}t_D^{\lambda}g_D(t_D)\right)}~~&~~\textup{if }D=D_\textup{max,}\\\\\textup{Spec}~\frac{\mathcal{O}[s_D,t_D,z_D]}{\left(\pi-s_Dt_D~,~z_D^2-s_D^{\omega_D(f)}t_D^{\omega_{P(D)}(f)}g_D(s_D,t_D)\right)}~~&~~\textup{else,}\end{cases}\]
 where  $\lambda\in \{0,1\}$ is such that $\lambda\equiv  \textup{deg}(f)~~\textup{(mod 2)}$ and the polynomials $f_D$ and $g_D$ are as in  \Cref{defi of the polynomials}.  
 
Denote by $\phi_{D,1}:\mathcal{U}_D\rightarrow U_D$ and $\phi_{D,2}:\mathcal{W}_D\rightarrow W_D$ the morphisms induced by the obvious inclusions of rings.
\end{definition}

 \begin{proposition}\label[proposition]{normalisation of charts}
The model $\mathcal{C}_\textup{disc}$ is covered by the charts $\mathcal{U}_D\setminus \phi_{D,1}^{-1}(\mathcal{P}_D)$ and $\mathcal{W}_D\setminus \phi_{D,2}^{-1}(\mathcal{P}_D\cup \mathcal{Q}_D)$ as $D$ ranges over all valid discs. The restriction of the normalisation morphism to a map $\mathcal{U}_D\setminus \phi_{D,1}^{-1}(\mathcal{P}_D)\rightarrow U_D\setminus\mathcal{P}_D$  (resp. $\mathcal{W}_D\setminus \phi_{D,2}^{-1}(\mathcal{P}_D\cup \mathcal{Q}_D)\rightarrow W_D\setminus (\mathcal{P}_D\cup \mathcal{Q}_D) $)  is given by $\phi_{D,1}$ (resp. $\phi_{D,2}$). Inside $\Knr(C)$ we have the relations
\[x_D=\frac{x-z_D}{\pi^{d_D}},\phantom{hello}t_D=1/x_D,\phantom{hi}\phantom{hii}s_Dt_D=\pi, \phantom{hello}y_D=\pi^{(\omega_D(f)-\nu_D(f))/2}y\]
and
\[z_D=\begin{cases}x_D^{-\lfloor \textup{deg}(f)/2\rfloor}y_D &~~\textup{if }D=D_{\textup{max}},\\ s_D^{(\omega_D(f)-\nu_D(f))/2}t_D^{(\omega_{P(D)}(f)-\nu_{P(D)}(f))/2}y~~&~~\textup{else.}\end{cases}\] 
 \end{proposition}
 

 \begin{proof}
 Since $\mathcal{Y}_{\textup{disc}}$ is covered by the open subschemes $U_D\setminus\mathcal{P}_D$ and $W_D\setminus (\mathcal{Q}_D\cup \mathcal{P}_D)$, and normalisation is local on the base, it suffices to show that for each valid disc $D$, $\phi_{D,1}:\mathcal{U}_D\setminus \phi_{D,1}^{-1}(\mathcal{P}_D)\rightarrow U_D\setminus \mathcal{P}_D$ realises the normalisation of $U_D\setminus \mathcal{P}_D$ in $K^{\textup{nr}}(C)$, as well as the analagous result for $\phi_{D,2}$. 
 We prove this for $\phi_{D,1}$, the argument in the other case being identical. Viewing $U_D\setminus \mathcal{P}_D$ as a subscheme of $\mathcal{Y}_{\textup{disc}}$ it follows from \Cref{disc polynomials remark2} that the divisor of $\pi^{\omega_D(f)}f_D(x_D)$ on $U_D\setminus \mathcal{P}_D$ is equal to $B\cap (U_D\setminus \mathcal{P}_D)$ and as such is regular. Note also that the ring $\mathcal{O}[x_D]$ is regular. Now for any $h\in \mathcal{O}[x_D]$ such that the open subscheme $\{h\neq 0\}$ of $U_D$ is contained in $U_D\setminus \mathcal{P}_D$, it follows that both the ring $A=\mathcal{O}[x_D]_h$ and the closed subscheme $V(\pi^{\omega_D(f)}f_D(x_D))$ of $\textup{Spec} A$ are regular.  From this it follows easily that the ring $A'=A[y]/(y^2-\pi^{\omega_D(f)}f_D(x_D))$ is regular (this is the algebraic result underpinning the proof of \Cref{main regular model theorem}). Moreover, appealing to \Cref{disc polynomials remark2} once again we see that  $\pi^{\omega_D(f)}f_D(x_D)$ and $f(x)$ differ by a square in $\Knr(x)=\textup{Frac}A$ (indeed, by an even power of $\pi$). Thus $\textup{Frac}(A')=\Knr(x)[y]/(y^2-f(x))=\Knr(C)$. Since $A'$ is regular, finite over $A$ and $\textup{Frac}(A')=\Knr(C)$, it follows that $A'$ is the integral closure of $A$ in $K(C)$. Thus the (map on schemes associated to the) inclusion of $A$ into $A'$ realises the normalisation of $\textup{Spec}A$ in $\Knr(C)$. It remains only to note that  $U_D\setminus \mathcal{P}_D$ is covered by the schemes $\mathcal{O}[x_D]_h$ as we vary $h$, and that normalisation is local on the base. 
 
 The relationship between the various variables follows from combining \Cref{disc polynomials remark} and \Cref{disc polynomials remark2}.
 \end{proof}
 
 \begin{remark}
The extension of the hyperelliptic involution to $\mathcal{C}_{\textup{disc}}$ acts on the charts of \Cref{normalisation of charts} as $y_D\mapsto -y_D$ (resp. $z_D\mapsto -z_D$).
 \end{remark}
 
 \begin{remark} \label[remark]{multiplicity 2 remark}
 The proof of \Cref{normalisation of charts} shows that locally over a sufficiently small $U=\textup{Spec}A\subseteq \mathcal{Y}_{\textup{disc}}$, $\phi^{-1}(U)\subseteq \mathcal{C}_{\textup{disc}}$ is given by $\textup{Spec}A'$ where $A'=A[u]/(u^2-t)$ for $t$ a local equation for $B$ on $U$.  From this we deduce the following. Suppose that $D$ is a valid disc with $\omega_D(f)=1$ (i.e. $\nu_D(f)$ odd). Then $\phi^{-1}(E_D)$ consists of a single component $Z$ of multiplicity $2$  in $\mathcal{C}_{\textup{disc},\kbar}$ and the restriction of $\phi$ to a map $Z_\textup{red}\rightarrow E_D$ is an isomorphism.\footnote{here $Z_\textup{red}$ denotes the component $Z$ with its reduced structure.} Indeed, our assumptions mean that $E_D$ is contained entirely in $B$, so that locally over some $U$ as above,  $E_D$ corresponds to a prime ideal $\mathfrak{p}$ of $A$  containing $t$. The fibre of $\phi$ over the generic point of $E_D$ is then $\textup{Spec}~\textbf{k}(\mathfrak{p})[u]/(u^2)$ where $\textbf{k}(\mathfrak{p})$ denotes the residue field at $\mathfrak{p}$. Thus there is a unique component over $E_D$ with multiplicity $2$ in  $\mathcal{C}_{\textup{disc},\kbar}$. Moreover, the unique prime $\mathfrak{q}$ of $A'$ lying over $\mathfrak{p}$ is generated by $\mathfrak{p}$ and $u$ whence the map $A/\mathfrak{p}\rightarrow A'/\mathfrak{q}$ is an isomorphism. As this is just the map $Z_\textup{red}\cap \phi^{-1}(U)\rightarrow E_D\cap U$ and such $U$ cover $\mathcal{Y}_{\textup{disc}}$, we are done. 
 \end{remark}
 
The following lemma describes the reduction $\textup{mod } \mathfrak{m}$ of the polynomials $f_D$ and $g_{D_\textup{max}}$ appearing in the above charts, and will facilitate in the study of the special fibre of $\mathcal{C}_{\textup{disc}}$. 
 
 \begin{definition} \label[definition]{the leading term defi}
 For a valid disc $D$ define  $c_D\in \kbar^{\times}$ as
\[c_D=\frac{c_f}{\pi^{v(c_f)}}\prod_{r\in \mathcal{R}\setminus D}\left(\frac{z_D-r}{\pi^{v(z_D-r)}}\right)\quad(\textup{mod }\mathfrak{m}).\]
For a proper cluster $\s$ for which $D(\s)$ is valid we set $c_\s=c_{D(\s)}$.
 \end{definition}
 
 \begin{lemma} \label[lemma]{reduced polynomial lemma}
Let $D$ be a valid disc and write $\s=D\cap \mathcal{R}$. 
\begin{itemize}
\item[(i)] We have
\[f_D(x_D)\quad (\textup{mod }\mathfrak{m}) =\begin{cases} c_\s\prod_{\s'<\s}\left(x_D-\textup{red}_D(\s')\right)^{|\s'|}~~&~~ D=D(\s) \\ c_D\left(x_D-\textup{red}_D(\s)\right)^{|\s|}~~&~~\textup{else.} \end{cases}\]
\item[(ii)] For $D=D_\textup{max}$ we have
\[ t_D^\lambda g_D(t_D)\quad(\textup{mod }\mathfrak{m}) =t_{D}^{2g+2-|D\cap \mathcal{R}|}c_D\prod_{r\in D\cap \mathcal{R}}\left(1-\textup{red}_{D}(r)t_{D}\right)\]
where, as in \Cref{the charts of Cdisc}, $\lambda\in \{0,1\}$ is such that $\lambda\equiv  \textup{deg}(f)~~\textup{(mod 2)}$.
\end{itemize}
\end{lemma}

\begin{proof}
By the definition of $f_D(x_D)$ and \Cref{formula for nu lemma} we have
\[f_D(x_D)\quad (\textup{mod }\mathfrak{m})  = \frac{c_f}{\pi^{v(c_f)}}\prod_{r\in \mathcal{R}}\left(\frac{\pi^{d_D}x_D+z_D-r}{\pi^{\textup{min}\{d_D,v(z_D-r)\}}}\right)~~(\textup{mod }\mathfrak{m}) .\]
 Noting that $r\in D$ if and only if $\textup{min}\{d_D,v(z_D-r)\}=d_D$ we have
 \[f_D(x_D)\quad (\textup{mod }\mathfrak{m}) =c_D\prod_{r\in D}\left(x_D-\textup{red}_D(r)\right)\]
from which part (i) follows. For (ii), note that when $D=D_\textup{max}$ we  have $t_D^\lambda g_D(t_D)=t_D^{2g+2}f_D(1/t_D)$. Now apply (i). 
\end{proof}

\begin{remark}
In (i), two factors $(x_D-\textup{red}_D(\s'))$ and $(x_D-\textup{red}_D(\s''))$  are distinct  for distinct $\s',\s''<\s$. 
\end{remark}

\subsection{Types of valid disc}

We now describe the special fibre of $\mathcal{C}_{\textup{disc}}$. To do this we will frequently break into cases, which we set out here. We remark that if all roots of $f(x)$ lie in $K^{\textup{nr}}$ then cases II, III and IV do not occur as then all proper clusters have integer depth.

\begin{definition} \label[definition]{types of valid disc}
Let $D$ be a valid disc. We define the \textit{type} of $D$ as follows.
\begin{itemize}
\item[(I.1)] $D=D(\s)$ for a non-\ub~cluster $\s$ with $\nu_\s$ even, and $D$ is not of type II.1,
\item[(I.2)]$D=D(\s)$ for an \ub~cluster $\s$ and $D$ is not of type II.2, 
\item[(II.1)] $D=D(\s)$ where $\s<\mathcal{R}$ is such that $\delta_\s=1/2$, $|\s|=2g$, and $\s$ is not \ub,
\item[(II.2)] $D=D(\s)$ where $\s<\mathcal{R}$ is such that $\delta_\s=1/2$, $|\s|=2g$, and $\s$ is  \ub,
\item[(III)] $D=D_\textup{max}$ and there is a cluster $\s<\mathcal{R}$ with $|\s|=2g$ and $\delta_{\s}=n+1/2$ for $n\geq 1$ an integer,
\item[(IV)] $D\cap \mathcal{R}=\t$ for $\t$ a twin with $d_\t=d_D+1/2$,
\item[(V)] $\nu_D(f)$ is odd (i.e. $\omega_D(f)=1$, equivalently either $D=D_\textup{max}$ and $\mathcal{R}=\s_1\sqcup \s_2$ is a  union of two odd proper children with $d_{\s_1}-d_\mathcal{R}$ odd, or both $\mathcal{R}\cap D$ and $d_D-d_{\mathcal{R}\cap D}$ are odd, c.f. \Cref{parity of nu thing}),
\item[(VI.1)] $D$ does not fall into cases I-V and $|D\cap \mathcal{R}|$ is odd,
\item[(VI.2)] $D$ does not fall into cases I-V and $|D\cap \mathcal{R}|$ is even.
\end{itemize}
\end{definition}

 \subsection{The special fibre of $\mathcal{C}_\textup{disc}$}\label{se:specialfibre}
 
In the following proposition we describe, for each component $E_D$ of the special fibre of $\mathcal{Y}_{\textup{disc}}$, the component(s) of the special fibre of $\mathcal{C}_{\textup{disc}}$ which lie over this.  Later in \Cref{main dual graph thingy} we explain how these components fit together, drawing on the description of the dual graph of $\mathcal{Y}_{\textup{disc},\bar{k}}$ afforded by \Cref{graph comparison 2}.

 In the statement below, for a proper cluster $\s$ we write $g(\s)$ for the genus of $\s$, thus $g(\s)$ is 0 if  $\s$ is \ub\, and is determined by $|\so| = 2g(\s) +1$ or $2g(\s)+2$ where $\so$ denotes the set of odd children of $\s$.

 \begin{proposition} \label[proposition]{components of cdisc description}
Let $D$ be a valid disc, $E_D$ the associated component of  the special fibre of $\mathcal{Y}_{\textup{disc}}$ and $\phi:\mathcal{C}_\textup{disc}\rightarrow \mathcal{Y}_{\textup{disc}}$ the normalisation morphism. Then  $\phi^{-1}(E_D)$ consists\footnote{here and below, by $\phi^{-1}(E_D)$ we formally mean the scheme $\mathcal{C}_{\textup{disc}}\times _{\mathcal{Y}_{\textup{disc}}}E_D$.}, according to the type of $D$, of:
\begin{itemize}
\item[(I.1,II.1)] one component  of multiplicity $1$, with geometric genus $g(\s)$, and   one node for each twin $\t<\s$ with $d_\t=d_\s+1/2$, plus one additional node if $D$ has type II.1. The normalisation of this component is the hyperelliptic curve
\[\widetilde{\Gamma}_\mathfrak{s}:y^2=c_\mathfrak{s}\prod_{\textup{odd }\s'<\s}(x-\textup{red}_\mathfrak{s}(\s')).\]
\item[(I.2,II.2)] two components of multiplicity 1, each isomorphic to $\P^1_{\kbar}$ and intersecting transversally at one point for each twin $\t$ with $d_\t=d_\s+1/2$, and at one additional point if $D$ has type $II.2$.
\item[(III,IV)] two components of multiplicity 1, isomorphic to $\P^1_{\kbar}$ and intersecting transversally at 1 point.
\item[(V)]   a single component    isomorphic to $\P^1_{\kbar}$, with multiplicity $2$  and self intersection $-1$.  
\item[(VI.1)] one component of multiplicity 1, isomorphic to $\P^1_{\kbar}$,
\item[(VI.2)] two disjoint  components of multiplicity 1, each isomorphic to $\P^1_{\kbar}$.
\end{itemize}
Moreover, all intersections between components of $\mathcal{C}_{\textup{disc},\kbar}$ are transversal. \label{need more than this for dual graph? normal crossings}
\end{proposition}

\begin{remark} \label[remark]{simplification over extension}
When all roots of $f(x)$ lie in $\Knr$ (e.g. if we are willing to construct the model only over a suitable extension of the base field) all components are individually smooth, no two components lying over the same component of $\mathcal{Y}_{\textup{disc}}$ intersect, and, as above, types II, III and IV never arise.
\end{remark}

 The proof of \Cref{components of cdisc description} is given after \Cref{number of components cor,singular points}.

\begin{lemma} \label[lemma]{number of components cor} \label[lemma]{exceptional curve lemma}
Let $D$ be a valid disc, $E_D$ the associated component of the special fibre of $\mathcal{Y}_{\textup{disc},\kbar}$ and $\phi:\mathcal{C}_\textup{disc}\rightarrow \mathcal{Y}_{\textup{disc}}$ the normalisation morphism. Then
\begin{itemize}
\item[(i)]  If $D=D(\s)$ for a proper cluster $\s$ with  $\nu_\s$ even then $\phi^{-1}(E_D)$ consists of two multiplicity 1 components if $\s$ is \ub~(types I.2 and II.2), and one otherwise (types I.1 and II.1). The (geometric) genus of the component(s) is $g(\s)$.

\item[(ii)] If $\omega_D(f)=0$ and $D$ is not of the form  $D(\s)$ for a proper cluster $\s$ then $\phi^{-1}(E_D)$  consists of two multiplicity 1 components of genus 0 if $|D\cap \mathcal{R}|$ is even (types III, IV and VI.2), and one otherwise (type VI.2). 

\item[(iii)]  If $D$ has type V then $\phi^{-1}(E_D)$ consists of a single multiplicity 2 component, which is an exceptional curve\footnote{isomorphic to $\P^1_{\kbar}$ with self intersection $-1$.}, and meets all other components  transversally.
\end{itemize}
\end{lemma}

\begin{proof}
(i) and (ii). In what follows, let us temporarily denote the reduction of $f_D(x_D)$ $\textup{mod }\mathfrak{m}$ as $\overline{f_D(x_D)}$. Let $\eta_D$ be the generic point of $E_D$. From \Cref{normalisation of charts} we see that the fibre  over $\eta_D$ is 
\[\textup{Spec}~\kbar(x_D)[y_D]/\left(y_D^2-\overline{f_D(x_D)}\right)\]
and as  $\overline{f_D(x_D)}$ is non-zero the associated component(s) has multiplicity 1. That the number of components is as claimed follows since, by \Cref{reduced polynomial lemma},  if $D$ is not the defining disc of  $\s=D\cap \mathcal{R}$ then  $\overline{f_D(x_D)}$ is a square in $\kbar(x_D)$ if and only if $|\s|$ is even, whilst if $D=D(\s)$  then $\overline{f_D(x_D)}$ is a square if and only $\s$ is \ub. To determine the geometric genus of the components note that when $\overline{f_D(x_D)}$ is a square the fibre over the generic point of $\eta_D$ is isomorphic to two copies of $\kbar(x)$ from which it follows that both components have genus 0. Now note that, for  $\overline{f_D(x_D)}$ non-square, the genus of the function field $\kbar(x_D)[y_D]/\left(y_D^2-\overline{f_D(x_D)}\right)$ is precisely the number of odd degree factors of $\overline{f_D(x_D)}$ and we conclude by \Cref{reduced polynomial lemma}.  

(iii).   By \Cref{multiplicity 2 remark} we see that if $\omega_D(f)=1$ then there is a unique component over $E_D$, $\Gamma_D$ say, which has multiplicity $2$ and is isomorphic to $\P^1_{\kbar}$. It remains to show that $\Gamma_D$ has self intersection $-1$. Combining parts (i) and (ii) with \Cref{parity of nu thing} we see that $E_D$ meets precisely two components of the special fibre of $\mathcal{Y}_{\textup{disc}}$, $E_{D_1}$ and $E_{D_2}$ say, and there is one multiplicity 1 component of the special fibre of $\mathcal{C}_{\textup{disc}}$ lying over each of these, $\Gamma_{D_1}$ and $\Gamma_{D_2}$ say. In particular, as divisors we have $\phi^*(E_{D_1})=\Gamma_{D_1}$,  $\phi^*(E_{D_2})=\Gamma_{D_2}$ and $\phi_{*}(\Gamma_D)=E_D$. 
 By the projection formula \cite[Theorem 9.2.12]{LiuA} we have
$\Gamma_D\cdot \Gamma_{D_1}=E_D\cdot E_{D_1}=1$
and  $\Gamma_D\cdot \Gamma_{D_2}=1$ similarly. Thus $\Gamma_D$ intersects $\Gamma_{D_1}$ and $\Gamma_{D_2}$ transversally.
Finally as the intersection of $\Gamma_D$ with the whole special fibre  
 is 0 it follows that $\Gamma_D^2=-1$.
\end{proof}

%

\begin{lemma} \label[lemma]{singular points}
Let $D$ be a valid disc not of type V. Then $\phi^{-1}(E_D)$ has precisely one singular point for each twin $\t$ with $d_\t=d_D+1/2$, unless $D=D_\textup{max}$, when it has an additional singular point (lying over the point at infinity on $E_{D_\textup{max}}$) if $\mathcal{R}$ is exceptional of type A (i.e. if $D_\textup{max}$ has type II or III). Each singular point of $\phi^{-1}(E_D)$ is an ordinary double point.  
\end{lemma}

\begin{proof}
 Let $P$ be a closed point of $\phi^{-1}(E_D)$, let $Q=\phi(P)$ and suppose that  $Q \in B\cap E_D$ (if not $\phi$ is \'etale at $P$  whence, since $E_D$ is smooth, $P$ is a smooth point of $\phi^{-1}(E_D)$). If $Q\in B_{\textup{ver}}$ then as  $\omega_D(f)=0$, $P$ is a point of intersection between $\phi^{-1}(E_D)$ and a component lying over some $E_{D'}$ with $\omega_{D'}(f)=1$. By \Cref{exceptional curve lemma} such intersections are transverse whence (see e.g. \cite[Lemma 10.3.1]{LiuA}) they are smooth points of the individual components. Next, suppose $Q\in B_{\textup{hor}}\cap E_D$ and is not the point at infinity on $E_{D_\textup{max}}$. Then by \Cref{branch locus lemma}  $Q$ corresponds to a point $x_D=\textup{red}_{D}(\s)$ where $\s$ is  a child of $D\cap \mathcal{R}$ and is either a singleton, or a twin with $d_\s=d_D+1/2$. Now by \Cref{normalisation of charts} and again writing $\overline{f_D(x_D)}$ for the reduction of $f_D(x_D)$ $\textup{mod }\mathfrak{m}$, locally around $P$, $\phi^{-1}(E_D)$ is given by \[\textup{Spec}~\kbar[x_D,y_D]/(y_D^2-\overline{f_D(x_D)})\] with $P$ the point $(\textup{red}_{D}(\s),0)$.  By \Cref{reduced polynomial lemma} $\textup{red}_D(z_{D'})$ is a root of multiplicity $1$ if $\s$ is a singleton, and $2$ if $\s$ is a twin. One checks (see e.g. \cite[Example 10.3.4]{LiuA}) that $P$ is smooth point in the first instance, and an ordinary double point  in the second. The case $D=D_\textup{max}$ and $Q$ the point at infinity on $E_{D_\textup{max}}$ is similar. Our assumption that $Q\in B_{\textup{hor}}$ forces $\textup{deg}(f)=2g+1$ or $\mathcal{R}$ exceptional. Arguing as above and using \Cref{reduced polynomial lemma} (ii) one sees that $P$ is smooth if $\textup{deg}(f)=2g+1$ or $\mathcal{R}$ is exceptional of type B, and is an ordinary double point if $\mathcal{R}$ is exceptional of type A.
\end{proof}

\begin{proof}[Proof of \Cref{components of cdisc description}]
 
 (I.1, II.1): By \Cref{number of components cor} $\phi^{-1}(E_D)$ consists of a unique component of geometric genus $g(\s)$ whose function field is  \[\kbar(x_D)[y_D]/(y_D^2-\overline{f_D(x_D)}).\] The equation for  $\widetilde{\Gamma}_\mathfrak{s}$  follows from \Cref{reduced polynomial lemma}. That the nodes of $\Gamma_\s$ are as claimed is \Cref{singular points}. (I.2, II.2): By \Cref{number of components cor} $\phi^{-1}(E_D)$ consists of two components each of geometric genus $0$ and by \Cref{singular points} the singular points of $\phi^{-1}(E_D)$ are all ordinary double points and are in bijection with the twins $\t<\s$ with $d_\t=d_\s+1/2$, plus one additional ordinary double point for type II.2. For each such, $P$ say, we have $\phi(P)\in B$ whence $P$ is the unique point lying over $\phi(P)$. In particular, it lies on both components of $\phi^{-1}(E_D)$. It now follows from \cite[Lemma 10.3.11]{LiuA} that $P$ is a smooth point of each individual component, and that these components intersect transversally at $P$. In particular, each component  is smooth hence isomorphic to $\P^1_{\kbar}$. (III, IV): Follows from \Cref{number of components cor} and \Cref{singular points} similarly to cases I.2 and II.2. (V): \Cref{exceptional curve lemma} (iii). (VI): Combine \Cref{number of components cor,singular points} (the latter shows $\phi^{-1}(E_D)$ is smooth). 
 
Finally we show that all intersections are transverse. Let $P\in \mathcal{C}_{\textup{disc},\kbar}$ be a point lying on distinct components $Z_1$ and $Z_2$. If $\phi(Z_1)=\phi(Z_2)$ then both $Z_1$ and $Z_2$ lie in $\phi^{-1}(E_D)$ for some valid disc $D$ of type I.2, II.2, III or IV, and the intersection is transverse as above. Similarly, if one of $Z_1$ or $Z_2$ has multiplicity $2$ we are done by \Cref{exceptional curve lemma} (iii). Otherwise $Q=\phi(P)$ is disjoint from $B$ whence $\phi$ is \'etale at $P$ and, moreover, $Q$ is a point of transverse intersection between the distinct components $\phi(Z_1)$ and $\phi(Z_2)$ of $\mathcal{Y}_{\textup{disc}}$. Thus the intersection is transverse in this case also. 
\end{proof}

\subsection{The minimal regular model of $C/\Knr$}
\label{ssminregmod}

Having constructed a regular model of $C/K^{\textup{nr}}$ and described the components of its special fibre, it  is a simple matter to describe the minimal regular model $\mathcal{C}_{\textup{min}}/\mathcal{O}$ of $C/K^{\textup{nr}}$, which we do now. In particular, we use the explicit description we obtain to show that $\mathcal{C}_{\textup{min}}$ is semistable. This proves the `semistability criterion implies semistability' part of \Cref{redcond} (1) (see also \Cref{the semistability theorem}).

\begin{theorem} \label[theorem]{semistability criterion implies semistability}
Let $C/K$ be a hyperelliptic curve satisfying the semistability criterion. Then the model of $C$ obtained from $\mathcal{C}_{\textup{disc}}$ by contracting each of the components corresponding to valid discs of type $V$ is semistable, and is the minimal  regular model  of $C/\Knr$. 
\end{theorem}

\begin{proof}
Note first that by \Cref{components of cdisc description} (iv) the components we contract are disjoint, and are all exceptional curves. Thus we may contract them and the resulting model, which we temporarily denote $\mathcal{C}/\mathcal{O}$, is again proper and regular. Moreover, again by \Cref{components of cdisc description}, the components contracted are precisely those with multiplicity greater than one. Fix one such, $\Gamma_D$ say. Then as in the proof of \Cref{exceptional curve lemma}(iii) it intersects precisely two components of $\mathcal{C}_{\textup{disc}, \kbar}$, say $Z_1$ and $Z_2$ (which do not themselves meet), both transversally. By \cite[Lemma 3.35]{LiuA}, after contracting $\Gamma_D$ to a point the components $Z_1$ and $Z_2$ intersect transversally there. Thus the special fibre of $\mathcal{C}$ is reduced with normal crossings (away from the contracted components the same is true for the special fibre of   $\mathcal{C}_{\textup{disc}}$ by \Cref{components of cdisc description}, and the contraction map is an isomorphism here). Thus  $\mathcal{C}$ is semistable.

Since $\mathcal{C}$ is a proper regular model of $C$, to show that $\mathcal{C}$ is the minimal such we must show there are no exceptional curves in its special fibre. Note that such components appear in the dual graph of $\mathcal{C}_{\bar{k}}$ as genus 0 vertices of degree~1. \Cref{main dual graph thingy} below shows that the dual graph of $\mathcal{C}$ is (after removing vertices of degree~2 and genus 0 from the vertex set) a hyperelliptic graph in the sense of \cite[Definition 3.2]{hyble} and in particular contains no genus 0 vertices of degree~1 as desired (the statement of \Cref{main dual graph thingy} refers to the dual graph of $\mathcal{C}_{\textup{min},\bar{k}}$ however the proof  in fact uses the description of $\mathcal{C}_{\textup{min},\bar{k}}$ as the model $\mathcal{C}$ considered here, and does not assume minimality in the proof). Alternatively one may proceed via \Cref{components of cdisc description} and the description of the dual graph of $\mathcal{Y}_{\textup{disc},\bar{k}}$ afforded by \Cref{graph comparison 2}. 
\end{proof}

\subsection{The dual graph of the special fibre of the minimal regular model}

Having shown that $\mathcal{C}_{\textup{min}}$ is semistable, we may talk about the dual graph of its special fibre, which we now describe.

\begin{notation}
Let $C/K$  be a hyperelliptic curve satisfying the semistability criterion. We denote by $\Upsilon_C$ the dual graph of the special fibre of its minimal regular model. 
\end{notation}

Recall from \Cref{sss:dualgraph} that $\Upsilon_C$ has vertex set the set  of irreducible components of $\mathcal{C}_{\textup{min},\bar{k}}$, edge set the set of ordinary double points of $\mathcal{C}_{\textup{min},\bar{k}}$, and,  for an ordinary double point $P$, the edge-endpoints at $P$ are the points lying above $P$ under the normalisation morphism \[n:\widetilde{\mathcal{C}_{\textup{min},\bar{k}}}\rightarrow \mathcal{C}_{\textup{min},\bar{k}}\]
(we refer to these points as the \emph{tangents} at $P$). The graph $\Upsilon_C$ carries a natural genus marking, with a vertex being given the genus of the component to which it corresponds. Finally, by uniqueness of the minimal regular model, the hyperelliptic involution on $C$ extends (necessarily uniquely) to an involution $\iota$ on $\mathcal{C}_{\textup{min}}$ which then acts on $\Upsilon_C$ via its action on components, ordinary double points and tangents. 

As we did in \Cref{dual graph of y sub} for $\mathcal{Y}_{\textup{disc}}$, we will describe the graph $\Upsilon_C$ using the combinatorial framework developed in \cite{hyble} and summarised in \Cref{hyble appendix}. In particular, as explained in \Cref{clreal},  to $C$ (or rather to $f(x)$), one may associate a  \emph{metric cluster picture }  $\Sigma$, which is a purely combinatorial object.  Then, via \Cref{BY tree const,TtoG}, one can associate to $\Sigma$ a BY tree $T_\Sigma$, and a \emph{metric hyperelliptic graph} $G_\Sigma$, where the latter is a metric graph coming equipped with a genus marking and involution, the quotient by which is canonically the result of halving the length of all edges in $T_\Sigma$. This graph is (by design) very closely related to $\Upsilon_C$, as the following result shows.

\begin{theorem} \label[theorem]{main dual graph thingy}
Let $C/K$ be a hyperelliptic curve satisfying the semistability criterion. Denote by $\widehat{\Upsilon_C}$ the graph obtained from $\Upsilon_C$ by removing from the vertex set all vertices of genus 0 and degree $2$, and let $\Sigma$ denote the metric cluster picture associated to $C$. Then there is a  genus preserving isomorphism of metric graphs $\widehat{\Upsilon_C}\cong G_\Sigma$
identifying the hyperelliptic involutions and inducing a canonical identification of the quotient $\widehat{\Upsilon_C}/\iota$ with the graph obtained from $T_\Sigma$ by halving the length of all edges.

In particular $\widehat{\Upsilon_C}$, along with its genus marking and hyperelliptic involution, is a hyperelliptic graph in the sense of \cite[Definition 3.2]{hyble}. 
\end{theorem}

\begin{proof}
The basic idea is that $G_\Sigma$ is by definition a certain ramified double cover of $T_\Sigma$ (up to slightly adjusting the metric), whilst our explicit construction of $\mathcal{C}_{\textup{min}}$ shows that $\Upsilon_C$ is essentially a ramified double cover of the dual graph of $\mathcal{Y}_{\textup{disc}}$. The latter is related to $T_\Sigma$ via \Cref{graph comparison 2}. We now flesh out the details. In what follows, we will need to understand the action of the hyperelliptic involution on $\mathcal{C}_{\textup{min}}$ and do this by noting that where $\mathcal{C}_{\textup{disc}}$ and $\mathcal{C}_{\textup{min}}$ are isomorphic the two extensions of the hyperelliptic involution must agree, since they do so on the generic fibre. 

We first describe the dual graph of $\mathcal{C}_{\textup{disc},\bar{k}}$, which we temporarily denote $\Omega_C$.\footnote{We have only defined the dual graph of a semistable curve - whilst $\mathcal{C}_{\textup{disc},\bar{k}}$ is not semistable it is when given its reduced structure (see \Cref{components of cdisc description}) and we formally mean the dual graph of this curve.} By \Cref{hyperelliptic involution extension}, whenever the fibre over a point of $\mathcal{Y}_{\textup{disc}}$ contains two points  the hyperelliptic involution swaps these. In particular, applying this to the generic point of a component $E_D$ of $\mathcal{Y}_{\textup{disc},\bar{k}}$, if $\phi^{-1}(E_D)$ consists of two components then $\iota$ swaps these. Similarly, if there are two points  lying over an intersection point of components $E_D$ and $E_{D'}$ then $\iota$ swaps these  also. 

Now consider the (topological) quotient graph $\Omega_C/\iota$. Its vertices either arise as $\iota$-orbits of vertices of $\Omega_C$ - we get one vertex in $\Omega_C/\iota$ for each valid disc this way, corresponding to $\phi^{-1}(E_D)$  -  or as the midpoint of an edge whose endpoints are swapped  (i.e. is `$\iota$-anti-invariant'). Similarly, the edges of $\Omega_C/\iota$ are $\iota$-orbits of edges of $\Omega_C$, but in the case of an $\iota$-anti-invariant edge $e$, the resulting edge of $\Omega_C/\iota$ has length half that of $e$. 

As $\iota$ swaps components if and only if they lie over the same component of $\mathcal{Y}_{\textup{disc}}$, an edge of $\Omega_C$ can only be $\iota$-anti-invariant if it corresponds to an ordinary double point on $\phi^{-1}(E_D)$ for some valid disc $D$ (as opposed to an ordinary double point arising as the intersection  between components whose images in $\mathcal{Y}_\textup{disc}$ are distinct). Moreover, each such edge is in fact $\iota$-anti-invariant. Indeed, for edges corresponding to  intersection points between two components of some $\phi^{-1}(E_D)$ this is clear, whilst for those corresponding to a node lying on a single component, say a node $P$, $\iota$ must swap the tangents at $P$ else by \cite[Proposition 3.48 (b)]{LiuA} $\phi(P)$ would not be a smooth point of $\mathcal{Y}_{\textup{disc}}$.

From the above discussion we see that $\Omega_C/\iota$ is the graph obtained from the dual graph  of $\mathcal{Y}_{\textup{disc}}$ by adding a vertex for each ordinary double point of $\mathcal{C}_{\textup{disc},\bar{k}}$ lying over a single component $E_D$, joined to the vertex  corresponding to $E_D$ by an edge of length $1/2$. It now follows from \Cref{singular points} and \Cref{graph comparison 2}  that, defining $\widehat{\Omega_C}$ as for $\widehat{\Upsilon_C}$, the quotient $\widehat{\Omega_C}/\iota$ is canonically the result of halving all yellow edges in $T_\Sigma$. Finally (c.f. \Cref{parity of nu thing}), one obtains $\widehat{\Upsilon_C}/\iota$ from $\widehat{\Omega_C}/\iota$ by halving the length of those edges corresponding to blue edges of $T_\Sigma$ to account for contracted components.

To conclude, we now note that by  \Cref{components of cdisc description} the map $\widehat{\Upsilon_C}\rightarrow \widehat{\Upsilon_C}/\iota$ ramifies precisely over the part of $\widehat{\Upsilon_C}/\iota$ which corresponds to the blue part of $T_\Sigma$, so that the canonical isomorphism  $\widehat{\Upsilon_C}/\iota\cong G_\Sigma/\iota$ above lifts to a (in general noncanonical) isomorphism $\widehat{\Upsilon_C}\cong G_\Sigma$ identifying the hyperelliptic involutions.
\end{proof}

\subsection{Explicit equations for the components of the special fibre of the minimal regular model}

By \Cref{main dual graph thingy} the components of the special fibre of $\mathcal{C}_{\textup{min}}$ are indexed by valid discs $D$ not of type V (though one disc can yield two components not one). Here we record explicit equations  for these components.  

\begin{definition}\label[definition]{components defi}
For a valid disc $D$ not of type $V$ we define $\Gamma_D$ to be the image in $\mathcal{C}_{\textup{min}}$ (under the contraction morphism)  of the closed subscheme $\phi^{-1}(E_D)$ of $\mathcal{C}_{\textup{disc}}$. Note that $\Gamma_D$ consists of either one or two components of $\mathcal{C}_{\textup{min},\kbar}$, and that these account for all components.
\end{definition} 

\begin{proposition} \label[proposition]{components of min}
Let $D$ be a valid disc not of type $V$. Then the associated component(s) $\Gamma_D$ of $\mathcal{C}_{\textup{min},\kbar}$ is given, according to the type of $D$, as follows\footnote{other than in cases II and III, the curves are understood to have the additonal chart at infinity corresponding to our conventions for (possibly singular) hyperelliptic curves set out in Appendix \ref{ap:apphyp}.}:
\begin{itemize}
\item[(I)]
$$
   Y^2 = c_\s
    \prod_{{\text{odd }}{{\mathfrak o} < \s}}(X-\red_\s(\mathfrak o))
    \prod_{{\text{twin }{\mathfrak t}<\s}\atop{\delta_\t=1/2}}(X-\red_\s(\mathfrak t))^2.
$$
\item[(II)] The glueing of the affine curves
\[Y^2=c_\s
    \prod_{{\text{odd }}{{\mathfrak o} < \s}}(X-\red_\s(\mathfrak o))
    \prod_{{\text{twin }{\mathfrak t}<\s}\atop{\delta_\t=1/2}}(X-\red_\s(\mathfrak t))^2\quad \quad (\dagger)\]
and
    \[Z^2=c_\s T^{2}\prod_{{\text{odd }}{{\mathfrak o} < \s}}(1-\red_\s(\mathfrak o)T)
    \prod_{{\text{twin }{\mathfrak t}<\s}\atop{\delta_\t=1/2}}(1-\red_\s(\mathfrak t)T)^2\]
over the subsets $Y\neq 0$ and $T\neq 0$ via $Z=YT^{n+1}$, where $2n$ is the degree of the right hand side of $(\dagger)$.
\item[(III)]The glueing of the affine curves
\[Y^2=c_D\quad \subseteq \quad \mathbb{A}^2_{X,Y}\]
and
\[Z^2=c_DT^2\]
over the subsets $X\neq 0$ and $T\neq 0$ via $X=1/T$, $Z=YT$,

\item[(IV)] 
$$
   Y^2 = c_D(X-\red_D(\mathfrak t))^2,
$$

\item[(VI.1)]   \[Y^2=c_D(X-\textup{red}_D(\s)),\] where $\s=D\cap \mathcal{R}$,
\item[(VI.2)]  \[Y^2=c_D\quad \subseteq \quad \mathbb{A}^2_{X,Y}.\] 
\end{itemize}
\end{proposition}

\begin{proof}
We will extract the result from the explicit equations for $\mathcal{C}_{\textup{disc}}$ given in \Cref{normalisation of charts}. Fix a valid disc $D$ and temporarily denote by $\Gamma_D'$ the curve with equations as in the statement, so that we must prove $\Gamma_D'\cong \Gamma_D$. By \Cref{normalisation of charts},  the scheme  $\mathcal{U}_{D}\setminus\phi^{-1}(\mathcal{P}_D)$ of \Cref{the charts of Cdisc} is an open subscheme of $\mathcal{C}_{\textup{disc}}$. Its special fibre is an open subset of $\phi^{-1}(E_D)$ and is disjoint from all components contracted when forming $\mathcal{C}_{\textup{min}}$. Thus its special fibre is naturally  an open subscheme of $\Gamma_D$ and  equations for it are readily extracted from \Cref{reduced polynomial lemma}. In particular, we may define a rational map $\psi:\Gamma_D'\rightarrow \Gamma_D$ by
\begin{equation}\label{disgusting map}
(x_0,y_0)\mapsto \left(x_0,y_0h_D(x_0)\right)
\end{equation}
where the right hand side is understood with respect to the variables $x_D,y_D$ for the chart $\mathcal{U}_{D}\setminus\phi^{-1}(\mathcal{P}_D)$ and we set
\begin{equation}\label{the poly h def}
h_D(x)=\begin{cases}\prod_{{\s'<\s}\atop{\delta_{\s'}>1/2}}(x-\textup{red}_{\s}(\s'))^{\lfloor \frac{|\s|}{2}\rfloor} ~~&~~\textup{if }D\textup{ has type I, II}\\ \left(x-\textup{red}_D(D\cap \mathcal{R})\right)^{\lfloor \frac{|D\cap \mathcal{R}|}{2}\rfloor}
~~&~~\textup{if } D\textup{ has type III, VI }\\ 1~~&~~\textup{if }D\textup{ has type IV}.\end{cases}\end{equation} 
Since $\mathcal{P}_D$ is precisely the set of points \[\{\left(\textup{red}_{\s}(\s'),0\right)~~|~~\s'<\s~~\textup{and}~~ \delta_{\s'}>1/2\}\] if $D$ has type I or II, and consists of the single point $\left(\textup{red}_{D}(D\cap \mathcal{R}),0\right)$ if $D$ has type III or VI, we see that $\psi$ is invertible on the open subset  $U=\Gamma_D\cap \left(\mathcal{U}_{D}\setminus\phi^{-1}(\mathcal{P}_D)\right)$ of $\Gamma_D$. Suppose first that $D$ is not of type II or III. Then by \Cref{singular points} (and its proof) all singular points of $\Gamma_D$ lie in $U$. 
Similarly, $\Gamma_D'$ is visibly smooth away from $\psi^{-1}(U)$ also. Thus $\psi$ is a rational map between complete (possibly reducible) curves $\Gamma_D$ and $\Gamma_D'$ which is an isomorphism away from finitely many smooth points. Thus $\psi$ is in fact an isomorphism everywhere.

The  case where $D$ has type II or III is proved identically with the caveat that  one must explicitly check that the map is an isomorphism over an open subset of the point at infinity on $\Gamma_D$ (since unlike the other cases the curves are not smooth here). This may be done by using the chart $\mathcal{W}_{D_\textup{max}}$ of \Cref{normalisation of charts} combined with \Cref{reduced polynomial lemma} (ii). 
\end{proof}

\subsection{Reducing points}

Since $\mathcal{C}_\textup{min}/\mathcal{O}$ is proper there is a natural reduction map $C(\Knr)\rightarrow \mathcal{C}_{\textup{min},\kbar}(\kbar)$ whose image consists precisely of the non-singular points (by regularity of $\mathcal{C}_{\textup{min}}$ and the fact that $\Knr$ is Henselian). We describe this map in the following proposition.

\begin{proposition} \label[proposition]{reducing points}
Let $C:y^2=f(x)$ be a hyperelliptic curve satisfying the semistability criterion and let $P=(x_0,y_0)\in C(\Knr)$.
\begin{itemize}
 \item[(i)] Suppose $x_0\in D_\textup{max}$ and let $D$ be a valid disc not of type V. Then $P$ reduces to $\Gamma_{D}$ if and only if $x_0\in D$ but $x_0\notin D'$ for any valid subdisc $D'\subseteq D$. In this case, in the variables $X,Y$ of \Cref{components of min}, the reduction of $P$ is, according to the type of $D$,  the following point on $\Gamma_D$: 
 \begin{itemize}
 
 \item[(I,II):] 
 \[\left(\textup{red}_D(x_0)~,~\overline{\pi^{-\nu_D(f)/2}y_0}\prod_{{\s'<\s}\atop{\delta_{\s'}>\frac{1}{2}}}\left(
  \red_\s(x_0)-\red_\s(\s')\right)^{-\lfloor \frac{|\c'|}{2}\rfloor}\right),\]

  \item[(III,VI):] 
  \[\left(\textup{red}_D(x_0)~,~\overline{\pi^{-\nu_D(f)/2}y_0}\left(\textup{red}_D(x_0)-\textup{red}_D(r)\right)^{-\lfloor \frac{|D\cap \mathcal{R}|}{2}\rfloor}\right)\]
  for any choice of $r\in D\cap\mathcal{R}$,
    \item[(IV):]
 \[\left(\textup{red}_D(x_0)~,~\overline{\pi^{-\nu_D(f)/2}y_0}\right).\]
 \end{itemize}
 \item[(ii)] Suppose $x_0\notin D_\textup{max}$. Then $\textup{red}(P)$ is a point at infinity on $\Gamma_{D_\textup{max}}$. There is a unique such point unless $|\mathcal{R}\cap D_\textup{max}|=2g+2$, in which case $D$ has type I 
and $P$ reduces to 
\[\left(0\quad,\quad\frac{\pi^{-\frac{v(c_f)}{2}}y_0}{(x_0-z_{\R})^{g+1}}\quad (\textup{mod }\mathfrak{m})\right)\]
in the variables for the chart at infinity of (the equation given in  \Cref{components of min}  for) $\Gamma_{D_\textup{max}}$.
\end{itemize}
\end{proposition}

\begin{proof}
As in \cite[Definition 10.1.3]{LiuA} $\textup{red}(P)$ is the unique point of intersection of $\overline{\{P\}}$ (the closure of $P$ in $\mathcal{C}_{\textup{min}}$) with $\mathcal{C}_{\textup{min},\kbar}$. In particular the reduction map may be computed locally. Now for a valid disc $D\neq D_\textup{max}$ with $\omega_D(f)=0$, the non-singular points of $\Gamma_D$ are all visible on the chart $\mathcal{U}_{D}\setminus\phi^{-1}(\mathcal{P}_D)$ of  \Cref{normalisation of charts}. Changing variables from $x,y$ to $x_D,y_D$ (c.f. \Cref{normalisation of charts}) we see that $P$ corresponds to the point 
\[\left(\frac{x_0-z_D}{\pi^{d_D}},\pi^{-\nu_D(f)/2}y_0\right)\]
on the generic fibre of $\mathcal{U}_D$. One checks readily that the closure of this point in $\mathcal{U}_D$ contains a point  of the special fibre if and only if  $(x_0-z_D)/\pi^{d_D}$ is integral, i.e. if and only if $x_0\in D$. When this is the case the point is given by 
\[\left(\textup{red}_D(x_0),\overline{\pi^{-\nu_D(f)/2}y_0}\right).\]
Moreover, this lies in $\mathcal{U}_D\setminus \phi_{D,1}^{-1}(\mathcal{P}_D)$ if and only if $x_0$ is not in any valid subdisc of $D$. To compete the proof it remains to change variables from $x_D,y_D$ to the variables for $\Gamma_D$ of \Cref{components of min}, which simply serves to multiply the second coordinate by $h_D(\textup{red}_D(x_0))^{-1}$ with $h_D$ as in \cref{the poly h def}.

The case $D=D_\textup{max}$ may be treated similarly, additionally considering the chart $\mathcal{W}_{D_\textup{max}}$ to prove (ii). 
\end{proof}

\subsection{The stable model of $C/\Knr$} \label{stable model subsection}

We finish the section by giving an explicit description of the stable model of a hyperelliptic curve satisfying the semistability criterion. As in the statement of \Cref{components of cdisc description}, for a principal cluster $\s$ we denote by $\widetilde{\Gamma}_\mathfrak{s}$ the hyperelliptic curve
\[\widetilde{\Gamma}_\mathfrak{s}:y^2=c_\mathfrak{s}\prod_{\textup{odd }\s'<\s}(x-\textup{red}_\mathfrak{s}(\s')).\]

\begin{theorem}\label[theorem]{stable model}
Let $C/K$ be a hyperelliptic curve satisfying the semistability criterion. Then the stable model $\mathcal{C}_{\textup{st}}/\mathcal{O}$ of $C/\Knr$ is obtained from $\mathcal{C}_{\textup{disc}}$ by contracting each component $\Gamma_D$ for which $D$ is not of the form $D(\s)$ for a principal cluster $\s$.\footnote{That such curves may be contracted will be shown during the course of the proof.} Its special fibre consists of one component with normalisation $\tilde{\Gamma}_\s$ for each non-\ub~principal cluster $\s$,  and two components with normalisation $\P^1_{\kbar}$ for each \ub~ principal cluster. Letting $\Sigma$ denote the metric cluster picture associated to $C$, the dual graph of the special fibre of $\mathcal{C}_{\textup{st}}$ is obtained from  $G_\Sigma$  by adjusting the length function to give each edge length~$1$.
\end{theorem}

\begin{proof}
By \Cref{semistability criterion implies semistability}, upon contracting components $\Gamma_D$ for which $\omega_D(f)=1$ we obtain the minimal regular model of $C/\Knr$.  The stable model is then obtained from this by contracting all components which have self-intersection $-2$ and are isomorphic to $\P^1_{\kbar}$ (\cite[Proposition 9.4.8, Theorem 10.3.34]{LiuA}). Since such components are precisely the ones which give vertices of genus 0 and degree $2$ in the dual graph, i.e. precisely the ones which, in the notation of \Cref{main dual graph thingy}, are removed from the vertex set when passing from $\Upsilon_{C}$ to $\widehat{\Upsilon_{C}}$, this proves the claim about the dual graph of $\mathcal{C}_{\textup{st}}$. It remains to show that the components which remain are precisely those corresponding to principal clusters. For this one may either argue via \Cref{components of cdisc description} and a case by case analysis or use \cite[Lemmas 5.5, 5.20]{hyble} to note that the vertices of $G_\Sigma$ correspond precisely to the principal clusters of $\Sigma$.
\end{proof}

\def\fr{{\mathfrak r}}
\def\R{{\mathcal R}}
\def\optional#1{{\color{cyan}#1}}
\def\d#1{d_#1}

\section{Galois action on the models of $C$}
\label{sgalois}

In this section we still work with a hyperelliptic curve $C/K:y^2=f(x)$ and write $\mathcal{R}\subseteq \bar{K}$ for the set of roots of $f(x)$. However we now no longer assume that $C/K$ satisfies the semistability criterion, and instead fix a finite Galois extension $F/K$  such that the semistability criterion is satisfied by $C/F$. For example we may take $F/K$ to be the Galois closure of the extension given by adjoining a square root of a uniformiser to $K(\mathcal{R})/K$, though the case where $F=K$ and $C/K$ satisfies the semistability criterion is still an important special case of the results of this section. We denote by $\pi_F$ a uniformiser for $F$ and write $e$ for the ramification index of $F/K$. 

Since $C/F$ satisfies the semistability criterion we may apply the constructions of \Cref{ydisc_section,scdisc} with $K$ replaced by $F$ throughout, so that in particular we have an explicit construction of the minimal regular model and stable model of $C$ over $F^{\textup{nr}}$, afforded by \Cref{semistability criterion implies semistability} and \Cref{stable model} respectively. The aim of this section is to describe explicitly the action of $G_K$ on these models (we recall how this action works in \Cref{start of action} below). To talk about these models we use the notation of \Cref{ydisc_section,scdisc} replacing $K$ by $F$ throughout. Thus for example we work with the collection of valid discs (\Cref{valid discs defi}) defined with respect to $F$ rather than $K$ and fix a choice of centre $z_D\in F^{\textup{nr}}$ for each valid disc $D$. We caution however that we continue to normalise the valuation $v$ on $\bar{K}$ with respect to $K$, which is the reason for the appearance of the ramification index $e$ in the formulae below. 

\subsection{Galois action on components} \label[section]{start of action}
Let $\cC_{\min}/\cO_{\Fnr}$ be the minimal regular model of $C$ over $\Fnr$. We will work with the explicit description of $\cC_{\min}/\cO_{\Fnr}$  afforded by \Cref{semistability criterion implies semistability} and \Cref{components of min}. Thus the components of its special fibre (or more precisely the $\iota$-orbits of components where $\iota$ denotes the hyperelliptic involution) are indexed by valid discs $D$ not of type V (\Cref{types of valid disc}), with $D$ corresponding to the component(s) $\Gamma_D$ of  \Cref{components defi}.

As explained in \Cref{ss:gcurves} (\cref{skew} in particular), the special fibre of $\cC_{\min}/\cO_{\Fnr}$ carries a natural action of the full Galois group $G_K$ (arising from uniqueness of the model) which is uniquely determined 
by `lift-act-reduce' on non-singular points. The following quantities will facilitate in describing this action. 

\begin{definition} \label[definition]{the characters defi}
We define, for $\sigma\in G_K$, 
\[\chi(\sigma) = \frac{\sigma(\pi_F)}{\pi_F} \quad(\textup{mod }\mathfrak{m}).\]
Given also a valid disc $D$ not of type V define
$$
  \beta_{D}(\sigma) = \frac{\sigma(z_D)-z_{\sigma D}}{\pi_F^{ed_{D}}} \quad(\textup{mod }\mathfrak{m}),
  $$
  and
  \[\lambda_D=\begin{cases}\frac{\nu_{\c}}{2} -  d_{\s}\!\sum_{{\s'<\s}\atop{\delta_{\c'}>\frac{1}{2}}} \lfloor \frac{|\c'|}{2}\rfloor~~&~~D~~\textup{ has type I, II}\\\\\frac{\nu_D(f)}{2}-d_D\lfloor \frac{|D\cap \mathcal{R}|}{2}\rfloor ~~&~~D~~\textup{ has type III, VI}\\\\\frac{\nu_D(f)}{2}~~&~~D~~\textup{ has type IV,}\end{cases}\]
  where, for types I and II, $\s$ denotes the cluster $D\cap \mathcal{R}$ and we recall that by definition we have $D=D(\s)$ in these instances.
\end{definition}

We now describe the Galois action on the $\Gamma_D$.

\begin{theorem}\label[theorem]{th:GeneralGaloisAction}
Let $C/K:y^2=f(x)$ be a hyperelliptic curve satisfying the semistability criterion over a finite Galois extension $F/K$ and let $D$ be a valid disc (the collection of such defined with respect to $F$) not of type V.  Then any $\sigma\in G_K$  maps  $ \Gamma_{D}$ to $\Gamma_{\sigma D}$ and for a  point $P=(x_0,y_0)\in \Gamma_D$, we have 
 $$
 \sigma(P)=  \left(\chi(\sigma)^{ed_{D}}\bar{\sigma}(x_0)+\beta_{D}(\sigma),\chi(\sigma)^{e\lambda_D}\bar{\sigma}(y_0)\right) \in \Cmp_{\sigma D}
 $$
 where $\bar{\sigma}$ denotes the map induced by $\sigma$ on the residue field $\kbar$ and the points are written with respect to the variables $X,Y$ for $\Gamma_D$ (resp $\Gamma_{\sigma D}$) of \Cref{components of min}.
\end{theorem}

\begin{proof}
As explained in \Cref{ss:gcurves}  it suffices to prove the result under the assumption that $P$ is non-singular. We follow the recipe of \Cref{ss:gcurves} to act on $P$. Let $\tilde{P}\in C(\Fnr)$ be a lift of $P$. Now let $h_D$ be as in \cref{the poly h def} and $\tilde{h}_D$ any lift of $h_D$ to a polynomial with coefficients in $\cO_{F^{\textup{nr}}}$. By \Cref{reducing points}, in the coordinates $x,y$ for $C$, $\tilde{P}$ has the form  
\[\left(\pi_F^{ed_D}\tilde{x}_0+z_D~,~\pi_F^{e\nu_D/2}\tilde{y}_0\tilde{h}_D(\tilde{x}_0)\right)\]
for some $\tilde{x}_0,\tilde{y}_0\in  \cO_{F^{\textup{nr}}}$ with $\tilde{x}_0\equiv x_0~~(\textup{mod }\mathfrak{m})$ and $\tilde{y}_0\equiv y_0~~(\textup{mod }\mathfrak{m})$ and such that the $x$-coordinate is not in any valid subdisc of $D$. Under $\sigma$ this maps to the point
\[\tilde{Q}=\left(\sigma(\pi_F)^{ed_D}\sigma(\tilde{x}_0)+\sigma(z_D)~,~\sigma(\pi_F)^{e\nu_D(f)/2}\sigma(\tilde{y}_0)\sigma(\tilde{h}_D(\tilde{x}_0))\right)\in C(\Fnr).\]
Note that the $x$-coordinate of $\tilde{Q}$ is in $\sigma D$ but not in any valid subdisc of $\sigma D$, since the set of valid discs is stable under $\sigma$. Thus this new point reduces to $\Gamma_{\sigma D}$ by \Cref{reducing points}. In particular $\sigma$ maps $\Gamma_D$ to $\Gamma_{\sigma D}$. Moreover, applying \Cref{reducing points} one last time and noting that $\sigma$ preserves depths and $\nu$, we find that $\tilde{Q}$ reduces to
\[\left(\chi(\sigma)^{ed_D}\bar{\sigma}(x_0)+\beta_D(\sigma),\frac{\chi(\sigma)^{e\nu_D(f)/2}\bar{\sigma}(y_0)\bar{\sigma}(h_D(x_0))}{h_{\sigma{D}}\left(\chi(\sigma)^{ed_D}\bar{\sigma}(x_0)+\beta_D(\sigma)\right)}\right)\in \Gamma_{\sigma D}.\]
In light of the definition of $\lambda_D$ it remains to show that 
\[\bar{\sigma}(h_D(x_0))=\chi(\sigma)^{-ed_D\textup{deg}(h_D)}h_{\sigma D}\left(\chi(\sigma)^{ed_D}\bar{\sigma}(x_0)+\beta_D(\sigma)\right).\]
Writing $h_D^{\bar{\sigma}}(x)$ for the polynomial obtained by applying $\bar{\sigma}$ to the coefficients of $h_D$ we wish to prove the polynomial identity
\[h_D^{\bar{\sigma}}(x)=\chi(\sigma)^{-ed_D\textup{deg}(h_D)}h_{\sigma D}\left(\chi(\sigma)^{ed_D}x+\beta_D(\sigma)\right).\]
Noting that for any $z\in \Kbar$ we have 
\begin{eqnarray}\label{sigma of reduction}\bar{\sigma}\textup{red}_D(z)&=&\bar{\sigma}\left(\frac{z-z_D}{\pi_F^{ed_D}}\right)
\end{eqnarray}
\begin{eqnarray*}\phantom{hellohellohello}~~~&=&\chi(\sigma)^{-ed_D}\left(\textup{red}_{\sigma D}(\bar{\sigma}z)-\beta_D(\sigma)\right)\end{eqnarray*}
and that $h_D^{\bar{\sigma}}$ is the monic polynomial whose (multi)set of roots is given by applying $\bar{\sigma}$ to the (multi)set of roots of $h_D$, the result follows.
\end{proof}

\subsection{Galois action on the normalisation of components}

In what follows, for any cluster $\s$ for which $D(\s)$ is a valid disc (in particular, for all principal clusters) we write $\Gamma_\s$ in place of $\Gamma_{D(\s)}$. Note that by \Cref{stable model} the $\Gamma_\s$ for $\s$ principal account for precisely those components which remain when passing from the minimal regular model of $C/F^{\textup{nr}}$ to the stable model, and by an abuse of notation we denote  the associated component(s) of the stable model by $\Gamma_\s$ also. Here we describe the Galois action on the normalisation of these components.

\begin{definition} \label[definition]{normal cmps}
For a  cluster $\c$  for which $D(\s)$ is valid (so in particular  for all principal clusters) define
$$
 \quad \tilde{\lambda}_\c = \frac{\nu_{\c}}{2} -  d_{\s}\!\sum_{\s'<\s} \lfloor \frac{|\c'|}{2}\rfloor. \quad
  $$
  Define also \[\tilde{\Gamma}_\s:y^2=c_\s\prod_{\textup{odd }\s'<\s}(x-\textup{red}_\s(\s')).\]
  By \Cref{components of cdisc description} this is the normalisation of $\Gamma_\s$ viewed either on the minimal regular model of $C/F^{\textup{nr}}$, or, for $\s$ principal, the stable model of $C/F^{\textup{nr}}$.
\end{definition}

\begin{cor}[of \Cref{th:GeneralGaloisAction}]\label[corollary]{galois action on normalisation corr}
Let $C/K:y^2=f(x)$ be a hyperelliptic curve satisfying the semistability criterion over a finite Galois extension $F/K$. Let $\sigma\in G_K$, $\s$  a principal cluster, and  $\Gamma_\s$ the associated component(s) of the special fibre of either the minimal regular model, or the stable model, of $C/F^{\textup{nr}}$. Then $\sigma$ maps $\Gamma_\s$ to $\Gamma_{\sigma \s}$. Moreover, if $\sigma$  stabilises $\s$ then the action of $\sigma$ on the normalisation $\tilde{\Gamma}_\s$ is,
\ for a point $P=(x_0,y_0)$ in the variables $x,y$ of \Cref{normal cmps}, given by  $$
\sigma(P) =  \left(\chi(\sigma)^{ed_{\c}}\bar{\sigma}(x_0)+\beta_{\c}(\sigma),\chi(\sigma)^{e\tilde{\lambda}_\c}\bar{\sigma}(y_0)\right).
 $$ 
\end{cor}

\begin{proof}
Combine \Cref{th:GeneralGaloisAction} with \Cref{le:apphyp} (iii).
\end{proof}

\subsection{Galois action on the dual graph}

As in \Cref{sss:dualgraph} the action of $G_K$ on the special fibre of the minimal regular model of $C/F^{\textup{nr}}$ induces an action on its dual graph $\Upsilon_C$ via the action on  components, ordinary double points, and tangents. Here we describe this action, beginning with the following lemma.

\begin{definition}{(cf. \Cref{de:epsilon})} \label{the choice of ell defi}
Let $\mathcal{E}$ denote the set of even clusters which do not have an \ub~ parent, excluding $\mathcal{R}$ unless $\mathcal{R}$ is \ub.  For each cluster $\s\in \mathcal{E}$, fix a square root $\theta_\s$ of 
\[c_f\prod_{r\notin \s}(z_\s-r).\]
Having made this choice define, for each $\sigma\in \textup{G}_K$ and $\s\in \mathcal{E}$,
\[\epsilon_\s(\sigma)=\frac{\sigma(\theta_\s)}{\theta_{\sigma \s}}\quad(\textup{mod }\mathfrak{m}).\]
\end{definition}

\begin{lemma} \label[lemma]{choices of tangents lemma}
For each $\s\in \mathcal{E}$, the above choice $\theta_\s$ of square root of 
$c_f\prod_{r\notin \s}(z_\s-r)$
determines:
\begin{itemize}
\item[(i)] if $\s$ is a twin with $\delta_\s=1/2$, a choice of tangent at the node $(\textup{red}_{P(\s)}(\s),0)$ on $\Gamma_{P(\s)}$,
\item[(ii)] if $\s$ has size $2g$, is not \ub, and $\delta_\s=1/2$ (i.e. $D(\s)$ has type II.1) a choice of tangent at the node at infinity on $\Gamma_\s$,
\item[(iii)] a choice of one of the two points at infinity on $\widetilde{\Gamma_\s}$ otherwise.
\end{itemize}
\end{lemma}

\begin{proof}
We begin with (iii) which is the simplest case. The  points at infinity on $\widetilde{\Gamma_\s}$ are $(0,\pm \sqrt{c_\s})$. Now we compute
\begin{equation} \label{1st cl relation} 
c_\s=\frac{\theta_\s^2}{\pi_F^{e(\nu_\s-|\s|d_\s)}}\quad(\textup{mod }\mathfrak{m}).
\end{equation}
By \Cref{integral clusters,nu parity app} $e(\nu_\s-|\s|d_\s)$ is even, so that our choice of $\theta_\s$ determines a square root of $c_\s$ and hence a choice of a point at infinity, namely the point
\[\left(0,\frac{\theta_\s}{\pi_F^{e(\nu_\s-|\s|d_\s)/2}}\quad(\textup{mod }\mathfrak{m})\right).\]
In case (ii)  (see \Cref{tangents app} and \Cref{components of min}), the two tangents at the node are similarly given by the points $(0,\pm \sqrt{c_\s})$ and we proceed as in case (iii).

Finally, in case (i),  (see \Cref{tangents app} and \Cref{components of min} again) the two tangents are the points
\[\left(\textup{red}_{P(\s)}(\s),\pm \sqrt{c_{P(\s)}\prod_{\textup{odd }\mathfrak{o}<P(\s)}\left(\textup{red}_{P(\s)}(\s)-\textup{red}_{P(\s)}(\mathfrak{o})\right)}\right)\in \widetilde{\Gamma}_{P(\s)}.\]
This time, we compute
\begingroup\smaller[1]
\begin{equation*} \label{2st cl relation eq}
c_{P(\s)}\prod_{\textup{odd }\mathfrak{o}<P(\s)}\left(\textup{red}_{P(\s)}(\s)-\textup{red}_{P(\s)}(\mathfrak{o})\right)=\frac{\theta_\s^2}{\pi_F^{e(\nu_{P(\t)}-2d_{P(\t)})}}\left(\prod_{\substack{\s'<P(\t)\\\s'\neq \t}}\left(\frac{z_\t-z_{\s'}}{\pi_F^{ed_{P(\t)}}}\right)^{-\lfloor \frac{|\s'|}{2}\rfloor}\right)^2\quad(\textup{mod }\mathfrak{m}).
\end{equation*}
\endgroup
Again$e(\nu_{P(\t)}-2d_{P(\t)})$ is even  (\Cref{integral clusters,nu parity app}) , whence our choice of $\theta_\s$ determines a choice of one of the tangents, namely 
\[\left(\textup{red}_{P(\s)}(\s),\frac{\theta_\s}{\pi_F^{e(\nu_{P(\t)}-2d_{P(t)})/2}}\prod_{\substack{\s'<P(\t)\\\s'\neq \t}}\left(\frac{z_\t-z_{\s'}}{\pi_F^{ed_{P(\t)}}}\right)^{-\lfloor \frac{|\s'|}{2}\rfloor}\quad(\textup{mod }\mathfrak{m})\right).\]
\end{proof}

We now return to describing the action of $G_K$ on $\Upsilon_C$. Let $\Sigma$ denote the metric cluster picture associated to $f(x)$  over $F$ (\Cref{clreal}), associated hyperelliptic graph $G_\Sigma$.  By \Cref{main dual graph thingy} we have 
$\widehat{\Upsilon_C}\cong G_\Sigma$ where $\widehat{\Upsilon_C}$ is the graph obtained from $\Upsilon_C$ by removing from the vertex set all vertices of genus 0 and degree $2$. As explained in \Cref{the automorphism of g}, to each pair $\rho=(\rho_0,\epsilon_\rho)$ where $\rho_0$ is a permutation of the set $\Sigma$ of proper clusters preserving sizes, inclusions and relative depths, and $\epsilon$ is a collection of signs $\epsilon_\rho(\s)\in \{\pm 1\}$ for each cluster $\s\in \mathcal{E}$, there is an associated automorphism $G(\rho)$ of $G_\Sigma$.

\begin{theorem} \label[theorem]{th:ActionOnDualGraph} 

Let $C/K:y^2=f(x)$ be a hyperelliptic curve satisfying the semistability criterion over a finite Galois extension $F/K$. Denote by $\Upsilon_C$  the dual graph of the special fibre of the minimal regular model of $C/F^{\textup{nr}}$ and $\Sigma$  the metric cluster picture associated to $f(x)$ over $F$. Fix a choice of $\theta_\s$ for each $\s\in \mathcal{E}$ as in \Cref{the choice of ell defi}.

Then is an isomorphism of metric graphs $\widehat{\Upsilon_C}\cong G_\Sigma$ under which the action of any $\sigma\in G_K$  corresponds to the automorphism 
\[\left(\rho(\sigma),(\epsilon_\s(\sigma))_{\s \in \mathcal{E}}\right)\]
of $G_\Sigma$, where $\rho(\sigma)$ is the permutation of the proper clusters of $\Sigma$ induced by the natural action of $\sigma$ on the set of roots of $f(x)$.
\end{theorem}

\begin{proof}
Fix an isomorphism $\widehat{\Upsilon_C}\cong G_\Sigma$ as in \Cref{main dual graph thingy}, so that the isomorphism identifies the respective hyperelliptic involutions (denoted $\iota$) and induces the canonical identification of the quotients  $\widehat{\Upsilon_C}/\iota$ and $G_\Sigma/\iota$ detailed there. Note that our choice of $\theta_\s$ for each $\s\in \mathcal{E}$ determines via \Cref{choices of tangents lemma} a choice $n_\s^+$ of (below $v_\Gamma$ denotes the vertex of $\Upsilon_C$ corresponding to a component $\Gamma$):
\begin{itemize}
\item an endpoint of the loop at $v_{\Gamma_{P(\s)}}$  associated to the node $(\textup{red}_{P(\s)}(\s),0)$ on $\Gamma_{P(\s)}$ if $\s$ is a twin with $\delta_\s=1/2$,
\item an  endpoint  of the loop at infinity on $v_{\Gamma_\s}$ when $|\s|=2g$ and $\delta_\s=1/2$,
\item one of the two vertices corresponding to $\Gamma_\s$ if $\s=\mathcal{R}$ (indeed, note that the points at infinity of $\Gamma_\s$ lie on different components),
\item an edge endpoint  at $v_{\Gamma_\s}$ for one of the two edges between  $v_\s$ and $v_{P(\s)}$ otherwise (here if $\Gamma_\s$ consists of two components then by this we mean a choice of an edge-endpoint for one of the two edges meeting one of two associated vertices).
\end{itemize}
Composing our chosen isomorphism $\widehat{\Upsilon_C}\cong G_\Sigma$ with an automorphism of $G_\Sigma$ of the form $(\textup{id},\eta)$ for an appropriate choice of $\eta$ we may assume that the choices $n_\s^+\in \widehat{\Upsilon_C}$ ($\s\in \mathcal{E}$)  get identified with the corresponding `plus' choice arising from the decomposition of $G_y$ into $G_y^+$ and $G_y^-$  (see \Cref{TtoG} for the definition of these objects). 

Now fix $\sigma\in \textup{Gal}(\bar{K}/K)$ and view it as an automorphism of $G_\Sigma$ via the identification above. Since $\sigma$ fixes $\Gamma_\mathcal{R}$ (or $\Gamma_\s$ if $\mathcal{R}$ has size $2g+2$ and a child $\s$ of size $2g+1$) and perserves genera, by \Cref{all things arise how we want} there is an automorphism $\tau=(\tau_0,\epsilon_\tau)$ of $\Sigma$ such that $\sigma$ acts as $G(\tau)$; we will show that $\tau$ must be as in the statement of the theorem. We now determine $\tau_0$ and $\epsilon_\tau$, using the explicit description of automorphisms of this form afforded by \Cref{explicit aut action}, to which we also refer for the definition of vertices $v^{\pm}_\bullet$ and (half-)edges $e^{\pm}_\bullet$ appearing below.

By \Cref{th:GeneralGaloisAction}, $\sigma$ maps $\Gamma_\s$ to $\Gamma_{\sigma\s}$ so that on $G_\Sigma$, $\sigma$ maps $\{v_{\s}^{\pm}\}$ to $\{v_{\sigma \s}^{\pm}\}$ for each principal cluster $\s$.  Similarly, by \Cref{th:GeneralGaloisAction} and \ref{sigma of reduction} we see that  for a twin $\t$, $\sigma$ maps a node $(\textup{red}_{P(\t)}(\t),0)$ on $\Gamma_{P(\t)}$ to the node $(\textup{red}_{\sigma P(\t)}(\sigma \t),0)$ on $\Gamma_{\sigma P(\t)}$, so that $\sigma$ maps $\{e_\t^{\pm}\}$ onto $\{e_{\sigma \t}^{\pm}\}$ for each twin $\t$. It follows that $\tau_0=\rho(\sigma)$ as desired.\footnote{Note that a permutation of the proper clusters preserving size and inclusion is determined by its action on principal clusters and twins.}

For the signs, fix a cluster $\s\in \mathcal{E}$ with $\delta_\s\neq 1/2$ (i.e. case (iii) of \Cref{choices of tangents lemma}).  Then $n_\s^+$ is the specified point at  infinity on $\Gamma_\s$. By \Cref{le:apphyp}(i) and \Cref{galois action on normalisation corr} we have
\[\sigma(n_\s^+)=\begin{cases} n_{\sigma \s}^+~~&~~\frac{\chi(\sigma)^{e\lambda_\s}}{\chi(\sigma)^{ed_\s(n+1)}}\cdot \frac{\sigma(\sqrt{c_\s})}{\sqrt{c_\s}}=1\\~~n_{\sigma \s}^-~~&~~\textup{ else,}\end{cases}\]
where  $2n+2$ is the degree of the defining polynomial of $\Gamma_\s$. 
Using \Cref{1st cl relation} and the definition of $\lambda_\s$, we compute
\[\frac{\chi(\sigma)^{e\lambda_\s}}{\chi(\sigma)^{ed_\s(n+1)}}\cdot \frac{\sigma(\sqrt{c_\s})}{\sqrt{c_\s}}=\frac{\sigma(\theta_\s)}{\theta_{\sigma\s}}~~\quad(\textup{mod }\mathfrak{m})~~=~~\epsilon_\s(\sigma).\]
Comparing this with the action of $G(\tau)$ on $G_\Sigma$ detailed in \Cref{explicit aut action} we see that $\epsilon_\tau(\s)=\epsilon_\s(\sigma)$ as desired.

Finally, for the nodes we compute using   \Cref{th:GeneralGaloisAction}, \Cref{choices of tangents lemma} and  \Cref{le:apphyp}(ii)  that $\epsilon_\tau(\s)=\epsilon_\s(\sigma)$ similarly.
\end{proof}

\def\fr{{\mathfrak r}}
\def\R{{\mathcal R}}
\def\optional#1{{\color{cyan}#1}}
\def\d#1{d_#1}


\section{The semistability criterion is equivalent to semistability}
\label{scriterion}

We now complete the proof that the semistability criterion (\Cref{semistability criterion}) is equivalent to semistability.

\begin{theorem} \label[theorem]{the semistability theorem}
Let $C:y^2=f(x)$ be a hyperelliptic curve over $K$. Then $C/K$ is semistable if and only if it satisfies the semistability criterion.
\end{theorem}

\begin{proof}
When $C/K$ satisfies the semistability criterion \Cref{semistability criterion implies semistability} 
gives an explicit semistable model of $C$ over $\OKnr$. Since semistability may be checked after unramified extension 
it follows that $C$ is semistable over $K$. 

Now suppose that $C/K$ is semistable. We will show that $K(\mathcal{R})/K$ is tamely ramified and that each principal cluster $\s$ is fixed by inertia, has $d_\s\in \mathbb{Z}$ and $\nu_\s\in 2\mathbb{Z}$. This is equivalent to the semistability criterion by \Cref{weakss}.

As $C/K$ is semistable so is its Jacobian $\textup{Jac}(C)/K$ (\cite[Theorem 2.4]{DM}) whence the inertia group of $K$ acts unipotently on the $2$-adic Tate module of $\textup{Jac}(C)$ (\cite[SGA$7_1$,IX,3.5/3.8]{Gro}). It follows that $K(J[2])/K$ is tamely ramified. As $K(\textup{Jac}(C)[2])=K(\mathcal{R})$ (see e.g. \cite[Lemma 2.1]{Cor})  $K(\mathcal{R})/K$ is tame.

Now consider the stable model $\cC_\textup{st}/\OKnr$ (which exists since $C/K$ is assumed semistable).  
Fix a tame extension $F/K$, ramification degree $e$ say, over which $C$ satisfies the semistability criterion 
(e.g. a quadratic ramified extension of $K(\mathcal{R})/K$) and set $I_{F/K}=\Gal(\Fnr/\Knr)$. 
By \cite[Lemma 10.3.30]{LiuA} the formation of the stable model commutes with base change, in other words
the stable model of $C$ over $\Fnr$ is
\begin{equation*}
\label{stable model equation}
  \cC'=\cC_\textup{st}\times_{\OKnr} \OFnr.
\end{equation*} 
In particular, the unique extension of the action of $I_{F/K}$ on $C/\Fnr$ to $\cC'$ is via the second factor 
 and becomes trivial upon passing to the special fibre.     

On the other hand, since $C$ satisfies the semistability criterion over $F$ we have an explicit description 
of the stable model over $\OFnr$ complete with action of $I_{F/K}$ on its special fibre afforded by 
\Cref{stable model} and \Cref{galois action on normalisation corr}. In order that this action be trivial we 
see from \Cref{galois action on normalisation corr} that each principal cluster must be fixed 
by $I_{F/K}$ (and hence the full inertia group of $K$) else $I_{F/K}$ would permute components of $\cC'_{\kbar}$. 
Moreover, since the character $\chi$  of \Cref{the characters defi} has exact order $e$ when restricted 
to $I_{F/K}$, for each principal cluster $\s$ we deduce from \Cref{galois action on normalisation corr}  
that both $d_\s$ and $\tilde{\lambda}_\s$ must be integers. Equivalently $d_\s\in \mathbb{Z}$ and 
$\nu_\s\in 2\mathbb{Z}$ as desired.
\end{proof}

\section{Special fibre of the minimal regular model}
\label{sfibre}

Here we collect and present the relevant notation and results from Sections \ref{ydisc_section}, \ref{scdisc} and \ref{sgalois} for the convenience of the reader. In particular we present the special fibre of the minimal regular model of $C/F^{\textup{nr}}$ in a self-contained manner, that does not refer to the constructions in \cite{hyble}. 

Let $C/K:y^2 =f(x)$ be a hyperelliptic curve and $F/K$ a finite Galois extension  over which $C$ becomes semistable. By Theorem \ref{the semistability theorem} $C/F$ satisfies the semistability criterion, so that all the constructions of Sections \ref{ydisc_section}, \ref{scdisc} and \ref{sgalois} are valid over $F$.

For the rest of this section we fix the following data. 

\begin{notation}
Fix as above a finite Galois extension $F/K$ over which $C$ is semistable and let $\pi_F$ denote a fixed choice of uniformiser of $F$. For each proper cluster $\s$, fix a centre $z_\s\in F^{\textup{nr}}$ (possible by \Cref{integral clusters}). Additionally, for every even cluster $\c\neq \mathcal{R}$ that does not have an \ub~parent, and for $\mathcal{R}$ if it is \ub,  fix  a square root $\theta_\s$ of 
$
 c_f\prod\nolimits_{r \notin \c} (z_\c-r).
$

We write  $e$ for the ramification degree of $F/K$ and $v$ for the valuation on $\bar{K}$ normalised with respect to $K$, so that, in particular, $v(\pi_F)=1/e$.
\end{notation}

\subsection{Components and characters}
\label{sclusters}

\begin{definition}\label{de:characters}
For $\sigma \in G_K$ set 
$$
 \chi(\sigma) = \frac{\sigma(\pi_F)}{\pi_F} \mod \mathfrak{m}.
$$
For principal clusters $\c$ define
$$
  \tilde{\lambda}_\c = \frac{\nu_{\c}}{2} -  d_{\s}\!\sum_{\s'<\s} \lfloor \frac{|\c'|}{2}\rfloor \qquad \text{ and }
$$
$$
  \alpha_\c(\sigma) = \chi(\sigma)^{ed_\c}, \qquad\quad
  \beta_{\c}(\sigma) = \frac{\sigma(z_{\c})-z_{\sigma\c}}{\pi_F^{ed_{\c}}} \mod \mathfrak{m}, \qquad \quad 
  \gamma_\c(\sigma) = \chi(\sigma)^{e\tilde{\lambda}_\c}.
$$

If $\c$ is either even or a cotwin, define 
$\epsilon_\s:G_K\to \{\pm 1\}$ by
$$
  \epsilon_{\c}(\sigma) \equiv \frac{\sigma(\theta_{\c^*})}{\theta_{\neck{(\sigma\s)}}} \mod \mathfrak{m}.
$$
For all other clusters $\s$, set $\epsilon_\s(\sigma)=0$. 
\end{definition}

\begin{remark}
\label{legaldisc}
On the inertia group $I_K<G_K$ the map $\chi$, and therefore the $\alpha_\s$ 
and $\gamma_\s$ as well, are independent of the choices of $F$, $\pi_F$ and $z_\s$, and are characters $I_K\to\bar k^\times$ which are trivial on wild inertia.
When restricted to the stabiliser $I_\s$, the character $\gamma_\s$ has order 
the prime-to-$p$ part of the denominator of $|I_K/I_\s|\,\tilde\lambda_\s$.
\end{remark}

\begin{definition}\label{de:gammas}\label{de:red}
For a principal cluster $\c$ 
define $c_{\c} \in \bar{k}^\times$ by $$c_\c =\hat{c}_f \prod_{r \notin \c}\widehat{(z_\c-r)} 
~~\mbox{ mod } \mathfrak{m}$$
and set
\begin{equation*}\label{redmap}
\begin{array}{llllll}
\red_\s(t)=  \frac{t-z_\s}{\pi_F^{ed_\s}}\mod\mathfrak{m}
         \end{array}
\end{equation*}
for those $t\in \bar{K}$ for which the  above formula makes sense.
For $\s'<\s$, by $\red_\s(\s')$ we mean $\red_\s(r)$ for any $r \in \s'$. 

If $\s$ is a principal cluster and $e\delta_\s\neq 1/2$ 
we define the hyperelliptic curve $\Cmp_\s/\bar k$ by 
$$
  \qquad \qquad
  \Cmp_\s :\>\> Y^2 = c_\s
    \prod_{{\text{odd }}{{\mathfrak o} < \s}}(X-\red_\s(\mathfrak o))
    \prod_{{\text{twin }{\mathfrak t}<\s}\atop{e\delta_{\t}=1/2}}(X-\red_\s(\mathfrak t))^2.
$$
If $\s$ is principal and $e\delta_\s=1/2$ we define the curve $\Gamma_\s/\bar{k}$ to be the glueing of the affine curves 
$$
\phantom{hahahaha}Y^2=c_\s
    \prod_{{\text{odd }}{{\mathfrak o} < \s}}(X-\red_\s(\mathfrak o))
    \prod_{{\text{twin }{\mathfrak t}<\s}\atop{e\delta_\t=1/2}}(X-\red_\s(\mathfrak t))^2\phantom{hahahaha} (\dagger)
$$
and
$$
Z^2=c_\s T^{2}\prod_{{\text{odd }}{{\mathfrak o} < \s}}(1-\red_\s(\mathfrak o)T)
    \prod_{{\text{twin }{\mathfrak t}<\s}\atop{e\delta_\t=1/2}}(1-\red_\s(\mathfrak t)T)^2
$$
over the subsets $Y\neq 0$ and $T\neq 0$ via $Z=YT^{n+1}$ where $n$ is half the degree of the right hand side of $(\dagger)$ (note that this is not the usual chart at infinity).
\end{definition}


%


As in Section \ref{scdisc}, each $\Cmp_\s$ corresponds to one or possibly two components of the special fibre 
of the minimal regular model of $C$ over $\OFnr$. The following theorem describes how these components fit 
together: roughly $\Cmp_\s$ and $\Cmp_{\s'}$ are linked by chains of curves isomorphic to $\mathbb{P}^1_{\bar{k}}$ whenever 
$\c' \< \c$ and there is a loop of such curves from $\Cmp_\s$ to itself for each twin or cotwin $\t \< \s$ or $\s\<\t$. It also
describes the corresponding Galois action and the reduction map. 

\begin{theorem}\label{th:DualGraph}
Let $F/K$ be an extension over which $C$ is semistable. 

\medskip
\noindent(1) Let $\Upsilon_C$ be the dual graph of the special fibre of the minimal regular model 
of $C$ over $\OFnr$. Then $\Upsilon_C$ has 
a vertex $v_\c$ corresponding to $\Cmp_\s$ for every non-\ub\ principal cluster and two vertices $v_{\c}^+, v_{\c}^-$ for each \ub\ principal cluster $\c$. These are linked by chains of edges as follows (where we write  $v_\c= v_{\c}^+=v_{\c}^-$ whenever $\c$ is not \ub).

\noindent\hskip+3mm
\begin{tabular}{|c|c|c|c|l|}
\hline
Name & From & To & Length & Conditions \\
\hline
$L_{\c'}$ &$v_{\c'}$&$v_{\c}$ &$\frac 12 \delta_{\s'}$ & $\c'<\c$ both principal, $\c'$ odd\\
\hline
$L_{\c'}^+$ &$v_{\c'}^+$&$v_{\c}^+$ &$\delta_{\s'}$ & $\c'<\c$, both principal, $\c'$ even \\
\hline
$L_{\c'}^-$ &$v_{\c'}^-$&$v_{\c}^-$ &$\delta_{\s'}$ & $\c'<\c$, both principal, $\c'$ even \\
\hline
$L_{\t}$ &$v_{\c}^-$&$v_{\c}^+$ &$2 \delta_\t$& $\c$ principal, $\t<\c$ twin\\
\hline
$L_{\t}$ &$v_{\c}^-$&$v_{\c}^+$ &$2 \delta_\s$& $\c$ principal, $\c<\t$ cotwin\\\hline
\end{tabular}\\

\noindent Moreover, if $\cR$ is not principal\\

\noindent\hskip+3mm
\begin{tabular}{|c|c|c|c|l|}
\hline
$L_{\c_1,\c_2}$ &$v_{\c_1}$&$v_{\c_2}$ &$\frac 12 (\delta_{\c_1}+\delta_{\c_2})$& $ \cR = \c_1\coprod\c_2$, with $\c_1, \c_2$ principal odd\\
\hline
$L_{\c_1,\c_2}^+$ &$v_{\c_1}^+$&$v_{\c_2}^+$ &$\delta_{\c_1}+\delta_{\c_2}$& $ \cR = \c_1\coprod\c_2$, with $\c_1, \c_2$ principal even\\
\hline
$L_{\c_1,\c_2}^-$ &$v_{\c_1}^-$&$v_{\c_2}^-$ &$\delta_{\c_1}+\delta_{\c_2}$& $ \cR = \c_1\coprod\c_2$, with $\c_1, \c_2$ principal even\\
\hline
$L_{\t}$ &$v_{\c}^-$&$v_{\c}^+$ &$2(\delta_\s+\delta_\t)$&  $\cR = \c \coprod \t$, with $\c$ principal even, $\t$ twin\\
\hline
\end{tabular}\\

\noindent(2) If $\sigma \in G_K$ then it acts on $\Upsilon_C$ by
\begin{itemize}
\item[$(i)$] 
$ \sigma(v_{\c}^\pm)=v_{\sigma(\c)}^{\pm\epsilon_{\c}(\sigma)}$;
\item[$(ii)$] 
$\sigma(L_{\c}^\pm)=L_{\sigma(\c)}^{\pm\epsilon_{\c}(\sigma)}$;
 

\item[$(iii)$] for $\t$ twin or cotwin  $\sigma(L_{\t})=\epsilon_{\t}(\sigma)L_{\sigma(\t)}$, where $-L$ denotes $L$ with reversed orientation; 
\end{itemize}
and the induced permutation on the remaining edges and vertices.

\noindent(3) If $\sigma \in G_K$ and $\c$ is a principal cluster then\footnote{recall that $\sigma$ acts as in \eqref{CbarnAction} on $\cC_{min,\bar{k}}$}
$\sigma$ maps  $ \Gamma_{\c}$ to $\Gamma_{\sigma\c}$ and
 $$
 \sigma|_{\Cmp_\s}(x,y) =  \left(\chi(\sigma)^{ed_{\c}}\bar{\sigma}(x)+\beta_{\c}(\sigma),\chi(\sigma)^{e\lambda_\c}\bar{\sigma}(y)\right) \in \Cmp_{\sigma\c},
 $$
where $ \lambda_\c = \frac{\nu_{\c}}{2} -  d_{\s}\!\sum \lfloor \frac{|\c'|}{2}\rfloor$, the sum taken over $\s'<\s$ with $\delta_{\c'}>\frac{e}{2}$.

\noindent(4) The point $(x,y) \in C(F^{nr})$ reduces to $\Cmp_\s$ if and only if:
\begin{enumerate}
\item[\textit{i)}]  $v(x-z_\c) \ge d_\c$ and $\red_{\c}(x) \neq \red_{\c}(\c')$ for any $\c'<\c$, or 
\item[\textit{ii)}] $|\c| \ge 2g+1$ and $v(x-z_\c) < d_\c$.
\end{enumerate}
Explicitly, for one of these points\footnote{in $(ii)$ if $|\s| =2g+2$ then the point reduces to one of the two points at infinity on $\Cmp_\s$, see Proposition \ref{reducing points}.$(i)$ to determine which one.}, \begingroup \Small$$(x,y)
 \mapsto
  \left(\red_\c(x), ~\pi_F^{-\frac{e\nu_\c}{2}}y\cdot \!\!\!\! \prod_{{\s'<\s}\atop{\delta_{\s'}>\frac{e}{2}}}\left(
  \red_\s(x)-\red_\s(\s')\right)^{-\lfloor \frac{|\c'|}{2}\rfloor}\right).$$ \endgroup 
  
\end{theorem}

%
%
%

\begin{proof}

(1) The dual graph of the special fibre is given by Theorem \ref{main dual graph thingy} and \Cref{explicit graph remark} gives the explicit description.

(2) follows by combining Theorem \ref{th:ActionOnDualGraph} with \Cref{explicit aut action}.

(3) This is Proposition \ref{th:GeneralGaloisAction}. 


(4) This is Proposition \ref{reducing points}.
\end{proof}

\begin{corollary}\label{numcomponents}
Let $C/K$ be a semistable hyperelliptic curve. Then the number of components in the special fibre if its minimal regular model over $\cO_{K^{\nr}}$ is
 $$
 m_C = \sum_{\substack{\s \neq \cR, \\ \text{odd, proper}}} \frac{\delta_{\s}}{2} +  \sum_{\substack{\s \neq \cR, \\ \text{even}}} 2\delta_{\s} +1-\rk H_1(\Upsilon_C,\Z).
 $$
\end{corollary}
\begin{proof}
This follows from the usual Euler characteristic formula for $H_1$ of a graph, and counting the total number of edges in part (1) of the theorem.
\end{proof}

We now describe the normalisation of each $\Cmp_\s$ as well as the induced Galois action. 

\begin{theorem}\label{de:gammatilde}\label{galactss}
For a principal cluster $\c$ the normalisation of $\Cmp_\s$ is given by
$$
\qquad \qquad
  \widetilde{\Cmp}_\s :\>\> Y^2 = c_\s
    \prod_{{\text{odd }}{{\mathfrak o} < \s}}(X-\red_\s(\mathfrak o)).
$$

$(i)$ If $\sigma \in G_K$  the associated map $ \widetilde{\Gamma}_{\c}$ to $\widetilde{\Gamma}_{\sigma\c}$ is given by
 $$
\sigma|_{\widetilde{\Cmp}_\s}(x,y) =  
 \left(\alpha_\c(\sigma)\sigma(x)+\beta_{\c}(\sigma),\gamma_\s(\sigma)\sigma(y)\right)
 $$
 
$(ii)$ If $\sigma \in I_\s$, the geometric automorphism of $\tilde\Gamma_{\c}(\bar{k})$ given by $\sigma$ 
%
$$
  \begin{tabular}{llllll}
  swaps two points at infinity & if $\epsilon_\s(\sigma)=-1$,\cr
   fixes two points at infinity & if $\epsilon_\s(\sigma)=1$,\cr
  fixes the unique point at infinity & if $\epsilon_\s(\sigma)=0$.
  \end{tabular}
$$

$(iii)$ If the point $P=(x,y) \in C(F^{nr})$ reduces to $\bar{P} \in \Cmp_\s$ then $\bar{P}$ corresponds to the point
\begingroup \Small
$$
   \left(\red_\c(x), ~\pi_F^{-\frac{e\nu_\c}{2}}y\cdot \prod_{\s'<\s}\left(
  \red_\s(x)-\red_\s(\s')\right)^{-\lfloor \frac{|\c'|}{2}\rfloor}\right) \in \widetilde{\Cmp}_\s.
  $$ \endgroup

\end{theorem}
\begin{proof}
(i). Combine Theorem \ref{th:DualGraph} (3) and \Cref{le:apphyp}(iii).

(ii). 
The case where there is a unique point at infinity is clear so suppose otherwise. 
\Cref{th:ActionOnDualGraph} gives the case $\s = \s^*$.

For the case $\s^*\neq \s$, note that by considering the action of $\sigma$ on the components of the special fibre of the minimal regular model of $C$ (c.f. Theorem \ref{th:DualGraph} (2)), one sees that the points at infinity on $\widetilde{\Gamma}_\mathfrak{s}$ are swapped by $\sigma$ if and only if the points at infinity  on $\widetilde{\Gamma}_{s^*}$ are, if and only if $\epsilon_\mathfrak{s}=-1$. 

(iii). The description of the normalisation of $\Gamma_{\mathfrak{s}}$ is standard and the normalization map between $\tilde{\Cmp}_\s$ and $\Cmp_\s$ is given by 
$$
(x,y) \mapsto \left(x, y \cdot  \prod_{{\text{twin }{\mathfrak t}<\s}\atop{\delta_{\mathfrak t}=\frac e2}}(x-\red_\s(\mathfrak t))\right)
$$
(c.f. \ref{normalisation equation2}). The claimed formula now follows from Theorem \ref{th:DualGraph} (4). 
\end{proof}

\begin{remark} \label[remark]{action is what you think remark}
We note that the formula for the action of $\sigma \in G_K$ 
becomes particularly simple in the following two settings:\\
(i) if $\s$ is a principal cluster and $\sigma \in I_\s$ then $\sigma$ acts on $\widetilde{\Gamma}_\s$ as the geometric automorphism
$$
  (x,y)\mapsto \left(\alpha_\s(\sigma)x+\beta_\s(\sigma),\gamma_\s(\sigma)y\right).
$$
(ii) suppose $\s$ is a principal cluster, $F=K$ (so that $C/K$ is semistable) and that $\sigma \in G_\s$. Then $\chi(\sigma)=\textup{id}$. If (as is possible by Lemma \ref{lem:invcentre}) we additionally pick our centre $z_\s$ for $\s$ to lie in $K_\s$, the subfield of $\Ks$ fixed by $G_\s$,  then we also have $\beta_\s(\sigma)=0$. Thus $\sigma$ acts on $\widetilde{\Gamma}_\s(\bar{k})$ via
$$
\sigma(x,y)=\left(\bar{\sigma}(x),\bar{\sigma}(y)\right)
$$
where $\bar{\sigma}$ denotes the automorphism of $\bar{k}$ induced by $\sigma$. (This is a manifestation of the fact that, when $C/K$ is semistable, all our constructions are $\textup{Gal}(K^{\textup{nr}}/K)$-equivariant.)
\end{remark}

\section{Homology of the dual graph of the special fibre}
\label{shomology}

The homology of the dual graph of the special fibre forms a part of the Galois representation of $C$ and determines several arithmetic invariants (see \eqref{tatedec}, Theorem \ref{th:rootnumber} and Lemma \ref{lemtam}). 
In this section we give a description of the homology in terms of clusters. In the notation of Theorem \ref{th:DualGraph}, the basic observation is that every even non-\ub\, cluster $\s$ starts off two chains $L_\s^+$ and $L_\s^-$ that eventually join back up (normally at $v_{P(\s^*)}$) to form a loop in $\Upsilon_C$.

\def\sh{\hat{\c}}
\begin{definition}\label{de:loops}
Let $C/F$ be a semistable hyperelliptic curve and $\Upsilon_C$ the dual graph of the special fibre of 
its minimal regular model over $\OFnr$ as in Theorem \ref{th:DualGraph}. 

Let $\s \ne \cR$ be an even non-\"ubereven cluster. If $\s^*\neq\cR$, we define the 1-chain $\ell_{\s}$ in $C_1(\Upsilon_C,\Z)$ to be the shortest path from $v_{P(\s^*)}$ to itself that passes through $v_\s$ and goes through the minus part of the graph before the plus part of the graph.  If $\s^*=\cR$, we define $\ell_{\s}$ to be the shortest path from $v_\cR^-$ to $v_\cR^+$ that passes through $v_\s$. Here
\begin{itemize}
 \item if $\s=\t$ is a twin or $P(\s^*)=\t$ is a cotwin, we write $v_{\t}$ for the point in the middle of $L_{\t}$;
 \item if $\s^*=\cR = \s_1 \coprod \s_2$ with $\s_i$ both principal even, we write $v_{\cR}^{+}$ and $v_{\cR}^-$ for the points in the middle of $L_{\c_1,\c_2}^+$ and $L_{\c_1,\c_2}^-$;   
 \item if $\s^*=\cR = \t \coprod \c$ with $\t$ a twin and $\c$ principal even, we write $v_{\cR}^{+}$ and $v_{\cR}^-$ for the points on $L_{\t}$ of distance $\delta_\s$ from $v_{\s}^+$ and $v_{\s}^-$, respectively.
\end{itemize}
\end{definition}

\begin{remark}
$\ell_{\c}$ is a loop (cycle) in $\Upsilon_C$ unless $\c^* = \cR$. In the latter case, it is a ``half loop'' in the sense that if $\ell_{\c}$, $\ell_{\c'}$ are two half loops then $ \ell_{\c} - \ell_{\c'}$ is a loop.
\end{remark}

Using the explicit description of the dual graph it is not hard to check that the loops described above form a basis for the homology of $\Upsilon_C$ and to track the action of Galois on them. This gives the following result on $H_1(\Upsilon_C,\Z)$. 

\begin{theorem}\label{th:Homology}
Let $C/K$ be a hyperelliptic curve and let $F/K$ be a Galois extension over which $C$ is semistable. 
Let $\Upsilon_C$ be the dual graph of the special fibre of the minimal regular model of 
$C$ over $\OFnr$. Let  $A$ be the set of even non-\ub\ clusters excluding $\cR$, and let $B$ be the subset of clusters $\c \in A$ such that $\c^*=\cR$. Then
\begin{enumerate}
\item $\rk_\Z( H_1(\Upsilon_C,\Z)) = \bigleftchoice {\#A}{\mbox{ if } \cR \mbox{ is not \ub,}}{\#A-1}{\mbox{ otherwise.}}$
\item $$H_1(\Upsilon_{C},\ZZ) = \Bigl\{ \sum_{\c\in A} a_{\c}\ell_{\c} \Bigm|\> a_\s\in\Z,\>\>\sum_{\c\in B} a_{\c}=0\Bigr\},$$ 
\item the length pairing is given by 
$$
\langle \ell_{\c_1},\ell_{\c_2} \rangle=\left\{
    \begin{array}{ll}
        0&  \mbox{ if } \c^*_1 \neq \c^*_2, \\
        2(d_{(\c_1\wedge\c_2)}-d_{P(\c^*_1)})& \mbox{ if } \c^*_1 =\c^*_2 \ne \cR, \\
        2(d_{(\c_1\wedge\c_2)}-d_{\cR)}& \mbox{ if } \c^*_1 =\c^*_2 = \cR.\\
    \end{array}
\right.
$$      
\item for $\sigma \in G_K$, $$\sigma(\ell_{\c}) = \epsilon_{\c}(\sigma) \ell_{\sigma(\c)}.$$
\end{enumerate}
\end{theorem}

\begin{proof}
This follows from Theorem \ref{th:ActionOnDualGraph}, which describes the $\Upsilon_C$ with the induced Galois action, and Theorem \ref{th:apphomology} and Remark \ref{rm:apphomology}, which describe the associated homology group.
\end{proof}

\begin{corollary}\label{th:condexpo}
Let $C/K$ be a semistable hyperelliptic curve. Let $A$ be the set of even, non-\ub\ clusters excluding $\cR$. Then the conductor exponent of Jac $C$ is 
$$
n_C = \#A - \leftchoice{1}{\cR \mbox{ \ub,}}{0}{\mbox{otherwise.}}
$$
\end{corollary}

\begin{proof}
Since $J=\Jac C$ is semistable,
$$
n_C = \dim(V_{\ell}J)-\dim(V_{\ell}(J)^{I}) = \rk_\Z(H_1(\Upsilon_C, \Z))
$$
by \eqref{tatedec}.
The result follows from Theorem \ref{th:Homology}.(1). 
\end{proof}

\begin{notation}
Let $G$ be a group acting on a set $X$ via the signed\footnote
{i.e. $G$ acts on $\{+x,-x | x\in X\}$ by $g(\pm x)=\pm\varepsilon_x(g) g(x)$ with $x\mapsto g(x)$ a $G$-action and $\varepsilon_x(g)\in\{\pm\}$ satisfying $\varepsilon_{x}(gh)=\varepsilon_{hx}(g)\varepsilon_x(h)$.}
permutation $(X,\varepsilon)$. For a ring $R$ we write $R[X,\varepsilon]$ for the corresponding signed permutation representation, and $R[X,\varepsilon]_0$ for its sum zero part.  
\end{notation}
\begin{corollary}\label{co:EtaleCohoToric}
Let $F/K$ be an extension over which $C$ is semistable. Let  $A$ be the set of even non-\ub\ clusters excluding $\cR$, and let $B$ be the subset of clusters $\c \in A$ such that $\c^*=\cR$. \\
1) $$
  H_1(\Upsilon_{C},\Z) \>\>\iso\>\> \Z[A\!\setminus \!B, \epsilon] \>\>
    \oplus \>\>\Z[B,\epsilon ]_0,
$$
2) 
$$
  \H(C,\Q_l)_t \>\>\>\>\iso\>\>\>\> \Q_l[E, \epsilon]\ominus\epsilon_\cR\ \>\>\>\>\iso\>\>\>\> \bigoplus_\s \Ind_{\Stab \s}^{G_K} \!\epsilon_\s \quad\ominus\epsilon_\cR,
$$
where $E$ is the set of even non-\ub\ clusters and  the sum is taken over representatives of $G_K$-orbits on $E$. 
\end{corollary}
\begin{proof}
1) Follows directly from parts (2) and (4) of Theorem \ref{th:Homology}. 


2) Tensoring  1) with $\Q_l$ and using \eqref{tatedec} we get 
\begin{align*}
  \H(C,\Q_l)_t \>\>\>\>&\iso\>\>\>\> \Q_l[A\!\setminus \!B, \epsilon] \>\>
    \oplus \>\>\Q_l[B,\epsilon]_0\\
 &\iso\>\>\>\> \Q_l[A\!\setminus \!B, \epsilon] \>\>
    \oplus \>\>\ \leftchoice{0}{B = \emptyset}{\Q_l[B,\epsilon ]\ominus \epsilon_\cR}{B \ne \emptyset}\\
   &\iso\>\>\>\> \Q_l[A, \epsilon] \>\>
    \ominus \>\>\ \leftchoice{0}{B = \emptyset,}{\epsilon_\cR}{B \ne \emptyset}\\
   &\iso\>\>\>\> \Q_l[E, \epsilon] \>\>
    \ominus \>\>\ \epsilon_\cR,   
\end{align*}
where $E$ is the set of even non-\ub\ clusters and the last isomorphism uses the fact that $B$ is empty if and only if $\cR$ is not \ub. 

Observe that when $G$ acts transitively on $X$ and $(X, \varepsilon)$ is a signed permutation then  $\Q_l[X, \varepsilon] \iso \Ind_{\Stab_t}^G\varepsilon_t$ for any point $t \in X$; here $\sigma(t) = \varepsilon_t(\sigma) t$ for $\sigma \in \Stab_t$. Hence 
$$
  \H(C,\Q_l)_t \>\>\>\>\iso\>\>\>\>  \bigoplus_\s \Ind_{\Stab \s}^{G_K} \!\epsilon_\s \quad\ominus\epsilon_\cR,
$$
where the sums are taken over representatives of $G_K$-orbits on $E$.
\end{proof}


\section{Galois Representation}
\label{sgalrep}

Having obtained an explicit description of the special fibre of the minimal 
regular model of $C$ over the field where it becomes semistable, together 
with the action of $G_K$, we are now in a position to extract the action of 
$G_K$ on $\H(C)=\H(C_{\bar K},\Q_l)$. 

Fix a prime $l\ne \vchar k$. As in \S\ref{sfibre}, we take
\begin{itemize}
\item 
$C/K$ a hyperelliptic curve;
\item
$F/K$ a finite Galois extension over which $C$ becomes semistable;
\item
$\Gamma_\s$ components of the special fibre $\cC_{\min,\kbar}$ 
of the minimal regular model of $C$  over $\OFnr$ (see Definition \ref{de:gammas});
\item
$G_K \acts \cC_{\min,\kbar}$ Galois action of \eqref{CbarnAction}; it induces the
action of the stabiliser $G_\s$ on $\Gamma_\s$, on its normalisation $\tilde\Gamma_\s$
and on the \'etale cohomology group $\H(\Cmpn_\s)$.
\end{itemize}

\begin{theorem} \label[theorem]{galois rep theorem}
\label{het}
Let $C/K$ be a hyperelliptic curve. Let $\H(C)=\H(C)_{ab}\oplus \H(C)_{t}\tensor\Sp_2$ 
be the decomposition into `toric' and `abelian' parts. Then
$$
\begin{array}{llllllllllll}
  \H(C)_{t} &=&\bigoplus_\s \Ind_{G_\s}^{G_K} \!\epsilon_\s \quad\ominus\epsilon_\cR,\\[3pt]
  \H(C)_{ab}&=&\bigoplus_\s \Ind_{G_\s}^{G_K} \H(\Cmpn_\s).
\end{array}
$$
The first sum is taken over representatives of $G_K$-orbits of even non-\"ubereven 
clusters. 
The second sum is taken over representatives of $G_K$-orbits of principal non-\"ubereven 
clusters. For every such cluster $\s$, there is an isomorphism of $I_\s$-modules
$$
  \H(\Cmpn_\s) \>\>\iso\>\> \tilde\gamma_s \otimes (\Q_l[\so] \ominus \triv) \quad\ominus \epsilon_\s,  
$$
where $\tilde\gamma_\s\colon I_\s\to \bar\Q_l^\times$ is any character whose order 
is the prime-to-$p$ part of the denominator of 
$|I_K/I_\s|\,\tilde\lambda_\s$.
\end{theorem}

%

\begin{proof}
By Theorem \ref{tatedec} we have the decomposition and the claim regarding the abelian
part. The statement about the toric part is Corollary \ref{co:EtaleCohoToric}(2).
The last claim is \cite[Thm. 1.2]{hq} combined with Theorem \ref{galactss}(i)
and Remark \ref{legaldisc};
note that \cite[Thm. 1.2]{hq} is phrased for $\C$- rather then $\Q_l$-representations,
but that does not affect the result.
\end{proof}

\begin{remark} \label[remark]{semsible action on cohomology remark}
When $C/K$ is semistable the full action of $G_\s$ (rather than just that of $I_\s$) on $\H(\Cmpn_\s)$   may be explicitly determined, as we now explain. For a proper cluster $\s$, write $K_\s$ for the subfield of $\Ks$ fixed by $G_\s$ and denote by $k_\s$ its residue field. Suppose (as is possible by \Cref{lem:invcentre}) that for each proper cluster $\s$ we have fixed our choice of centre $z_\s$ to lie in $K_\s$. Then for any principal cluster $\s$ the coefficients of 
$$
\widetilde{\Gamma}_\mathfrak{s}:y^2=c_\mathfrak{s}\prod_{\textup{odd }\mathfrak{o}<\s}(x-\textup{red}_\mathfrak{s}(\mathfrak{o}))
$$
lie in $k_\s$. Moreover, by \Cref{action is what you think remark} (ii) the action of $G_\s$ on $\widetilde{\Gamma}_\mathfrak{s}(\bar{k})$ (arising from \Cref{saction}) is simply given by $(x,y)\mapsto (\bar{\sigma}(x),\bar{\sigma}(y))$ where $\bar{\sigma}$ denotes the automorphism of $\bar{k}$ induced by $\sigma$  (whence $G_\s$ acts through $\textup{Gal}(K_\s^{\textup{nr}}/K_\s)$). In particular, upon identifying $\textup{Gal}(K_\s^{\textup{nr}}/K_\s)$ with $\textup{Gal}(\bar{k}/k_\s)$, the induced action on $\H(\Cmpn_\s)$ is precisely the usual action of $\textup{Gal}(\bar{k}/k_\s)$ on $\H(\Cmpn_\s)$ coming from viewing $\widetilde{\Gamma}_\s$ as a curve defined over $k_\s$ 
given by the above formula.
One may then recover the Frobenius eigenvalues for this action on $\H(\Cmpn_\s)$ from point counts on $\widetilde{\Gamma}_\s$ over extensions of  $k_\s$ in the usual way.
\end{remark}

\begin{theorem}\label{reduction}
Let $C/K$ be a hyperelliptic curves.
Write $\Jac C$ for its Jacobian. Then 
\begin{enumerate}
\item
$C$ is semistable $\iff$ $\Jac C$ semistable $\iff$ 
$C/K$ satisfies the semistability criterion.
\item
$C$ has good reduction $\iff$ $K(\cR)/K$ is unramified, there are no proper
clusters of size $<2g+1$ and $\nu_s\in 2\Z$ for the unique principal 
cluster.
\item
$C$ has potentially good reduction $\iff$ 
there are no proper clusters of size $<2g+1$.
\item
$C$ is tame
$\iff$ $\Jac C$ is tame $\iff$ $K(\cR)/K$ is tame.
\item
$\Jac C$ has good reduction  $\iff$  $K(\cR)/K$ is unramified, 
all clusters $\s\ne\cR$ are odd, and principal clusters 
have $\nu_\s \in 2\Z$.
\item
$\Jac C$ has potentially good reduction  $\iff$ all clusters 
$\s\ne\cR$ are odd.
\item
The potential toric rank of $\Jac C$ equals the number of even non-\"ubereven clusters 
excluding $\cR$, less 1 if $\cR$ is \"ubereven.
\item
$\Jac C$ has potentially totally toric reduction 
$\iff$ every cluster has at most two odd children. 
\end{enumerate}
\end{theorem}

\begin{proof}
(1) As $g\ge 2$, $C$ is semistable if and only if its Jacobian is
\cite[Thm. 1.2]{DM}. The equivalence with the semistability criterion is proved in 
Theorem \ref{the semistability theorem}.

(2),(3) Using that good reduction is in particular semistable, these follow from (1)
and Theorem \ref{th:DualGraph} which gives the description of the special fibre 
for semistable curves in terms of principal clusters.

(4) Follows directly from (1).

(5),(6) Recall that $\Jac C$ has good reduction if and only if inertia $I_K$ acts trivially 
on the $l$-adic Tate module $V_l\Jac C$ (for some $l\ne p$), 
by the N\'eron-Ogg-Shafarevich criterion \cite[\S2]{ST}.
Now apply Theorem \ref{het} that gives the inertia action on $V_l\Jac C$.

(7),(8) Apply Theorem \ref{het}. For (8), note that the condition `at most two odd
children' is equivalent to all components $\Gamma_\s$ from principal clusters having
genus 0.
\end{proof}

\section{Conductor}
\label{sconductor}

In this section we derive a formula for the conductor of a hyperelliptic curve $C/K$
in terms of clusters (Theorem \ref{condmain}).


\begin{lemma}
\label{lemabmain}
Let $k$ be a field of characteristic $\ne 2$, and $C/k$ a hyperelliptic curve given by
$$
  Y^2 = c\prod_{r\in R}(X-r), \qquad R\subset\bar k.
$$
Let $G\subset\Aut_k C$ be an affine group of automorphisms acting as
$$
  g(X) = \alpha(g) X+\beta(g), \qquad g(Y)=\gamma(g)\,Y \qquad \qquad (g\in G).
$$
Let $\tilde\gamma: G\to \bar\Q_l^\times$ be a character with $\ker\tilde\gamma=\ker\gamma$. We have:
\begin{itemize}
\item 
If $\ord_2(\order(\gamma))>\ord_2(\order(\alpha))$ then
$\tilde\gamma \otimes (\Q_l[R] \ominus \triv)$ has trivial $G$-invariants.
\item
If $\ord_2(\order(\gamma))\le \ord_2(\order(\alpha))$ then
$$
  \tilde\gamma \otimes (\Q_l[R] \ominus \triv) \>\iso\> \Q_l[R] \ominus \bigleftchoice 
    \triv{\text{if $\gamma$ has odd order,}}
    {\tilde\gamma}{\text{if $\gamma$ has even order.}}
$$
as $G$-modules.
\item
If $|R|\le 2$, then $\tilde\gamma \otimes (\Q_l[R] \ominus \triv)\oplus \triv$
is the permutation representation of $G$ on the (one or two) points at infinity of $C$.
\end{itemize} 
\end{lemma}

\begin{proof}
If $\ord_2(\order(\gamma))>\ord_2(\order(\alpha))$, then $\Q_l[R] \ominus \triv$ 
contains no 1-dimensional characters of order equal to the order of $\gamma$.
Therefore $\tilde\gamma \otimes (\Q_l[R] \ominus \triv)$ has no $G$-invariants.

Suppose $\ord_2(\order(\gamma))\le \ord_2(\order(\alpha))$. 
Then we are in the setup of \cite[Thm 4.1]{hq}, and by \cite[Lemma 4.4 (2)]{hq} we have 
$$
  \tilde\gamma \otimes (\Q_l[R] \ominus \triv) \>\iso\> \Q_l[R] \ominus \bigleftchoice 
    \triv{\text{if $\Q_l[R]$ contains an irregular orbit of $G$,}}
    {\tilde\gamma}{\text{if $\Q_l[R]\iso\Q_l[G]^{\oplus r}$ for some $r$.}}
$$
In the first (irregular orbit) case, $\gamma^2=\alpha$ by \cite{hq} Prop. 2.2 (5b), and it follows 
that $\gamma$ has odd order. 
In the second (regular) case, $\gamma^2=\triv$ by \cite{hq} Prop. 2.2 (5a). 
Hence, either $\gamma=\triv$ and the claim is trivial, or $\gamma$ has even order.

The last claim follows from \cite[Thm 4.1]{hq}, since $C$ has genus 0 and trivial $\H$ 
in this case.
\end{proof}

Now we go back to the setting of a hyperelliptic curve $C/K$. Recall from Definition \ref{de:characters} that we defined $\tilde{\lambda}_\s$ and characters $\alpha_\s, \beta_\s$ and $\gamma_\s$ for all principal clusters $\s$. In what follows we extend these definitions to all proper clusters $\s$ by the same formulae.\footnote{The formulae of Definition \ref{de:characters} only make sense when $ed_\s$ and $e\tilde{\lambda}_\s$ are integers. However this is always the case when $F$ is suitably large and (cf Remark \ref{legaldisc}) these characters, when defined, are independent of the choice of $F$.} 

\begin{lemma}
\label{lemets}
Let $\s$ be a proper non-\"ubereven cluster, and
$\tilde\gamma_\s: I_\s\to \bar\Q_l^\times$ a character with $\ker\tilde\gamma_\s=\ker\gamma_\s$.
\begin{itemize}
\item
If $\ord_2\denom(|I/I_\s|\tilde\lambda_\s)>\ord_2\denom (|I/I_\s|d_\s)$ then
$\tilde\gamma_\s \otimes (\Q_l[\so] \ominus \triv)$ has trivial $I_\s$-invariants. 
\item If $\ord_2\denom(|I/I_\s|\tilde\lambda_\s)\le\ord_2\denom (|I/I_\s|d_\s)$ then
$$
  \tilde\gamma_\s \otimes (\Q_l[\so] \ominus \triv) \>\iso\> \Q_l[\so] \ominus \bigleftchoice 
    \triv{\text{if $\ord_2(|I/I_\s|\tilde\lambda_\s)\ge 0$,}}
    {\tilde\gamma_\s}{\text{if $\ord_2(|I/I_\s|\tilde\lambda_\s)<0$,}}
$$
as $I_\s$-modules.
\item If $|\so|=2$ and $\s$ is not a cotwin, then $\tilde\gamma_\s \otimes (\Q_l[\so] \ominus \triv) = \epsilon_\s$.
\item If $\cR$ is a cotwin of size $2g\!+\!2$, then $\tilde\gamma_\cR \otimes (\Q_l[\cR_0] \ominus \triv) = \epsilon^0_\cR$, where $\epsilon^0_{\cR}(\sigma)=\frac{\sigma(\sqrt{c_f})}{\sqrt{c_f}}$ for $\sigma\in I_K$. 
\end{itemize}
\end{lemma}

\begin{proof}
Let $F$ be a finite Galois extension over which $C$ becomes semistable, and such that $\cR\subset F$.
In what follows we consider all representations as those of $I_\s$. First let $\s$ be a principal cluster. 
By Theorem \ref{de:gammatilde}, we have the 
curve 
$$
  \tilde\Gamma_\s\colon y^2=c_\s\prod_{\textup{odd }\s'<\s}(x-\textup{red}_\s(\s'))
$$
over $\bar{k}$ together with $I_\s$-action
$$
  g(X) = \alpha_\s(g) X+\beta_\s(g), \qquad g(Y)=\gamma_s(g)\,Y \qquad \qquad (g\in I_\s).
$$
Observe that $\order(\gamma_\s)$ is the prime-to-$p$ part of the denominator of $|I/I_\s|\tilde\lambda_\s$,
and $\order(\alpha_\s)$ is the prime-to-$p$ part of the denominator of $|I/I_\s|d_\s$.
The claim follows from Lemma \ref{lemabmain} and Theorem \ref{galactss}(ii).

When $\s$ is not principal we argue similarly. Since $\cR \subset F$ the disc $D(\s)$ (the minimal disc 
cutting out $\s$, see Notation \ref{no:disc}) is valid in the sense of Definition \ref{valid discs defi}. 
In particular this disc contributes a component $\Cmp_{D(\s)}$ to the special fibre of the minimal 
regular model of $C$ over $\OFnr$ (see Proposition \ref{components of min}; if $e\nu_\s$ is not even, enlarge $F$ to make it so). The normalisation of this component is given by the same equation as for $\s$ principal as is the action of $I_\s$ (see Proposition \ref{th:GeneralGaloisAction}). Now we conclude as before; that $\epsilon_\s$ corresponds to the action on the points at infinity follows from the explicit action of $I_\s$ on the dual graph $\Upsilon_C$ in Thereom \ref{th:DualGraph}(2) in the non-cotwin case, and by the formula for $\Gamma_{D(\cR)}$ in the cotwin case.

\end{proof}

\begin{theorem}[Conductor]
\label{condmain}
Let $C/K$ be a hyperelliptic curve.
Decompose the conductor exponent of $\Jac C$ into the tame part and a wild part,
$n_C = n_{C,\text{tame}}+n_{C,\text{wild}}$. Then
$$
  n_{C,\text{wild}} = \sum_{r\in R/\!/G_K} v_K(\Delta_{K(r)/K})-[K(r):K]+f_{K(r)/K},
$$
and $n_{C,\text{tame}} = 2g-\dim \H(C)^{I_K}$ with
$$
\begin{array}{llllllllllll}
 \dim \H(C)^{I_K} &=& 
  \#\bigl\{\s\subsetneq \cR \text{ odd} \,\bigm|\,  \xi_{P(\s)}(\tilde\lambda_{P(\s)})\le \xi_{P(\s)}(d_{P(\s)})\bigr\}/{I_K} \cr
    &-& \#\bigl\{\s \text{ proper non-\"ubereven} \,\bigm|\, \xi_{P(\s)}(\tilde\lambda_{P(\s)})=0 \bigr\}/{I_K} \cr
&-& \leftchoice{1}{\text{if $|\cR|$ is even and $v(c_f)$ is even,}}{0}{\text{otherwise;}} 
\end{array}
$$
here $\xi_\s(a)=\ord_2(b)$ where $b$ is the denominator of $|I_K/I_{\s}|a$, with $\xi_\s(0)=0$.
\end{theorem}

\begin{proof}
(Tame part) 
Let $\epsilon^0_{\cR}(\sigma)=\frac{\sigma(\sqrt{c_f})}{\sqrt{c_f}}$ for $\sigma\in I_K$ if $\cR$ is even, and the zero representation if $\cR$ is odd. Note that $\epsilon_\cR^0$ is the same as $\epsilon_\cR$, except when $\cR$ is a cotwin.

By Theorem \ref{hetmain} we have to compute the dimension of inertia invariants on
$$  
  \H(C)_{ab}\oplus \H(C)_t\otimes \Sp_2  = 
  \bigoplus_{\s\in P} \Ind_{I_\s}^{I_K} (\tilde\gamma_\s \otimes (\Q_l[\so] \ominus \triv)\ominus\epsilon_\s) 
    \oplus
  (\Sp_2 \tensor (\bigoplus_{s\in E}\Ind_{I_\s}^{I_K} \!\epsilon_\s\quad \ominus\epsilon_\cR)),
$$
where $P$ and $E$ are sets of $I_K$-orbit representatives on principal non-\ub\ clusters and on even non-\ub\ clusters, respectively.

By Lemma \ref{lemets} we can expand the first sum to 
$$
\H(C)_{ab} = \bigoplus_{\s\in T} \Ind_{I_\s}^{I_K} (\tilde\gamma_\s \otimes (\Q_l[\so] \ominus \triv)\ominus\epsilon_\s) \oplus \epsilon_\cR\ominus \epsilon^0_\cR
$$
where $T$ is a the set of representatives of $I_K$-orbits on all proper non-\ub\ clusters excluding cotwins of odd size.
Since $\epsilon_\s$ is the zero representation for odd clusters (excepts cotwins), we can similarly expand the second sum to
$$
\H(C)_{t} = \Sp_2 \tensor (\bigoplus_{s\in T}\Ind_{I_\s}^{I_K} \!\epsilon_\s\quad \ominus\epsilon_\cR),
$$

Taking $I_F$ invariants on $\H(C)_{ab}\oplus \H(C)_t\otimes \Sp_2$ replaces $\Sp_2$ by $\triv$, so we are left with $I_{F/K}$-invariants~on
$$  
  \bigoplus_{\s\in T} \Ind_{I_\s}^{I_K} (\tilde\gamma_\s \otimes (\Q_l[\so] \ominus \triv)) \quad \ominus\epsilon^0_\cR.
$$
Note that if $\s$ is a cotwin of odd size then $\Q_l[\so]$ is the trivial representation so this expression is the same as
$$  
  \bigoplus_{\s\in S} \Ind_{I_\s}^{I_K} (\tilde\gamma_\s \otimes (\Q_l[\so] \ominus \triv)) \quad \ominus\epsilon^0_\cR,
$$
where $S$ is a set of representatives of $I_K$-orbits on proper non-\"ubereven clusters.

By Frobenius reciprocity, we find
$$
  \dim \H(C)^{I_K} = \sum_{\s\in S}\> \langle \triv, \tilde\gamma_\s \otimes (\Q_l[\so] \ominus \triv) \rangle_{I_\s}
    \quad - \langle \triv,\epsilon_\cR^0 \rangle_{I_K}.
$$
Let $S_0\subset S$ be the set of those clusters for which 
$$
\xi_{P(\s)}(\tilde\lambda_{P(\s)})\le \xi_{P(\s)}(d_{P(\s)}).
$$

By Lemma \ref{lemets}, $\gamma_\s \otimes (\Q_l[\so] \ominus \triv)$ 
has no $I_\s$-invariants if $\s\in S\setminus S_0$. 
Otherwise, by the same lemma,
$$
  \gamma_\s \otimes (\Q_l[\so] \ominus \triv) \qquad\iso\qquad 
     \Q_l[\so] \ominus \triv \qquad\text{or}\qquad\Q_l[\so] \ominus \gamma_\s,
$$
depending on whether $\xi_\s(\tilde\lambda_\s)$ is $=0$ or $>0$, respectively.
Therefore,
$$
  \dim \H(C)^{I_K} = \sum_{\s\in S_0}\> \langle \triv, \Q_l[\so] \rangle_{I_\s} \>\>-\>\>
    \#\{\s\in S_0\,|\,\xi_\s(\tilde\lambda_\s)= 0\} \>\>-\>\> \langle \triv,\epsilon_\cR^0 \rangle.
$$
Note that 
$$
\langle \triv,\epsilon_\cR^0 \rangle=\leftchoice{1}{\text{if $|\cR|$ is even and $v(c_f)$ is even,}}{0}{\text{otherwise,}}
$$
is the last term in the statement.
Write $U_\s$ for the set of those odd clusters $\r$ such that $P(\r)$ lies in the $I_K$-orbit of $\s$. Then
$$
  \sum_{\s\in S_0}\> \langle \triv, \Q_l[\so] \rangle_{I_\s} = 
  \sum_{\s\in S_0}\> \langle \triv, \Q_l[U_\s] \rangle_{I_K}
$$
and 
$$
  \bigcup_{\s\in S_0} U_{\s} = \{\r\subsetneq R \text{ odd} \,\bigm|\, \xi_{P(\s)}(\tilde\lambda_{P(\s)})\le \xi_{P(\s)}(d_{P(\s)}) \bigr\}.
$$
Counting $I_K$-orbits gives the claim.

(Wild part) 
By the standard description of the 2-torsion of Jacobians of hyperelliptic curves 
(see e.g. Cornelissen \cite[Lemma 2.1]{Cor}),
$$
  (\Jac C)[2] \>\iso\>
  \F_2[R] \ominus \leftchoice {\triv}{\text{$2g+1$ odd,}}{\triv\oplus\triv}{\text{$2g+2$ even,}}
$$ 
as $G_K$-modules. Let $W$ be the wild inertia subgroup of $G_K$. Hence
$$
  V_2\Jac C \>\iso\>
  \Q_2[R] \ominus \leftchoice {\triv}{\text{$2g+1$ odd,}}{\triv\oplus\triv}{\text{$2g+2$ even,}}
$$ 
as $W$-modules, since $W$ acts on them through a finite group of odd order.
(Representation theory modulo $p$ agrees with complex representation theory for finite groups of 
order prime to $p$.) 

The left-hand side has the same wild part of the conductor as $\H(C)$. 
For $r\in R$ write $X_r$ for its $G_K$ orbit. 
The conductor of the right-hand side~is
$$
  n_{\Q_l[R]} = \sum_{r\in R/\!/G_K} n_{\Q_l[X_r]}
            = \sum_{r\in R/\!/G_K} v_K(\Delta_{K(r)/K}),
$$
and its tame part is
$$
  n_{\Q_l[R],\tame} = \sum_{r\in R/\!/G_K} [K(r):K] \!-\! \dim \Q_l[X_r]^{I_K}.
$$
Because $\dim\Q_l[X_r]^{I_K}$ is the number of $I_K$-orbits on $X_r$, which is the residue degree of $K(r)/K$,
we have
$$  
  n_{\Q_l[R],\wild} = \sum_{r\in R/\!/G_K} v_K(\Delta_{K(r)/K})\!-\![K(r):K]\!+\!f_{K(r)/K}.
$$
%
%
%
%
%
%
\end{proof}

%
%

%

\section{Deficiency}
\label{sdeficiency}


In this section we explain how to see whether a semistable hyperelliptic curve is deficient in term of its cluster picture. We first recall some standard results regarding deficiency of curves (see e.g. \cite{PS} Section 8).
\begin{definition}\label{de:deficiency}
Let $X/K$ be a geometrically integral, smooth and proper curve of genus $g$. 
Then $X$ is \emph{deficient} over $K$ if and only if it has no rational divisor of degree $g-1$;
equivalently, the index $I_2$ does not divide $g-1$, where 
$$
  I_2=\gcd \bigl\{[K(P)\!:\!K] ~|~ P \in X(\Kbar)\bigr\}.
$$
\end{definition}

\begin{lemma}\label{le:deficiency}
Let $X/K$ be a geometrically integral, smooth and proper curve of genus $g$. Let $Y_1, ... Y_r$ denote 
the components of the special fibre of the minimal regular model of $X$ over $\cO_K$, 
and $d_1,..., d_r$ their multiplicities.  Let 
$$
e_i =[\kbar\cap k(Y_i):k]= \text{length of }G_k\!\text{-orbit} \text{ of }Y_i
$$
and
$$
I_3 = \gcd\{d_ie_i ~|~ i =1,...,r\}.
$$
Then $I_2=I_3$.
\end{lemma}
\begin{proof}
This is Remark 1 after Lemma 16 in \cite{PS}. 
\end{proof}

\begin{lemma}\label{le:HyperDeficiency}
Let $C/K$ be a hyperelliptic curve of genus $g$. The following are equivalent:
 
1) $C$ is deficient over $K$, 

2) $C$ has even genus and has no rational point over any odd degree extension of $K$, 

3) $C$ has even genus and every component of the special fibre of its minimal regular model has either even multiplicity or a $G_k$-orbit of even length. 
\end{lemma}
\begin{proof}
Since $C$ is hyperelliptic, it has either one or two points at infinity. In particular, it has a rational divisor of degree 2 so $I_2 = 1$ or $I_2=2$. Hence if $g$ is odd then $C$ is not deficient. 

$1 \iff 2$ is clear from Definition \ref{de:deficiency}. 

$1\iff 3$ is clear from Lemma \ref{le:deficiency}.
\end{proof}

\begin{theorem}\label{th:deficiency}
Let $C/K$ be a semistable hyperelliptic curve. Then
$C$ is deficient if and only if it has even genus and either
\begin{enumerate}
\item $\cR = \s_1 \coprod \s_2$ with $\s_1, \s_2$ odd $G_k$-conjugate clusters, and $\delta_{\s_1}$ is odd, or
\item $\cR$ is \ub\ with $\epsilon_{\cR}(\Frob) = -1$ and for all non-\ub\ $\s$ such that $\neck{\s} =\cR$, either $d_\s \notin \Z$ or the $\Frob$-orbit of $\c$ has even size, or
\item $\cR$ is a cotwin with a principal \ub\ child $\r$  with $\epsilon_{\r}(\Frob) = -1$, and for all $\s$ such that $\neck{\s} =\r$, either $d_\s \notin \Z$ or the $\Frob$-orbit of $\s$ has even size.
\end{enumerate}
\end{theorem}

\begin{proof}
Since $C/K$ is semistable, all the components of the special fibre of its minimal regular model have multiplicity 1. Thus, by Lemma \ref{le:HyperDeficiency}, $C$ is deficient if and only if every component has a $G_k$-orbit
of even length. 

The result follows from the explicit description of dual graph of the special fibre of the minimal regular model of $C$ together with the action of Frobenius given in Theorem \ref{th:DualGraph}. From the description of the Frobenius action on the vertices $v_\s, v_\s^{\pm}$ and the various chains of edges, we see that all components will have even-length orbits under Frobenius if and only if: i) every principal non-\ub\ cluster has an even-length orbit under Frobenius, ii) every principal \ub\ cluster $\s$ either has an even-length orbit or $\epsilon_\s(\Frob)=-1$, iii) every twin and cotwin $\t$ either has an even-length orbit or $\epsilon_\t(\Frob)=-1$ and $d_\t\not\in\Z$,  iv) if $\cR=\s_1 \coprod \s_2$ is a disjoint union of two even clusters then $\epsilon(\cR)=-1$, v) if  $\cR=\s_1 \coprod \s_2$ is a disjoint union of two proper odd clusters then $\frac12 (\delta_{\s_1}\!+\!\delta_{\s_2})$ is odd and $\s_1$ and $\s_2$ are swapped by Frobenius.

Since $\cR$ itself cannot have a non-trivial orbit under Frobenius, it follows that it must either be a union of two odd clusters, \ub\ or a cotwin (and, in the cotwin case, its principal child must similarly be \ub ). It follows by inspection that one of (1), (2) and (3) must hold for the curve to be deficient.

For the converse, observe that if a cluster $\s$ has an ancestor with an even-length Frobenius orbit, then so does $\s$ itself. In particular (1) automatically forces (i)--(v) to hold. Similarly, if (2) (respectively (3)) holds, then every proper cluster $\s$ with $\s^*\neq\cR$ (respectively $\s^*\neq\r$) will necessarily have a non-\ub\ ancestor $\mathfrak{a}$ with $\mathfrak{a}^*=\cR$ (respectively $\mathfrak{a}=\r$), and, since $\mathfrak{a}$ is not a twin or cotwin, $d_{\mathfrak{a}}\in\Z$ by the semistability criterion. Thus $\mathfrak{a}$ must have a Frobenius-orbit of even length, and hence so does $\s$. It follows that if either (2) or (3) holds, then so do (i)--(iv), and the curve is deficient.
\end{proof}

%

\def\fr{{\mathfrak r}}
\def\R{{\mathcal R}}



\section{Integral Weierstrass models}\label{s:integral}

As we shall see in \S\ref{s:integral}--\ref{s:MWEquation}, cluster pictures are very well suited for studying Weierstrass equations of hyperelliptic curves, including discriminants and minimal Weierstrass equations. We begin by a criterion for checking whether a Weierstrass equation is integral, that is whether $f(x)\in\cO_K[x]$. Since the cluster picture of a polynomial does not change under a substitution $x\mapsto x-t$, it is clearly not possible to determine whether $f(x)\in\cO_K[x]$ from the cluster picture. However, up to such shifts in the $x$-coordinate, this turns out to be possible.

Recall first the definition of an (abstract) cluster picture:

\begin{definition}
Let $\cR$ be a finite set, $\Sigma$ a collection of non-empty subsets of $\cR$ (called {\em clusters}), and $d_\s\in\Q$ to every $\s\in\Sigma$ of size $> 1$ (called the depth of $\s$).
Then $\Sigma$ (or $(\Sigma,\cR,d)$) is a \emph{cluster picture} if
\begin{enumerate}
\item
Every singleton is a cluster, and $\cR$ is a cluster. 
\item
Two clusters are either disjoint or contained in one another.
\item
$d_\t >d_\s$ if $\t \subsetneq \s$.
\end{enumerate}
Two cluster pictures $(\Sigma,\cR,d)$ and $(\Sigma',\cR',d')$ are isomorphic if there is a bijection $\phi\colon \cR\to \cR'$ which induces a bijection from $\Sigma$ to $\Sigma'$ and $d_\s = d'_{\phi(\s)}$. 
We say a group $G$ acts on $(\Sigma,\cR,d)$ if it acts by isomorphisms\footnote{This is subtly different from the action of a group by automorphisms on metric cluster pictures, which we use specifically for semistable curves; see Definition \ref{auts of clusters defi}.}.
\end{definition}

We refer the reader to Table \ref{clusternotation} for our standard notation for clusters, including the notions of child/parent and $\s\wedge\s'$.


\begin{definition}\label{def:integralCP}
Let $(\Sigma,\cR,d)$ be a cluster picture with an action 
of $G_K$,
and let $n\in\Z$. 
We say that the pair $(\Sigma,n)$ is {\em integral} if either
\begin{itemize}
\item $n \ge 0$ and $d_{\cR} \ge 0$, or 
\item there is a $G_K$-stable proper cluster $\s$ with $d_\s \le 0$ and $$n+(|\s|\!-\!|\t|)d_\s+\sum\nolimits_{r\notin\s} d_{r\wedge \s} \ge 0$$
for some $\t$ which is either empty or a $G_K$-stable child $\t<\s$ with either $|\t|=1$ or $d_{\t} \ge 0$.
\end{itemize}
\end{definition}

%
%

\begin{theorem}\label{th:integralpoly}
Let $C:y^2 = f(x)$ be a hyperelliptic curve over $K$, and let $(\Sigma,\cR,d)$ be the associated cluster picture. \\
(1) If $f(x)\in\cO_K[x]$ then $(\Sigma,v(c_f))$ is integral.\\
(2) Conversely, if $(\Sigma, v(c_f))$ is integral and $G_K$ acts tamely on $\cR$  then $f(x-z) \in \cO_K[x]$ for some $z \in K$.\end{theorem}

\begin{lemma}\label{le:NewtonPoly}
Let $f(x) \in K[x]$. Then $f(x) \in \cO_{K}[x]$ if and only if 
$$
\sum_{r\in \cR} \min\{0, v(r)\} \ge -v(c_f).
$$
\end{lemma}
\begin{proof}
 $f(x) \in \cO_K[x]$ if and only if every point in the Newton polygon of $\frac{f(x)}{c_f}$ lies above $-v(c_f)$. Equivalently the sum of the positive slopes is less than or equal to $v(c_f)$, i.e. 
$
\sum_{r\in \cR} \min\{0, v(r)\} \ge -v(c_f).
$
\end{proof}

\begin{proof}[Proof of Theorem \ref{th:integralpoly}]
(1)
As $f(x)\in\cO_K[x]$, clearly $v(c_f) \ge 0$, so we may suppose $d_{\cR}<0$. If $0$ is not a centre for $\cR$, then every $r\in\cR$ has $v(r)<d_\cR$, so that by Lemma \ref{le:NewtonPoly}
$$
|\cR|d_\cR > \sum_{r\in\cR}\min\{0,v(r)\}\ge-v(c_f).
$$
Thus, taking $\s=\cR$ and $\t=\emptyset$, shows that $(\Sigma,v(c_f))$ is integral.

Henceforth suppose that $d_\cR<0$ and that 0 is a centre for $\cR$.
Let $\s$ be the smallest proper cluster with centre $z_\s=0$ and depth $d_\s\le 0$ and let $\t =  \{r \in \s \mid v(r) > d_\s\}$. 
Note that, by minimality of $\s$, $\t$ is either empty or a child of $\s$ (with $d_\t\ge 0$ if $|\t|>1$).
In particular $v(r)<d_\s\le 0$ for every $r\notin\s$, and $v(r)=d_\s\le 0$ for every $r\in\s\setminus\t$.
Thus
$$
(|\s|-|\t|)d_\s + \sum_{r \notin \s} d_{r\wedge\s}   = \sum_{r\in\s\setminus\t}d_\s +\sum_{r\in\t}0+ \sum_{r \notin \s} d_{r\wedge\s} =
$$
$$
=  \sum_{r \in \s\setminus \t} v(r) + \sum_{r \in \t} 0 + \sum_{r \notin \s} v(r) \ge  \sum_{r\in \cR} \min\{0, v(r)\}\ge -v(c_f),
$$
where the last step comes from Lemma \ref{le:NewtonPoly}.
Finally, note that $\s$ and $\t$ are $G_K$-stable, and that either $|\t|\le 1$ or $\t<\s$ with $d_\t\ge 0$. The result follows.

(2) If $d_{\cR} \ge 0$, by Lemma \ref{lem:invcentre} we can pick a centre $z=z_\cR\in K$ for $\cR$. Then the roots of $f(x-z)$ are all integral. Since by integrality $v(c_f) \ge 0$, we must have $f(x-z) \in \cO_K[x]$, as required.

If $d_\cR<0$, consider $\s$ and $\t$ as in the definition of integrality of $(\Sigma,v(c_f))$. 
If $\t\neq\emptyset$, by Lemma \ref{lem:invcentre} we can pick a centre $z=z_\t=z_\s\in K$ for $\t$ and $\s$.
If $\t$ is empty, pick $z=z_\s\in K$ to be a centre for $\s$, using the same Lemma.
Shifting $x$ to $x-z$, we may thus assume that 
$z_\s=0$ is a centre for $\s$, and that
$z_\t=0$ if $\t\neq\emptyset$. 
By assumption, $d_\s\le 0$, so we have $v(r)<d_\s\le 0$ for all $r\not\in\s$; moreover, $d_s\le v(r)$ for all $r\in\s$, and $0\le v(r)$ for $r\in\t$. 
Hence
$$
-v(c_f) \le  (|\s|-|\t|)d_\s + \sum_{r \notin \s} d_{r\wedge\s}
= |\s\setminus\t|d_\s + \sum_{r \notin \s} v(r)=
$$
$$
\sum_{r \in \s\setminus \t} d_\s+ \sum_{r \in \t} 0 + \sum_{r \notin \s} v(r) 
\le \sum_{r\in \cR} \min\{0, v(r)\}.
$$
The result follows from Lemma \ref{le:NewtonPoly}.
\end{proof}

%

We record a further result that will be useful for understanding the minimal Weierstrass equation of a hyperelliptic curve. In view of Theorem \ref{th:integralpoly} it gives a simple criterion for checking whether $p^n f(x-z)$ has integral coefficients for some $z\in K$, provided all the roots of $f(x)$ lie in $K$. 

\begin{lemma}
\label{lem:optimaln}
Let $\Sigma$ be a cluster picture with trivial $G_K$-action, such that $d_\s\in\Z$ for every cluster $\s$ and $d_{\cR}\le 0$.
Let $n$ be the minimal integer which makes $(\Sigma,n)$ integral.\\ 
(i) Then $n\!=\!\min_\t(-\sum_{r\notin\t}d_{r\wedge\t})$, where $\t$ ranges over all clusters with $d_{P(\t)}\!\le\!0$ that are either singletons or have $d_\t>0$.\\
(ii) If there is a cluster of size $\ge\frac{|\cR|}{2}$ and depth $\ge0$ then $n=-\sum_{r\notin\t}d_{r\wedge\t}$, where $\t$ is the maximal such cluster (either of two, if there are two such).\\
(ii') If there is a cluster of size $\ge\frac{|\cR|}{2}$ and depth $>0$ then $n=-\sum_{r\notin\t}d_{r\wedge\t}$, where $\t$ is the maximal such cluster (either of two, if there are two such).
\end{lemma}
\begin{proof}
If $d_\cR=0$ then $n=0$ and the results follow directly from the definition, so we may suppose that $d_\cR<0$.
Note that for $\t<\s$,
$$
(|\s|-|\t|)d_\s+\sum_{r\notin\s}d_{r\wedge \s} =
\sum_{r\in\s\setminus\t}d_\s+\sum_{r\notin\s}d_{r\wedge \s} =  \sum_{r\notin\t} d_{r\wedge\s} = \sum_{r\notin\t} d_{r\wedge{\t}}.
$$

(i) 
Suppose that $\s$ and $\t$ are the clusters from Definition \ref{def:integralCP} corresponding to $n$. Recall that $G_K$ acts trivially, so all clusters are $G_K$-stable.
If $\t=\emptyset$ then $\s$ cannot have a child $\s'$ with $d_{\s'}\le 0$, since $\sum_{r\in\cR} d_{r\wedge\s} < \sum_{r\in\cR} d_{r\wedge\s'}$, and so $n$ is not the minimal possible. Thus if $\t=\emptyset$, then every child $\t'<\s$ is either a singleton or has $d_{\t'}>0$, and has $(|\s|\!-\!|\t|)d_\s\le (|\s|\!-\!|\t'|)d_\s$ and hence $n+(|\s|\!-\!|\t'|)d_\s+\sum_{r\notin\s}d_{r\wedge \s}\ge 0$. In other words, we may assume that $\t\neq\emptyset$.
%
%
The required formula now follows.

(ii) 
If $d_\t=0$ then for any child $\t'<\t$ we have $\sum_{r\notin\t} d_{r\wedge\t}=\sum_{r\notin\t'} d_{r\wedge\t'}$. If $d_\t>0$ then by maximality $d_{P(\t)}\le 0$. In either case, $n\le -\sum_{r\notin\t}d_{r\wedge\t}$ by (i).

If $\s\not\subseteq\t$ is a cluster with $d_{P(\s)}\le 0$ 
then necessarily $|(\s\wedge\t)\setminus\t|\le\frac{\cR}{2}\le |\t|$ and
$$
 \sum_{r\notin\s} d_{r\wedge\s} = 
 \sum_{r\notin\s\wedge\t} d_{r\wedge(\s\wedge\t)} + \sum_{r\in(\s\wedge\t)\setminus\s} d_{r\wedge\s} \le
 \sum_{r\notin\s\wedge\t} d_{r\wedge(\s\wedge\t)} + |\t|d_{\s\wedge\t} \le
$$
$$
 \sum_{r\notin\s\wedge\t} d_{r\wedge(\s\wedge\t)} + |(\s\wedge\t)\!\setminus\!\t|d_{\s\wedge\t} \le
 \sum_{r\notin\s\wedge\t} d_{r\wedge(\s\wedge\t)} + \sum_{r\in(\s\wedge\t)\setminus\t} d_{r\wedge\t} =
 \sum_{r\notin\t} d_{r\wedge\t}.
$$
Hence $\t$ gives the optimal bound in the expression in (i).

(ii') Same as (ii) without the $d_\t=0$ case.
\end{proof}

\section{Isomorphisms of curves and cluster pictures}\label{s:isomorphisms}

Different models of the same hyperelliptic curve may have different cluster pictures.
In this section we show that there is a good equivalence relation on cluster pictures that is respected by isomorphisms between hyperelliptic curves: isomorphic curves have ``equivalent'' cluster pictures (Theorem \ref{th:isoequiv}) and, conversely, every cluster picture in the equivalence class is realised by some curve over $\Kbar$ (Corollary \ref{realiseequiv}). We will look at $K$-isomorphism classes of (semistable) curves in the next section.




\begin{definition}
\label{d:equivalent}
Two cluster pictures $(\Sigma,\cR,d)$ and $(\Sigma', \cR', d')$ are \emph{equivalent} if $(\Sigma', \cR', d')$ is isomorphic to a cluster picture obtained from $(\Sigma,\cR,d)$ in a finite number of the following steps:
\begin{itemize}[leftmargin=*]
\item \emph{increasing the depth of all clusters by $m\in\Q$:} \\
$\cR' = \cR$, $\Sigma'=\Sigma$ and $d_\s'= d_\s + m$ for all $\s \in \Sigma$,

\item  \emph{adding a root $r$}, when $|\cR|$ is odd:\\ 
$\cR' = \cR \cup \{r\}$,  
$\Sigma'=\Sigma\cup \{\{r\},\cR'\}\setminus\{\cR\}$, $d'_\s = d_\s$ for all proper $\s\in \Sigma'\setminus{\{\cR'}\}$ and  $d'_{\cR'}=d_{\cR}$,

\item \emph{removing a root $r\in\cR$,} when $|\cR|$ is even, $\{r\}<\cR$ and $\cR\setminus\{r\}\notin\Sigma$:\\ 
$\cR' = \cR \setminus \{r\}$,
$\Sigma'=\Sigma\cup \{\cR'\}\setminus\{\cR,\{r\}\}$, $d'_\s = d_\s$ for all proper $\s\in \Sigma'\setminus{\{\cR'}\}$ and  $d'_{\cR'}=d_{\cR}$,




\item \emph{redistributing the depth between $\c$ and $\s^{c}\!=\!\cR\setminus \c$ to $d'_\s=d_\s+m$,} when $|\cR|$ is even, $\c <\cR$ and $-\delta_\s\le m\le \delta_{\s^c}$: \\
$\cR'=\cR$, $\Sigma'=\Sigma\cup \{\s,\s^{c}\}$, $d'_\cR = d_\cR$, 
$$
 d'_\t = d_\t+m \qquad \text{for proper clusters } \t\subseteq\s
$$ 
$$
\>\> d'_\t = d_\t-m \qquad \text{for proper clusters } \t\subseteq\s^c
$$ 
Here we set $\delta_{\s^{c}} =0$ if $\s^{c} \notin \Sigma$, and $d_\t=+\infty$ if $|\t|=1$; if in the resulting cluster picture $\delta'_\s=0$, we remove $\s$ from $\Sigma'$, and similarly for $\s^c$.
\end{itemize}
\end{definition}

\begin{example} 
For any given $n \in \Q$ the following cluster pictures form an equivalence class:
$$
\begin{array}{c}
\clusterpicture            
  \Root {1} {first} {r1};
  \Root {1} {r1} {r2};
  \Root {3} {r2} {r3};
  \Root {} {r3} {r4};
  \Root {} {r4} {r5};
  \Root {1} {r5} {r6};
  \ClusterD c1[{n}] = (r1)(r2);
  \Cluster c2 = (r3)(r4)(r5)(c1);
  \Cluster c3 = (r6)(c2);
\endclusterpicture -\clusterpicture            
  \Root {1} {first} {r1};
 \Root {} {r1} {r2};
  \Root {3} {r2} {r3};
  \Root {} {r3} {r4};
  \Root {} {r4} {r5};
  \Root {} {r5} {r6};
  \ClusterD c1[{n}] = (r1)(r2);
  \Cluster c2 = (r3)(r4)(r5)(r6)(c1);
\endclusterpicture- \RebclnDce - \RebclnCce -  \RebclnGce\cr
\bigm|\qquad\qquad\qquad\qquad\qquad\quad\>\>\>\>\bigm|\\[2pt]
\clusterpicture            
  \Root {1} {first} {r1};
  \Root {} {r1} {r2};
  \Root {3} {r2} {r3};
  \Root {} {r3} {r4};
  \Root {} {r4} {r5};
  \ClusterD c1[{n}] = (r1)(r2);
  \Cluster c2 = (r3)(r4)(r5)(c1);
\endclusterpicture \qquad\qquad\qquad\qquad \RebclnFce
\end{array}
$$
Here the subscripts on clusters specify their relative depths, $t$ can take all values in the range $0<t<n$, and all clusters of size 5 and 6 can have arbitrary depths. Horizontal lines correspond to cluster pictures that are related by redistributing the depth of a child $\s<\cR$ (possibly a singleton) and $\cR\setminus\s$, and vertical lines to those related by adding or removing a root.
\end{example}

\begin{remark}
This agrees with the notion of equivalence in \cite{hyble} for ``metric cluster pictures'' (\cite{hyble} Definitions 3.43, 3.45). Metric cluster pictures do not carry a depth function on clusters, but only a relative depth, which is accounted for by the first of the steps in Definition \ref{d:equivalent}. The other steps then correspond to the moves (iii), (iv) and (i/ii) in the definitions in \cite{hyble}.

\end{remark}


\begin{theorem}\label{th:isoequiv}
If $C_1: y^2=f_1(x)$ and $C_2: y^2=f_2(x)$ are isomorphic hyperelliptic curves over $K$, 
then their cluster pictures are equivalent.
\end{theorem}

\begin{proof}
Note that if $F/K$ is a finite extention then the cluster pictures of $C_1$ and $C_2$ are equivalent over $K$ if and only if they are equivalent over $F$. So we may assume that $C_i/K$ are semistable. 
Then by Theorem \ref{main dual graph thingy}, the two metric hyperelliptic graphs 
$G_{\Sigma_{C_1}}$ and $G_{\Sigma_{C_2}}$ are isomorphic (see \ref{TtoG} for the notation), 
and by \cite{hyble} Thm 5.1, the cluster pictures $\Sigma_{C_1}$ and $\Sigma_{C_2}$ are equivalent.
\end{proof}

\begin{lemma}\label{magicformula}
(i) For $x,y\in\Kbar$,
$$
 v\Bigl(\frac1x - \frac1y\Bigr)=-v(x)-v(y)+v(x-y).
$$
(ii) 
Let $f(x)\in K[x]$ be a separable polynomial with cluster picture $\Sigma_f=(\cR,\Sigma,d)$. Suppose $\s<\cR$ and that all $r\in\s$ have $v(r)=a$ and all $r\in\cR\setminus\s$ have $v(r)=b$. 

Let $\cR'=\{\frac1r:r\in\cR\}$. Then $\psi:r\mapsto\frac1r$ induces a 1-to-1 correspondence between $\Sigma\cup\{\s,\cR\!\setminus\!\s\}$ and $\Sigma'\cup\{\s',\cR\!\setminus\!\s'\}$, where $\s'=\{\frac1r:r\in\s\}$. Moreover, $d'_{\psi(\t)}=d_{\t}-2a$ for clusters $\t\subset\s$, $d'_{\psi(\t)}=d_{\t}-2b$ for clusters $\t\subset\cR\setminus\s$, and $d'_{\cR'}=d_{\cR}-a-b$.
\end{lemma}
\begin{proof}
(i) Clear, since $\frac1x-\frac1y=\frac{y-x}{xy}$. \\ (ii) Follows directly from (i).
\end{proof}


\begin{proposition}\label{mobius}
Let $f(x)\in K[x]$ be a separable polynomial with roots $\cR\subset\bar K$, 
such that $G_K$ acts tamely on $\cR$, 
and let $\Sigma$ be the associated cluster picture. 
Suppose $\Sigma'$ is a cluster picture obtained from $\Sigma$ by one of the following constructions:
\begin{enumerate}
\item Increasing the depth of all clusters by some $n\in\Z$;
\item Adding a root to $\Sigma$, provided $|\cR|$ is odd, $d_\cR\in\Z$ and $|k|> \#\{\s<\cR: \s \text{ is $G_K$-stable}\}$;
\item Redistributing the depth from $\s$ to $\cR\setminus\s$ to eliminate $\s$ and then changing the depth of $\cR$ to 0, provided
$|\cR|$ is even, $\s<\cR$ is $G_K$-stable with $d_\s\in\Z$ and $|k|> \#\{\t<\s: \t \text{ is $G_K$-stable}\}$;
\item Redistributing the depth between $\s$ and $\cR\setminus\s$ by decreasing the depth of $\s$ by 1, provided
$|\cR|$ is even, $\s<\cR$ is $G_K$-stable with $d_\cR, d_\s\in\Z$  and $|k|> \#\{\t<\s: \t \text{ is $G_K$-stable}\}$.
\item Removing a root from $\cR$, provided $|\cR|$ is even, $d_\cR\in\Z$ and $f(x)$ has a root $r\in K$ which does not lie in any proper cluster other than $\cR$.
\end{enumerate}

Then there is a M\"obius transformation $\phi(z)=\frac{az+b}{cz+d}$ with $a, b, c, d\in K$, such that $\Sigma'$ is the cluster picture of $\cR'=\{\phi(r): r\in\cR\}\setminus\{\infty\}$ if $|\cR|$ is even and of $\cR'=\{\phi(r): r\in\cR\cup\{\infty\}\}\setminus\{\infty\}$ if $|\cR|$ is odd.

Moreover, if $y^2=f(x)$ is a hyperelliptic curve, then there is a $K$-isomorphic curve given by a Weierstrass model whose cluster picture is $\Sigma'$.
\end{proposition}

\begin{proof}
Depending on the case 1--5, the M\"obius transformation $\phi$ can be obtained as follows:\\
1. Take $\phi(z)=\pi^n z$.
\\ 2. Since $G_K$ acts tamely, Lemma \ref{lem:invcentre} shows that $\cR$ has a center $z_\cR\in K$; shifting by this and applying (1) we may assume that $z_\cR=0$ and $d_\cR=0$. Shifting further by some $y\in\cO_K$, we may assume that all $r\in\cR$ are units in $\cO_{\Kbar}$. Now $\phi(z)=\frac{1}{z}$ has $\cR'=\{\frac1r :r\in\cR\}\cup\{0\}$, which, by Lemma \ref{magicformula} (i), has the required properties. 
\\ 3. As in the proof of (2), we may assume that $d_\s=0$, $z_\s=0$ and that all roots $r\in\s$ are units. All other roots $r\in\cR\setminus\s$ then have valuation $v(r)=d_\cR$. By Lemma \ref{magicformula}, $\phi(z)=\frac1z$ gives the required cluster picture.
\\ 4. As in the proof of (2), we may assume that $d_\cR=0$, $z_\s=0$ and that all roots $r\in\s$ have valuation 1. Note that all $r\in\cR\setminus\s$ have valuation 0. By Lemma \ref{magicformula}, $\phi(z)=\frac{\pi_K}{z}$ gives the required cluster picture.
\\ 5. Shifting by $r$ and applying (1), we may assume that $z_\cR=r=0$ and $d_\cR=0$. Since $r$ does not lie in any proper subcluster of $\cR$, all the other roots of $f(x)$ must be units.  By Lemma \ref{magicformula}, $\phi(z)=\frac1z$ gives the required cluster picture.

Finally observe that if $y^2=f(x)$ is a hyperelliptic curve, then a change of variables of the form $x=\frac{ax'+b}{cx'+d}$, $y=\frac{y'}{(cx'+d)^{g+1}}$ for a M\"obius transformation $\psi(z)=\frac{az+b}{cz+d}$ with $a, b, c, d\in K$, gives a model for $C/K$ of the form $y'^2=g(x')$. The set of roots of $g(x)$ is precisely $\cR'=\{\psi^{-1}(r):r\in\cR\}\setminus\{\infty\}$ if $|\cR|$ is even and $\cR'=\{\psi^{-1}(r):r\in\cR\cup\{\infty\}\}\setminus\{\infty\}$ if $|\cR|$ is odd. Setting $\phi=\psi^{-1}$ for $\phi$ as in the first part gives the desired model.
\end{proof}

\begin{corollary}\label{realiseequiv}
Let $C: y^2=f(x)$ be a hyperelliptic curve over $K$ and $\Sigma$ its cluster picture. If $\Sigma'$ is equivalent to $\Sigma$, then there is a $\Kbar$-isomorphic hyperelliptic curve $C'/\Kbar:y^2=g(x)$ whose cluster picture is $\Sigma'$.
\end{corollary}
\begin{proof}
Replacing $K$ by a sufficiently large extension we may assume that Galois acts trivially on the roots of $f(x)$ and that the depths of all clusters in $\Sigma$ and $\Sigma'$ are integers. Now Proposition \ref{mobius} lets us realise all the equivalence steps from $\Sigma$ to $\Sigma'$ in Definition~\ref{d:equivalent} through isomorphisms of curves.
\end{proof}

\section{Canonical cluster picture of a semistable curve}\label{s:balancing}

As any given hyperelliptic curve can admit many different models, it is often desirable to be able to put it in some canonical form. It turns out that every equivalence class of cluster pictures has a canonical representative (Lemma \ref{balanced}). 
Unfortunately, this canonical cluster picture does not always correspond to a Weierstrass model for $C$ defined over $K$. 
However, for semistable curves this is nearly the case (Theorem \ref{Ombalanseringsatz}). Moreover, if the genus of a semistable hyperelliptic curve is even, then this cluster picture does come from a Weierstrass model over $K$ (Corollary \ref{co:evengenus}). As we shall see in \S\ref{se:bible}, this canonical cluster picture is particularly well suited for classifying all reduction types of semistable hyperelliptic curves in any given genus.


\begin{lemma}\label{balanced}
Let $(\Sigma,\cR,d)$ be a cluster picture with $|\cR|=2g+1$ or $2g+2$. There is a unique equivalent (``balanced'') cluster picture $(\Sigma^{b},\cR^b,d^b)$ such that (i) $d^b_{\cR^b}=0$, (ii) $\Sigma^b$ has no clusters of size $>g+1$, (iii) either  $\Sigma^b$ has no clusters of size $g\!+\!1$ or it has two such clusters, in which case they have equal depth.
\end{lemma}

\begin{proof}
This follows from \cite{hyble} Theorem 5.1.
\end{proof}


%
%

\begin{theorem}
\label{Ombalanseringsatz}
Let $C/K$ be a semistable hyperelliptic curve and suppose that 
$|k|>$ number of $G_K$-stable children of every cluster of size $\ge g+1$ in the cluster picture of $C$. 
Then there is a $K$-isomorphic 
curve $C'/\cO_K: y^2=f(x)$ with $deg(f) =2g+2$ such that 
\begin{itemize}[leftmargin=*]
\item the top cluster of $C'$ has depth 0; 
\item the cluster picture of $C'$ has no cluster of size $>g+1$;
\item either there is at most one cluster of size $g+1$ and $v(c_f)=0$, or $\Frob_K$ swaps two clusters of size $g+1$ and $v(c_f)=0$ or $1$.
\end{itemize}
\end{theorem}

\begin{proof}
Let $(\Sigma,\cR, d)$ be the cluster picture of $C/K$. Recall that as $C/K$ is semistable, it satisfies the semistability criterion (Theorem \ref{the semistability theorem}, Definition \ref{semistability criterion}). In particular, the inertia group cannot permute proper clusters in $\Sigma$.

If Frobenius swaps two clusters of size $g+1$, then by the semistability criterion $d_{\cR}\in\Z$ (e.g. since $\s<\cR$ is principal and so has integral depth, and $\delta_\s\in\Z$ by Proposition \ref{th:equivss} (2)). Applying a transformation of the form $x'=\pi_K^k x$, $y'=\pi_K^n x$ gives a model over $K$ with the required cluster picture. 

Suppose henceforth that Frobenius does not swap clusters of size $g+1$, and consequently that all clusters of size $\ge g+1$ are $G_K$-stable. 
We now change the model for $C$ by repeatedly applying Proposition \ref{mobius}: (2) creates a cluster picture with an even number of roots, then (1) or (3) makes the depth of the top cluster 0, and finally a repeated use of (4) removes all clusters of size $>g+1$ and leaves at most one cluster of size $g+1$ (principal clusters have integral relative depth by Proposition \ref{th:equivss} using that the depth of the top cluster is now $0$ for the case when it is a cotwin). The resulting model necessarily has $v(c_f)\in 2\Z$ (semistability criterion for the top cluster, as it now has depth 0), and hence a change of variables of the form $y=\pi^k y$ gives a model over $K$ with the required cluster picture. 

Finally, Theorem \ref{th:integralpoly} shows that shifting the $x$-coordinate by a suitable element of $K$ gives a model over $\cO_K$.
\end{proof}

\begin{corollary}\label{co:evengenus}
Let $C/K$ be as in Theorem \ref{Ombalanseringsatz}. If $C$ has even genus then there is a $K$-isomorphic 
curve $C'/\cO_K: y^2=f(x)$ such that 
\begin{itemize}[leftmargin=*]
\item the top cluster of $C'$ has size $2g+2$ and depth 0; 
\item the cluster picture of $C'$ has no cluster of size $>g+1$;
\item either $v(c_f)=0$ and there is no cluster of size $g+1$, or $v(c_f) \in \{0,1\}$ and there are two clusters of size $g+1$ with equal depths.
\end{itemize}
Any other $K$-isomorphic curve $C''/\cO_K$ satisfying (1), (2) and (3) has the same cluster picture and valuation of leading term as $C'$.
\end{corollary}

\begin{proof}
The existence of $C'$ follows from the theorem and, in the case when the theorem yields a cluster picture with a cluster of size $g+1$, Proposition~\ref{mobius}(4).
The fact that, in this special case, the two resulting clusters $\s, \s'$of size $g+1$ can be made to have equal depth follows from the semistability criterion, which shows that $\delta_\s, \delta_{\s'}\in\Z$ and $(g+1)\delta_\s\equiv (g+1)\delta_{\s'} \mod 2$  (Theorem \ref{the semistability theorem}, Definition \ref{semistability criterion}).

Uniqueness of the cluster picture follows from the fact that isomorphic curves have equivalent cluster pictures (Theorem \ref{th:isoequiv}) and uniqueness of balanced cluster pictures (Lemma \ref{balanced}).
\end{proof}

\begin{remark}\label{balansingact}
Theorem \ref{cor:basicminimal} and Proposition \ref{pr:excepmini} show that the models obtained in the theorem and the corollary are  minimal Weierstrass  equations.
\end{remark}

\section{Discriminant}\label{s:discriminant}


Recall the definition of the discriminant of a hyperelliptic curve:

\begin{definition}(See \cite{Liu}, Section 2)\label{de:DiscHEC}
Let $C:y^2=f(x)$ be a hyperelliptic curve of genus $g$ over $K$. The {\em discriminant} $\Delta_C$ of $C$ is 
$$
\Delta_C 
=  16^g c_f^{4g+2}\text{disc}\Bigl(\frac{1}{c_f}f(x)\Bigr).
$$
\end{definition}

Our main result on the discriminant is that one can easily read off its valuation from the cluster picture of $C$ and, when $C/K$ is semistable, one can moreover read off the valuation of the discriminant of its minimal Weierstrass model:

\begin{theorem}
\label{thmmindisc}
Let $C:y^2=f(x)$ be a hyperelliptic curve of genus $g$ over $K$, and let $(\Sigma,\cR,d)$ be the associated cluster picture. 
Then
$$
 v(\Delta_C)=
   v(c_f)(4g+2) + \sum_{\s \text{ proper}} d_\s \Bigl(|\s|^2-\mathop{\Sigma}\limits_{\s'<\s}|\s'|^2\Bigr)
$$
If $C/K$ is semistable and $|k|>2g\!+\!1$, then the valuation of the discriminant $\Delta_C^{\min}$ of a minimal Weierstrass model of $C$ is determined by the formula
$$
  \frac{v(\Delta_C)-v(\Delta_C^{\min})}{4g+2}=v(c_f)-E+ d_{\cR}(|\cR|-g-1)+\sum_{g+1<|\s|<|\cR|}\delta_\s   (|\s|-g-1),
$$
where $E=0$ unless $\Sigma$ has two clusters of size $g+1$ that are permuted by Frobenius and $v(c_f)$ is odd, in which case $E=1$.
\end{theorem}

The first part of the theorem follows directly from Lemmas \ref{comparedisc} and \ref{le:discriminant} below. The second part will be proved at the end of section \ref{s:MWEquation}, after we find a description for minimal Weierstrass equations in terms of cluster pictures.

\begin{definition}
Let $(\Sigma,\cR,d)$ be a cluster picture and let $n\in\Z$. 
The {\em discriminant} $\Delta_{\Sigma,n}$ of the pair $(\Sigma,n)$ is the fractional ideal of $K$ with valuation
$$
v(\Delta_{\Sigma,n}) = 
n(4g+2)+ 2 \sum _{r\neq r'\in\cR} d_{r\wedge r'}.
$$
\end{definition}

\begin{lemma}\label{comparedisc}
Let $C:y^2=f(x)$ be a hyperelliptic curve over $K$, let $\Sigma$ be the associated cluster picture and $n=v(c_f)$. Then
$$
v(\Delta_{C})=v(\Delta_{\Sigma,n}).
$$
\end{lemma}

\begin{proof}
Clear.
\end{proof}


\begin{lemma}\label{le:discriminant} Let $(\Sigma, \cR, d)$ be a cluster picture. Then
 $$ 2 \sum _{r\neq r'\in\cR} d_{r\wedge r'}
   =  \sum_{\s, |\s|>1} d_\s (|\s|^2-\mathop{\Sigma}\limits_{\s'<\s}|\s'|^2)
   =  \sum_{\s} |\s|^2 \left\{\scriptstyle\begin{array}{lll}
      -d_{P(\s)} & \text{if} & |\s|=1, \cr   
      d_\s & \text{if} & \s=\cR, \cr   
      \delta_{\s} & \rlap{\text{otherwise.}} \cr   
      \end{array}\right.
     $$
\end{lemma}

\begin{proof}
We have

$$ 
 2 \sum _{r\neq r'\in\cR} d_{r\wedge r'}= \sum_{r\in \cR} \sum_{\s\ni r, |\s|>1} d_\s (|\s|-|\text{child of $\s$ containing $r$}|)
$$
$$
= \sum_{\s, |\s|>1} d_\s (|\s|^2-\mathop{\Sigma}\limits_{\s'<\s}|\s'|^2)=\sum_{\s, |\s|>1} d_\s |\s|^2-\sum_{\s'\ne \cR} d_{P(\s')}|\s'|^2.
$$
\end{proof}

In the remainder of this section we establish some results on how the discriminant changes under the different steps yielding equivalent cluster pictures (see Definition \ref{d:equivalent}). As equivalence of cluster pictures is closely tied to isomorphisms of hyperelliptic curves (see Theorem \ref{th:isoequiv} and Proposition \ref{mobius}), this will be the key to understanding minimal Weierstrass equations and their discriminants. Recall also that we have a notion of integrality for cluster pictures (see Definition \ref{def:integralCP}), which gives a way of determining whether a cluster picture corresponds to an integral Weierstrass equations.

\begin{lemma}
\label{le:discchange}
Let $(\Sigma,\cR,d)$ and $(\Sigma',\cR',d')$ be cluster pictures. \\
(i) If $\Sigma'$ is obtained from $\Sigma$ by increasing the depth of all clusters to $d_\s'=d_\s+t$, then
$$
 v(\Delta_{\Sigma',n})= v(\Delta_{\Sigma,n})+t|\cR|(|\cR|-1).
$$
(ii) If $\Sigma$ has odd size and $\Sigma'$ is obtained by adding a root to $\Sigma$, then
$$
 v(\Delta_{\Sigma',n})= v(\Delta_{\Sigma,n})+2d_{\cR}|\cR|.
$$
(iii) If $\Sigma$ has even size then for $k\in\Z$,
$$
 v(\Delta_{\Sigma,n+k})= v(\Delta_{\Sigma,n})+2k(|\cR|-1).
$$
If $\Sigma$ has odd size, then 
$$
v(\Delta_{\Sigma,n+k})= v(\Delta_{\Sigma,n})+2k|\cR|.
$$
(iv) If $\Sigma$ has even size and $\Sigma'$ is obtained by redistributing the depth between $\s<\cR$ and $\cR\setminus\s$ to $d_\s'=d_\s-t$ and $d_{\cR\setminus\s}'=d_{\cR\setminus\s}+t$, then
$$
 v(\Delta_{\Sigma',n})= v(\Delta_{\Sigma,n})+ t(|\cR|-2|\s|)(|\cR|-1).
$$
\end{lemma}

\begin{proof}
Write $|\cR|$ as $|\cR|=2g+1$ or $2g+2$ for some $g\in\Z$.\\
(i)
$$
v(\Delta_{\Sigma',n})=n(4g+2) + \sum_{r\neq r'\in\cR}d'_{r\wedge r'} 
= n(4g+2) + \sum_{r\neq r'\in\cR}d_{r\wedge r'} + |\cR|(|\cR|-1)t.
$$
(ii)
$$
v(\Delta_{\Sigma',n})=n(4g+2) + \sum_{r\neq r'\in\cR'}d'_{r\wedge r'} = 
n(4g+2) + \sum_{r\neq r'\in\cR}d_{r\wedge r'} + 2|\cR|d_{\cR}.
$$
(iii)
$$
v(\Delta_{\Sigma,n+k})=(n+k)(4g+2) + \sum_{r\neq r'\in\cR}d_{r\wedge r'} 
= n(4g+2) + \sum_{r\neq r'\in\cR}d_{r\wedge r'} + k(4g+2).
$$
(iv)
$$
v(\Delta_{\Sigma',n})=n(4g+2) + \sum_{r\neq r'\in\cR}d'_{r\wedge r'} =
$$
$$
= n(4g+2) + \sum_{r\neq r'\in\s}d'_{r\wedge r'} + \sum_{r\neq r'\in\cR\setminus\s}d'_{r\wedge r'} + 2\sum_{r\in\s, r'\in\cR\setminus\s}d'_{r\wedge r'} =
$$
$$
=  v(\Delta_{\Sigma,n})- t|\s|(|\s|-1) + t|\cR\setminus\s|(\cR\setminus\s|-1)
=  v(\Delta_{\Sigma,n}) + t(|\cR|-1)(|\cR|-2|\s|).
$$
\end{proof}

The following proposition identifies a number of scenarios when we can manipulate {\em integral} cluster pictures to decrease the valuation of their discriminant.

\begin{proposition}\label{balancingdisc}
Let $(\Sigma,\cR,d)$ and $(\Sigma',\cR',d')$ be cluster pictures all of whose cluster depths are integers, and let $G_K$ act trivially on both $\Sigma$ and $\Sigma'$. Let $n, n'\in\Z$ be the minimal integers such that $(\Sigma,n)$ and $(\Sigma',n')$ are integral.\\
(i) If $d_\cR\ge 1$ and $\Sigma'$ is obtained from $\Sigma$ by decreasing the depth of all clusters by 1, then 
$$
 n'=n=0 \qquad \text{and}\qquad v(\Delta_{\Sigma',n'})=v(\Delta_{\Sigma,n}) - |\cR|(|\cR|-1).
$$
(ii) If $\Sigma$ has odd size, $d_\cR\le 0$ and $\Sigma'$ is obtained by adding a root to $\Sigma$, then
$$
 n'=n-d_\cR \quad \text{ and } \quad v(\Delta_{\Sigma',n'})=v(\Delta_{\Sigma,n}).
$$
(iii) If $\Sigma$ has even size, $d_\cR\le 0$, $\Sigma$ has a cluster $\t$ with $|\t|\ge \frac{|\cR|}{2}$, $d_\t>0$ and $d_{P(\t)}\le 0$, and $\Sigma'$ is obtained by redistributing the depth of the child $\s<\cR$ containing $\t$ and $\cR\setminus\s$ to $d_\s'=d_\s-1$ and $d_{\cR\setminus\s}'=d_{\cR\setminus\s}+1$, then
$$
 n'=n+|\s|-|\t| \quad \text{ and }  \quad v(\Delta_{\Sigma',n'})=v(\Delta_{\Sigma,n})-(2|\t|-|\cR|) (|\cR|-1).
$$
(iv) If $\Sigma$ has even size, $d_\cR< 0$,  $\Sigma$ has no cluster $\t$ with $|\t|\ge \frac{|\cR|}{2}$ and $d_\t\ge 0$, and $\Sigma'$ is obtained by increasing the depth of all clusters by 1, then 
$$
 n'<n-\frac{|\cR|}{2} \quad \text{ and } \quad  v(\Delta_{\Sigma',n'})<v(\Delta_{\Sigma,n}),
$$
(v) If $\Sigma$ has even size, $d_\cR< 0$,  $\Sigma$ has a cluster $\t$ with $|\t|\ge \frac{|\cR|}{2}$ and $d_\t=0$, and $\Sigma'$ is obtained by increasing the depth of all clusters by 1, then 
$$
 n'= n-|\cR|+|\t| \quad \text{ and } \quad \quad v(\Delta_{\Sigma',n'})=v(\Delta_{\Sigma,n})+(2|\t|-|\cR|) (|\cR|-1).
$$
\end{proposition}

\begin{proof}
(i) The claim for $n,n'$ is clear. The rest follows from Lemma \ref{le:discchange}(i).\\
(ii) The claim for $n'$ follows from Lemma \ref{lem:optimaln}(i). Lemma \ref{le:discchange}(ii,iii) then gives
$$
  v(\Delta_{\Sigma',n'})=v(\Delta_{\Sigma,n}) + 2d_\cR |\cR| - 2d_\cR(|\cR'|-1)=v(\Delta_{\Sigma,n}).
$$
(iii) Since the cluster depths are assumed to be integers, we must have $d'_\t\ge 0$, so by Lemma \ref{lem:optimaln}(ii, ii'),
$$
 n' = - \sum_{r\notin\t} d_{r\wedge\t}' = - \sum_{r\notin\s}d_{r\wedge\t}' - \sum_{r\in\s\setminus\t}d_{r\wedge\t}'
= - \sum_{r\notin\s}d_{r\wedge\t} - \sum_{r\in\s\setminus\t}d_{r\wedge\t} + |\s\setminus\t| = n+|\s\setminus\t|.
$$
 Thus by Lemma \ref{le:discchange}(iii,iv),
$$
  v(\Delta_{\Sigma',n'})=v(\Delta_{\Sigma,n}) + (|\cR|-2|\s|) (|\cR|-1) + 2(|\s|-|\t|)(|\cR|-1)
$$
$$
  = v(\Delta_{\Sigma,n})+ (|\cR|-2|\t|) (|\cR|-1).
$$
(iv) 
By Lemma \ref{lem:optimaln}(i) we can write $n=-\sum_{r\notin\t} d_{r\wedge\t}$ for some $\t$ which has $d_{P(\t)}\le 0$ and is either a singleton or has $d_\t>0$.
Since necessarily $|\t|<\frac{|\cR|}{2}$, we must have $-\sum_{r\notin\t}d'_{r\wedge\t}< -\sum_{r\notin\t}d_{r\wedge\t}- \frac{|\cR|}{2}$. 
Moreover, if $d'_{P(\t)}>0$ then necessarily $d_{P(\t)}=0$ and we also have
$$
 -\sum_{r\notin P(\t)}d'_{r\wedge P(\t)}< 
 -\sum_{r\notin P(\t)}d_{r\wedge P(\t)}- \frac{|\cR|}{2}=-\sum_{r\notin\t}d_{r\wedge\t}- \frac{|\cR|}{2}.
$$
By Lemma \ref{lem:optimaln}(i) it follows that
$$
 n'<n - \frac{|\cR|}{2}.
$$
Thus by Lemma \ref{le:discchange}(i,iii)
$$
  v(\Delta_{\Sigma',n'})=v(\Delta_{\Sigma,n}) + |\cR|(|\cR|-1) + 2(n'-n)(|\cR|-1)<v(\Delta_{\Sigma,n}) .
$$
(v) 
By Lemma \ref{lem:optimaln} (ii',ii),
$$
  n' = -\sum_{r\notin\t}d_{r\wedge\t}' = -|\cR\setminus\t|-\sum_{r\notin\t}d_{r\wedge\t} = n-|\cR|+|\t|,
$$
and by Lemma \ref{le:discchange}(i,iii)
$$
  v(\Delta_{\Sigma',n'})=v(\Delta_{\Sigma,n}) + |\cR|(|\cR|-1) + 2(n'-n)(|\cR|-1)=v(\Delta_{\Sigma,n}) + (2|\t|-|\cR|)(|\cR|-1).
$$
\end{proof}

We end this section with a result that effectively gives sufficient criteria for a cluster picture to correspond to a minimal Weierstrass equation.

\begin{corollary}\label{pr:genericcase}
Suppose $(\Sigma,\cR,d)$ is a cluster picture with a trivial action of $G_K$,
with $|\cR|=2g\!+\!2$, with $d_{\cR}=0$, with $d_\s\in\Z$ for every proper cluster $\s$, and with no cluster $\s\neq\cR$ of size $|\s|\!>\!g\!+\!1$. Then for every integral $(\Sigma',n')$ with $\Sigma'$ equivalent to $\Sigma$, with $d_{\s}\in\Z$ for every~cluster~$\s$, and with trivial $G_K$-action,
$$
v(\Delta_{\Sigma,0}) \le v(\Delta_{\Sigma',n'}),
$$
with equality if and only if 
$\Sigma'$ has no cluster of size $>g+1$ with depth $>0$, but has some cluster $\t$ of size $\ge g+1$ with $d_\t'\ge 0$, and 
$n'=-\sum_{r\notin\t}d_{r\wedge\t}'$.
\end{corollary}

\begin{proof}
By \cite{hyble} Thm 5.1, there is a unique (``balanced'') cluster picture $(\Sigma^{b},\cR^{b},d^{b})$ equivalent to $\Sigma$ with $\R^b$ even, $d^{b}_{\cR^{b}}=0$, no clusters other than $\cR^{b}$ of size $>g+1$, and either zero or two clusters of size $g+1$, in which case they have equal depth. Note that, by uniqueness, either $\Sigma=\Sigma^{b}$, or $\Sigma^{b}$ is obtained from $\Sigma$ by redistributing the depths of a child $\s<\cR$ of size $g+1$ and $\cR\setminus\s$ so that they get equal depth. 
Thus by Lemma \ref{le:discchange}(iv)
$$
v(\Delta_{\Sigma^{b},0}) = v(\Delta_{\Sigma,0}).
$$

Now consider $\Sigma'$. From the definition of the discriminant and Lemma~\ref{lem:optimaln}(ii), we may assume that $n'$ is the minimal integer that makes $(\Sigma',n')$ integral.

By Proposition \ref{balancingdisc}(i,ii), it suffices to prove the result when $\Sigma'$ has even size and $d_\cR'\le 0$.

By Proposition \ref{balancingdisc}(iv), we may further assume that either $d_{\cR'}'\!=\!0$, or $d_{\cR'}'<0$ and $\Sigma'$ has a cluster of size $\ge g\!+\!1$ with depth $\ge 0$.

By Proposition \ref{balancingdisc}(iii), we may further further assume that $\Sigma'$ has no cluster $\s$ with $|\s|>g\!+\! 1$ and $d'_\s>0$.

Suppose $d'_{\cR'}<0$. If $\Sigma'$ has a cluster $\s$ of size $g\!+\! 1$ and depth $d'_\s>0$, then repeatedly applying  Proposition \ref{balancingdisc}(iii), reduces the problem to the case when $\Sigma'$ has a cluster $\s$ of size $g\!+\! 1$ and depth $d'_\s=0$. Now if $\s$ is a cluster of size $\ge g\!+\!1$ and depth $d'_\s=0$, then repeatedly applying Proposition \ref{balancingdisc}(v) followed by (iii) (which eventually eliminates the cluster by pushing $d'_\cR$ up to 0) reduces the problem to the case when $d'_{\cR'}=0$.

Finally, suppose $d'_{\cR'}=0$, so, in particular, $n'=0$. If $\Sigma'$ has no cluster of size $g\!+\! 1$, then $\Sigma'=\Sigma^{b}$. If $\Sigma'$ does have such a cluster, then Lemma~\ref{le:discchange}(iv) and uniqueness of $\Sigma^{b}$ show that $v(\Delta_{\Sigma',0})=v(\Delta_{\Sigma^{b},0})$. The result follows.
\end{proof}

\section{Minimal Weierstrass equations}
\label{s:MWEquation}


\begin{theorem}\label{cor:basicminimal}
Let $C:y^2=f(x)$ be a hyperelliptic curve over $K$ with $f(x)\in\cO_K[x]$, and let $\Sigma$ be the associated cluster picture. 
If $d_{\cR}=0$, $v(c_f)=0$ and $\Sigma$ has no cluster $\s\neq\cR$ of size $|\s|>g+1$, then $C$ is a minimal Weierstrass equation.
\end{theorem}

\begin{proof}
Suppose another integral Weierstrass model $C':y^2=g(x)$ for $C$ has cluster picture $\Sigma'$. Over a suitable field extension $F/K$, the depths of all clusters of $\Sigma$ and $\Sigma'$ are integers, and $\Gal(\bar{F}/F)$ acts trivially on both cluster pictures. By Theorems \ref{th:integralpoly} and \ref{th:isoequiv}, $(\Sigma',v_F(c_g))$ is integral and $\Sigma'$ is equivalent to $\Sigma$ over $F$. By Proposition \ref{pr:genericcase}, $v_F(\Delta_{\Sigma,0})\le v_F(\Delta_{\Sigma',v_F(c_g)})$, and so the same inequality holds over $K$. By Lemma \ref{comparedisc}, $v(\Delta_C)\le v(\Delta_{C'})$, as required.
\end{proof}

For semistable hyperelliptic curves, we can give a full characterisation of minimal Weierstrass models in terms of cluster pictures:

\begin{theorem}\label{th:mainmini}
Let $C:y^2=f(x)$ be a semistable hyperelliptic curve over~$K$, and let $(\Sigma,\cR,d)$ be the associated cluster picture. 
Suppose that every cluster of $\Sigma$ of size $>g+1$ has at most $|k|\!-\!1$ $G_K$-stable children.

Then there is  some $z\in K$ such that $y^2=f(x-z)$ is a minimal Weierstrass model if and only if one of the following conditions holds
\begin{enumerate}
\item $\Sigma$ has two clusters of size $g+1$ that are swapped by Frobenius, $d_\cR=0$ and $v(c_f)\in\{0,1\}$,
\item $\Sigma$ has no cluster of size $>\!g\!+\!1$ with depth $>0$, but has some~$G_K$-stable cluster $\t$ of size $|\t|\ge g+1$ with  $d_\t\ge 0$ and $v(c_f)=-\sum_{r\notin\t}d_{r\wedge\t}$.
\end{enumerate}
In particular, if $f(x)\in\cO_K[x]$ then $y^2=f(x)$ is a minimal Weierstrass equation if and only if $\Sigma$ satisfies either (1) or (2), 
\end{theorem}

The rest of this section is devoted to the proof of this theorem. Our approach is to translate everything in terms of cluster pictures (semistability, integral Weierstrass equations, discriminants) and work mainly on that level. The two cases for the cluster picture in the above theorem are quite distinct and it will be convenient to use the following terminology:

\begin{definition}
The cluster picture of a semistable hyperelliptic curve $y^2=f(x)$ of genus $g$ is {\em{exceptional}} if it has two clusters of size $g+1$ that are swapped by Frobenius.
\end{definition}

\begin{proposition}\label{pr:exceptional}
Let $C:y^2=f(x)$ be a semistable hyperelliptic curve over $K$, let $\Sigma$ be the associated cluster picture and $n=v(c_f)$. Suppose $C'$ is another Weierstrass model for $C$ with corresponding $(\Sigma',n')$. Then $\Sigma$ is exceptional if and only if $\Sigma'$ is. If so, then, $n'\equiv n \mod 2$ and 
$\Sigma'$ is obtained from $\Sigma$ by increasing the depths of all clusters 
by some $d\in\Z$;
moreover 
$$
v(\Delta_{\Sigma,n}) = v(\Delta_{\Sigma',n'})  - 2(n'-n)(|\cR|-1) - d|\cR|(|\cR|-1).
$$
\end{proposition}
\begin{proof}
Since $C'$ is another Weierstrass model for $C$, their special fibres have the same dual graphs with the same automorphisms induced by Frobenius. By Theorem \ref{th:DualGraph} this graph $\Upsilon_C$ modulo the hyperelliptic involution has a unique fixed point under the Frobenius action. In the terminology of \cite{hyble}, it follows that they have the same open hyperelliptic graph (\cite{hyble} Proposition 5.7 with $\tilde{G}=\Upsilon_C$) and hence the same metric cluster picture with the same automorphisms induced by Frobenius (\cite{hyble} Thm 4.2). In particular, $\Sigma'$ is obtained from $\Sigma$ by 
increasing the depth of all clusters by some $d\in\Z$,
and with the same action of Frobenius on clusters.
 
Thus $\Sigma$ is exceptional if and only if $\Sigma'$ is.  It follows from the semistability criterion (Theorem \ref{the semistability theorem}, Definition \ref{semistability criterion}) that $n' \equiv n \mod 2$.
Finally the valuation of the discriminant follows from Lemma \ref{le:discchange}(i,iii). 
\end{proof}

%
%
%

\begin{proposition}\label{pr:excepmini}
Let $C:y^2=f(x) \in \cO_K[x]$ be a semistable hyperelliptic curve, and let $n=v(c_f)$. Suppose the associated cluster picture $\Sigma$ is exceptional. 
Then the Weierstrass equation is minimal if and only if $d_{\cR} =0$ and $n\in\{0,1\}$. 
\end{proposition}

\begin{proof}
First suppose $\Sigma$ is exceptional with $d_{\cR}=0$ and $n\in\{0,1\}$. 
Let $C'$ be another integral Weierstrass model with corresponding $(\Sigma', n')$. 
By Proposition \ref{pr:exceptional} $\Sigma'$ is exceptional, 
$n'\equiv n \mod 2$, and
$$
 v(\Delta_{\Sigma,n}) = v(\Delta_{\Sigma',n'})  - 2(n'-n)(|\cR|-1) - d'_{\cR'}|\cR|(|\cR|-1).
$$
By Theorem \ref{th:integralpoly} $(\Sigma',n')$ is integral, so either $n'\ge0$ and $d'_{\cR'}\ge 0$, or $d'_{\cR'} < 0$ and $n'+|\cR|d'_{\cR'}\ge 0$. In the first case, since $n'\equiv n\mod 2$, we clearly have $ v(\Delta_{\Sigma,n}) \le v(\Delta_{\Sigma',n'})$ with equality if and only if $n'=n$ and $d'_{\cR'} = 0$. 
In the second case, $d'_{\cR'} \in \Z$ by the semistability criterion for $C'$  (since $\s<\cR'$ is principal and so has integral depth, and $\delta_\s\in\Z$ by Proposition \ref{th:equivss} (2)), so $n'\!\ge\! -|\cR|d'_{\cR}>2$ and
we again obtain
$$
 v(\Delta_{\Sigma,n}) \le v(\Delta_{\Sigma',n'})  - 2(n'-n)(|\cR|-1) + n'(|\cR|-1)< v(\Delta_{\Sigma',n'}).
$$
Since the discriminants of the cluster pictures are those of the curves (Lemma \ref{comparedisc}), it follows that $C$ is a minimal Weierstrass equation.



For the converse, 
suppose $C:y^2=f(x)$ is a minimal Weierstrass model and $\Sigma$ is exceptional. A substitution of the form $y_1 = p^a y$ and $x_1= p^bx$ for suitable $a,b \in \Z$, gives a new model $C'$ whose cluster picture $\Sigma'$ is also exceptional,
$d'_{\cR'}=0$ and the corresponding valuation $n'$ is either 0 or 1. Theorem \ref{th:integralpoly} gives a new {\em integral} model $C''$ with the same cluster picture, $d''_{\cR''}=0$ and $n''\in\{0,1\}$. By the first part of the proof, $C''$ is a minimal Weierstrass model (so $v(\Delta_{\Sigma'',n''})=v(\Delta_{\Sigma,n})$), $n=n''$ and $d_{\cR}=0$.
\end{proof}

We now turn to cluster pictures that are not exceptional.

\begin{lemma}\label{pr:notexceptional}
Let $C:y^2=f(x)$ be a semistable hyperelliptic curve over $K$ with cluster picture $\Sigma$. Suppose that $\Sigma$ is not exceptional and that every cluster of $\Sigma$ of size $>g+1$ has at most $|k|\!-\!1$ $G_K$-stable children.
Then $C$ admits an integral Weierstrass model $y^2=h(x)$ with $v(c_h)=0$ and whose cluster picture $(\Sigma',\cR',d')$ has $|\cR'|=2g\!+\!2$, $d'_{\cR'}\!=\!0$ and has no cluster $\s\neq\cR'$ of size $> g\!+\!1$.
\end{lemma}

\begin{proof}
This follows from Theorem \ref{Ombalanseringsatz}, and the fact that by Proposition \ref{pr:exceptional}, $\Sigma'$ cannot be exceptional.
\end{proof}

\begin{proposition}\label{pr:nexcepmini}
Let $C:y^2=f(x) \in \cO_K[x]$ be a semistable hyperelliptic curve. Suppose the associated cluster picture $\Sigma$ is not exceptional and that every cluster of $\Sigma$ of size $>g+1$ has at most $|k|\!-\!1$ $G_K$-stable children.
Then the Weierstrass equation is minimal if and only if 
$\Sigma$ has no cluster of size $>g+1$ with depth $>0$, but has some cluster $\t$ of size $\ge g+1$ with $d_\t\ge 0$ and 
$-\sum_{r\notin\t}d_{r\wedge\t}=v(c_f)$.

\end{proposition}

\begin{proof}
By Lemma \ref{pr:notexceptional}, there exists another integral model $C_{\min}:y^2=g(x)$ with $v(c_g)=0$ and whose associated cluster picture $\Sigma_{\min}$ has its top cluster of depth 0, and has no other clusters of size $>g+1$.  
By Theorem \ref{cor:basicminimal} $C_{\min}$ is a minimal Weierstrass equation, so 
$$
v(\Delta_{C})\ge v(\Delta_{C_{\min}}).
$$

It remains to show that the claimed condition for having equality is correct.
Passing to a suitable field extension if necessary, we may assume that all the clusters in $\Sigma$ have integer depth and that the Galois group acts trivially on $\Sigma$ and on $\Sigma_{\min}$. 
By Theorem \ref{th:isoequiv}, $\Sigma$ is equivalent to $\Sigma_{\min}$, so 
by Corollary \ref{pr:genericcase} 
$v(\Delta_{C})=v(\Delta_{C_{\min}})$  if and only if 
$\Sigma$ has no cluster of size $>g+1$ with depth $>0$, but has some cluster $\t$ of size $\ge g+1$ with $d_\t\ge 0$, and 
$v(c_f)=-\sum_{r\notin\t}d_{r\wedge\t}$.




\end{proof}

\begin{proof}[Proof of Theorem \ref{th:mainmini}] 
Note first that, by definition of being exceptional, conditions (1) and (2) are mutually exclusive.
Note also that $y^2=f(x)$ and $y^2=f(x-z)$ have the same cluster picture and the same valuation of the leading term of the defining polynomial.

Suppose $\Sigma$ is exceptional. If $y^2=f(x-z)$ is a minimal Weierstrass equation, then Proposition \ref{pr:excepmini} shows that $d_\cR=0$ and $v(c_f)\in\{0,1\}$. Conversely, if $d_\cR=0$ and $v(c_f)\in\{0,1\}$ then $(\Sigma,v(c_f))$ is integral. By the semistability criterion $G_K$ acts tamely on $\cR$, so by Theorem \ref{th:integralpoly}(2) there exists $z \in K$ such that $f(x-z) \in \cO_K[x]$. By Proposition \ref{pr:excepmini}, $y^2=f(x-z)$ is then a minimal Weierstrass equation.

Suppose $\Sigma$ is not exceptional and $y^2=f(x-z)$ is a minimal Weierstrass equation. Then Proposition \ref{pr:nexcepmini} shows that $\Sigma$ has no cluster of size $>g+1$ with depth $>0$, but has some cluster $\t$ of size $\ge g+1$ with $d_\t\ge 0$, and $v(c_f)=-\sum_{r\notin\t}d_{r\wedge\t}$. Since $\Sigma$ is not exceptional, $\t$ is necessarily $G_K$-stable, so condition (2) holds.

Finally, suppose $\Sigma$ is not exceptional and satisfies (2). If $\t=\cR$ then $d_\t=0$ and $n=0$, so $(\Sigma,n)$ is integral.
If $\t\neq\cR$, then 
$$
\sum_{r\notin\t}d_{r\wedge\t}=|P(\t)\setminus\t|d_{P(\t)}+\sum_{r\notin P(\t)}d_{r\wedge P(\t)},
$$
so $(\Sigma,n)$ is again integral (with $\s=P(\t)$ in Definition \ref{def:integralCP}). By the semistability criterion $G_K$ acts tamely on $\cR$, so by Theorem \ref{th:integralpoly}(2) there exists $z \in K$ such that $f(x-z) \in \cO_K[x]$, and by Proposition \ref{pr:nexcepmini}, $y^2=f(x-z)$ is a minimal Weierstrass equation.
\end{proof}

\begin{proof}[Proof of Theorem \ref{thmmindisc}]
The first part of the theorem follows from Lemmas \ref{comparedisc} and \ref{le:discriminant}. From now on, suppose $C/K$ is semistable. 

First suppose that $\Sigma$ is exceptional. A substitution of the form $x'=\pi_K^a$, $y'=\pi_K^b$ gives a curve whose cluster picture is obtained from $\Sigma$ by increasing the depth of all clusters so that $\cR$ gets depth 0, and whose valuation of the leading term of the defining polynomial is either $0$ or $1$. By Theorem \ref{th:mainmini}(i) this is the cluster picture of a minimal Weierstrass equation of $C$. The result follows from Lemma~\ref{le:discchange}(i,iii).

Now suppose that $\Sigma$ is not exceptional. Let $(\Sigma^b,\cR^b,d^b)$ be the equivalent balanced cluster picture (in the sense of Lemma \ref{balanced}).
Let $C'$ be the Weierstrass model given by Theorem \ref{Ombalanseringsatz}. By Theorem \ref{th:mainmini}, this is a minimal Weierstrass model. Its cluster picture is either already balanced, or is obtained from the balanced one by redistributing the depth of a child of $\cR$ of size $g\!+\!1$. Thus, by Lemmas \ref{comparedisc} and \ref{le:discchange}(iv),
$$
  v(\Delta_C^{\min}) = v(\Delta_{C'}) = v(\Delta_{\Sigma,0}) = v(\Delta_{\Sigma^b,0}).
$$
The balanced cluster picture $(\Sigma^b,0)$ is also obtained from $(\Sigma,v(c_f))$ itself by
(i) adding a root if necessary to make $\cR$ have size $2g+2$, 
(ii) increasing the depth of all clusters so that $\cR$ has depth 0,
(iii) changing the valuation of the leading term to 0, and (iv) repeatedly redistributing the depth of a child of $\cR$ of size $\ge g\!+\!1$ to eliminate all clusters of size $>g\!+\!1$ and make the ones of size $g+1$ have equal depths. Thus by Lemma \ref{le:discchange}, 
$$
v(\Delta_C^{\min}) = v(\Delta_{\Sigma^b,0}) = 
  v(\Delta_C) +\choice{0}{\text{if } |\cR| \text{ even}}{2(2g+1) d_{\cR}}{\text{if } |\cR| \text{ odd}}-
$$
$$
(2g+2)(2g+1)d_\cR-2v(c)(2g+1)-2(2g+1)\sum_{g+1<|\s|<|\cR|}\delta_\s   (|\s|-g-1),
$$
which simplifies to give the required expression.
\end{proof}

\def\cdepthscale{0.5}            

\section{Reduction types and a classification in genus 2}\label{se:bible}

In this section, we propose a notion of the reduction type of a semistable hyperelliptic curve. For elliptic curves these 
types would correspond to good, split multiplicative of type $\I_n$ or non-split multiplicative of type $\I_n$. 
Our main input is the cluster picture of the curve which determines the combinatorics of its special fibre as well as several other invariants. Since our residue field is not algebraically closed, we need to keep track of the action of Frobenius on proper clusters and the sign $\epsilon_\s(\Frob)$ for even clusters. For elliptic curves with multiplicative reduction this sign will distinguish between split and non-split cases. 


\begin{definition}\label{de:reductiontypeGeneral}
By the \emph{reduction type} of a semistable curve of genus $\ge 2$ over a non-archimedean local field we mean (the isomorphism class of) the dual graph of the special fibre of its minimal regular model with Frobenius action and with a genus associated to every vertex. 

For hyperelliptic curves there is a notation for such types given in \cite[\S8]{hyble}.   
\end{definition}

%

\begin{lemma}\label{le:reductiontypeGeneral}
The reduction type determines the conductor exponent, whether the curve is deficient and the Tamagawa number and root number of its Jacobian. 
\end{lemma}

\begin{proof}
By \eqref{tatedec}, Theorem \ref{th:rootnumber}, Definition \ref{de:deficiency} and Lemma \ref{le:deficiency} and Lemma \ref{lemtam}, the dual graph and its homology  determine these invariants. 
\end{proof}

\begin{example} Table \ref{tb:reductiontypegenus2} lists all reduction types of semistable genus 2 curves together with their labels. 
Here the numbers inside the vertices indicate their genus (no number meaning genus 0). We draw an edge of length $n$ to indicate a chain of $n$ edges between $n-1$ genus 0 vertices. The black arrows represent the action of Frobenius on the graph.

\def\tgrNA{\hetype{2}}
\def\tgrGA{\raise-3pt\hbox{\begin{tikzpicture}[scale=\GraphScale]
  \GraphVertices
  \Vertex[x=0.000,y=0.000,L=2]{1};
  \BlueEdges
\end{tikzpicture}}}
\def\tgrNB{\hetype{1_n^+}}
\def\tgrGB{\raise-7pt\hbox{\begin{tikzpicture}[scale=\GraphScale]
  \GraphVertices
  \Vertex[x=1.50,y=0.000,L=1]{1};
  \coordinate (2) at (0.000,0.000);
  \BlueEdges
  \LoopW(1)
\GraphEdgeSignS(1)(2){0.5}{n}\end{tikzpicture}}}
\def\tgrNC{\hetype{1_n^-}}
\def\tgrGC{\raise-7pt\hbox{\begin{tikzpicture}[scale=\GraphScale]
  \GraphVertices
  \Vertex[x=1.50,y=0.000,L=1]{1};
  \coordinate (2) at (0.000,0.000);
  \BlueEdges
  \LoopW(1)
\GraphEdgeSignS(1)(2){0.5}{n}\ESwapOfs1212{}{0,-0.25}{0,0.25}{0.5}\end{tikzpicture}}}
\def\tgrNBe{\hetype{1_n^{\epsilon}}}
\def\tgrND{\hetype{I_{n,m}^{+,+}}}
\def\tgrGD{\raise-7pt\hbox{\begin{tikzpicture}[scale=\GraphScale]
  \GraphVertices
  \Vertex[x=1.50,y=0.000,L=\relax]{1};
  \coordinate (2) at (3.00,0.000);
  \coordinate (3) at (0.000,0.000);
  \BlueEdges
  \LoopE(1)
  \LoopW(1)
\GraphEdgeSignS(1)(3){0.5}{n}\GraphEdgeSignS(1)(2){0.5}{m}\end{tikzpicture}}}
\def\tgrNE{\hetype{I_{n,m}^{+,-}}}
\def\tgrGE{\raise-7pt\hbox{\begin{tikzpicture}[scale=\GraphScale]
  \GraphVertices
  \Vertex[x=1.50,y=0.000,L=\relax]{1};
  \coordinate (2) at (3.00,0.000);
  \coordinate (3) at (0.000,0.000);
  \BlueEdges
  \LoopE(1)
  \LoopW(1)
\GraphEdgeSignS(1)(3){0.5}{n}\GraphEdgeSignS(1)(2){0.5}{m}\ESwapOfs1212{}{0,-0.25}{0,0.25}{0.5}\end{tikzpicture}}}
\def\tgrNF{\hetype{I_{n,m}^{-,-}}}
\def\tgrGF{\raise-7pt\hbox{\begin{tikzpicture}[scale=\GraphScale]
  \GraphVertices
  \Vertex[x=1.50,y=0.000,L=\relax]{1};
  \coordinate (2) at (3.00,0.000);
  \coordinate (3) at (0.000,0.000);
  \BlueEdges
  \LoopE(1)
  \LoopW(1)
\GraphEdgeSignS(1)(3){0.5}{n}\GraphEdgeSignS(1)(2){0.5}{m}\ESwapOfs1313{}{0,-0.25}{0,0.25}{0.5}\ESwapOfs1212{}{0,-0.25}{0,0.25}{0.5}\end{tikzpicture}}}
\def\tgrNDed{\hetype{I_{n,m}^{\epsilon,\delta}}}
\def\tgrNG{\hetype{I_{n\FrobL n}^+}}
\def\tgrGG{\raise-7pt\hbox{\begin{tikzpicture}[scale=\GraphScale]
  \GraphVertices
  \Vertex[x=1.50,y=0.000,L=\relax]{1};
  \coordinate (2) at (3.00,0.000);
  \coordinate (3) at (0.000,0.000);
  \BlueEdges
  \LoopE(1)
  \LoopW(1)
\GraphEdgeSignS(1)(3){0.5}{n}\GraphEdgeSignS(1)(2){0.5}{n}\ESwapOfs1312{in=160,out=20}{0.2,0.3}{-0.2,0.3}{0.5}\end{tikzpicture}}}
\def\tgrNH{\hetype{I_{n\FrobL n}^-}}
\def\tgrGH{\raise-7pt\hbox{\begin{tikzpicture}[scale=\GraphScale]
  \GraphVertices
  \Vertex[x=1.50,y=0.000,L=\relax]{1};
  \coordinate (2) at (3.00,0.000);
  \coordinate (3) at (0.000,0.000);
  \BlueEdges
  \LoopE(1)
  \LoopW(1)
\GraphEdgeSignS(1)(3){0.5}{n}\GraphEdgeSignS(1)(2){0.5}{n}\EArrOfs1312{in=150,out=30}{0.1,0.29}{0,0.35}{0.5}\EArrOfs1213{in=-60,out=-60}{0,0.2}{0.3,-0.25}{0.5}\end{tikzpicture}}}
\def\tgrNGe{\hetype{I_{n\FrobL n}^{\epsilon}}}
\def\tgrNI{\hetype{U_{n,m,k}^+}}
\def\tgrGI{\raise-7pt\hbox{\begin{tikzpicture}[scale=\GraphScale]
  \GraphVertices
  \coordinate (2) at (3.00,0.000);
  \coordinate (3) at (0.000,0.000);
  \coordinate (4) at (1.50,1.00);
  \Vertex[x=1.75,y=0.250,L=\relax]{1+}
  \Vertex[x=1.25,y=-0.250,L=\relax]{1-}
  \BlueEdges
  \EdgeE(1)
  \EdgeW(1)
  \EdgeS(1)
\GraphEdgeSignS(1)(2){0.6}{n}\GraphEdgeSignS(1)(3){0.6}{m}\GraphEdgeSignW(1)(4){0.6}{k}\end{tikzpicture}}}
\def\tgrNJ{\hetype{U_{n,m,k}^-}}
\def\tgrGJ{\raise-7pt\hbox{\begin{tikzpicture}[scale=\GraphScale]
  \GraphVertices
  \coordinate (2) at (3.00,0.000);
  \coordinate (3) at (0.000,0.000);
  \coordinate (4) at (1.50,1.00);
  \Vertex[x=1.75,y=0.250,L=\relax]{1+}
  \Vertex[x=1.25,y=-0.250,L=\relax]{1-}
  \BlueEdges
  \EdgeE(1)
  \EdgeW(1)
  \EdgeS(1)
\GraphEdgeSignS(1)(2){0.6}{n}\GraphEdgeSignS(1)(3){0.6}{m}\GraphEdgeSignW(1)(4){0.6}{k}\ESwapOfs1111{}{-0.18,-0.18}{0.18,0.18}{0.5}\end{tikzpicture}}}
\def\tgrNIe{\hetype{U_{n,m,k}^{\epsilon}}}
\def\tgrNK{\hetype{U_{n\FrobL n,k}^+}}
\def\tgrGK{\raise-7pt\hbox{\begin{tikzpicture}[scale=\GraphScale]
  \GraphVertices
  \coordinate (2) at (3.00,0.000);
  \coordinate (3) at (0.000,0.000);
  \coordinate (4) at (1.50,1.00);
  \Vertex[x=1.75,y=0.250,L=\relax]{1+}
  \Vertex[x=1.25,y=-0.250,L=\relax]{1-}
  \BlueEdges
  \EdgeE(1)
  \EdgeW(1)
  \EdgeS(1)
\GraphEdgeSignS(1)(2){0.7}{n}\GraphEdgeSignS(1)(3){0.7}{n}\GraphEdgeSignW(1)(4){0.6}{k}\ESwapOfs1312{in=-120,out=-60}{0,-0.3}{0,-0.3}{0.5}\end{tikzpicture}}}
\def\tgrNL{\hetype{U_{n\FrobL n,k}^-}}
\def\tgrGL{\raise-7pt\hbox{\begin{tikzpicture}[scale=\GraphScale]
  \GraphVertices
  \coordinate (2) at (3.00,0.000);
  \coordinate (3) at (0.000,0.000);
  \coordinate (4) at (1.50,1.00);
  \Vertex[x=1.75,y=0.250,L=\relax]{1+}
  \Vertex[x=1.25,y=-0.250,L=\relax]{1-}
  \BlueEdges
  \EdgeE(1)
  \EdgeW(1)
  \EdgeS(1)
\GraphEdgeSignS(1)(2){0.7}{n}\GraphEdgeSignS(1)(3){0.7}{n}\GraphEdgeSignW(1)(4){0.6}{k}\ESwapOfs1111{}{-0.18,-0.18}{0.18,0.18}{0.5}\ESwapOfs1312{in=-120,out=-60}{0,-0.3}{0,-0.3}{0.5}\end{tikzpicture}}}
\def\tgrNKe{\hetype{U_{n\FrobL n,k}^{\epsilon}}}
\def\tgrNM{\hetype{U_{n\FrobL n\FrobL n}^+}}
\def\tgrGM{\raise-7pt\hbox{\begin{tikzpicture}[scale=\GraphScale]
  \GraphVertices
  \coordinate (2) at (3.00,0.000);
  \coordinate (3) at (0.000,0.000);
  \coordinate (4) at (1.50,1.00);
  \Vertex[x=1.75,y=0.250,L=\relax]{1+}
  \Vertex[x=1.25,y=-0.250,L=\relax]{1-}
  \BlueEdges
  \EdgeE(1)
  \EdgeW(1)
  \EdgeS(1)
\GraphEdgeSignS(1)(2){0.7}{n}\GraphEdgeSignS(1)(3){0.7}{n}\GraphEdgeSignW(1)(4){0.8}{n}\EArrOfs1314{in=180,out=90}{-0.3,0.1}{-0.3,0}{0.5}\EArrOfs1412{in=90,out=0}{0.3,0.3}{0,0.3}{0.5}\EArrOfs1213{in=-60,out=-120}{0,-0.3}{0,-0.3}{0.5}\end{tikzpicture}}}
\def\tgrNN{\hetype{U_{n\FrobL n\FrobL n}^-}}
\def\tgrGN{\raise-7pt\hbox{\begin{tikzpicture}[scale=\GraphScale]
  \GraphVertices
  \coordinate (2) at (3.00,0.000);
  \coordinate (3) at (0.000,0.000);
  \coordinate (4) at (1.50,1.00);
  \Vertex[x=1.75,y=0.250,L=\relax]{1+}
  \Vertex[x=1.25,y=-0.250,L=\relax]{1-}
  \BlueEdges
  \EdgeE(1)
  \EdgeW(1)
  \EdgeS(1)
\GraphEdgeSignS(1)(2){0.7}{n}\GraphEdgeSignS(1)(3){0.7}{n}\GraphEdgeSignW(1)(4){0.8}{n}\ESwapOfs1111{}{-0.18,-0.18}{0.18,0.18}{0.5}\EArrOfs1314{in=180,out=90}{-0.3,0.1}{-0.3,0}{0.5}\EArrOfs1412{in=90,out=0}{0.3,0.3}{0,0.3}{0.5}\EArrOfs1213{in=-60,out=-120}{0,-0.3}{0,-0.3}{0.5}\end{tikzpicture}}}
\def\tgrNMe{\hetype{U_{n\FrobL n\FrobL n}^{\epsilon}}}
\def\tgrNO{\hetype{I_n^+x_{\protect\scalebox{0.6}{$\!\!\scriptstyle r\>$}}I_m^+}}
\def\tgrGO{\raise-7pt\hbox{\begin{tikzpicture}[scale=\GraphScale]
  \GraphVertices
  \Vertex[x=3.00,y=0.000,L=\relax]{1};
  \Vertex[x=1.50,y=0.000,L=\relax]{2};
  \coordinate (3) at (4.50,0.000);
  \coordinate (4) at (0.000,0.000);
  \BlueEdges
  \Edge(1)(2)
  \LoopE(1)
  \LoopW(2)
\GraphEdgeSignS(2)(4){0.5}{n}\GraphEdgeSignS(1)(3){0.5}{m}\GraphEdgeSignS($(1)+(0,0.3)$)(2){0.5}{r}\end{tikzpicture}}}
\def\tgrNP{\hetype{I_n^+x_{\protect\scalebox{0.6}{$\!\!\scriptstyle r\>$}}I_m^-}}
\def\tgrGP{\raise-7pt\hbox{\begin{tikzpicture}[scale=\GraphScale]
  \GraphVertices
  \Vertex[x=3.00,y=0.000,L=\relax]{1};
  \Vertex[x=1.50,y=0.000,L=\relax]{2};
  \coordinate (3) at (4.50,0.000);
  \coordinate (4) at (0.000,0.000);
  \BlueEdges
  \Edge(1)(2)
  \LoopE(1)
  \LoopW(2)
\GraphEdgeSignS(2)(4){0.5}{n}\GraphEdgeSignS(1)(3){0.5}{m}\ESwapOfs1313{}{0,-0.25}{0,0.25}{0.5}\GraphEdgeSignS($(1)+(0,0.3)$)(2){0.5}{r}\end{tikzpicture}}}
\def\tgrNQ{\hetype{I_n^-x_{\protect\scalebox{0.6}{$\!\!\scriptstyle r\>$}}I_m^-}}
\def\tgrGQ{\raise-7pt\hbox{\begin{tikzpicture}[scale=\GraphScale]
  \GraphVertices
  \Vertex[x=3.00,y=0.000,L=\relax]{1};
  \Vertex[x=1.50,y=0.000,L=\relax]{2};
  \coordinate (3) at (4.50,0.000);
  \coordinate (4) at (0.000,0.000);
  \BlueEdges
  \Edge(1)(2)
  \LoopE(1)
  \LoopW(2)
\GraphEdgeSignS(2)(4){0.5}{n}\GraphEdgeSignS(1)(3){0.5}{m}\ESwapOfs2424{}{0,-0.25}{0,0.25}{0.5}\ESwapOfs1313{}{0,-0.25}{0,0.25}{0.5}\GraphEdgeSignS($(1)+(0,0.3)$)(2){0.5}{r}\end{tikzpicture}}}
\def\tgrNOed{\hetype{I_n^{\epsilon}x_{\protect\scalebox{0.6}{$\!\!\scriptstyle r\>$}}I_m^{\delta}}}
\def\tgrNR{\hetype{I_n^+\FrobX_{\protect\scalebox{0.6}{$\!\scriptstyle r\>$}} I_n}}
\def\tgrGR{\raise-7pt\hbox{\begin{tikzpicture}[scale=\GraphScale]
  \GraphVertices
  \Vertex[x=3.00,y=0.000,L=\relax]{1};
  \Vertex[x=1.50,y=0.000,L=\relax]{2};
  \coordinate (3) at (4.50,0.000);
  \coordinate (4) at (0.000,0.000);
  \BlueEdges
  \Edge(1)(2)
  \LoopE(1)
  \LoopW(2)
\GraphEdgeSignS(2)(4){0.5}{n}\GraphEdgeSignS(1)(3){0.5}{n}\ESwapOfs2413{in=160,out=20}{0.2,0.3}{-0.2,0.3}{0.5}\GraphEdgeSignS(1)(2){0.5}{r}\end{tikzpicture}}}
\def\tgrNS{\hetype{I_n^-\FrobX_{\protect\scalebox{0.6}{$\!\scriptstyle r\>$}} I_n}}
\def\tgrGS{\raise-7pt\hbox{\begin{tikzpicture}[scale=\GraphScale]
  \GraphVertices
  \Vertex[x=3.00,y=0.000,L=\relax]{1};
  \Vertex[x=1.50,y=0.000,L=\relax]{2};
  \coordinate (3) at (4.50,0.000);
  \coordinate (4) at (0.000,0.000);
  \BlueEdges
  \Edge(1)(2)
  \LoopE(1)
  \LoopW(2)
\GraphEdgeSignS(2)(4){0.5}{n}\GraphEdgeSignS(1)(3){0.5}{n}\EArrOfs2413{in=150,out=30}{0.1,0.29}{0,0.35}{0.5}\EArrOfs1324{in=-60,out=-60}{0,0.2}{0.3,-0.25}{0.5}\GraphEdgeSignS(1)(2){0.5}{r}\end{tikzpicture}}}
\def\tgrNRe{\hetype{I_n^{\epsilon}\FrobX_{\protect\scalebox{0.6}{$\!\scriptstyle r\>$}} I_n}}
\def\tgrNT{\hetype{1x_{\protect\scalebox{0.6}{$\!\!\scriptstyle r\>$}}I_n^+}}
\def\tgrGT{\raise-7pt\hbox{\begin{tikzpicture}[scale=\GraphScale]
  \GraphVertices
  \Vertex[x=3.00,y=0.000,L=1]{1};
  \Vertex[x=1.50,y=0.000,L=\relax]{2};
  \coordinate (3) at (0.000,0.000);
  \BlueEdges
  \Edge(1)(2)
  \LoopW(2)
\GraphEdgeSignS(2)(3){0.5}{n}\GraphEdgeSignS($(1)+(0,0.3)$)(2){0.5}{r}\end{tikzpicture}}}
\def\tgrNU{\hetype{1x_{\protect\scalebox{0.6}{$\!\!\scriptstyle r\>$}}I_n^-}}
\def\tgrGU{\raise-7pt\hbox{\begin{tikzpicture}[scale=\GraphScale]
  \GraphVertices
  \Vertex[x=3.00,y=0.000,L=1]{1};
  \Vertex[x=1.50,y=0.000,L=\relax]{2};
  \coordinate (3) at (0.000,0.000);
  \BlueEdges
  \Edge(1)(2)
  \LoopW(2)
\GraphEdgeSignS(2)(3){0.5}{n}\ESwapOfs2323{}{0,-0.25}{0,0.25}{0.5}\GraphEdgeSignS($(1)+(0,0.3)$)(2){0.5}{r}\end{tikzpicture}}}
\def\tgrNTe{\hetype{1x_{\protect\scalebox{0.6}{$\!\!\scriptstyle r\>$}}I_n^{\epsilon}}}
\def\tgrGV{\raise-7pt\hbox{\begin{tikzpicture}[scale=\GraphScale]
  \GraphVertices
  \Vertex[x=1.50,y=0.000,L=1]{1};
  \Vertex[x=0.000,y=0.000,L=1]{2};
  \BlueEdges
  \Edge(1)(2)
\GraphEdgeSignS($(1)+(0,0.3)$)(2){0.5}{r}\end{tikzpicture}}}
\def\tgrNV{\hetype{1x_{\protect\scalebox{0.6}{$\!\!\scriptstyle r\>$}}1}}
\def\tgrGV{\raise-7pt\hbox{\begin{tikzpicture}[scale=\GraphScale]
  \GraphVertices
  \Vertex[x=1.50,y=0.000,L=1]{1};
  \Vertex[x=0.000,y=0.000,L=1]{2};
  \BlueEdges
  \Edge(1)(2)
\GraphEdgeSignS($(1)+(0,0.3)$)(2){0.5}{r}\end{tikzpicture}}}
\def\tgrNW{\hetype{1\FrobX_{\protect\scalebox{0.6}{$\!\scriptstyle r\>$}} 1}}
\def\tgrGW{\raise-7pt\hbox{\begin{tikzpicture}[scale=\GraphScale]
  \GraphVertices
  \Vertex[x=1.50,y=0.000,L=1]{1};
  \Vertex[x=0.000,y=0.000,L=1]{2};
  \BlueEdges
  \Edge(1)(2)
\GraphEdgeSignS(1)(2){0.5}{r}\VSwap{1}{2}{in=30,out=150}{2}\end{tikzpicture}}}


\begin{table}[H]
\begingroup
$$
\begin{array}{|c|c|c|c|c|c|}
\hline
\tgrGA&\tgrGB&\tgrGD&\tgrGH&\tgrGL&\tgrGP\\
\tgrNA&\tgrNB&\tgrND&\tgrNH&\tgrNL&\tgrNP\\
\hline
\tgrGV&\tgrGC&\tgrGE&\tgrGI&\tgrGM&\tgrGQ\\
\tgrNV&\tgrNC&\tgrNE&\tgrNI&\tgrNM&\tgrNQ\\
\hline
\tgrGW&\tgrGT&\tgrGF&\tgrGJ&\tgrGN&\tgrGR\\
\tgrNW&\tgrNT&\tgrNF&\tgrNJ&\tgrNN&\tgrNR\\
\hline
&\tgrGU&\tgrGG&\tgrGK&\tgrGO&\tgrGS\\
&\tgrNU&\tgrNG&\tgrNK&\tgrNO&\tgrNS\\
\hline
\end{array}
$$
\endgroup
\caption{Reduction types of semistable genus 2 curves}
\label{tb:reductiontypegenus2}
\end{table}
\end{example}

\begin{theorem}\label{th:classification}
The cluster picture of a semistable hyperelliptic curve over $K$ together with the action of Frobenius on clusters and $\epsilon_\s(\Frob)$ for even clusters determine the reduction type of the curve.
\end{theorem}
\begin{proof}
This is clear from the definition of reduction type and Theorem \ref{th:DualGraph}. 
\end{proof}

It follows that one can classify all reduction types of semistable curves of a given genus via their cluster pictures with this extra data. Note that different cluster pictures can give the same reduction type.

\begin{notation}\label{no:adddata}
Given the cluster picture of a semistable hyperelliptic curve, we write the relative depth on all proper clusters (except for $\cR$ which is decorated with its depth) at the bottom right corner of the cluster. For every even cluster $\s$ such that $\s= \s^*$ we write a sign $+$ or $-$ on its top right corner to indicate $\epsilon_\s(\Frob)$. For every cluster, we link its children that are in the same Frobenius orbit by lines.

Note that the definition of $\epsilon_\s$ (see Definition \ref{de:epsilon}) depends on a choice of sign for $\theta_\s$. A different choice of sign will change the sign parameter on $\s$ and on $\Frob(\s)$ if these are different clusters.
%
%
%
\end{notation}

\begin{example}
Suppose $p \equiv 7 \mod 8$. The curve
$$
y^2=(x+1)(x-1)(x-(i+p))(x-(i-p))(x-(-i+p))(x-(-i-p))
$$ has the following cluster picture
$
\clusterpicture            
  \Root {1} {first} {r1};
  \Root {} {r1} {r2};
  \Root {1} {r2} {r3};
  \Root {} {r3} {r4};
  \ClusterLD c1[{}][{1}] = (r3)(r4);
  \Root {1} {c1} {r5};
  \Root {} {r5} {r6};
  \ClusterLD c2[{}][{1}] = (r5)(r6);
  \ClusterD c3[0] = (r1)(r2)(c1)(c2);
\endclusterpicture
$, 
where $i$ is a square root of -1 in $\overline{\Q}_p$. The two twins are $\t_1 = \{i+p, i-p\}$ and $\t_2 = \{-i+p, -i-p\}$ and hence are swapped by Frobenius. Take $\theta_{\t_1} =\theta_{\t_2} = 2\sqrt{2}$. Since  $p \equiv 7 \mod 8$, $\epsilon_{\t_1}(\Frob) = \epsilon_{\t_2}(\Frob) = +1$. We draw this data as 
$
\clusterpicture            
  \Root {1} {first} {r1};
  \Root {} {r1} {r2};
  \Root {1} {r2} {r3};
  \Root {} {r3} {r4};
  \ClusterLD c1[{+}][{1}] = (r3)(r4);
  \Root {1} {c1} {r5};
  \Root {} {r5} {r6};
  \ClusterLD c2[{+}][{1}] = (r5)(r6);
  \ClusterD c3[0] = (r1)(r2)(c1)(c2);
  \frob(c1t)(c2);
\endclusterpicture
$.
Note that if we had chosen $\theta_{\t_2} = -2\sqrt{2}$ we would have obtained $\epsilon_{\t_1}(\Frob) = \epsilon_{\t_2}(\Frob) = -1$ and the cluster picture 
$
\clusterpicture            
  \Root {1} {first} {r1};
  \Root {} {r1} {r2};
  \Root {1} {r2} {r3};
  \Root {} {r3} {r4};
  \ClusterLD c1[{-}][{1}] = (r3)(r4);
  \Root {1} {c1} {r5};
  \Root {} {r5} {r6};
  \ClusterLD c2[{-}][{1}] = (r5)(r6);
  \ClusterD c3[0] = (r1)(r2)(c1)(c2);
  \frob(c1t)(c2);
\endclusterpicture
$. 
This is the reason why we consider these two the same type. 

Finally note that if $p \equiv 3 \mod 8$ then $\Frob(\sqrt{2}) = -\sqrt{2}$ and the corresponding cluster pictures would be
$
\clusterpicture            
  \Root {1} {first} {r1};
  \Root {} {r1} {r2};
  \Root {1} {r2} {r3};
  \Root {} {r3} {r4};
  \ClusterLD c1[{-}][{1}] = (r3)(r4);
  \Root {1} {c1} {r5};
  \Root {} {r5} {r6};
  \ClusterLD c2[{+}][{1}] = (r5)(r6);
  \ClusterD c3[0] = (r1)(r2)(c1)(c2);
  \frob(c1t)(c2);
\endclusterpicture$ and $
\clusterpicture            
  \Root {1} {first} {r1};
  \Root {} {r1} {r2};
  \Root {1} {r2} {r3};
  \Root {} {r3} {r4};
  \ClusterLD c1[{+}][{1}] = (r3)(r4);
  \Root {1} {c1} {r5};
  \Root {} {r5} {r6};
  \ClusterLD c2[{-}][{1}] = (r5)(r6);
  \ClusterD c3[0] = (r1)(r2)(c1)(c2);
  \frob(c1t)(c2);
\endclusterpicture.
$
\end{example}

\cite[\S9]{hyble} explains how to list all the reduction types of semistable hyperelliptic curves of arbitrary genus $g$. Given a hyperelliptic curve, in order to find its reduction type in that list, we first construct its cluster picture together with the additional data as in Notation \ref{no:adddata} and either we use Theorem \ref{th:DualGraph} or use Table 4.20 of \cite{hyble} to identify (the core of its open) hyperelliptic graph with automorphism induced by Frobenius. Theorem \ref{main dual graph thingy} guarantees that the latter produces the correct reduction type. 

Finally, note that Theorem \ref{Ombalanseringsatz} shows that every semistable hyperelliptic curve has a $K$-rational model with a distinguished cluster picture. For instance in genus 2, curves with cluster pictures that have no clusters of size 4 or 5 cover all $K$-isomorphism classes (see Theorem \ref{th:unbal}.(2)). In general, hyperelliptic curves of genus $g$ with a cluster picture that has no cluster $\s$ of size $g+1<|\s|<2g+2$ cover all $K$-isomorphism classes.

In the rest of this section we give a complete classification for semistable genus 2 curves.

%

\begin{theorem}\label{th:unbal}
Let C 
be a hyperelliptic curve over $K$ of genus 2 with cluster picture~$\Sigma_C$.

1) $C/K$ is semistable if and only if (a) $\Sigma_C$ is one of the pictures in Table \ref{tb:unbal} with $n,m,k,r,t \in \Z$
, (b) the thick black cluster $\s$ has depth $d_\s \in \Z$ and $\nu_\s \in 2\Z$, and (c) wild inertia does not permute any root. In this case its reduction type is the one given in Table \ref{tb:unbal}.

2) If $C$ is semistable and $|k|>5$ then there is an isomorphic curve $C'/K$ such that $\Sigma_{C'}$ is the top cluster picture in the second column of the same reduction type as $\Sigma_C$, with the same parameters $n,m,k,r,\epsilon, \delta$ and with $t = r$. 


3) The Namikawa-Ueno type of a semistable genus 2 curve is as indicated in Table \ref{tb:unbal}.
\end{theorem}

\def\ee{\epsilon}
\def\nn{\frac{n}{2}}
\def\mm{\frac{m}{2}}
\def\kk{\frac{k}{2}}
\def\dd{\delta}

\def\tgrNA{\hetype{2}}
\def\tgrNB{\hetype{1_n^+}}
\def\tgrNC{\hetype{1_n^-}}
\def\tgrNBe{\hetype{1_n^{\epsilon}}}
\def\tgrND{\hetype{I_{n,m}^{+,+}}}
\def\tgrNE{\hetype{I_{n,m}^{+,-}}}
\def\tgrNF{\hetype{I_{n,m}^{-,-}}}
\def\tgrNDed{\hetype{I_{n,m}^{\epsilon,\delta}}}
\def\tgrNG{\hetype{I_{n\FrobL n}^+}}
\def\tgrNH{\hetype{I_{n\FrobL n}^-}}
\def\tgrNGe{\hetype{I_{n\FrobL n}^{\epsilon}}}
\def\tgrNI{\hetype{U_{n,m,k}^+}}
\def\tgrNJ{\hetype{U_{n,m,k}^-}}
\def\tgrNIe{\hetype{U_{n,m,k}^{\epsilon}}}
\def\tgrNK{\hetype{U_{n\FrobL n,k}^+}}
\def\tgrNL{\hetype{U_{n\FrobL n,k}^-}}
\def\tgrNKe{\hetype{U_{n\FrobL n,k}^{\epsilon}}}
\def\tgrNM{\hetype{U_{n\FrobL n\FrobL n}^+}}
\def\tgrNN{\hetype{U_{n\FrobL n\FrobL n}^-}}
\def\tgrNMe{\hetype{U_{n\FrobL n\FrobL n}^{\epsilon}}}
\def\tgrNO{\hetype{I_n^+x_{\protect\scalebox{0.6}{$\!\!\scriptstyle \frac{r + s}{2}\>$}}I_m^+}}
\def\tgrNP{\hetype{I_n^+x_{\protect\scalebox{0.6}{$\!\!\scriptstyle \frac{r + s}{2}\>$}}I_m^-}}
\def\tgrNQ{\hetype{I_n^-x_{\protect\scalebox{0.6}{$\!\!\scriptstyle \frac{r + s}{2}\>$}}I_m^-}}
\def\tgrNOed{\hetype{I_n^{\epsilon}x_{\protect\scalebox{0.6}{$\!\!\scriptstyle r\>$}}I_m^{\delta}}}
\def\tgrNR{\hetype{I_n^+\FrobX_{\protect\scalebox{0.6}{$\!\scriptstyle r\>$}} I_n}}
\def\tgrNS{\hetype{I_n^-\FrobX_{\protect\scalebox{0.6}{$\!\scriptstyle r\>$}} I_n}}
\def\tgrNRe{\hetype{I_n^{\epsilon}\FrobX_{\protect\scalebox{0.6}{$\!\scriptstyle r\>$}} I_n}}
\def\tgrNT{\hetype{1x_{\protect\scalebox{0.6}{$\!\!\scriptstyle \frac{r + s}{2}\>$}}I_n^+}}
\def\tgrNU{\hetype{1x_{\protect\scalebox{0.6}{$\!\!\scriptstyle \frac{r + s}{2}\>$}}I_n^-}}
\def\tgrNTe{\hetype{1x_{\protect\scalebox{0.6}{$\!\!\scriptstyle r\>$}}I_n^{\epsilon}}}
\def\tgrNV{\hetype{1x_{\protect\scalebox{0.6}{$\!\!\scriptstyle r\>$}}1}}
\def\tgrNW{\hetype{1\FrobX_{\protect\scalebox{0.6}{$\!\scriptstyle r\>$}} 1}}

\def\dsh{\hbox{--}}

\def\cltopskip{5pt}              
\def\clbottomskip{5pt}           

\begin{table}
$$
\hskip-3mm
\scalebox{0.8}{$
\begin{array}{|c|c|c@{\>\>\>}c@{\>\>\>}c@{\>\>\>}c@{\>\>\>}c|}
\hline
\text{N-U}&\text{Type} & |\cR|=5 & \s\<\cR  \text{ sizes}  & \s\<\cR \text{ sizes} & \s\<\cR \text{ sizes}  & \s\<\cR \text{ sizes}    \\
\text{Type}&\text{} & & \le 3 &4, 1, 1 &4, 2 & 5, 1  \\
\hline
 \hetype{I_{0-0-0}} &\tgrNA&\clgAc &  \clgBalc &&& \clgCc  \\
\hline
&\tgrNV & \clooAc & \clooBalc   &&&\clooCc  \\
 \hetype{I_0\dsh I_0\dsh r}&&&\clooDc&&&\\
 \cline{2-7}
&\tgrNW&&\cloofc  & & &   \\
\hline
 \hetype{I_{n-0-0}} &\tgrNBe &\clnAce  & \clnBalce &  \clnCce &  \clnDce &\clnEce\\
 && \clnFce &   & & & \clnGce   \\
 \hline
&\tgrNTe & \cloInAce & \cloInBalce  &\cloInCce&\cloInDce& \cloInEce \\
  \hetype{I_n\dsh I_0\dsh r}&& \cloInFce & \cloInGce && &  \cloInHce \\
 && \cloInIce  &\cloInJce & && \cloInKce\\
 \hline
 &\tgrNDed &  \clnmAce   & \clnmBalce & \clnmCce &\clnmDce &\clnmEce\\
  \hetype{I_{n-m-0}}&&\clnmFce & & &  &\clnmGce\\
 \cline{2-7}
&\tgrNGe & \clnnAce &\clnnBalce    &&&  \clnnCce  \\
 \hline
&\tgrNIe & \clUnmkAce & \clUnmkBalce & \clUnmkCce &\clUnmkDce & \clUnmkEce  \\
  \cline{2-7}
\hetype{I_{n-m-k}}&\tgrNKe & \clUnnkAce & \clUnnkBalce & \clUnnkCce &\clUnnkDce & \clUnnkEce  \\
 \cline{2-7}
&\tgrNMe & & \clUnnnBalce &  & &  \\
 \hline
&\tgrNOed & \clInImAce & \clInImBalce &\clInImCce &\clInImDce  &\clInImEce \\
\hetype{I_n\dsh I_m\dsh r}&&\clInImFce &\clInImGce &&& \clInImHce\\
  \cline{2-7}
&\tgrNRe & & \clInInBalce &  & &   \\
 \hline
 \end{array}$}
$$
\hbox{\footnotesize{Notation: 
here $\eta \in \{\pm1\}$ and $t \in \Z$ are arbitrary.}}
\caption{Cluster pictures for semistable genus 2 curves}
\label{tb:unbal}
 \end{table}
 
 \def\cltopskip{1pt}              
\def\clbottomskip{1pt}           

\begin{proof}
The table contains all possible cluster pictures for genus 2 curves, 
with all possible choices of a permutation of proper clusters and choice of signs on even clusters of the form $\s^*$.

1) The semistability claim follows from Theorem \ref{the semistability theorem} and Proposition \ref{th:equivss} with the thick black cluster as choice for $\s$. The reduction type follows from Theorem \ref{th:DualGraph}. 
 
2) 
By Corollary \ref{co:evengenus} there is a model for the curve whose cluster picture is balanced in the sense of Lemma \ref{balanced}. By Theorem \ref{th:isoequiv} their cluster pictures are equivalent and in particular their relative depths are related as in the table. Moreover, the special fibres have isomorphic dual graphs with the same action of Frobenius, which pins down the Frobenius action on clusters and the sign parameters as given in the table.

3) The Namikawa-Ueno type is determined by the dual graph of the special fibre of the minimal regular model.
\end{proof}


%
The arithmetic invariants of genus 2 semistable hyperelliptic curves depend on the reduction type as follows.

\begin{theorem}\label{th:g2bible}
Let $K$ be a local field of odd residue characteristic and $C/K$ a semistable hyperelliptic curve of genus 2. Then the reduction type of $C/K$ is one of the ones listed in Table  \ref{tb:genus2bible}. 

Any genus 2 curve $y^2= f(x)$ with one of the cluster pictures from the table is semistable of the corresponding reduction type. If $f(x) \in \cO_K[x]$ then this is a minimal Weierstrass model. 
Moreover if $|k| > 5$, then every semistable $C/K$ admits a model $y^2 = f(x)$, with $f(x) \in \cO_K[x]$ and one of the listed cluster pictures. 

The invariants of the curve and its Jacobian are as given in the table (the value of $v(\Delta_{min})$ is conditional on $|k| > 5$).
In the table 
$m$ is the number of components in the special fibre of the minimal regular model of $C$, 
$n$ is the conductor exponent, 
$w$ is the local root number, 
$c$ is the Tamagawa number of $Jac(C)$. 
Def is - or + depending on whether the curve is deficient or not; $(-)^r$ means deficient if and only $r$ is odd.
The column $H_1(\Upsilon_C, \Z)$ lists the isomorphism class of the lattice together with automorphism (induced by Frobenius) and pairing (induced by the length pairing on $\Upsilon_C$), in the notation of Theorem 1.2.2 in \cite{DB}.
\end{theorem}

\def\tgrNA{\hetype{2}}
\def\tgrGA{\raise-3pt\hbox{\begin{tikzpicture}[scale=\GraphScale]
  \GraphVertices
  \Vertex[x=0.000,y=0.000,L=2]{1};
  \BlueEdges
\end{tikzpicture}}}
\def\tgrNB{\hetype{1_n^+}}
\def\tgrGB{\raise-7pt\hbox{\begin{tikzpicture}[scale=\GraphScale]
  \GraphVertices
  \Vertex[x=1.50,y=0.000,L=1]{1};
  \coordinate (2) at (0.000,0.000);
  \BlueEdges
  \LoopW(1)
\GraphEdgeSignS(1)(2){0.5}{n}\end{tikzpicture}}}
\def\tgrNC{\hetype{1_n^-}}
\def\tgrGC{\raise-7pt\hbox{\begin{tikzpicture}[scale=\GraphScale]
  \GraphVertices
  \Vertex[x=1.50,y=0.000,L=1]{1};
  \coordinate (2) at (0.000,0.000);
  \BlueEdges
  \LoopW(1)
\GraphEdgeSignS(1)(2){0.5}{n}\ESwapOfs1212{}{0,-0.25}{0,0.25}{0.5}\end{tikzpicture}}}
\def\tgrNBe{\hetype{1_n^{\epsilon}}}
\def\tgrND{\hetype{I_{n,m}^{+,+}}}
\def\tgrGD{\raise-7pt\hbox{\begin{tikzpicture}[scale=\GraphScale]
  \GraphVertices
  \Vertex[x=1.50,y=0.000,L=\relax]{1};
  \coordinate (2) at (3.00,0.000);
  \coordinate (3) at (0.000,0.000);
  \BlueEdges
  \LoopE(1)
  \LoopW(1)
\GraphEdgeSignS(1)(3){0.5}{n}\GraphEdgeSignS(1)(2){0.5}{m}\end{tikzpicture}}}
\def\tgrNE{\hetype{I_{n,m}^{+,-}}}
\def\tgrGE{\raise-7pt\hbox{\begin{tikzpicture}[scale=\GraphScale]
  \GraphVertices
  \Vertex[x=1.50,y=0.000,L=\relax]{1};
  \coordinate (2) at (3.00,0.000);
  \coordinate (3) at (0.000,0.000);
  \BlueEdges
  \LoopE(1)
  \LoopW(1)
\GraphEdgeSignS(1)(3){0.5}{n}\GraphEdgeSignS(1)(2){0.5}{m}\ESwapOfs1212{}{0,-0.25}{0,0.25}{0.5}\end{tikzpicture}}}
\def\tgrNF{\hetype{I_{n,m}^{-,-}}}
\def\tgrGF{\raise-7pt\hbox{\begin{tikzpicture}[scale=\GraphScale]
  \GraphVertices
  \Vertex[x=1.50,y=0.000,L=\relax]{1};
  \coordinate (2) at (3.00,0.000);
  \coordinate (3) at (0.000,0.000);
  \BlueEdges
  \LoopE(1)
  \LoopW(1)
\GraphEdgeSignS(1)(3){0.5}{n}\GraphEdgeSignS(1)(2){0.5}{m}\ESwapOfs1313{}{0,-0.25}{0,0.25}{0.5}\ESwapOfs1212{}{0,-0.25}{0,0.25}{0.5}\end{tikzpicture}}}
\def\tgrNDed{\hetype{I_{n,m}^{\epsilon,\delta}}}
\def\tgrNG{\hetype{I_{n\FrobL n}^+}}
\def\tgrGG{\raise-7pt\hbox{\begin{tikzpicture}[scale=\GraphScale]
  \GraphVertices
  \Vertex[x=1.50,y=0.000,L=\relax]{1};
  \coordinate (2) at (3.00,0.000);
  \coordinate (3) at (0.000,0.000);
  \BlueEdges
  \LoopE(1)
  \LoopW(1)
\GraphEdgeSignS(1)(3){0.5}{n}\GraphEdgeSignS(1)(2){0.5}{n}\ESwapOfs1312{in=160,out=20}{0.2,0.3}{-0.2,0.3}{0.5}\end{tikzpicture}}}
\def\tgrNH{\hetype{I_{n\FrobL n}^-}}
\def\tgrGH{\raise-7pt\hbox{\begin{tikzpicture}[scale=\GraphScale]
  \GraphVertices
  \Vertex[x=1.50,y=0.000,L=\relax]{1};
  \coordinate (2) at (3.00,0.000);
  \coordinate (3) at (0.000,0.000);
  \BlueEdges
  \LoopE(1)
  \LoopW(1)
\GraphEdgeSignS(1)(3){0.5}{n}\GraphEdgeSignS(1)(2){0.5}{n}\EArrOfs1312{in=150,out=30}{0.1,0.29}{0,0.35}{0.5}\EArrOfs1213{in=-60,out=-60}{0,0.2}{0.3,-0.25}{0.5}\end{tikzpicture}}}
\def\tgrNGe{\hetype{I_{n\FrobL n}^{\epsilon}}}
\def\tgrNI{\hetype{U_{n,m,k}^+}}
\def\tgrGI{\raise-7pt\hbox{\begin{tikzpicture}[scale=\GraphScale]
  \GraphVertices
  \coordinate (2) at (3.00,0.000);
  \coordinate (3) at (0.000,0.000);
  \coordinate (4) at (1.50,1.00);
  \Vertex[x=1.75,y=0.250,L=\relax]{1+}
  \Vertex[x=1.25,y=-0.250,L=\relax]{1-}
  \BlueEdges
  \EdgeE(1)
  \EdgeW(1)
  \EdgeS(1)
\GraphEdgeSignS(1)(2){0.6}{n}\GraphEdgeSignS(1)(3){0.6}{m}\GraphEdgeSignW(1)(4){0.6}{k}\end{tikzpicture}}}
\def\tgrNJ{\hetype{U_{n,m,k}^-}}
\def\tgrGJ{\raise-7pt\hbox{\begin{tikzpicture}[scale=\GraphScale]
  \GraphVertices
  \coordinate (2) at (3.00,0.000);
  \coordinate (3) at (0.000,0.000);
  \coordinate (4) at (1.50,1.00);
  \Vertex[x=1.75,y=0.250,L=\relax]{1+}
  \Vertex[x=1.25,y=-0.250,L=\relax]{1-}
  \BlueEdges
  \EdgeE(1)
  \EdgeW(1)
  \EdgeS(1)
\GraphEdgeSignS(1)(2){0.6}{n}\GraphEdgeSignS(1)(3){0.6}{m}\GraphEdgeSignW(1)(4){0.6}{k}\ESwapOfs1111{}{-0.18,-0.18}{0.18,0.18}{0.5}\end{tikzpicture}}}
\def\tgrNIe{\hetype{U_{n,m,k}^{\epsilon}}}
\def\tgrNK{\hetype{U_{n\FrobL n,k}^+}}
\def\tgrGK{\raise-7pt\hbox{\begin{tikzpicture}[scale=\GraphScale]
  \GraphVertices
  \coordinate (2) at (3.00,0.000);
  \coordinate (3) at (0.000,0.000);
  \coordinate (4) at (1.50,1.00);
  \Vertex[x=1.75,y=0.250,L=\relax]{1+}
  \Vertex[x=1.25,y=-0.250,L=\relax]{1-}
  \BlueEdges
  \EdgeE(1)
  \EdgeW(1)
  \EdgeS(1)
\GraphEdgeSignS(1)(2){0.7}{n}\GraphEdgeSignS(1)(3){0.7}{n}\GraphEdgeSignW(1)(4){0.6}{k}\ESwapOfs1312{in=-120,out=-60}{0,-0.3}{0,-0.3}{0.5}\end{tikzpicture}}}
\def\tgrNL{\hetype{U_{n\FrobL n,k}^-}}
\def\tgrGL{\raise-7pt\hbox{\begin{tikzpicture}[scale=\GraphScale]
  \GraphVertices
  \coordinate (2) at (3.00,0.000);
  \coordinate (3) at (0.000,0.000);
  \coordinate (4) at (1.50,1.00);
  \Vertex[x=1.75,y=0.250,L=\relax]{1+}
  \Vertex[x=1.25,y=-0.250,L=\relax]{1-}
  \BlueEdges
  \EdgeE(1)
  \EdgeW(1)
  \EdgeS(1)
\GraphEdgeSignS(1)(2){0.7}{n}\GraphEdgeSignS(1)(3){0.7}{n}\GraphEdgeSignW(1)(4){0.6}{k}\ESwapOfs1111{}{-0.18,-0.18}{0.18,0.18}{0.5}\ESwapOfs1312{in=-120,out=-60}{0,-0.3}{0,-0.3}{0.5}\end{tikzpicture}}}
\def\tgrNKe{\hetype{U_{n\FrobL n,k}^{\epsilon}}}
\def\tgrNM{\hetype{U_{n\FrobL n\FrobL n}^+}}
\def\tgrGM{\raise-7pt\hbox{\begin{tikzpicture}[scale=\GraphScale]
  \GraphVertices
  \coordinate (2) at (3.00,0.000);
  \coordinate (3) at (0.000,0.000);
  \coordinate (4) at (1.50,1.00);
  \Vertex[x=1.75,y=0.250,L=\relax]{1+}
  \Vertex[x=1.25,y=-0.250,L=\relax]{1-}
  \BlueEdges
  \EdgeE(1)
  \EdgeW(1)
  \EdgeS(1)
\GraphEdgeSignS(1)(2){0.7}{n}\GraphEdgeSignS(1)(3){0.7}{n}\GraphEdgeSignW(1)(4){0.8}{n}\EArrOfs1314{in=180,out=90}{-0.3,0.1}{-0.3,0}{0.5}\EArrOfs1412{in=90,out=0}{0.3,0.3}{0,0.3}{0.5}\EArrOfs1213{in=-60,out=-120}{0,-0.3}{0,-0.3}{0.5}\end{tikzpicture}}}
\def\tgrNN{\hetype{U_{n\FrobL n\FrobL n}^-}}
\def\tgrGN{\raise-7pt\hbox{\begin{tikzpicture}[scale=\GraphScale]
  \GraphVertices
  \coordinate (2) at (3.00,0.000);
  \coordinate (3) at (0.000,0.000);
  \coordinate (4) at (1.50,1.00);
  \Vertex[x=1.75,y=0.250,L=\relax]{1+}
  \Vertex[x=1.25,y=-0.250,L=\relax]{1-}
  \BlueEdges
  \EdgeE(1)
  \EdgeW(1)
  \EdgeS(1)
\GraphEdgeSignS(1)(2){0.7}{n}\GraphEdgeSignS(1)(3){0.7}{n}\GraphEdgeSignW(1)(4){0.8}{n}\ESwapOfs1111{}{-0.18,-0.18}{0.18,0.18}{0.5}\EArrOfs1314{in=180,out=90}{-0.3,0.1}{-0.3,0}{0.5}\EArrOfs1412{in=90,out=0}{0.3,0.3}{0,0.3}{0.5}\EArrOfs1213{in=-60,out=-120}{0,-0.3}{0,-0.3}{0.5}\end{tikzpicture}}}
\def\tgrNMe{\hetype{U_{n\FrobL n\FrobL n}^{\epsilon}}}
\def\tgrNO{\hetype{I_n^+x_{\protect\scalebox{0.6}{$\!\!\scriptstyle r\>$}}I_m^+}}
\def\tgrGO{\raise-7pt\hbox{\begin{tikzpicture}[scale=\GraphScale]
  \GraphVertices
  \Vertex[x=3.00,y=0.000,L=\relax]{1};
  \Vertex[x=1.50,y=0.000,L=\relax]{2};
  \coordinate (3) at (4.50,0.000);
  \coordinate (4) at (0.000,0.000);
  \BlueEdges
  \Edge(1)(2)
  \LoopE(1)
  \LoopW(2)
\GraphEdgeSignS(2)(4){0.5}{n}\GraphEdgeSignS(1)(3){0.5}{m}\GraphEdgeSignS($(1)+(0,0.3)$)(2){0.5}{r}\end{tikzpicture}}}
\def\tgrNP{\hetype{I_n^+x_{\protect\scalebox{0.6}{$\!\!\scriptstyle r\>$}}I_m^-}}
\def\tgrGP{\raise-7pt\hbox{\begin{tikzpicture}[scale=\GraphScale]
  \GraphVertices
  \Vertex[x=3.00,y=0.000,L=\relax]{1};
  \Vertex[x=1.50,y=0.000,L=\relax]{2};
  \coordinate (3) at (4.50,0.000);
  \coordinate (4) at (0.000,0.000);
  \BlueEdges
  \Edge(1)(2)
  \LoopE(1)
  \LoopW(2)
\GraphEdgeSignS(2)(4){0.5}{n}\GraphEdgeSignS(1)(3){0.5}{m}\ESwapOfs1313{}{0,-0.25}{0,0.25}{0.5}\GraphEdgeSignS($(1)+(0,0.3)$)(2){0.5}{r}\end{tikzpicture}}}
\def\tgrNQ{\hetype{I_n^-x_{\protect\scalebox{0.6}{$\!\!\scriptstyle r\>$}}I_m^-}}
\def\tgrGQ{\raise-7pt\hbox{\begin{tikzpicture}[scale=\GraphScale]
  \GraphVertices
  \Vertex[x=3.00,y=0.000,L=\relax]{1};
  \Vertex[x=1.50,y=0.000,L=\relax]{2};
  \coordinate (3) at (4.50,0.000);
  \coordinate (4) at (0.000,0.000);
  \BlueEdges
  \Edge(1)(2)
  \LoopE(1)
  \LoopW(2)
\GraphEdgeSignS(2)(4){0.5}{n}\GraphEdgeSignS(1)(3){0.5}{m}\ESwapOfs2424{}{0,-0.25}{0,0.25}{0.5}\ESwapOfs1313{}{0,-0.25}{0,0.25}{0.5}\GraphEdgeSignS($(1)+(0,0.3)$)(2){0.5}{r}\end{tikzpicture}}}
\def\tgrNOed{\hetype{I_n^{\epsilon}x_{\protect\scalebox{0.6}{$\!\!\scriptstyle r\>$}}I_m^{\delta}}}
\def\tgrNR{\hetype{I_n^+\FrobX_{\protect\scalebox{0.6}{$\!\scriptstyle r\>$}} I_n}}
\def\tgrGR{\raise-7pt\hbox{\begin{tikzpicture}[scale=\GraphScale]
  \GraphVertices
  \Vertex[x=3.00,y=0.000,L=\relax]{1};
  \Vertex[x=1.50,y=0.000,L=\relax]{2};
  \coordinate (3) at (4.50,0.000);
  \coordinate (4) at (0.000,0.000);
  \BlueEdges
  \Edge(1)(2)
  \LoopE(1)
  \LoopW(2)
\GraphEdgeSignS(2)(4){0.5}{n}\GraphEdgeSignS(1)(3){0.5}{n}\ESwapOfs2413{in=160,out=20}{0.2,0.3}{-0.2,0.3}{0.5}\GraphEdgeSignS(1)(2){0.5}{r}\end{tikzpicture}}}
\def\tgrNS{\hetype{I_n^-\FrobX_{\protect\scalebox{0.6}{$\!\scriptstyle r\>$}} I_n}}
\def\tgrGS{\raise-7pt\hbox{\begin{tikzpicture}[scale=\GraphScale]
  \GraphVertices
  \Vertex[x=3.00,y=0.000,L=\relax]{1};
  \Vertex[x=1.50,y=0.000,L=\relax]{2};
  \coordinate (3) at (4.50,0.000);
  \coordinate (4) at (0.000,0.000);
  \BlueEdges
  \Edge(1)(2)
  \LoopE(1)
  \LoopW(2)
\GraphEdgeSignS(2)(4){0.5}{n}\GraphEdgeSignS(1)(3){0.5}{n}\EArrOfs2413{in=150,out=30}{0.1,0.29}{0,0.35}{0.5}\EArrOfs1324{in=-60,out=-60}{0,0.2}{0.3,-0.25}{0.5}\GraphEdgeSignS(1)(2){0.5}{r}\end{tikzpicture}}}
\def\tgrNRe{\hetype{I_n^{\epsilon}\FrobX_{\protect\scalebox{0.6}{$\!\scriptstyle r\>$}} I_n}}
\def\tgrNT{\hetype{1x_{\protect\scalebox{0.6}{$\!\!\scriptstyle r\>$}}I_n^+}}
\def\tgrGT{\raise-7pt\hbox{\begin{tikzpicture}[scale=\GraphScale]
  \GraphVertices
  \Vertex[x=3.00,y=0.000,L=1]{1};
  \Vertex[x=1.50,y=0.000,L=\relax]{2};
  \coordinate (3) at (0.000,0.000);
  \BlueEdges
  \Edge(1)(2)
  \LoopW(2)
\GraphEdgeSignS(2)(3){0.5}{n}\GraphEdgeSignS($(1)+(0,0.3)$)(2){0.5}{r}\end{tikzpicture}}}
\def\tgrNU{\hetype{1x_{\protect\scalebox{0.6}{$\!\!\scriptstyle r\>$}}I_n^-}}
\def\tgrGU{\raise-7pt\hbox{\begin{tikzpicture}[scale=\GraphScale]
  \GraphVertices
  \Vertex[x=3.00,y=0.000,L=1]{1};
  \Vertex[x=1.50,y=0.000,L=\relax]{2};
  \coordinate (3) at (0.000,0.000);
  \BlueEdges
  \Edge(1)(2)
  \LoopW(2)
\GraphEdgeSignS(2)(3){0.5}{n}\ESwapOfs2323{}{0,-0.25}{0,0.25}{0.5}\GraphEdgeSignS($(1)+(0,0.3)$)(2){0.5}{r}\end{tikzpicture}}}
\def\tgrNTe{\hetype{1x_{\protect\scalebox{0.6}{$\!\!\scriptstyle r\>$}}I_n^{\epsilon}}}
\def\tgrGV{\raise-7pt\hbox{\begin{tikzpicture}[scale=\GraphScale]
  \GraphVertices
  \Vertex[x=1.50,y=0.000,L=1]{1};
  \Vertex[x=0.000,y=0.000,L=1]{2};
  \BlueEdges
  \Edge(1)(2)
\GraphEdgeSignS($(1)+(0,0.3)$)(2){0.5}{r}\end{tikzpicture}}}
\def\tgrNV{\hetype{1x_{\protect\scalebox{0.6}{$\!\!\scriptstyle r\>$}}1}}
\def\tgrGV{\raise-7pt\hbox{\begin{tikzpicture}[scale=\GraphScale]
  \GraphVertices
  \Vertex[x=1.50,y=0.000,L=1]{1};
  \Vertex[x=0.000,y=0.000,L=1]{2};
  \BlueEdges
  \Edge(1)(2)
\GraphEdgeSignS($(1)+(0,0.3)$)(2){0.5}{r}\end{tikzpicture}}}
\def\tgrNW{\hetype{1\FrobX_{\protect\scalebox{0.6}{$\!\scriptstyle r\>$}} 1}}
\def\tgrGW{\raise-7pt\hbox{\begin{tikzpicture}[scale=\GraphScale]
  \GraphVertices
  \Vertex[x=1.50,y=0.000,L=1]{1};
  \Vertex[x=0.000,y=0.000,L=1]{2};
  \BlueEdges
  \Edge(1)(2)
\GraphEdgeSignS(1)(2){0.5}{r}\VSwap{1}{2}{in=30,out=150}{2}\end{tikzpicture}}}

\begin{table}
\def\cltopskip{3pt}              
\def\clbottomskip{3pt}           
\def\graphdslabelscale{1.0}      
\def\cdepthscale{0.6}            
$$
\hskip-3mm
\scalebox{0.7}{$
\begin{array}{| c|c r|c|c|c|c|c|c|c|c|}
\hline
\text{Type} & \Sigma_C &\text{\llap{$v(c_f)$}}& \Upsilon_C & m_C &H_1(\Upsilon_C, \Z) & n&  w & c & \text{Def}   & \tiny{v(\Delta_{min})}\\
\hline
\tgrNA&\clgBal	& 0 &\tgrGA& 1 & - & 0 & + & 1 & + &0\\
\hline
\tgrNV&\clooBal&  \bar{r} &\tgrGV& r+1 & - & 0 & + & 1 & + &12r\\
\tgrNW&\cloof&  \bar{r} &\tgrGW& r+1 & - & 0 & + & 1 & (-)^r &12r+10\bar{r}\\
\hline
\tgrNB&\clnBals& 0 &\tgrGB& n & [1:n] & 1 & - & n & + &n\\
\tgrNC&\clnBalns& 0 &\tgrGC& n & [2:n] & 1 & + & \widetilde{n} & +&n \\
\hline
\tgrNT&\cloInBals&  \bar{r} &\tgrGT& n+r & [1:n] & 1 & - & n & + & 12r+n \\
\tgrNU&\cloInBalns&  \bar{r} &\tgrGU& n+r & [2:n] & 1 & + & \widetilde{n} & +  & 12r+n\\
\hline
\tgrND&\clnmBalss& 0 &\tgrGD& n+m-1 & [1.1:n,m] & 2 & + & nm & + &n+m\\
\tgrNE&\clnmBalsns& 0 &\tgrGE& n+m-1 & [1.2_A:n,m] & 2 & - & n\widetilde{m} & + &n+m \\
\tgrNF&\clnmBalnsns& 0 &\tgrGF& n+m-1 & [2.2:n,m] & 2 & + & \widetilde{n}\widetilde{m} & +&n+m \\
\tgrNG&\clnnBals& 0 &\tgrGG& 2n-1 & [1.2_B:n,n] & 2 & - & n & +&2n \\
\tgrNH&\clnnBalns& 0 &\tgrGH& 2n-1 & [4:n] & 2 & + & \widetilde{n} & + &2n\\
\hline
\tgrNI&\clUnmkBals& 0 &\tgrGI& n\!+\!m\!+\!k\!-\!1 & [1.1:d,t/d] & 2 & + & t & + &n+m+k \\
\tgrNJ&\clUnmkBalns& 0 &\tgrGJ& n\!+\!m\!+\!k\!-\!1 & [2.2:d,t/d] & 2 & + & \widetilde{t/d}\!\cdot\!\widetilde{d} & (\!-\!)^{nmk} &n+m+k \\
\tgrNK&\clUnnkBals& 0 &\tgrGK& 2n+k-1 & [1.2_B\!:\! n\!+\!2k,n] & 2 & - & n \!+\! 2k & + &2n+k\\
\tgrNL&\clUnnkBalns& 0 &\tgrGL& 2n+k-1 & [1.2_B\!:\!n, n\!+\!2k] & 2 & - & n & (-)^k &2n+k\\
\tgrNM&\clUnnnBals& 0 &\tgrGM& 3n-1 & [3:n] & 2 & + & 3 & + &3n\\
\tgrNN&\clUnnnBalns& 0 &\tgrGN& 3n-1 & [6:n] & 2 & + & 1 & (-)^n & 3n \\
\hline
\tgrNO&\clInImBalss&  \bar{r} &\tgrGO& n\!+\!m\!+\!r\!-\!1 & [1.1:n,m] & 2 & + & nm & + &12r\!+\!n\!+\!m\\
\tgrNP&\clInImBalsns&  \bar{r} &\tgrGP& n\!+\!m\!+\!r\!-\!1 & [1.2_A:n,m] & 2 & - & n\widetilde{m} & +  &12r\!+\!n\!+\!m\\
\tgrNQ&\clInImBalnsns&  \bar{r} &\tgrGQ& n\!+\!m\!+\!r\!-\!1 & [2.2:n,m] & 2 & + & \widetilde{n}\widetilde{m}  & +&12r\!+\!n\!+\!m \\
\tgrNR&\clInInBals&  \bar{r} &\tgrGR& 2n+r-1 & [1.2_B:n,n] & 2 & - & n & (-)^r &12r\!+\!2n \!+\!10\bar{r}\\
\tgrNS&\clInInBalns&  \bar{r} &\tgrGS&  2n+r-1 & [4:n] & 2 & + & \widetilde{n} & (-)^r &12r\!+\!2n\!+\!10\bar{r}\\
\hline
\end{array}$}
$$
\hbox{Notation: 
$\bar{r} = 0$ if $2|r$ and $\bar{r} = 1$ if $2\nmid r$; $\tilde n=2$ if $2|n$ and $\tilde n=1$ if $2\nmid n$;}\\
\hbox to 25em{\hfill $d=\gcd(m,n,k)$; $t=nm+nk+mk$.}
\caption{Local invariants of semistable genus 2 curves}
\label{tb:genus2bible}
\end{table}

\begin{proof}
The completeness of the list of reduction types, and that curves with such cluster pictures are semistable and have these types follow from Theorem \ref{th:unbal}. The claim regarding minimality of the model follows from Theorem \ref{th:mainmini}. 
If $|k|>5$, Theorem \ref{th:unbal} shows that $C$ admits a model with the cluster picture corresponding to this type.
In this case,  the value for the valuation of the minimal discriminant follows from Theorem~\ref{thmmindisc}.


The number of components $m$ is clear from the dual graph $\Upsilon_C$. The conductor exponent $n$, the root number $w$ and deficiency Def follow from Theorems \ref{th:condexpo}, \ref{th:rootnumber} and \ref{th:deficiency}. The Tamagawa number is explicitly determined by the lattice type by Theorem 1.2.2 in \cite{DB}, using the fact that $H_1(\Upsilon_C,\Z)$ is isomorphic to the character lattice of the torus in the Raynaud parametrisation (Lemma \ref{le:isolattice}).

Finally, it remains to prove the claim for the lattice type of  $H_1(\Upsilon_C,\Z)$.
The dimension of the lattice and the eigenvalues of Frobenius come from the cluster picture (with the extra data) as given by Corollary \ref{co:EtaleCohoToric}. 

Suppose that $\cR$ is not \ub. Then the pairing with respect to the basis given in Theorem \ref{th:Homology} is diagonal with values given by twice the relative depth of the corresponding twins; the dual lattice and the lattice type follow except for the types $\I^+_{n_{\tilde{~}} n}$ and $\I_n^+\tilde{\times}_r\I_n$. For these two cases, let $\t_1, \t_2$ be the two twins and $\ell_{\t_1}, \ell_{\t_2}$ the corresponding loops which generate $H_1(\Upsilon_C,\Z)$. The Frobenius invariant/anti-invariant loops are generated by $\ell_{\t_1} + \ell_{\t_2}$ and  $\ell_{\t_1} - \ell_{\t_2}$ respectively. It follows that $H_1(\Upsilon_C,\Z)$ is not spanned by these and hence the type is $[1.2_B:*,*]$. However, over an unramified quadratic extension, it becomes  a $[1.1:n,n]$ and hence by Theorem 1.2.2 in \cite{DB} with $f=2$, the type is a $[1.2_B:n,n]$.

Suppose $\cR$ is \ub. 
For the case $U^+_{n_{\tilde{~}} n,k}$, let $\ell^+= \ell_{\t_1}+\ell_{\t_2}-2\ell_{\t_3}$, $\ell^- = \ell_{\t_1}-\ell_{\t_2}$ be $\Z$-generators for the space of invariant and anti-invariant loops respectively. Note that $\frac12(\ell^+ + \ell^-) \in H_1(\Upsilon_C,\Z)$ so that the lattice type is $[1.2_B:*,*]$. Now, $\langle a\ell^+, \ell^+\rangle = a(2n+4k)$ and $\langle a\ell^+, \ell_{\t_1}-\ell_{\t_3}\rangle = a(n+2k)$ so that $a\ell^+ \in H_1(\Upsilon_C,\Z)^{\vee}$ if and only if $a \in \frac{1}{n+2k}\Z$. Similarly,  $\langle a\ell^-, \ell^+\rangle =0$ and $\langle a\ell^-, \ell_{\t_1}-\ell_{\t_3}\rangle = an$ so that $a\ell^- \in H_1(\Upsilon_C,\Z)^{\vee}$ if and only if $a \in \frac{1}{n}\Z$. It follows that the lattice type is $[1.2_B:n+2k,n]$. The case $U^-_{n_{\tilde{~}} n,k}$ follows by swapping the roles of $\ell^+$ and $\ell^-$. 

For the case $U_{n,m,k}^{+}$, let $\t_1, \t_2, \t_3$ be the three twins and $\ell_{\t_1}, \ell_{\t_2}, \ell_{\t_3}$ be the corresponding half loops. Choose the loops $h_1 = \ell_{\t_1}- \ell_{\t_3}$ and $h_2 = \ell_{\t_2}- \ell_{\t_3}$ as a basis for $H_1(\Upsilon_C,\Z)$. Then for $a, b \in \Q$, $\ell = ah_1+bh_2 \in H_1(\Upsilon_C,\Z)^{\vee}$ if and only if $\langle \ell, h_1\rangle = a(n+k) + bk\in \Z$ and $\langle \ell, h_2 \rangle =b(m+k)+ak \in \Z$ i.e. 
$$
\begin{pmatrix}n+k & k\\ k&m+k \end{pmatrix}\begin{pmatrix} a\\b \end{pmatrix} \in \Z^2.
$$
By properties of Smith normal forms, there exists a $\Z$-basis $g_1, g_2$ of $H_1(\Upsilon_C,\Z)$ such that $ug_1+vg_2 \in H_1(\Upsilon_C,\Z)^{\vee}$ if and only if 
$$
\begin{pmatrix}gcd(n+k, m+k, k) & 0\\ 0& \frac{det(M)}{gcd(n+k, m+k, k)} \end{pmatrix}\begin{pmatrix} u\\v \end{pmatrix}  =\begin{pmatrix}d & 0\\ 0& \frac{t}{d} \end{pmatrix}\begin{pmatrix} u\\v \end{pmatrix} \in \Z^2,
$$
where $M = \begin{pmatrix}n+k & k\\ k&m+k \end{pmatrix}$.
It follows that the lattice type is $[1.1:d,t/d]$.

In the cases $U_{n,m,k}^-$, $U_{n_{\tilde{~}} n_{\tilde{~}} n}^+$ and $U_{n_{\tilde{~}} n_{\tilde{~}} n}^-$, the eigenvalues of Frobenius are $(-1,-1)$, $ (\zeta_3, \zeta_3^{-1})$ and $(\zeta_6, \zeta_6^{-1})$, respectively. The lattices become $[1.1:d,t/d]$ after an unramified extension of degree 2, 3 and 6, respectively. It follows from Theorem 1.2.2 in \cite{DB} using the $f=2,3$ columns that the original lattice types are $[2.2:d,t/d]$, $[3:n]$ and $[6:n]$, respectively.

\end{proof}

\section{Local constancy of arithmetic invariants}
\label{scontiunity}

Observe that all our main results on a hyperelliptic curve $C: y^2 = f(x)$ are obtained from fairly coarse data coming from the roots of $f(x)$.
It follows that small $p$-adic perturbations of the roots of $f(x)$ do not change the arithmetic invariants of $C$. 
Here is a precise formulation:

\begin{theorem}
\label{continuity}
Suppose $C_1/K: y^2=f_1(x)$, $C_2/K: y^2=f_2(x)$ are two hyperelliptic curves, such that
\begin{itemize}
\item[(a)]
The leading coefficients $c_{f_1}$ of $f_1$ and $c_{f_2}$ of $f_2$ satisfy 
$\frac{c_{f_1}}{c_{f_2}}\in K^{\times 2}$.
\item[(b)]
There is a Galois-equivariant bijection $\phi\colon$\{roots of $f_1$\}$\to$\{roots of $f_2$\}
such that $\frac{\phi(r_i)-\phi(r_j)}{r_i-r_j}\equiv 1$ mod $\m$ 
for all roots $r_i\ne r_j$ of $f_1$. 
\end{itemize}
Then
\begin{enumerate}
\item
$C_1$ and $C_2$ acquire semistable reduction over the same extensions of~$K$.
\item
If $C_1, C_2$ are semistable over a finite Galois extension $F$ of $K$, 
then the special fibres of their minimal regular models over $\cO_{\Fnr}$ are isomorphic
as curves with the semilinear action of $G_K$ given by \eqref{skew}.
\item
$\H(C_1) \iso \H(C_2)$ as $G_K$-modules, for every $l\ne p$.
\item 
If $C_1, C_2$ are semistable and $|k|>2g+1$, then the valuation of their minimal discriminants are equal.
\end{enumerate}
\end{theorem}

\begin{remark}\label{re:sameinvariants}
By (3) $C_1$ and $C_2$ share the same conductor exponent and root number. By (2) if $C_1$ and $C_2$ are semistable 
then they have the same reduction type (in the sense of Definition \ref{de:reductiontypeGeneral}). 
It follows that $C_1$ is deficient if and only if $C_2$ is, and that their Jacobians have the same Tamagawa number.
\end{remark}

\begin{proof}[Proof of Theorem]

First note that by (a) and after a change of variable, we may assume that $c_{f_1} = c_{f_2}$. 
Moreover by (b), $\phi$ induces an isomorphism of cluster pictures, preserving depths and 
the Galois action on the roots. 

(1) Follows from the semistability criterion (Theorem \ref{the semistability theorem}).

(2) There is a one-to-one correspondence between valid discs (cf. \S\ref{ydisc_section})
for $C_1$ and $C_2$ over $F$, defined as follows. By the semistability criterion, every proper 
cluster $\s$ for $C_1$ contains either an $I_F$-invariant root $r$
or a twin consisting of $I_F$-conjugate roots $r_1$, $r_2$. 
Let $z_\s=r$ in the first case, and $z_\s=(r_1+r_2)/2$ in the second case. Then every valid 
disc for $C_1$ has a centre $z_\s\in \Fnr$ of this type. It corresponds to a valid disc 
of $C_2$ that has centre $\phi(r)$, respectively, $(\phi(r_1)+\phi(r_2))/2$, and the same radius. 
This gives the one-to-one correspondence $D\leftrightarrow \phi(D)$.

Next, we claim that the reduction maps agree, that is $\red_D(r)=\red_{\phi(D)}(\phi(r))$ for every root 
$r$ and valid disc $D$ of $C_1$. As the radii are the same, this is equivalent to
$$
  \frac{r-z_D}{\phi(r)-z_{\phi(D)}} \equiv 1 \mod \m.
$$
This is clear from (b) if $z_D$ is a root. Suppose $z_D=(r_1+r_2)/2$. If $r\in\{r_1,r_2\}$, this is again 
clear. Otherwise,
$$
  \frac{r-z_D}{\phi(r)-z_{\phi(D)}} \equiv 
  \frac{r-r_1}{\phi(r)-z_{\phi(r_1)}} \equiv 1 \mod \m,
$$
because $r-z_D=(r-r_1)+\frac{r_1-r_2}2$, and the second term has higher valuation than the first
(and similarly for $\phi(r)-z_{\phi(D)}$).
It follows from Definition \ref{de:epsilon} and \ref{de:characters} that all 
$\alpha$, $\beta$, $\gamma$, $\epsilon$ are the same for the corresponding discs
(using the same argument as above for $\beta$ in the case of non-root centres),
and by Proposition \ref{components of min} and Theorem \ref{th:GeneralGaloisAction},
the special fibres are the same, with the same Galois action.
%
%

(3) By \eqref{tatedec}, the Tate module $V_l \Jac(C)$, and hence $\H(C)$, is determined as a Galois module by the special fibre
of the minimal regular model over $F$ together with the $G_K$-action \eqref{CbarnAction}. These are the same 
for $C=C_1$ and $C=C_2$ by (2).

(4) Since the cluster pictures are the same, so are the minimal discriminants by 
Lemma~\ref{comparedisc} and Theorem \ref{thmmindisc}.
\end{proof}

\begin{corollary}\label{co:continuity}
Suppose $C_1:y^2 = c_1f_1(x)$ and $C_2:y^2 = c_2f_2(x)$ are two hyperelliptic curves with $c_1, c_2 \in K^{\times}$ and $f_1(x), f_2(x) \in \cO_K[x]$ monic polynomials.
If $\frac{c_1}{c_2} \in K^{\times 2}$ and $f_1(x) \equiv f_2(x) \mod \pi^{d+1}$ where $d$ is the largest depth among the depths of all proper clusters of $C_1$, then 
\begin{itemize}
\item $\H(C_1) \iso \H(C_2)$ as $G_K$-modules for every $l\ne p$, $C_1$ and $C_2$ have the same conductor exponent and the same root number. 
\item If $C_1$ is semistable then so is $C_2$. In this case, the special fibres of their minimal regular 
models over $\cO_{K^{nr}}$ are isomorphic as curves with an action of Frobenius,
their Jacobians have the same Tamagawa number, $C_2$ is deficient if and only if $C_1$ is and, 
if $|k|>\deg f_1(x)$, the valuations of their minimal discriminants are equal.
\end{itemize}
\end{corollary}

\begin{proof}
By hypothesis, the condition (a) of Theorem \ref{continuity} holds. 
Also, as $f_1(x)$ and $f_2(x)$ are monic and congruent mod $\pi$, they have the same degree.  

Let $F$ be the splitting field of $f_1(x)$ and $\alpha_1, ..., \alpha_n \in F$ its roots.  
Note that $\alpha_i \mod \pi^{d+1}$ is a root of $f_2(x) \mod \pi^{d+1}$ for all $i$. 
By definition of depth, these are all distinct so that by Hensel's Lemma, 
the roots $\beta_1,...\beta_n$ of $f_2(x)$ can be ordered so that $\alpha_i \equiv \beta_i \mod \pi^{d+1}$. 
Now if $\sigma(\alpha_i) = \alpha_j$ then $\sigma(\beta_i) \equiv \beta_j \mod \pi^{d+1}$ and 
hence $\sigma(\beta_i) = \beta_j$. Finally by choice of $d$, we have 
$\beta_i-\beta_j\equiv \alpha_i-\alpha_j \not \equiv 0 \mod \pi^{d+1}$, 
so that $\frac{\beta_i-\beta_j}{\alpha_i-\alpha_j }\equiv 1 \mod \m$. The result follows from 
Theorem \ref{continuity} and Remark \ref{re:sameinvariants}.
\end{proof}


\def\sectionname{Appendix}

\def\prehyp{\text{(possibly singular) hyperelliptic curve}}
\def\polx{\sigma(x)}
\def\poly{\sigma(y)}
\def\polt{\sigma(t)}
\def\polv{\sigma(v)}

\def\thesection{A}

\section{Hyperelliptic curves} \label{ap:apphyp}


Let $K$ be a field with $\textup{char}(K)\neq 2$. By a \prehyp  ~\[C:Y^2 = f(X),\] where $f(X)\in K[X]$ is of degree $2g$ or $2g+1$, has leading coefficient $c_f$ and has at worst double roots, we mean the projective curve given by glueing the pair of affine patches $$U_X:Y^2 = f(X) \phantom{hi}\text{ and }\quad U_T:V^2 = T^{2g+2}f(\frac{1}{T})$$ along $X = \frac{1}{T}$ and $Y = \frac{V}{T^{g+1}}$. 
By the \textit{points at infinity} on $C$ we mean the points of  $C\setminus U_X$, i.e. the points with $T=0$ on $U_T$. If $\textup{deg}(f)=2g+1$ there is a unique such, \[P^{\infty}=(0,0),\] whilst if $\textup{deg}(f)=2g+2$  then \[P^{\infty}_{\pm\sqrt{c_f}}=(0,\pm \sqrt{c_f})\] are the two points on $U_T$ with $T=0$. Note that the points at infinity are always nonsingular. The singular points of $C$, all of which are nodes, are precisely those of the form  $(r,0)$ on $U_X$ where $r$ is a double root of $f(X)$.

Write $f(X)=c_fg(X)h(X)^2$ with $g(X), h(X)$ monic and square free. Then the normalization of $C$ is the hyperelliptic curve 
\begin{equation} \label{normalisation equation}
\widetilde{C}:Y^2 = c_fg(X)
\end{equation}
 and the canonical morphism $\widetilde{C}\rightarrow C$ is given (on the chart $U_X$) by \begin{equation} \label{normalisation equation2}
 (x,y) 
\mapsto (x, yh(x)).
\end{equation}
The points on $\widetilde{C}$ above a node $(r,0)$ are 
\begin{equation} \label{tangents app}
N_r^{\pm \sqrt{c_fg(r)}} =(r, \pm\sqrt{c_fg(r)}).
\end{equation} 

Given a morphism $\phi:C_1\rightarrow C_2$ of hyperelliptic curves, we denote by $\tilde{\phi}$ the unique morphism $ \widetilde{C_1}\rightarrow \widetilde{C_2}$ making the diagram 
$$
\begin{tikzcd}[sep=1.2em]
  \widetilde{C_1}\dar\rar& \widetilde{C_2} 
   \dar\\[0em]
 C_1\rar&C_2 
\end{tikzcd}
$$
commute. 

\begin{remark}
We allow the case where every root of $f(X)$ is a double root, in which case $C$ is not geometrically connected. The discussion above and lemma below, however, remain valid as stated.
\end{remark}

\begin{lemma}  \label[lemma]{le:apphyp}
Suppose $K$ has characteristic $p>2$ and let $\sigma$ be a positive integer power of the Frobenius map on $\Kbar$ sending $x$ to $x^p$. 
Let $C_1:Y^2=f_1(X)$ and $C_2:Y^2=f_2(X)$  be two \prehyp s. Denote the affine charts for $C_1$ (resp. $C_2$) by $U_{X,1}$ and $U_{T,1}$ (resp. $U_{X,2}$ and $U_{T,2}$). Suppose $\phi:C_1\rightarrow C_2$ is a morphism given as a map $U_{X,1}\rightarrow U_{X,2}$ on $\Kbar$-points by 
\[(x,y)\mapsto (\alpha \polx+\beta,\gamma \poly), \qquad \text{$\alpha,\gamma \in \Kbar^{\times}$ and $\beta \in \Kbar$.}\]
 \\ 
$(i)$ As a (rational) map  $U_{T,1}\rightarrow U_{T,2}$, $\phi$ is given on $\Kbar$-points by the formula
$$
(t,v) \mapsto \left(\frac{\polt}{\alpha+\polt\beta},\frac{\gamma \sigma(v)}{(\alpha+\sigma(t)\beta)^{g+1}}\right).
$$
In particular, if $\textup{deg}(f_1)$ is odd then $P^\infty\mapsto P^\infty$ whilst if $\textup{deg}(f_1)$ is even then $P_{\infty}^{\sqrt{c_{f_1}}} \mapsto P_{\infty}^{\epsilon\sqrt{c_{f_2}}}$ where $\epsilon = \frac{\gamma \sigma(\sqrt{c_{f_1}})}{\alpha^{g+1}\sqrt{c_{f_2}}}\in \{\pm 1\}$.\\
$(ii)$ For $i=1,2$, write $f_i(X)=h_i(X)^2g_i(X)$ with $h_i,g_i$ monic and squarefree. The morphism $\tilde{\phi}:\widetilde{C_1}\rightarrow \widetilde{C_2}$ is given explicitly on $\Kbar$-points by the formula
\[(x,y)\mapsto \left(\alpha\sigma(x)+\beta,\gamma \alpha^{-\textup{deg}(h_1)}\sigma(y)\right).\]
In particular we have
\[\tilde{\phi}\left(N_r^{\pm \sqrt{c_{f_1}g_1(r)}}\right)=N_{r'}^{\pm \kappa_r\sqrt{c_{f_2}g_2(r')}}\]
where $r'=\alpha \sigma(r)+\beta$ and $\kappa_r=\gamma \alpha^{-\textup{deg}(h_1)}\frac{\sigma(\sqrt{c_{f_1}g_1(r)})}{\sqrt{c_{f_2}g_2(r')}}\in \{\pm 1\}.$
\end{lemma}

\begin{proof}
 (i). Let $(t,v)\in U_{T,1}$ with $t\neq 0$. This corresponds to the point $(1/t,v/t^{g+1})\in U_{X,1}$ which under $\phi$ is mapped to the point 
\[\left(\frac{\alpha+\beta \sigma(t)}{\sigma(t)},\frac{\gamma \sigma(v)}{\sigma(t)^{g+1}}\right)\in U_{x,2}.\]
Changing variables we see that this corresponds to the point
\[\left(\frac{\polt}{\alpha+\polt\beta},\frac{\gamma \sigma(v)}{(\alpha+\sigma(t)\beta)^{g+1}}\right)\in U_{T,2}.\]
Since this formula describes a rational map which is defined at $t=0$ it gives the desired expression for the morphism on $U_{T,1}$, as well as the claim about the points at infinity.

(iii). We first claim that
\begin{equation} \label{identity}
f_2(X)=\gamma^2f_1^\sigma\left(\frac{X-\beta}{\alpha}\right),\end{equation}
where $f_1^\sigma(X)$ is the result of applying $\sigma$ to the coefficients of $f_1(X)$. In particular 
\[c_{f_2}=(\gamma/\alpha^{g+1})^2\sigma(c_{f_1}).\] 
Indeed, since $\phi$ is a morphism, for all $(x,y)\in C_1(\Kbar)$ we must have
\[\gamma^2\sigma(y^2)=f_2(\alpha\sigma(x)+\beta),\]
or equivalently
\[\gamma^2f_1^\sigma(\sigma(x))=f_2(\alpha\sigma(x)+\beta).\]
Since both the $x$-coordinate map $U_{x,1}(\Kbar)\rightarrow \Kbar$ and $\sigma:\Kbar\rightarrow \Kbar$ are surjective we deduce that 
\[\gamma^2f_1^\sigma(x)=f_2(\alpha x+\beta)\]
holds for all $x\in \Kbar$ and is thus a polynomial identity, from which \cref{identity} follows.

That the claimed formula for $\tilde{\phi}$ gives a morphism $\tilde{C}_1\rightarrow \tilde{C}_2$ making the diagram commute, we use that from \Cref{identity} one has
\[h_2(X)=\alpha^{\textup{deg}(h_1)}h_1^\sigma\left(\frac{X-\beta}{\alpha}\right).\]
\end{proof}

\def\thesection{B}

\section{Centres of clusters}
\label{s:app:centres}




\begin{lemma}\label[lemma]{lem:invcentre}
Let $f(x)\in K[x]$ be a squarefree polynomial with set of roots~$\mathcal{R}$. 
Let $\s$ be a proper cluster, $G_\s = \Stab_{G_K}(\s)$ and $K_\s = \overline{K}^{G_\s}$.
If there is a root $z_0\in\s$ such that $K_\s(z_0)/K_\s$ is tame, then there is a centre for $\s$ which lies in $K_\s$.
In particular, if $K(\mathcal{R})/K$ is tamely ramified,
then for every proper cluster $\s$ there is a centre which lies in $K_\s$.
\end{lemma}

\begin{proof}
%
By assumption $z_0$ lies in a tame extension of $K_{\c}$ and hence in $K_{\c}^{nr}(\sqrt[m]{\pi_\c})$ for some $p\nmid m$ and uniformiser $\pi_\c$ of $K_\c$ (we fix here a choice of $\sqrt[m]{\pi_\c}$). Write the $p$-adic expansion of $z_0$ as
$$
 z_0 = a_t \sqrt[m]{\pi_\c}^t + a_{t+1} \sqrt[m]{\pi_\c}^{t+1} + \ldots
$$
for a suitable $t\in\Z$ and $a_t\in K_\c^{nr}$ roots of unity of order prime to $p$. 

For $\sigma\in G_\c$ we have $\sigma(z_0)\equiv z_0 \mod \pi_K^{d_\c}$. In other words the terms in the $p$-adic expansions of $z_0$ and $\sigma z_0$ agree up to $\sqrt[m]{\pi_\c}^{e_{K_\c/K} m d_\c}$.
Define
$$
 z = \sum_{t<e_{K_\c/K} m d_\c} a_t \sqrt[m]{\pi_\c}^t.
$$
Clearly $z$ is a centre for $\c$ and it suffices to check that it is $G_\c$-invariant.
Suppose not, and that $a_{u}\sqrt[m]{\pi_\c}^u$ is the lowest valuation term the expression which is not $G_\c$-invariant. If $m\nmid u$ then there is some element $\sigma$ of tame inertia of $K_\c$ which fixes $a_u\in K_\c^{nr}$ and maps $\sqrt[m]{\pi_\c}^u$ to $\zeta\sqrt[m]{\pi_\c}^u$ with a root of unity $\zeta\neq 1$; this contradicts the fact that $\sigma z_0\equiv z_0 \mod \sqrt[m]{\pi_\c}^{e_{K_\c/K} m d_\c}$. If $m|u$ then $\sqrt[m]{\pi_s}^u\in K_\c$, so we must have $a_u\not\in K_\c$; but in this case the Frobenius element $\phi$ similarly scales $a_u\sqrt[m]{\pi_\c}^u$ by a non-trivial root of unity of order prime to $p$, which contradicts
$\phi z_0\equiv z_0 \mod \sqrt[m]{\pi_\c}^{e_{K_\c/K} m d_\c}$.
\end{proof}

\def\thesection{C}

\section{Equivalent semistability conditions}
\label{s:app:equi}

Throughout this appendix $C/K:y^2 = f(x)$ is a hyperelliptic curve with 
$$
f(x) = c_f\prod_{r \in \cR}(x-r).
$$
We give two equivalent formulations to the semistability criterion (Propositions \ref{weakss} and \ref{th:equivss}). In view of Theorem \ref{the semistability theorem} these provide equivalent conditions for $C/K$ to be semistable. 
\begin{definition}[= Definition \ref{semistability criterion}]\label[definition]{semistability criterion appendix}
We say that $C/K$ satisfies the \emph{semistability criterion} if
the following conditions hold:
\begin{enumerate}
\item
The extension $K(\cR)/K$ has ramification degree at most 2.
\item
Every proper cluster is $I_K$-invariant. 
\item
Every principal cluster $\s$ has $d_\c \in\Z$ and 
$\nu_\s \in 2\Z.$
\end{enumerate}
\end{definition}

\begin{lemma}\label{le:InertiaOrbit}
Suppose $K(\mathcal{R})/K$ is tamely ramified and $\sigma r\ne r$ 
for some $\sigma\in I_K$ and $r\in\cR$. 
Then $v(r-\sigma r)\notin\Z$, and $|\text{Orbit}_{I_K}(r)|\,v(r-\sigma r)\in\Z$.
\end{lemma}

\begin{proof}
Write
$$
  r = a_1 \pi^{b_1} + a_2 \pi^{b_2} \ldots
$$
with rational $b_1<b_2<...$ and 
$a_i\in K^{nr}$ roots of unity of order prime to~$p$.
The expansion of $\sigma r$ differs from that of $r$ only at those $b_i$ 
that are not in $\Z$, hence $v(r-\sigma r)\notin\Z$.
Also, the size of the orbit $|\text{Orbit}_{I_K}(r)|$ is 
the lowest common multiple of the denominators of the $b_i$ (when written in lowest terms), and so 
$|\text{Orbit}_{I_K}(r)|\,v(r-\sigma r)\in\Z$.
\end{proof}

\begin{lemma}\label{le:TamelyRam} 
Suppose $K(\mathcal{R})/K$ is tamely ramified and $d_\s\in\Z$ for every principal cluster
$\s\subset\cR$ and for $\cR$ itself when $\cR=\s_1\coprod s_2$ is a union of two clusters. Then 

(1) $e_{K(\cR)/K}\le 2$, 

(2) all proper clusters are inertia invariant, 

(3) a root $r$ is fixed by inertia unless $r\in \t$ for a twin $\t$ or $\cR = \mathfrak{c} \cup \{r,r'\}$ is a cotwin with $\mathfrak{c}$ its principal child. 
\end{lemma}

\begin{proof}
Let $r$ and $\sigma r$ be two inertia conjugate roots. By Lemma \ref{le:InertiaOrbit}, $v(r-\sigma r) \notin \Z$, so the depth of $\s =\{r\} \wedge \{ \sigma r\}$ is not an integer. Note that $\s$ cannot be a cotwin of odd size, since its singleton root (by construction $r$ or $\sigma r$) cannot be Galois conjugate to a root in its principal child.
Thus $\s$ is either a twin or $\s  = \cR $ is a cotwin of the form $\mathfrak{c} \cup \{r,\sigma r\}$ where $\mathfrak{c}$ is its principal child. It follows that inertia can only swap roots inside twins or the two singletons inside a cotwin. The lemma follows.
\end{proof}

\begin{proposition}\label[proposition]{weakss}
Let $C/K$ be a hyperelliptic curve. Then $C/K$ satisfies the semistability criterion if and only if
\begin{enumerate}
\item
The extension $K(\cR)/K$ is tamely ramified.
\item
Every principal cluster $\s$ is $I_K$-invariant, has $d_\c \in\Z$ and 
$\nu_\s \in 2\Z$. 
\end{enumerate}
\end{proposition}

\begin{proof}
Clearly if $C/K$ satisfies the semistability criterion then (1) and (2) hold.
For the converse, by Lemma \ref{le:TamelyRam} it suffices to show that $d_\cR\in\Z$ if $\cR$ is a union of two clusters.
Suppose $\cR=\s_1 \coprod \s_2$ is a union of two clusters. At least one of the $\s_i$ is principal, so, by hypothesis, they cannot be permuted by $I_K$. By Lemma \ref{lem:invcentre}, $\s_1$ and $\s_2$ have centres $z_{\s_1}, z_{\s_2}\in K$ (taking $z_{\s}=r$ if $\s=\{r\}$ is a singleton), and hence $d_{\cR}=v(z_{\s_1}-z_{\s_2})\in\Z$.
\end{proof}

\begin{lemma}\label[lemma]{change of nu cluster}
For any cluster $\s$, $\nu_{\c} = \nu_{P(\s)} + |\s|\delta_{\s}.$
\end{lemma}
\begin{proof}
By definition of $\nu$, 
$$
\nu_{\c}= v(c_f)+|\c|d_{\c}+ \sum_{r \notin \c} v(z_{\c}-r) = v(c_f)+|\c|d_{\c}+\sum_{r \notin P(\s)} v(z_{\c}-r) + \sum_{r \in P(\s) \setminus \c} v(z_{\c}-r)
$$
$$
 =v(c_f) +  |P(\s)|d_{P(\s)}+ \sum_{r \notin P(\s)} v(z_{P(\s)}-r) + |\c|d_{\c} - |P(\s)|d_{P(\s)} + (|P(\s)|-|\s|)d_{P(\c)} 
$$
$$
= \nu_{P(\c)} +|\s|\delta_{\c}.
$$
\end{proof}


\begin{lemma}\label[lemma]{le:AdamsLemma} 
The following are equivalent:

(1) There exists a principal cluster $\s$ with  $d_\s \in \Z$ and $\nu_\s \in 2\Z$ and 
for all other principal clusters $\s', \s''\ne \cR$, 
\begin{itemize}
\item[a)] 
$\delta_{\s'} \in \Z$ if $\s'$ is even and $P(\s')$ is principal, 
\item[b)] 
$\delta_{\s'} \in 2\Z$ if $\s'$ is odd and $P(\s')$ is principal,
\item[c)] 
 $\delta_{\s'} - \delta_{\s^{''}} \in 2\Z$ if $\cR = \s' \coprod \s^{''}$ and $\s', \s^{''}$ odd.
\item[d)] 
$\delta_{\s'} - \delta_{\s^{''}} \in \Z$ if $\cR = \s' \coprod \s^{''}$ and $\s', \s^{''}$ even.
\end{itemize}

(2) all principal clusters $\s$ have $d_\s \in \Z$ and $\nu_s \in 2\Z$.
\end{lemma}

\begin{proof}
For all proper clusters $\s$ we have $d_s=d_{P(\s)}+\delta_\s$ by definition of $\delta_\s$, and $\nu_{\c} = \nu_{P(\s)} + |\s|\delta_{\s}$, by \Cref{change of nu cluster}. The result follows from a simple case-by-case check, and the fact that going to parent and child clusters allows one to move from any principal cluster to any other one, moving only through principal clusters and possibly through $\cR$ when it is a union of two odd or two even clusters.
\end{proof}


\begin{proposition}\label[proposition]{th:equivss}
Let $C/K$ be a hyperelliptic curve and let $\s$ be a principal cluster. 
Then $C/K$ satisfies the semistability criterion if and only if 
\begin{itemize}
\item[1)] There exists a principal cluster $\s$ with $d_\s \in \Z$ and $\nu_\s \in 2\Z$, 
\item[2)] for all proper clusters $\s', \s''\ne \cR$, 
\begin{itemize}
\item[a)] 
$\delta_{\s'} \in \Z$ if $|\s'|>2$ is even and $P(\s')$ is not a cotwin, 
\item[b)] 
$\delta_{\s'} \in 2\Z$ if $|\s'|$ is odd and $P(\s')$ is principal,
\item[c)] 
$\delta_{\s'} \in \frac12\Z$ if $|\s'|=2$,
\item[d)]
 $\delta_{\s'} \in  \frac12\Z$ if $|\s'|= 2g$ and $P(\s')$ is a cotwin,
\item[e)] 
$\delta_{\s'}, \delta_{\s''} \in \Z$ and $\delta_{\s'} + \delta_{\s''} \in 2\Z$ if $\cR = \s' \coprod \s''$ and $\s', \s''$ odd,
\item[f)] 
$\delta_{\s'}, \delta_{\s''} \in \Z$ if $\cR = \s' \coprod \s''$ and $\s', \s''$ even principal,
\end{itemize}
\item[3)] wild inertia acts trivially on the roots. 
\end{itemize}
\end{proposition}
\begin{proof}
Suppose $C/K$ satisfies (1)--(3). By (1) and (2) the curve satisfies the hypotheses of Lemma \ref{le:AdamsLemma}(1), and hence all principal clusters $\s'$ have $d_{\s'}\in\Z$ and $\nu_{\s'}\in2\Z$.
By (2e), (2f), if $\cR$ is a union of two clusters then $d_\cR\in\Z$, so by Lemma \ref{le:TamelyRam} all proper clusters are inertia invariant and the ramification degree of $K(\cR)/K$ is at most 2.

Conversely, suppose that $C/K$ satisfies the semistability criterion. Then (1) and (3) trivially hold.
If $\cR=\s_1 \coprod \s_2$ is a union of two clusters, then at least one of the $\s_i$ is principal, so they cannot be permuted by $I_K$. By Lemma \ref{lem:invcentre}, $\s_1$ and $\s_2$ have centres $z_{\s_1}, z_{\s_2}\in K$ (taking $z_{\s}=r$ if $\s=\{r\}$ is a singleton), and hence $d_{\cR}=v(z_{\s_1}-z_{\s_2})\in\Z$. Thus (2a), (2b), (2e) and (2f) hold because $d_\s\in\Z$ and $\nu_\s\in 2\Z$ for every principal cluster $\s$. Finally, (2c) and (2d) follow from Lemma \ref{le:InertiaOrbit}.
\end{proof}

\def\fr{{\mathfrak r}}

\def\thesection{D}
 \def\sh{\hat{\c}}

\section{Metric cluster pictures, hyperelliptic graphs and BY trees} \label[appendix]{hyble appendix}

Here we summarise various definitions and constructions from \cite{hyble}. Specifically, we recall the combinatorial notion of  \textit{metric cluster picture} and the process for associating  a metric BY tree and metric hyperelliptic graph to each such. The relevance  to this paper is that, for the metric cluster picture associated to a semistable hyperelliptic curve over a local field $K$ of odd residue characteristic (\Cref{explicit poly example}), the resulting hyperelliptic graph is precisely the dual graph of (the special fibre of) its minimal regular model over $K^{\textup{nr}}$ (\Cref{main dual graph thingy}). 

We caution that  our notation differs slightly from that of \cite{hyble}. Where there are differences we indicate this immediately after the relevant defintion and note in particular that, for a metric cluster picture $\Sigma$, we write $T_\Sigma$ (resp. $G_\Sigma$) for the graph denoted in op. cit. as $\widetilde{\underline{T}(\Sigma)}$ (resp. $\widetilde{\underline{G}(\underline{T}(\Sigma))}$).

We adopt the same definitions and conventions for metric graphs as in \Cref{sss:dualgraph}. In particular we allow graphs to have loops and multiple edges and  automorphisms of a metric graph $G$  are homotopy classes of homeomorphisms $G\rightarrow G$ preserving vertices and lengths (which may permute multiple edges and reverse direction of loops). 



\subsection{Cluster pictures}

\begin{definition}[Cluster picture] \label[definition]{defclpic}

Let $X$ be a finite set and $\Sigma$ a collection of non-empty subsets of $X$;
elements of $\Sigma$ are \emph{clusters}. 
Then $\Sigma$ (or $(X,\Sigma)$) is a \emph{cluster picture} if
\begin{enumerate}
\item
every singleton is a cluster, and $X$ is a cluster,
\item
any two clusters are either disjoint or one is contained in the other.
\end{enumerate}

A cluster picture $(X,\Sigma)$ is \textit{metric} if it is equipped with a
distance $\delta(\s,\fr)=\delta(\fr,\s)\in\R_{>0}$ for every pair of proper clusters $\s<\fr$. This extends to a distance function between all pairs of proper clusters in the natural way; see \cite[Definition 3.45]{hyble}. 
\end{definition}

We talk about properties of clusters using the notation and terminology set out in \Cref{clusternotation} (see also \cite[Section 3]{hyble}), and do not recall these terms here in the interest of space. In particular $(X,\Sigma)$ has \textit{genus} $g$ if $|X|\in \{2g+1,2g+2\}$. In this appendix we restrict to cluster pictures of  genus at least $2$.
%
%

\begin{example}\label[example]{clreal}
Let $C/K:y^2=f(x)$ be a hyperelliptic curve and denote by $\mathcal{R}\subseteq \bar{K}$ the set of roots of $f(x)$.
Then the non-empty subsets of $\mathcal{R}$ cut out by discs form a metric cluster picture, where for proper clusters $\s<\s'$  we set $\delta(\s,\s')=d_\s-d_{\s'}$ where   $d_\s = \min_{r,r'\in\s}\{v(r-r')\}$ is the \emph{depth} of the cluster $\s$.
\end{example}

\begin{example} \label[example]{explicit poly example}
As a concrete example of the above, take $K=\mathbb{Q}_p$ for $p$ odd and consider the monic polynomial $f(x)$ with set of roots $\cR=\{1,1\!+\!p^2,1\!-\!p^2,p,0,p^3,-p^3\}$, so that the resulting hyperelliptic curve $C:y^2=f(x)$ is the one considered in \Cref{introexample}. There are four proper clusters:
$$
 \cR,\quad
 \s_1=\{1,1\!+\!p^2,1\!-\!p^2\}, \quad
 \s_2=\{p,0,p^3,-p^3\}, \quad
  \s_3=\{0,p^3,-p^3\}, \quad 
 $$
of depths 0, 2, 1 and  3 respectively. 
We represent this pictorially by drawing roots $r \in\cR$ as
\smash{\raise4pt\hbox{\clusterpicture\Root{1}{first}{r1};\endclusterpicture}},
and drawing ovals around roots to represent a cluster:
$$
\scalebox{1.6}{\clusterpicture            
  \Root {1} {first} {r1};
  \Root {} {r1} {r2};
  \Root {} {r2} {r3};
  \Root {3} {r3} {r4};
  \Root {1} {r4} {r5};
  \Root {} {r5} {r6};
  \Root {} {r6} {r7};
  \ClusterLDName c1[][2][\s_1] = (r1)(r2)(r3);
  \ClusterLDName c2[][2][\s_3] = (r5)(r6)(r7);
  \ClusterLDName c3[][1][\s_2] = (r4)(c2)(c2n);
  \ClusterLDName c4[][0][\cR] = (c1)(c2)(c3)(c1n)(c2n)(c3n);
\endclusterpicture}
$$
the roots ordered as they appear in the definition of $\cR$. The subscript of the top cluster $\cR$ is its depth and for all other clusters it is their ``relative depth'': the difference between their depth and that of their parent cluster. 
\end{example}

\begin{definition}[Automorphisms of cluster pictures] \label[definition]{auts of clusters defi}
\label{autc}
An \emph{automorphism} of $\Sigma$ is  a pair $\sigma=(\sigma_0,\epsilon_\sigma)$ where $\sigma_0$ is a permutation of the proper clusters  preserving sizes, inclusions and, in the metric case, distances, and $\epsilon_\sigma$ is a collection of signs $\epsilon_\sigma(\s)\in \{\pm 1\}$ for even clusters $\s\in\Sigma$
such that $\epsilon_\sigma(\s')=\epsilon_\sigma(\s)$ whenever $\s$ is \"ubereven and $\s'<\s.$ 

  We compose automorphisms by the cocycle rule
$$
  (\alpha, \epsilon_\alpha)\circ (\beta, \epsilon_\beta) = 
    \bigl(\alpha\circ\beta, \s\mapsto\epsilon_\beta(\s)\epsilon_\alpha(\beta(\s))\bigr).
$$
\end{definition}

\begin{remark}
Let $\mathcal{E}$ denote the set of even clusters which do not have an \ub~ parent, excluding $\mathcal{R}$ unless $\mathcal{R}$ is itself \ub. 
Then to give a collection of signs $\epsilon_\sigma(\s)$ as in \Cref{auts of clusters defi} is equivalent to specifying  $\epsilon_\sigma(\s)$ for $\s\in \mathcal{E}$, with no additional compatibility. 
\end{remark}

\subsection{The BY tree associated to a cluster picture}

Let $\Sigma$ be a metric cluster picture. We associate to $\Sigma$ a finite tree $T_\Sigma$, equipped with a genus marking $g: V(T_\Sigma)\to\Z_{\ge 0}$ on vertices and  a 2-colouring blue/yellow on vertices and edges as follows.  

\begin{definition}[$T_\Sigma$] \label[definition]{BY tree const} 
\label{StoT}~
Let $(X,\Sigma)$ be a metric cluster picture. We define $T_\Sigma$, the \textit{BY tree associated to }$\Sigma$, as follows.
First take the graph with:
\begin{itemize}
\item 
a vertex $v_\s$ for every proper cluster $\s$, excluding $\s=X$ when $|X|=2g+2$ and has a child of size $2g+1$, coloured yellow
if $\s$ is \ub\ and blue otherwise,
\item 
an edge linking $v_{\s}$ to $v_{P(\s)}$ for every proper cluster $\s \neq X$, yellow of length $2\delta_\s$ if $\s$ is even, and blue of length $\delta_\s$ if $\s$ is odd.
\end{itemize}
To obtain $T_\Sigma$  from this graph we remove, if $|X|=2g+2$ and $X$ is a disjoint union of two proper children, the (degree $2$) vertex $v_{X}$ from the vertex set\footnote{we will freely still refer to $v_X$ in this case, understanding that it is simply a point on $T_\Sigma$ rather than a vertex.} (keeping the underlying topological space the same). 
We define the \emph{genus} of a vertex $v_\s$ as $g(v_\c)=g(\s)$. 
\end{definition}

Writing $T=T_\Sigma$, as a topological space $T=T_b\coprod T_y$ with
$T_b$ the \emph{blue part}, and $T_y$ the \emph{yellow part}.
Note that all leaves are blue and that  $T_b\subset T$ is closed.

\begin{remark}
$T_\Sigma$ is a (metric) BY tree in the sense of \cite[Definition 3.18]{hyble}; in the notation of op. cit. (see Construction 4.13, Proposition 5.7) it is precisely the graph $\widetilde{\underline{T}(\Sigma)}$.
\end{remark}
 
\begin{example}
Consider the cluster picture associated to the polynomial of \Cref{explicit poly example}. The associated metric BY tree is

$$
\scalebox{1.5}{
\begin{tikzpicture}[scale=\GraphScale]      
  \BlueVertices
  \Vertex[x=4.50,y=0.000,L=1]{1};
  \Vertex[x=0.000,y=0.000,L=1]{2};
  \Vertex[x=3.00,y=0.000,L=\relax]{3};
  \Vertex[x=1.50,y=0.000,L=\relax]{4};
  \BlueEdges
  \Edge(1)(3)
  \TreeEdgeSignN(1)(3){0.5}{2}
  \Edge(2)(4)
  \TreeEdgeSignN(2)(4){0.5}{2}
  \YellowEdges
  \Edge(3)(4)
  \TreeEdgeSignN(3)(4){0.5}{2}
  \EdgeSign(1)(1){0.5}(0,0.5){v_{\s_3}}
  \EdgeSign(2)(2){0.5}(0,0.5){v_{\s_1}}
  \EdgeSign(3)(3){0.5}(0,0.37){v_{\s_2}}
  \EdgeSign(4)(4){0.5}(0,0.37){v_{\cR}}
\end{tikzpicture}}
$$
where the yellow edge is  squiggly for the benefit of viewing  in black and white, the number above an edge is its length, and the number on a vertex its genus. 
\end{example}
 
\subsection{The hyperelliptic graph associated to a metric cluster picture}

Let $\Sigma=(X,\Sigma)$ be a cluster picture. We associate to $\Sigma$ a metric graph $G_\Sigma$, equipped with a \emph{genus} marking $g: V(G_\Sigma)\to \mathbb{Z}_{\ge 0}$ and an involution\footnote{graph isomorphism of order $\leq$ 2.} $\iota$ as follows. 

\begin{definition}[$G_\Sigma$]\label[definition]{TtoG}
Let $\Sigma$ be a metric cluster picture and $T=T_\Sigma$  the associated metric BY tree.  Define $G_\Sigma$, the \emph{hyperelliptic graph associated to }$\Sigma$, to be the topological space (complete with metric) given by glueing two disjoint copies $T^+$ and $T^-$ of $T$ along their common blue parts. Thus $G=G_\Sigma$ comes with a natural map $\pi:G\rightarrow T$ making it into a double cover of $T$ ramified along $T_b$, as well as an involution $\iota$ swapping $T^+$ and $T^-$.  We make $G_T$ into a (metric) graph by, for each $v_\s\in V(T)$ not a genus 0 leaf (equiv. $\s$ principal, see \cite[Lemma 5.20]{hyble}), declaring each element $x$ of $\pi^{-1}(v_\s)$  to be a vertex of genus $g(\s)$. We denote this vertex of $G$ by $v_\s$ if $x$ is the unique element of $\pi^{-1}(v_\s)$, and otherwise denote it $v_\s^+$ (resp. $v_\s^-$) if $x\in T^+$ (resp. $x\in T^-$).
Finally, we adjust the metric by halving the lengths of all edges.

We write $G_b$ for  those points in $G$ fixed by $\iota$ and $G_y$ for $G\setminus G_b$. Further, write $G_y^+$ for the points in $G_y$ which come from $T^+$ and $G_y^-$ for the points coming from $T^-$. 
\end{definition}

\begin{remark}
The graph $G_\Sigma$ is a \textit{hyperelliptic graph} in the sense of \cite[Definition 3.2]{hyble}, so that in particular  all vertices of genus $0$ necessarily have degree at least $3$. Specifically it is the hyperelliptic graph $\widetilde{\underline{G}(\underline{T}(\Sigma))}$ (see op. cit. Construction 4.8, Lemma 5.5). 
\end{remark}

\begin{remark} \label[remark]{explicit graph remark}
The graph $G_\Sigma$ may be described somewhat more concretely as follows.

For every non-\ub\ principal cluster there is a vertex $v_\c$, and for each  \ub\ principal cluster $\c$ the are two vertices $v_{\c}^+$ and $v_{\c}^-$. These are linked by edges as follows (where we write  $v_\c= v_{\c}^+=v_{\c}^-$ whenever $\c$ is not \ub):

\noindent\hskip+3mm
\begin{tabular}{|c|c|c|c|l|}
\hline
Name & From & To & Length & Conditions \\
\hline
$L_{\c'}$ &$v_{\c'}$&$v_{\c}$ &$\frac 12 \delta_{\s'}$ & $\c'<\c$ both principal, $\c'$ odd\\
\hline
$L_{\c'}^+$ &$v_{\c'}^+$&$v_{\c}^+$ &$\delta_{\s'}$ & $\c'<\c$, both principal, $\c'$ even \\
\hline
$L_{\c'}^-$ &$v_{\c'}^-$&$v_{\c}^-$ &$\delta_{\s'}$ & $\c'<\c$, both principal, $\c'$ even \\
\hline
$L_{\t}$ &$v_{\c}^-$&$v_{\c}^+$ &$2 \delta_\t$& $\c$ principal, $\t<\c$ twin\\
\hline
$L_{\t}$ &$v_{\c}^-$&$v_{\c}^+$ &$2 \delta_\s$& $\c$ principal, $\c<\t$ cotwin\\\hline
\end{tabular}\\

\noindent and if $\cR$ is not principal additionally:\\

\noindent\hskip+3mm
\begin{tabular}{|c|c|c|c|l|}
\hline
$L_{\c_1,\c_2}$ &$v_{\c_1}$&$v_{\c_2}$ &$\frac 12 (\delta_{\c_1}+\delta_{\c_2})$& $ \cR = \c_1\coprod\c_2$, with $\c_1, \c_2$ principal odd\\
\hline
$L_{\c_1,\c_2}^+$ &$v_{\c_1}^+$&$v_{\c_2}^+$ &$\delta_{\c_1}+\delta_{\c_2}$& $ \cR = \c_1\coprod\c_2$, with $\c_1, \c_2$ principal even\\
\hline
$L_{\c_1,\c_2}^-$ &$v_{\c_1}^-$&$v_{\c_2}^-$ &$\delta_{\c_1}+\delta_{\c_2}$& $ \cR = \c_1\coprod\c_2$, with $\c_1, \c_2$ principal even\\
\hline
$L_{\t}$ &$v_{\c}^-$&$v_{\c}^+$ &$2(\delta_\s+\delta_\t)$&  $\cR = \c \coprod \t$, with $\c$ principal even, $\t$ twin\\
\hline
\end{tabular}\\
\end{remark}

\begin{example}

The hyperelliptic graph associated to the cluster picture of \Cref{explicit poly example} is
$$
\scalebox{1.5}{
\begin{tikzpicture}[scale=\GraphScale]     
  \GraphVertices
  \Vertex[x=4.50,y=0.000,L=1]{1};
  \Vertex[x=0.000,y=0.000,L=1]{2};
  \Vertex[x=3.00,y=0.000,L=\relax]{3};
  \Vertex[x=1.50,y=0.000,L=\relax]{4};
  \BlueEdges
  \Edge(1)(3)
  \TreeEdgeSignN(1)(3){0.5}{1}
  \Edge(2)(4)
  \TreeEdgeSignN(2)(4){0.5}{1}
  \BendedEdges
  \Edge(3)(4)
  \Edge(4)(3)
  \GraphEdgeSignN(3)(4){0.5}{1}
  \GraphEdgeSignS(4)(3){0.5}{1}
  \BlueEdges
  \EdgeSign(1)(1){0.5}(0,0.5){v_{\s_3}}
  \EdgeSign(2)(2){0.5}(0,0.5){v_{\s_1}}
  \EdgeSign(3)(3){0.5}(0,0.37){v_{\s_2}}
  \EdgeSign(4)(4){0.5}(0,0.37){v_{\cR}}
\end{tikzpicture}
}$$
where the number above an edge indicates its length, and the number on a vertex its genus. In particular, by \Cref{main dual graph thingy}, for $p$ an odd prime this is the dual graph of the hyperelliptic curve 
\[C/\mathbb{Q}_p:y^2=x(x-1)(x-(1+p^2))(x-(1-p^2))(x-p^3)(x+p^3).\]
\end{example}

\subsection{Automorphisms of $T_\Sigma$ and $G_\Sigma$}

Let $\Sigma=(X,\Sigma)$ be a metric cluster picture. We now explain how to produce an automorphism of $T_\Sigma$ (resp. $G_\Sigma$) from an automorphism of $\Sigma$.

\subsubsection{Automorphisms of $T_\Sigma$}
By an \textit{automorphism} of $T=T_\Sigma$ we mean a pair  $(\sigma_0, \epsilon_\sigma)$ where 
\begin{itemize}
\item 
$\sigma_0$ is a graph automorphism of $T$ that preserves genera, colour and distances,
\item
$\epsilon_\sigma(Z)\in\{\pm 1\}$ is a collection of signs for every connected component $Z$ of 
the yellow part $T_y\subset T$. 
\end{itemize}
As for cluster pictures we compose automorphisms by the cocycle rule
$$
  (\alpha, \epsilon_\alpha)\circ (\beta, \epsilon_\beta) = 
    \bigl(\alpha\circ\beta, \bullet\mapsto\epsilon_\beta(\bullet)\epsilon_\alpha(\beta(\bullet))\bigr).
 $$
 (This is precisely the notion of automorphism for BY trees used in \cite[Definition 3.27]{hyble}.)
 
 \begin{definition}[$T(\sigma)$]
 Let $\sigma=(\sigma_0,\epsilon_{\sigma})$ be  an automorphism of $\Sigma$. Define the automorphism $T(\sigma)=(T(\sigma)_0,\epsilon_{T(\sigma)})$ of $T$ as follows. For a vertex $v_\s$ of $T$, set $T(\sigma)_0(v_\s)=v_{\sigma_0(\s)}$. To define $\epsilon_{T(\sigma)}$ for a yellow component $Z$ of $T_y$, pick (as is always possible) an even cluster $\s$ such that the edge between $v_\s$ and $v_{P(\s)}$ (half-edge if $P(\s)=X$ and $v_X$ is removed from the vertex set in the construction of $T$) lies in $Z$. Set $\epsilon_{T(\sigma)}(Z)=\epsilon_{\sigma}(\s)$. The compatibility of signs on even clusters ensures this is well defined. 
 \end{definition}
 
 \begin{remark} \label[remark]{earlier remark}
 The automorphism $T(\sigma)$ of $T_\Sigma$ is precisely the result of restricting the automorphism $\underline{T}(\sigma)$ of the open BY tree $\underline{T}(\Sigma)$, as defined in \cite[Construction 4.13]{hyble}, to its core $\widetilde{T(\Sigma)}$. In particular (see \cite[Proposition 4.14]{hyble}), the association $\sigma \mapsto T(\sigma)$ is a homomorphisms, and every automorphism of $T_\Sigma$ fixing $v_X$ (or $v_\s$ is $X$ has size $2g+2$ and a child $\s$ of size $2g+1$) arises this way. 
 \end{remark}
 
 \subsubsection{Automorphisms of $G_\Sigma$}

 By an \textit{automorphism} of $G_\Sigma$ we mean a graph automorphism preserving the genus marking. 
 
\begin{definition}[$G(\sigma)$] \label[definition]{the automorphism of g}
Let $\sigma=(\sigma_0,\epsilon_\sigma)$ be an automorphism of $\Sigma$, and $T(\sigma)=(T(\sigma)_0,\epsilon_{T(\sigma)})$ the associated automorphism of $T=T_\Sigma$. Denote by $\pi:G\rightarrow T$ the quotient map and for a connected component $Z$ of $G_y$,  denote by $\bar{Z}$ the component $\pi(Z)$ of $T_y$. We define $G(\sigma)$ to be the unique automorphism of $G$ such that:
\begin{itemize}
\item
$G(\sigma)$ commutes with $\iota$ and induces the graph automorphism $T(\sigma)_0$ (temporarily denoted $\rho$) on the quotient $T$,
\item for a connected component $Z$ of $G_y^+$, we have 
\[G(\sigma)(Z)=\begin{cases} \pi^{-1}(\rho(\bar{Z}))\cap G_y^+~~&~~\epsilon_{T(\sigma)}(\bar{Z})=1,\\\pi^{-1}(\rho(\bar{Z}))\cap G_y^-~~&~~\epsilon_{T(\sigma)}(Z)=-1.\end{cases}\]
\end{itemize} 
\end{definition}

\begin{remark} \label{explicit aut action}
Explicitly, for a non \ub~ principal cluster $\s$ we have $G(\sigma)(v_\s)=v_{\sigma \s}$. Similarly, for an \ub~ principal cluster $\s$ we have 
\[G(\sigma)(v_\s^+)=\begin{cases}v_{\sigma \s}^+~~&~~\epsilon_\sigma(\s)=1\\ v_{\sigma \s}^-~~&~~\epsilon_\sigma(\s)=-1.\end{cases}\]
For the edges, for a proper cluster $\s$ of size $<2g+1$, write $e_\s\in G_b$ (resp. $e_\s^+\in G_y^+$ and $e_\s^-\in G_y^-$) for the edge(s) between $\pi^{-1}(v_\s)$ and $\pi^{-1}(v_{P(\s)})$ (if $\s=\t$ is  a twin, then by $e_\t^{\pm}$ we mean the two half-edges which get glued at $\pi^{-1}(v_\t)$ to form a loop). Then for $\s$ odd we have $G(\sigma)(e_\s)=e_{\sigma \s}$, whilst for $\s$ even we have
\[G(\sigma)(e_\s^+)=\begin{cases} e_{\sigma \s}^+ ~~&~~\epsilon_\sigma(\s)=1\\e_{\sigma \s}^-~~&~~\epsilon_\sigma(\s)=-1.\end{cases}\]
\end{remark}

\begin{remark} \label{all things arise how we want}
The automorphism $G(\sigma)$ of $G_\Sigma$ is precisely the result of restricting the automorphism $\underline{G}(\sigma)$ of the hyperelliptic graph $\underline{G}(\underline{T}(\Sigma))$, as defined in \cite[Construction 4.8]{hyble}, to its core $\widetilde{\underline{G}(\underline{T}(\Sigma))}$. In particular (see \cite[Proposition 4.11]{hyble} and \Cref{earlier remark}), the association $\sigma \mapsto G(\sigma)$ is a homomorphism, and every automorphism of $G_\Sigma$ fixing $\pi^{-1}(v_X)$ as a set (or $\pi^{-1}(v_\s)$ if $X$ has size $2g+2$ and a child $\s$ of size $2g+1$) arises this way. 
\end{remark}

\subsection{The homology of $G_\Sigma$}
In \cite[Section 6]{hyble} an explicit  description of the first singular homology group $H_1(G_\Sigma,\mathbb{Z})$, along with its length pairing (\cite[Section 2.2.2]{hyble}) and automorphism action is given. Here we recall the result.

\begin{theorem} \label{th:apphomology}
Let $\Sigma$ be a metric cluster picture,  $A$ the set of even non-\ub\ clusters excluding $X$, and $B$ the subset of clusters $\c \in A$ such that $\c^*=X$.  Then there is a canonical isomorphism 
\[H_1(G_\Sigma,\mathbb{Z})\cong\left \{\sum_{\s \in A}\lambda_\s\ell_\s\in \mathbb{Z}^A~~\mid ~~\sum_{\s \in B}\lambda_\s=0 \right \}.\]
The length pairing  is given by
$$
\langle \ell_{\s_1},\ell_{\s_2} \rangle=\begin{cases}

        0~~&  \mbox{ } ~~\quad\c^*_1 \neq \c^*_2, \\
        2(\delta({\c_1\wedge\c_2},{P(\c^*_1)})& \mbox{ } \quad \c^*_1 =\c^*_2 \ne X, \\
        2(\delta({\c_1\wedge\c_2},{X}))& \mbox{} \quad\quad ~~\c^*_1 =\c^*_2 = X.\\
    \end{cases}
$$ 
For an automorphism $\sigma=(\sigma_0,\epsilon_\sigma)$ of $\Sigma$, the action of $G(\sigma)$ on $H_1(G_\Sigma,\mathbb{Z})$ is given by
$$\sigma(\ell_\s) = \epsilon_{\s}(\sigma)\ell_{\sigma\s}.$$
\end{theorem}

\begin{proof}
This is \cite[Theorem 6.1]{hyble} (see op. cit. Definitions 3.16, 3.31 and 3.48 for the definitions of the lattices $\Lambda_{\bullet}$ appearing in the statement). We remark that, writing $T=T_\Sigma$, the proof passes through a canonical identification (op. cit. Proposition 6.6) of $H_1(G_\Sigma,\mathbb{Z})$ with the relative homology group $H_1(T,T_b,\mathbb{Z})$, equivariant for the natural actions of automorphisms and preserving the respective length pairings.
\end{proof}

\begin{remark} \label{rm:apphomology}
Unwinding the isomorphism in \cite[Theorem 6.1]{hyble} yields the following explict description of the basis elements $\ell_\s$:   
for $\c \ne \cR$ an even non-\"ubereven cluster   $\ell_{\c}\in C_1(G_\Sigma,\Z)$
 is the shortest path going from $v_{\sh}^-$ to $v_\s$ through $G_\Sigma^-$ before going on to  $v_{\sh}^+$ through $G_\Sigma^+$, where here we set $\hat\c=P(\neck{\c})$  if  $\neck{\c} \ne \cR$, and $\hat\c=\cR$  otherwise.
 (In exceptional cases, for a cluster $\s$ appearing above, $v_\s$ (resp. $v_{\sh}^+$ , $v_{\sh}^-$) may not be in the vertex set of $G_\Sigma$, and we must interpret it as the obvious point on an edge instead.)
Note that $\ell_{\c}$ is a loop  in $G_\Sigma$ unless $\c^* = \cR$, in which case it is a ``half loop'' in the sense that if $\ell_{\c}$, $\ell_{\c'}$ are two such then $ \ell_{\c} - \ell_{\c'}$ is a loop.
\end{remark}

\end{document}